\documentclass{amsart}
  
\usepackage{amssymb,amsfonts,mathtools,thmtools}
\usepackage[margin=1.5in]{geometry}
\usepackage{color}
\usepackage[shortlabels]{enumitem}
\usepackage{mathrsfs}
\usepackage{eucal}
\usepackage{graphicx}
\usepackage{float}
\usepackage{url} 
\usepackage{bbm}
\usepackage[dvipsnames]{xcolor}
\definecolor{allrefcolors}{rgb}{.05,.45,.6}
\usepackage[pagebackref,linktocpage=true,colorlinks=true,allcolors=allrefcolors,bookmarksopen,bookmarksdepth=3]{hyperref}
\usepackage[nameinlink,noabbrev,capitalise]{cleveref}
\usepackage{tikz,tikz-cd}
\usepackage{enumitem}
\setenumerate{label=(\roman*),topsep=1pt,itemsep=1pt,partopsep=0pt,parsep=0pt}

\newtheorem{thm}{Theorem}[section]
\newtheorem{prop}[thm]{Proposition}
\newtheorem{lem}[thm]{Lemma}
\newtheorem{cor}[thm]{Corollary}

\theoremstyle{definition}
\newtheorem{defn}[thm]{Definition}

\newtheorem{exmp}[thm]{Example}
\newtheorem{notn}[thm]{Notation}
\newtheorem{asmpt}[thm]{Assumption}

\theoremstyle{remark}
\newtheorem{rem}[thm]{Remark}

\crefname{thm}{Theorem}{Theorems}
\crefname{prop}{Proposition}{Propositions}
\crefname{lem}{Lemma}{Lemmas}
\crefname{cor}{Corollary}{Corollaries}
\crefname{quest}{Question}{Questions}
\crefname{clm}{Claim}{Claims}

\crefname{defn}{Definition}{Definitions}
\crefname{con}{Construction}{Constructions}
\crefname{exmp}{Example}{Examples}
\crefname{notn}{Notation}{Notations}
\crefname{asmpt}{Assumption}{Assumptions}

\crefname{rem}{Remark}{Remarks}
\crefname{warn}{Warning}{Warnings}

\crefname{section}{Section}{Sections}
\crefname{subsection}{Subsection}{Subsections}

\crefname{sec}{Section}{Sections}
\crefname{subsec}{Subsection}{Subsections}
\crefname{eqn}{Equation}{Equations}
\crefname{part}{Part}{Parts}

\numberwithin{equation}{section}

\newcommand{\IC}{\mathbb{C}}
\newcommand{\ID}{\mathbb{D}}

\newcommand{\IL}{\mathbb{L}}

\newcommand{\IR}{\mathbb{R}}
\newcommand{\IS}{\mathbb{S}}

\newcommand{\IZ}{\mathbb{Z}}

\newcommand{\sA}{\mathcal{A}}
\newcommand{\sB}{\mathcal{B}}
\newcommand{\sC}{\mathcal{C}}
\newcommand{\sD}{\mathcal{D}}
\newcommand{\sE}{\mathcal{E}}
\newcommand{\sF}{\mathcal{F}}
\newcommand{\sG}{\mathcal{G}}
\newcommand{\sH}{\mathcal{H}}
\newcommand{\sI}{\mathcal{I}}
\newcommand{\sJ}{\mathcal{J}}
\newcommand{\sK}{\mathcal{K}}
\newcommand{\sKJ}{\mathcal{KJ}}
\newcommand{\sKV}{\mathcal{KV}}

\newcommand{\sM}{\mathcal{M}}
\newcommand{\sN}{\mathcal{N}}
\newcommand{\sO}{\mathcal{O}}
\newcommand{\sP}{\mathcal{P}}

\newcommand{\sR}{\mathcal{R}}

\newcommand{\sU}{\mathcal{U}}
\newcommand{\sV}{\mathcal{V}}
\newcommand{\sW}{\mathcal{W}}
\newcommand{\sX}{\mathcal{X}}
\newcommand{\sY}{\mathcal{Y}}

\newcommand{\fm}{\mathfrak{m}}

\newcommand{\fn}{\mathfrak{n}}
\newcommand{\fo}{\mathfrak{o}}

\renewcommand{\mod}[1]{#1{\operatorname{\mathbf{-mod}}}}
\newcommand{\homod}[1]{#1{\operatorname{\mathbf{-homod}}}}

\renewcommand{\line}[1]{#1{\operatorname{\mathbf{-line}}}}
\newcommand{\hooklongrightarrow}{\lhook\joinrel\longrightarrow}
\newcommand{\hooklongleftarrow}{\longleftarrow\joinrel\rhook}
\newcommand{\BO}{\mathrm{BO}}
\newcommand{\OO}{\mathrm{O}}
\newcommand{\osr}{\overline{\mathcal{R}}}
\newcommand{\Wu}[2]{\overline W^u_{#2}(#1)}

\newcommand{\tw}{\mathrm{tw}}
\newcommand{\OC}{\mathcal{OC}}
\newcommand{\ang}[1]{\langle #1 \rangle}
\newcommand{\un}{\underline}

\DeclareMathOperator{\Core}{Core}
\DeclareMathOperator*{\colim}{colim}
\DeclareMathOperator{\Ho}{Ho}
\DeclareMathOperator*{\hocolim}{hocolim}

\DeclareMathOperator{\maps}{Map}
\DeclareMathOperator{\Nbhd}{Nbhd}
\newcommand{\id}{\mathbbm{1}}
\DeclareMathOperator{\im}{im}
\DeclareMathOperator{\Ob}{Ob}
\DeclareMathOperator{\ev}{ev}
\DeclareMathOperator{\crit}{crit}
\DeclareMathOperator{\codim}{codim}
\DeclareMathOperator{\ind}{ind}

\DeclareMathOperator{\Spec}{\mathbf{Spec}}

\DeclareMathOperator{\BU}{BU}

\DeclareMathOperator{\BPin}{BPin}

\DeclareMathOperator{\BGL}{BGL}
\DeclareMathOperator{\EGL}{EGL}

\DeclareMathOperator{\lag}{U/O}

\DeclareMathOperator{\B}{B}
\DeclareMathOperator{\GL}{GL}
\DeclareMathOperator{\Gr}{Gr}
\DeclareMathOperator{\Hw}{Hw}
\DeclareMathOperator{\Ends}{Ends}
\DeclareMathOperator{\End}{End}

\DeclareMathOperator*{\holim}{holim}
\DeclareMathOperator{\hofib}{hofib}
\DeclareMathOperator{\hocofib}{hocofib}

\begin{document}

\title{Spectral equivalence of nearby Lagrangians}

\author{Johan Asplund}
\address{Department of Mathematics, Stony Brook University, 100 Nicolls Road, Stony Brook, NY 11794, USA}
\email{johan.asplund@stonybrook.edu}
\author{Yash Deshmukh}
\address{School of Mathematics, Institute for Advanced Study, 1 Einstein Drive, Princeton, NJ 08540, USA}
\email{deshmukh@ias.edu}
\author{Alex Pieloch}
\address{Department of Mathematics, Massachusetts Institute of Technology, 77 Massachusetts Avenue, Cambridge, MA 02139, USA}
\email{pieloch@mit.edu}

\begin{abstract}
	Let $R$ be a commutative ring spectrum. We construct the wrapped Donaldson--Fukaya category with coefficients in $R$ of any stably polarized Liouville sector. We show that any two $R$-orientable and isomorphic objects admit $R$-orientations so that their $R$-fundamental classes coincide.  Our main result is that any closed exact Lagrangian $R$-brane in the cotangent bundle of a closed manifold is isomorphic to an $R$-brane structure on the zero section in the wrapped Donaldson--Fukaya category, generalizing a well-known result over the integers.  To achieve this, we prove that the Floer homotopy type of the cotangent fiber is given by the stable homotopy type of the based loop space of the zero section.
\end{abstract}
\maketitle
\tableofcontents

\section{Introduction}\label{sec:intro}
\subsection{Statement of results}

Let $R$ be a commutative ring spectrum\footnote{We use ring spectra and commutative ring spectra to mean $A_\infty$ and $E_\infty$ algebras in spectra. See \cref{sec:background_spectra} for our conventions.} and let $X$ be a stably polarized Liouville sector.\footnote{That is, after stabilizing (if necessary) the tangent bundle $TX$ is equipped with a global Lagrangian distribution.}  Given a Lagrangian $L$ in $X$, the polarization of $X$ gives rise to a Lagrangian Gauss map $\sG_L \colon L \to \lag$.

\begin{defn}[{\cref{defn:Brane}}]\label{defn:intro_brane}
A \emph{Lagrangian $R$-brane} in $X$ is a graded exact conical Lagrangian $L \subset X$ equipped with a choice of lift of its Lagrangian Gauss map $\sG_L \colon L \to \lag$ to a map $\sG_L^\# \colon L \to (\lag)^\#$, where $(\lag)^\#$ denotes the the homotopy fiber of the following composition:
\[ \lag \overset{\beta^{-1}}{\longrightarrow} {\B}^2{\OO} \overset{{\B}^2{J}}{\longrightarrow} {\B}^2{\GL_1}(\IS) \longrightarrow {\B}^2{\GL_1}(R), \]
where $\GL_1(R)$ and $\GL_1(\IS)$ denote the spaces of units associated to the ring spectra $R$ and $\IS$ (see \cref{defn:space_of_units}), $\beta^{-1}$ is given by (the inverse of) the Bott periodicity map followed by the projection, and $J$ is the $J$-homomorphism.
\end{defn}

The first goal of this paper is to construct the Donaldson--Fukaya category with Lagrangian $R$-branes as objects, with coefficients in a commutative ring spectrum. To each pair of Lagrangian $R$-branes $L_0$ and $L_1$ we associate an $R$-oriented flow category $\sC\sW(L_0,L_1;R)$.  The Cohen--Jones--Segal geometric realization of this flow category is an of $R$-module spectrum, which we denote by $HW(L_0,L_1;R)$.  As with Floer theory over discrete rings, there is a homotopy associative product
\[ \mu^2 \colon HW(L_0,L_1;R) \wedge_R HW(L_1,L_2;R) \longrightarrow HW(L_0,L_2;R). \]

\begin{thm}[\cref{dfn:wfuk_K}]\label{thm:intro_main}
Suppose that $X$ is a stably polarized Liouville sector. There is a category $\sW(X;R)$ that is enriched over the homotopy category of $R$-modules with objects given by Lagrangian $R$-branes and morphisms given by the homotopy classes of $R$-modules $HW(-,-;R)$, and the composition is induced from the product $\mu^2$. 
\end{thm}
\begin{rem}
    As stated, $\sW(X;R)$ is a category that is enriched over the homotopy category of $R$-modules.  One can apply the stable homotopy groups functor $\pi_\bullet(-)$ to each of the morphism spectra to obtain a category that is enriched over the category of $\pi_\bullet(R)$-modules. In particular, letting $R$ be the Eilenberg--MacLane spectrum of the integers recovers the classical wrapped Floer cohomology with integer coefficients:
    \[ \pi_\bullet(HW(L_0,L_1;H\IZ)) \cong HW^{-\bullet}(L_0,L_1;\IZ).\]
\end{rem}
\begin{rem}\label{rem:comparison_ps_large}
    We compare our category $\sW(X;R)$ with other versions that has recently appeared in the literature.
    
    In the case $R=\IS$, the set of objects of $\sW(X;\IS)$ is larger than in the category considered by Porcelli--Smith \cite{porcelli2024bordism} and Large \cite{large2021spectral}; they only consider stably framed Liouville manifolds, and each object come equipped with a null-homotopy of the Lagrangian Gauss map that is compatible with the background stable framing.
    
    On the other hand, in \cite{porcelli2024spectral} Porcelli--Smith construct a Donaldson--Fukaya category $\sF(X;(\varTheta,\varPhi))$ associated to a graded tangential pair $\varTheta \to \varPhi$ lying over $\BO \xrightarrow{c} \BU$. The category defined in \cite{porcelli2024bordism} corresponds to $\varTheta = \varPhi = \{\mathrm{pt}\}$. In their language, our category $\sW(X;R)$ can conjecturally be obtained from theirs in the case $(\varTheta,\varPhi) = ((\lag)^{\#},\BO)$ via a certain base change (and taking duals), see \cref{rem:relation_ps} for details.
\end{rem}

The second goal of this paper is to study this category in the context of cotangent bundles and nearby Lagrangians.  To this end, let $Q$ be a closed manifold and recall that $T^*Q$ is a Liouville manifold; we equip it with the tautological polarization.\footnote{The tautological polarization is the one corresponding to the global Lagrangian distribution given by the tangent spaces of the cotangent fibers.}  A \emph{nearby Lagrangian} in $T^*Q$ is a closed exact Lagrangian submanifold $L \subset T^*Q$.  It is well-known that every relatively pin nearby Lagrangian in $T^*Q$ is Floer theoretically equivalent in the usual wrapped Fukaya category with integer coefficients to the zero section equipped with a particular choice of $\IZ$-brane structure \cite{fukaya2008symplectic,abouzaid2012nearby}.  Our main theorem is the following generalization.

\begin{thm}[{\cref{thm:modules_main}}]\label{thm:EquivToZero}
Let $R$ be a commutative ring spectrum. Let $L \subset T^* Q$ be a nearby Lagrangian that is equipped with an $R$-brane structure.  Then $L$ is isomorphic in $\sW(T^*Q;R)$ to the zero section $Q \subset T^*Q$ equipped with a choice of $R$-brane structure.
\end{thm}

\begin{rem}
   For a nearby Lagrangian $L$, the composition ${\B^2}{J} \circ \beta^{-1} \circ \sG_L$ in \cref{defn:intro_brane} is null-homotopic (see \cite[Corollary 1.3]{jin2020microlocal} or \cite[Theorem B]{abouzaid2020twisted}).  Hence every nearby Lagrangian admits an $\IS$-brane structure and consequently an $R$-brane structure. Therefore, in the context of cotangent bundles and nearby Lagrangians, one can always work with coefficients in $\IS$. Defining the Donaldson--Fukaya category for a general commutative ring spectrum is nevertheless useful as will be shown in the sequel to this paper \cite{MSpin}, in which we work with coefficients in the Thom spectrum $\mathrm{MO}\langle k\rangle$ of the $(k-1)$-connected cover of the orthogonal group.
\end{rem}
\begin{rem}
    Our proof of \cref{thm:EquivToZero} relies on relating the Floer homotopy theory of a cotangent fiber with the stable homotopy theory of the based loop space $\varOmega Q$ and using Whitehead-type theorems to bootstrap off of the analogous result known over discrete rings, see \cref{sec:sketch_proof} for more details.
\end{rem}
\begin{rem}
    During the preparation of the current version of this paper, Porcelli--Smith \cite[Theorem 1.2]{porcelli2025bordism} proved a generalization of \cref{thm:EquivToZero} using obstruction-theoretic methods. Our proof avoids the use of obstruction theory using geometric arguments, but is restricted to cotangent bundles.
\end{rem}

\subsection{Properties of the category}
    \begin{thm}[{\cref{thm:inclusion_pushforward}}]
        An inclusion of stably polarized Liouville sectors $X\hookrightarrow X'$ induces a pushforward functor $\sW(X;R) \to \sW(X';R)$.
    \end{thm}
    \begin{thm}[{\cref{thm:subcrit_trivial}}]
        If $X$ is a stably polarized subcritical Weinstein sector, then for any $L_0,L_1 \in \Ob(\sW(X;R))$ we have that $HW(L_0,L_1;R)$ is the zero $R$-module.
    \end{thm}
    \begin{thm}
        Let $L$ and $K$ be two closed Lagrangian $R$-branes which are isomorphic in $\sW(X;R)$. 
        \begin{enumerate}
            \item (\cref{lem:PSS}) There is an equivalence of $R$-modules $L^{-TL} \wedge R \simeq K^{-TK} \wedge R$.
            \item (\cref{thm:R-ori_lags_represent_same_class}) If $L$ is $R$-oriented, then $K$ admits an $R$-orientation so that the $R$-fundamental classes of $L$ and $K$, determined by these $R$-orientations, coincide in $H_n(X;R)$.
        \end{enumerate}
    \end{thm}
    
\subsection{Sketch of the proof of the main theorem}\label{sec:sketch_proof}
The idea behind the proof of \cref{thm:EquivToZero} is similar to the well-known proof of the analogous result over $\IZ$ \cite{fukaya2008symplectic,abouzaid2012nearby}. In this section, we assume for simplicity that $R$ is connective.  The general case is deduced from the connective case by using connective covers, see \cref{rem:reduce_to_connective}. Let $F$ denote a cotangent fiber.

\begin{thm}[\cref{prop:evinv}]\label{thm:intro_hw_loop}
    There is an equivalence of $R$-module spectra
    \[
    HW(F,F;R) \simeq \varSigma^\infty_+ \varOmega Q \wedge R.
    \]
    Moreover, this equivalence intertwines the product $\mu^2$ and the Pontryagin product up to homotopy.
\end{thm}
 The composition $\mu^2$ in the Donaldson--Fukaya category a priori only gives a homotopy $R$-algebra structure on $HW(F,F;R)$. \cref{thm:intro_hw_loop} allows us to equip $HW(F,F;R)$ with a (highly structured) $R$-algebra structure such that there is an equivalence of $R$-algebra spectra $HW(F,F;R) \simeq \varSigma^\infty_+ \varOmega Q \wedge R$ (\cref{cor:hw_loops_alg}).
Moreover, in \cref{sec:str_lift}, we show that the Yoneda functor
\begin{align*}
    \sY_F \colon \sF(T^*Q;R) &\longrightarrow \homod{HW(F,F;R)} \cong \homod{(\varSigma^\infty_+ \varOmega Q \wedge R)}\\
    L &\longmapsto HW(F,L;R).
\end{align*}
admits a lift to a functor
\begin{align*}
	\sY_{\varOmega Q} \colon \sF(T^*Q;R) &\longrightarrow \Ho \mod{(\varSigma^\infty_+ \varOmega Q \wedge R)}\\
	L &\longmapsto HW(\varOmega Q,L;R).
\end{align*}
(see \cref{sec:str_lift} for notation). The Lagrangian Gauss map of the zero section admits a canonical null-homotopy, determining a distinguished $R$-brane structure on $Q$. The $R$-brane structure $Q^\#$ thus corresponds to the choice of an $R$-line bundle on $Q$; namely, a homotopy class of maps $Q_+ \to \BGL_1(R)$. Looping such a map determines a ring map $(\varOmega Q)_+ \to \GL_1(R)$.  We show that the $(\varSigma^\infty_+ \varOmega Q \wedge R)$-module $HW(\varOmega Q,Q^\#;R)$ coincides with the module structure on the rank one free $R$-module $HW(F,Q^\#;R)$ induced by this ring map (see \cref{lem:TechnicalHelp}).

\begin{lem}\label{lem:ModuleActionRealized}
Every $(\varSigma^\infty_+ \varOmega Q \wedge R)$-module structure on a rank one free $R$-module
is, up to isomorphism, of the form $HW(\varOmega Q,Q^\#;R)$ for some $R$-brane structure $Q^\#$ on the zero section $Q$.
\end{lem}

Let $\sF(T^*Q;R) \subset \sW(T^*Q;R)$ denote the full subcategory of closed Lagrangian $R$-branes. It follows that in the category $\Ho \mod{(\varSigma^\infty_+ \varOmega Q \wedge R)}$, every $(\varSigma^\infty_+ \varOmega Q \wedge R)$-module that is isomorphic to $R$ as an $R$-module is obtained from the image of the Yoneda functor
\begin{align*}
	\sY_{\varOmega Q} \colon \sF(T^*Q;R) &\longrightarrow \Ho \mod{(\varSigma^\infty_+ \varOmega Q \wedge R)}\\
	L &\longmapsto HW(\varOmega Q,L;R).
\end{align*}
Consequently, we reduce the proof of \cref{thm:EquivToZero} to proving the following lemma:  
\begin{lem}[{\cref{lma:rank1} and \cref{lma:whitehead_cats}}]\label{lem:TechnicalHelp}
Let $L$ and $K$ be two nearby Lagrangian $R$-branes.
\begin{enumerate}
    \item There is an equivalence of $R$-modules $HW(\varOmega Q,L;R) \simeq \varSigma^\ell R$ for some $\ell \in \IZ$.
    \item The functor $\sY_{\varOmega Q}$ induces an equivalence of morphism spaces:
        \[ HW(L,K;R) \overset{\simeq}{\longrightarrow} F_{\varSigma^\infty_+ \varOmega Q \wedge R}(HW(F,L;R),HW(F,K;R)). \]
\end{enumerate}
\end{lem}

\begin{proof}[Proof of \cref{thm:EquivToZero}]
    By \cref{lem:ModuleActionRealized} and \cref{lem:TechnicalHelp}(i), $\sY_{\varOmega Q}(L)$ is isomorphic to $\sY_{\varOmega Q}(Q^\#)$ in $\Ho \mod{(\varSigma^\infty_+ \varOmega Q \wedge R)}$ for some choice of $R$-brane structure $Q^\#$ on $Q$.  By \cref{lem:TechnicalHelp}(ii), $\sY_{\varOmega Q}$ is fully faithful and hence $L$ is isomorphic to $Q^\#$ in $\sW(T^*Q;R)$.
\end{proof}

To prove \cref{lem:TechnicalHelp}, we use Whitehead-type theorems to bootstrap results over discrete rings to $R$. Let $k \coloneqq \pi_0(R)$.  Under the connectivity hypothesis on $R$, there is an $\IS$-algebra map $\Hw \colon R \to Hk$, and if $M$ is an $R$-module, this induces a \emph{Hurewicz map},
\[
    \Hw_M \colon M \overset{\simeq}{\longrightarrow} M \wedge_R R \xrightarrow{\id \wedge_R \Hw} M \wedge_R Hk.
\]
This Hurewicz map leads to $R$-module versions of the classical Whitehead theorem.  In particular:
\begin{enumerate}
    \item (\cref{thm:whitehead}) If $M$ and $M'$ are connective $R$-modules and there is an $R$-module map $\varphi \colon M \to M'$ such that
\[ \varphi_\bullet \colon \pi_i(M \wedge_R Hk) \longrightarrow \pi_i(M' \wedge_R Hk)\]
is an equivalence for all $i \in \IZ$, then $\varphi$ is an equivalence of $R$-modules.
    \item (\cref{cor:whitehead_cor}) If $M$ is an $(n-1)$-connective $R$-module and
\[
	\pi_\bullet(M \wedge_R Hk) = \begin{cases}
		k, & \bullet = n \\
		0, & \text{else}
	\end{cases},
\]
then there is an equivalence of $R$-modules $\varSigma^{n} R \to M$.
\end{enumerate}
The homotopy class of $R$-modules $HW(L_0,L_1;R)$ is related via the Hurewicz map to the usual construction of the wrapped Floer cohomology with coefficients in $k$ (e.g. \cite{ganatra2020covariantly}) by
\[ \pi_\bullet(HW(L_0,L_1;R) \wedge_R Hk) \cong HW^{-\bullet}(L_0,L_1;k). \]

\begin{proof}[Proof of \cref{lem:TechnicalHelp}]
    \begin{enumerate}
        \item It is known that $HW(F,L;k) = k[\ell]$ for some $\ell \in \IZ$ (see \cite[Section 1]{fukaya2008symplectic} or \cite[Appendix C]{abouzaid2012nearby}). Item (ii) above then gives that $HW(F,L;R) \simeq \varSigma^\ell R$.
        \item Since the functor 
        \[ \pi_\bullet(\sY_{\varOmega Q} \wedge_R Hk) \colon \sF(T^\ast Q;k) \longrightarrow \mod{\pi_\bullet(\varSigma^\infty_+ \varOmega Q \wedge_R Hk)} \]
        is fully faithful on nearby Lagrangians that admit $k$-brane structure (using standard Floer theory over discrete commutative rings, see \cite[Theorem 1.13]{ganatra2022sectorial} and \cite[Lemma C.1]{abouzaid2012nearby}), item (i) above yields the fully faithfulness of $\sY_F$.
    \end{enumerate}
\end{proof}

\begin{rem}
    The usage of Whitehead-type theorems to obtain results in Floer homotopy theory from Floer homology is not new.  For example, this is the approach taken by Abouzaid--Kragh in \cite{abouzaidkragh_immersions}.
\end{rem}

\subsection{Outline}

The paper is organized as follows:  The first half is dedicated to establishing the requisite data needed to define the spectral Donaldson--Fukaya category. In \cref{sec:background_flow}, we discuss some background material on flow categories; this involves definitions of flow categories, flow multimodules, and bordisms along with notions of compatible $R$-orientations on each of these structures. We also introduce the notion of flow categories with local systems which is used later in \cref{sec:str_lift}.  In \cref{sec:cjs}, we review the definition of the Cohen--Jones--Segal realization functor and study its functoriality properties with respect to flow multimodules and bordisms.  Finally, we give a discussion of how the above structures naturally arise in the context of Morse theory.  In \cref{sec:AbstractBranes}, we define abstract $R$-branes, $R$-orientations on index bundles, and spaces of abstract strip-caps, all of which are preliminaries for establishing $R$-orientations.  In \cref{sec:Polarizations}, we review the notions of stable polarizations and Lagrangian Gauss maps, and define the notion of an $R$-brane structure. In \cref{sec:ModuliSpaces}, we define the moduli spaces of $J$-holomorphic curves that naturally arise in our constructions. In \cref{sec:canonical_orientations}, we construct canonical $R$-orientations for moduli spaces of $J$-holomorphic curves.

The second half of the paper applies the above sections to define the Donaldson--Fukaya category over $R$ and study nearby Lagrangians.  In \cref{sec:donaldson-fukaya}, we combine the above sections to give the construction of the spectral Donaldson--Fukaya category.  In \cref{sec:loops}, we specialize to cotangent bundles and establish an isomorphism between $\varSigma^\infty_+ \varOmega Q \wedge R$ and the wrapped Floer homotopy type of a cotangent fiber. In \cref{sec:str_lift}, we discuss the construction of structured lifts necessary in the proof of the main theorem. In \cref{sec:modules}, we give the proof of our main theorem \cref{thm:EquivToZero}. 

Finally, we have a string of appendices: In \cref{sec:geom_background}, we give some geometric background on Liouville sectors.  In \cref{sec:background_spectra}, we give a background discussion on $R$-orientations of vector bundles and associated structures. In \cref{sec:background_connective} we discuss connective covers.

\subsection*{Acknowledgments}
We thank Mohammed Abouzaid, Andrew Blumberg, and Oleg Lazarev for helpful conversations. We also thank Noah Porcelli and Ivan Smith for pointing out an incomplete argument in an early draft of this paper, and a gap in the first version.
Part of this work was completed at a Summer Collaborators Program hosted by the Institute for Advanced Study; we are grateful for their hospitality and financial support. JA was supported by the Knut and Alice Wallenberg Foundation and the Swedish Royal Academy of Sciences. YD thanks Max Planck Institute for Mathematics in Bonn and Institute for Advanced Study for their hospitality and financial support. YD was partially supported by NSF grant DMS-1926686. AP was supported by a National Science Foundation Postdoctoral Research Fellowship through NSF grant DMS-2202941.

\section{Background on flow categories}\label{sec:background_flow}
	Flow categories were first introduced by Cohen--Jones--Segal \cite{cohen1995floer} originally used to package Morse-theoretic information. Informally, the objects of the flow category associated to a Morse function are the of critical points of the Morse function and morphisms are moduli spaces of broken flow lines of the Morse function (of all dimensions). 
 
    The purpose of this section is to provide definitions of flow categories, flow multimodules, bordisms, and the notion of $R$-orientations on such. Although the definitions have appeared elsewhere in the literature, the conventions about e.g.\@ gradings and orientations are not uniform. We thus find it convenient to spell out our definitions to set the conventions.
    \subsection{$\ang{k}$-manifolds}
		We recall the definition of a $\ang{k}$-manifold from \cite{jänich1968on,laures2000cobordism}.
		\begin{defn}
			\begin{itemize}
				\item A \emph{manifold with corners} is a topological manifold with boundary which has a smooth atlas locally modeled on open subsets of $\IR_+^n \coloneqq (\IR_{\geq 0})^n$.
				\item The \emph{codimension} of a point $x$ in a manifold with corners $M$ is the number of zeroes in $\phi(x) \in \IR^n_+$ for any chart $\phi$ on $M$ containing $x$.  It is denoted $c(x)$.
				\item A \emph{codimension $k$ connected face} of a manifold with corners $M$ is the closure of a connected component of the subset $\{x \in M \mid c(x) = k\}$. We refer to codimension $1$ connected face as simply a \emph{connected face}.  A \emph{face} of $M$ is any disjoint union of connected faces of $M$.
			\end{itemize}
			\end{defn}
		The definition of a $\ang{k}$-manifold imposes restrictions on how the faces overlap.
		\begin{defn}[$\ang{k}$-manifold]\label{def:k_mfd}
			A \emph{$\ang{k}$-manifold} $M$ is a manifold with corners and a $k$-tuple of faces $(\partial_1 M, \ldots ,\partial_k M)$ of $M$ that satisfy the following conditions:
			\begin{enumerate}
				\item Each $x \in M$ belongs to exactly $c(x)$ of the $\partial_i M$.
				\item $\partial_1 M \cup \cdots \cup \partial_k M = \partial M$.
				\item For $i \neq j$, $\partial_i M \cap \partial_j M$ is a face of both $\partial_i M$ and $\partial_j M$.
			\end{enumerate}
		\end{defn}
		\begin{rem}
            Item (iii) makes sense because each $\partial_i M$ is a manifold with corners. Moreover, it follows from the definition that each face $\partial_i M$ of a $\ang k$-manifold is itself a $\ang{k-1}$-manifold with faces $\partial_i M \cap \partial_j M$ for $i \neq j$.
		\end{rem}
	\subsection{Collars and gluing}
		\begin{defn}[Coherent system of collars]\label{dfn:k-mfd_collars}
			A \emph{coherent system of collars} for a $\ang{k}$-manifold $M$ is a choice of smooth embeddings 
			\[
				\kappa_{F,G} \colon F \times [0,\varepsilon)^{\codim F-\codim G} \hooklongrightarrow G, 
			\]
            for every inclusion of faces $F \subset G$ of $M$,
            such that 
			\begin{enumerate}
				\item $\kappa_{F,G}|_{F \times \left\{0\right\}}$ is the inclusion map onto $F \subset \Nbhd_G(F)$.
				\item The following diagram commutes for any pair of faces $F \subset G$
				\[
					\begin{tikzcd}[row sep=scriptsize, column sep=1cm]
						F \times [0,\varepsilon)^{\codim F} \rar{\kappa_{F,G} \times \id} \drar[swap]{\kappa_{F,M}} & G \times [0,\varepsilon)^{\codim G} \dar{\kappa_{G,M}} \\
					    {} & M
					\end{tikzcd}.
				\]
			\end{enumerate}
		\end{defn}

		\begin{lem}[{\cite[Proposition 2.1.7]{laures2000cobordism}}]
			Any $\ang{k}$-manifold admits a coherent system of collars.
            \qed
		\end{lem}
		\begin{defn}\label{dfn:gluing}
            Let $I$ and $J$ be two finite sets, and let $\iota \colon J \to I^2 \smallsetminus \varDelta$ be a function with components $\iota_1(j)$ and $\iota_2(j)$, where $\varDelta \subset I^2$ denotes the diagonal. Let $\{M_i\}_{i\in I}$ and $\{N_j\}_{j\in J}$ be two collections of $\ang k$-manifolds. Assume that for every $j \in J$ that there are maps
            \begin{equation}\label{eq:abs_gluing_maps}
                M_{\iota_1(j)} \overset{\nu_{j,1}}{\longleftarrow} N_j \overset{\nu_{j,2}}{\longrightarrow} M_{\iota_2(j)},
            \end{equation}
            that satisfies the following:
            \begin{enumerate}
                \item $\nu_{j,1}$ and $\nu_{j,2}$ are embeddings onto codimension $1$ faces of $M_{\iota_1(j)}$ and $M_{\iota_2(j)}$, respectively.
                \item Any codimension $1$ connected face of any $M_i$ is in the image of at most one of the maps $\nu_{j,\ell}$.
            \end{enumerate}
			Define the \emph{gluing of $\{M_i\}_{i\in I}$ along the maps \eqref{eq:abs_gluing_maps}} to be the topological space $X$ that is defined as the colimit of the following diagram
			\[
				\begin{tikzcd}[row sep=scriptsize, column sep=1.5cm]
					\displaystyle \bigsqcup_{j\in J}N_j \rar[shift left]{\bigsqcup_{j\in J}\nu_{j,1}} \rar[shift right,swap]{\bigsqcup_{j\in J}\nu_{j,2}} & \displaystyle\bigsqcup_{i\in I}M_i
				\end{tikzcd}.
			\]
		\end{defn}
		\begin{lem}[{\cite[Lemmas 3.21 and 3.22]{porcelli2024bordism}}]\label{lem:smooth_gluing}
			The topological space $X$ in \cref{dfn:gluing} admits the structure of a $\ang k$-manifold.
            \qed
		\end{lem}
\subsection{Flow categories}\label{subsec:FlowCategories}
       Throughout this section fix a commutative ring spectrum $R$. One may wish to review the discussion of $R$-orientations in \cref{sec:background_spectra}.
  
  \begin{defn}[Flow category]\label{dfn:flow_cat}
			A \emph{flow category} is a (non-unital) topologically enriched category $\sM$ consisting of the following: 
			\begin{enumerate}
				\item A set of objects $\Ob(\sM)$ equipped with a function $\mu \colon \Ob(\sM) \to \IZ$, called the \emph{grading function}.
				\item For each tuple $a,b \in \Ob(\sM)$, a space of morphisms $\sM(a,b)$ which  is a $(\mu(a) - \mu(b) -1)$-dimensional smooth manifold with corners.
				\item For each tuple $a,b,c \in \Ob(\sM)$, a composition map
					\[ \mu_{abc} \colon \sM(a,b) \times \sM(b,c) \longrightarrow \sM(a,c),\]
					that is a diffeomorphism onto a face of $\mathcal M(a,c)$.
                \item For each $a\in \Ob(\sM)$, $\bigsqcup_{b\in \Ob(\sM)} \sM(a,b)$ is compact.
			\end{enumerate}
			We moreover require for each $a,c\in \Ob(\sM)$ that $\sM(a,c)$ is a $\ang{\mu(a) - \mu(c)-1}$-manifold with faces given by
			\[
				\partial_i \sM(a,c) \coloneqq \bigsqcup_{\substack{b \in \Ob(\sM) \\ \mu(a) - \mu (b) = i}} \sM(a,b) \times \sM(b,c).
			\]
		\end{defn}
        \begin{rem}
            For Floer theory outside of the exact setting, condition (iv) in \cref{dfn:flow_cat} is in general not satisfied. One way of adapting the definition to such settings is to demand the existence of a locally constant energy functional 
            \[ \sA \colon \bigsqcup_{p,q \in \Ob(\sM)} \sM(p,q) \longrightarrow \IR,\]
            that is additive with respect to compositions, such that for any $p \in \Ob(\sM)$, its restriction to $\bigsqcup_q \sM(p,q) \longrightarrow \IR$ is proper (cf.\@ \cite{abouzaid2024foundation}). Similar modifications can be made to the definitions of flow bimodules and bordisms in \cref{sec:morph_flow_cats,sec:bordism_of_flow_bimods}, outside of the exact setting.
        \end{rem}
		\begin{defn}\label{dfn:metric_flow_cat}
			A \emph{metric} on a flow category $\sM$ is a choice of a Riemannian metric on each morphism space $\sM(a,b)$ such that the product metric on $\sM(a,b) \times \sM(b,c)$ coincides with the pullback of metric from $\sM(a,c)$ under the composition map $\mu_{abc}$.
		\end{defn}
		A metric on a flow category $\sM$ induces isomorphisms 
		\begin{equation}\label{eq:tancomp}
			\kappa_{abc} \colon T\sM(a,b) \oplus \underline \IR \oplus T\sM(b,c) \overset{\cong}{\longrightarrow} \mu_{abc}^* T\sM(a,c) .
		\end{equation}
		Define the index bundle of a flow category $\sM$ by
		\[ I(a,b) \coloneqq \underline \IR \oplus T\sM(a,b)\]
		for every $a,b \in \Ob(\sM)$. We fix isomorphisms
		\begin{align}\label{eq:indexiso}
			\rho_{abc} \colon I(a,b) \oplus I(b,c) &\overset{\cong}{\longrightarrow}  \mu_{abc}^* I(a,c)\\ \nonumber
            (\alpha \oplus u) \oplus (\beta \oplus v) &\longmapsto \alpha \oplus \kappa_{abc}(u \oplus \beta \oplus v).
		\end{align}
        These isomorphisms are compatible with the composition maps in $\sM$ in the sense that the following diagram of vector bundles over $\sM(a,b) \times \sM(b,c) \times \sM(c,d)$ commutes:
		\begin{equation}\label{eq:indexcom}
			\begin{tikzcd}
				I(a,b) \oplus I(b,c) \oplus I(c,d) \rar{\rho_{abc} \oplus \id} \dar{\id \oplus \rho_{bcd}} & I(a,c) \oplus I(c,d) \dar{\rho_{acd}}\\
				I(a,b) \oplus I(b,d) \rar{\rho_{abd}} & I(a,d)
			\end{tikzcd}.
		\end{equation}

		Let $I_R(a,b)$ denote the associated $R$-line bundle of the vector bundle $I(a,b)$ over $\sM(a,b)$ (see \cref{dfn:assoc_rline}). It follows from \eqref{eq:indexiso} and \cref{rem:assoc_vector_bundles}(ii) that we have an isomorphism
		\begin{equation}\label{eq:index_flow_cat_cmpat}
			\rho_R^{abc} \colon (\mu^\ast_{abc}I(a,c))_R \overset{\simeq}{\longrightarrow} I_R(a,b) \otimes_R I_R(b,c),
		\end{equation}
        of $R$-line bundles over $\sM(a,b) \times \sM(b,c)$.
		\begin{defn}[$R$-orientation on a flow category]\label{defn:rorcoh}
			An \emph{$R$-orientation} on a flow category consists of a rank one free $R$-module $\fo(a)$ (\cref{defn:free_rmod}) for every $a \in \Ob(\sM)$ and an isomorphism of $R$-line bundles over $\sM(a,b)$
			\begin{equation}\label{eq:floworient}
				 \fo(a,b) \colon \fo(a) \overset{\simeq}{\longrightarrow} I_R(a,b) \otimes_R \fo(b),
			\end{equation}
            for each pair $a,b\in \Ob(\sM)$, that are compatible with compositions in the sense that the following diagram of $R$-line bundles over $\sM(a,b) \times \sM(a,c)$ commutes
			\begin{equation}\label{eq:rorcoh}
				\begin{tikzcd}[row sep=scriptsize, column sep=1.5cm]
					\fo(a) \dar{\fo(a,c)} \rar{\fo(a,b)} & I_R(a,b) \otimes_R \fo(b) \dar{\id \otimes_R \fo(b,c)} \\
					I_R(a,c) \otimes_R \fo(c) \rar{\rho^{abc}_R \otimes_R \id} & I_R(a,b) \otimes_R I_R(b,c) \otimes_R \fo(c)
				\end{tikzcd}.
			\end{equation}
            We denote the choice of this data by $\fo$, and call the tuple $(\sM,\fo)$ an \emph{$R$-oriented flow category}.
		\end{defn}

		By \eqref{eq:floworient} and after passing to Thom spectra, there is an induced equivalence of $R$-module spectra (see \cref{rem:thommul})
		\begin{equation}\label{eq:floworientthom}
		    \sM(a,b) \wedge \fo(a) \overset{\simeq}{\longrightarrow} \sM(a,b)^{I_R(a,b)} \wedge_R \fo(b).
		\end{equation}
		By \cref{rem:assoc_vector_bundles}(i) we have
		\[
			\sM(a,b)^{I_R(a,b)} \wedge_R \fo(b) \simeq (\sM(a,b)^{I(a,b)} \wedge R) \wedge_R \fo(b) \simeq \sM(a,b)^{I(a,b)} \wedge \fo(b) .
		\]
		We therefore get a map of $R$-modules
		\begin{equation}\label{eq:flowthom}
			\mathcal M^{-\fo(a,b)} \colon \fo(a) \wedge \sM(a,b)^{-I(a,b)} \longrightarrow \fo(b),
		\end{equation}
		via the composition
        \[
        \fo(a) \longrightarrow \sM(a,b) \wedge \fo(a) \overset{\text{\eqref{eq:floworientthom}}}{\longrightarrow} \sM(a,b)^{I_R(a,b)} \wedge_R \fo(b) \simeq \sM(a,b)^{I(a,b)} \wedge \fo(b)
        \]
		\begin{defn}[Thom spectrum of a flow category]\label{dfn:Thom_spec_flow_cat}
			The \emph{Thom spectrum} of a flow category $\sM$ is the spectrally enriched (non-unital) category $\sM^{-I}$ that is defined to have the same objects as $\sM$ and morphisms given by the Thom spectra
			\[ \sM^{-I} (a,b) \coloneqq \sM(a,b)^{-I(a,b)}.\]
		\end{defn}
		\begin{rem}
            Associativity of the composition in $\sM^{-I}$ follows from \eqref{eq:indexcom}.
        \end{rem}

		\begin{defn}[$R$-module system on $\sM$]\label{def:kmodsys}
			An \emph{$R$-module system} on a flow category $\sM$ is a spectrally enriched functor $\fo \colon \sM^{-I} \to \mod{R}$. 
		\end{defn}
		\begin{lem}\label{lma:Rori_to_spec_system}
			An $R$-orientation on a flow category $\mathcal M$ defines an $R$-module system on $\mathcal M$ by the assignment $a \mapsto \fo(a) $ with morphisms
			\[
				    \mathcal M^{-\fo(a,b)} \colon \fo (a) \wedge \mathcal M^{-I}(a,b) \longrightarrow \fo (b),
			\]
			as in \eqref{eq:flowthom}.
		\end{lem}
		\begin{proof}
			Indeed, \eqref{eq:indexcom} and \eqref{eq:flowthom} yield a commutative diagram:
			\[
				\begin{tikzcd}[row sep=scriptsize, column sep=1.5cm]
					\fo(a) \wedge \sM^{-I}(a,b) \wedge \sM^{-I}(b,c)\rar{\mathcal M^{-\fo(a,b)} \wedge \id} \dar{\id \wedge \sM^{-\rho_{abc}}} & \fo(b) \wedge \sM^{-I}(b,c) \dar{\mathcal M^{-\fo(b,c)}} \\ 
					\fo(a) \wedge \sM^{-I}(a,c) \rar{\mathcal M^{-\fo(a,c)}} & \fo(c) \\
				\end{tikzcd},
			\]
			where $\mathcal M^{-\rho_{abc}}$ is the induced map on Thom spectra from the map $\rho_{abc}$ defined in \eqref{eq:indexiso}. Hence the assignment defines a functor, as claimed.
		\end{proof}

		\subsection{Flow bimodules}\label{sec:morph_flow_cats}
			We now discuss how to define morphisms between flow categories. The right notion in this context is that of a bimodule of two flow categories. 

			\begin{defn}[Flow bimodule]\label{dfn:flow_bimodule}
				Let $\sM_1$ and $\sM_2$ be two flow categories. A \emph{flow bimodule of degree $k$} $\sN \colon \sM_1 \to \sM_2$ is an assignment of a $(\mu_1(a) - \mu_2(b)+k)$-dimensional manifold with corners $\sN(a,b)$ for each tuple $(a,b) \in \Ob(\sM_1) \times \Ob(\sM_2)$, such that for any $a\in \Ob(\sM_1)$, $\bigsqcup_{b\in \Ob(\sM_2)} \sN(a,b)$ is compact. A flow bimodule is furthermore equipped with action maps
				\begin{align*}
					\nu^1_{a a' b} \colon \sM_1(a,a') \times \sN(a', b)  &\longrightarrow \sN(a,b), \quad \forall (a,a',b) \in \Ob(\sM_1)^2 \times \Ob(\sM_2) \\
					\nu^2_{a b' b} \colon \sN(a, b') \times \sM_2(b', b) &\longrightarrow \sN(a,b), \quad \forall (a,b',b) \in \Ob(\sM_1) \times \Ob(\sM_2)^2,
				\end{align*}
				each of which is a diffeomorphism onto a face of $\mathcal N(a,b)$. The action maps are required to satisfy the following:
				\begin{description}
					\item[\textsc{Compatibility}] For $a,a',a'' \in \Ob(\mathcal M_1)$ and $b \in \Ob(\mathcal M_2)$ the following diagram is commutative
					\[
						\begin{tikzcd}[row sep=scriptsize, column sep=1.5cm]
							\mathcal M_1(a,a') \times \mathcal M_1(a',a'') \times \mathcal N(a'',b) \rar{\id \times \nu^1_{a' a'' b}} \dar{\mu^1_{a a' a''} \times \id} & \mathcal M_1(a,a') \times \mathcal N(a',b) \dar{\nu^1_{a a' b}} \\
							\mathcal M_1(a,a'') \times \mathcal N(a'',b) \rar{\nu^1_{a a'' b}} & \mathcal N(a,b)
						\end{tikzcd}.
					\]
					Similarly, for $a \in \Ob(\mathcal M_1)$ and $b,b',b'' \in \Ob(\mathcal M_2)$ the following diagram is commutative
					\[
						\begin{tikzcd}[row sep=scriptsize, column sep=1.5cm]
							\mathcal N(a,b'') \times \mathcal M_2(b'',b') \times \mathcal M_2(b'',b) \rar{\nu^2_{a b'' b'} \times \id} \dar{\id \times \mu^2_{b'' b' b}} & \mathcal N(a,b') \times \mathcal M_2(b',b) \dar{\nu^2_{a b' b}} \\
							\mathcal N(a,b'') \times \mathcal M_2(b'',b) \rar{\nu^2_{a b'' b}} & \mathcal N(a,b)
						\end{tikzcd}
					\]
					\item[\textsc{Commutativity}] For $a,a' \in \Ob(\mathcal M_1)$ and $b,b' \in \Ob(\mathcal M_2)$, the following diagram is commutative
					\[
						\begin{tikzcd}[row sep=scriptsize, column sep=1.5cm]
							\mathcal M_1(a,a') \times \mathcal N(a',b') \times \mathcal M_2(b',b) \dar{\id \times \nu^2_{a' b' b}} \rar{\nu^1_{a a' b'} \times \id} & \mathcal N(a,b') \times \mathcal M_2(b',b) \dar{\nu^2_{a b' b}} \\ \mathcal M_1(a,a') \times \mathcal N(a',b) \rar{\nu^1_{a a' b}} & \mathcal N(a,b)
						\end{tikzcd}
					\]
				\end{description}
				Moreover we require that $\sN(a,b)$ is a $\ang{\mu_1(a)-\mu_2(b)+k}$-manifold with faces given by
				\[
					\partial_i \sN(a,b) \coloneqq \bigsqcup_{\substack{a' \in \Ob(\sM_1) \\ \mu_1(a) - \mu_1(a') = i}} (\sM_1(a,a') \times \sN(a', b)) \sqcup \bigsqcup_{\substack{b' \in \Ob(\sM_2) \\ \mu_1(a) - \mu_2(b') + 1 + k = i}} (\sN(a,b') \times \sM_2(b', b)).
                \]
			\end{defn}

            \begin{notn}
                We will often refer to a flow bimodule of degree $0$ simply as a \emph{flow bimodule}.
            \end{notn}
            \begin{rem}\label{rem:bimod_deg}
            Given a flow category $\sM$, let $\sM[k]$ denote the flow category with the same objects and morphisms, with the translated grading function $\mu_{\sM[k]} (-) = \mu_\sM (-)$. A degree $k$ flow bimodule $\sM_1 \to \sM_2$ is equivalent to a degree $0$ flow bimodule $\sM_1[k] \to \sM_2$.
            \end{rem}

			Define the index bundle of the flow bimodule $\sN \colon \mathcal M_1 \to \mathcal M_2$ by
			\[ I(a,b) \coloneqq T\sN(a,b),\]
			for every $(a,b) \in \Ob(\sM_1) \times \Ob(\sM_2)$. We denote the associated $R$-line bundle by $I_R(a,b)$. Denote by $I_1(a,a')$ and $I_2(b,b')$ the index bundles of the flow categories $\mathcal M_1$ and $\mathcal M_2$, respectively. We denote their associated $R$-line bundles by $I_{1,R}(a,a')$ and $I_{2,R}(b,b')$, respectively. Similar to \cref{dfn:metric_flow_cat} and \eqref{eq:indexiso}, fixing a metric on $\mathcal N$ yields isomorphisms of vector bundles over $\sM_1(a,b) \times \sN(b,c)$ and $\sN(a,b) \times \sM_2(b,c)$, respectively
			\begin{align}\label{eq:index_bun_map1}
				\sigma^1_{abc} \colon (\nu^1_{abc})^\ast I(a,c) &\overset{\cong}{\longrightarrow} I_1(a,b) \oplus I(b,c) \\ \label{eq:index_bun_map2}
				\sigma^2_{abc} \colon (\nu^2_{abc})^\ast I(a,c) &\overset{\cong}{\longrightarrow} I(a,b) \oplus I_2(b,c).
			\end{align}
			By abuse of notation, we use the same notation for the induced isomorphisms of $R$-line bundles over $\sM_1(a,b) \times \sN(b,c)$ and $\sN(a,b) \times \sM_2(b,c)$, respectively
			\begin{align*}
				\sigma^1_{abc} \colon ((\nu^1_{abc})^\ast I(a,b))_R &\overset{\simeq}{\longrightarrow} I_{1,R}(a,b) \otimes_R I_R(b,c) \\
				\sigma^2_{abc} \colon ((\nu^2_{abc})^\ast I(a,b))_R &\overset{\simeq}{\longrightarrow} I_R(a,b) \otimes_R I_{2,R}(b,c)
			\end{align*}
			\begin{lem}\label{lem:index_maps_compat}
				Let $\mathcal N \colon \mathcal M_1 \to \mathcal M_2$ be a flow bimodule. The maps $\sigma^1_{abc}$ and $\sigma^2_{abc}$ satisfies the following:
				\begin{description}
					\item[Commutativity]
						For any $a,a' \in \Ob(\mathcal M_1)$ and $b,b' \in \Ob(\mathcal M_2)$, the following diagram of isomorphisms $R$-line bundles over $\sM_1(a,a') \times \sN(a',b') \times \sM_2(b',b)$ commutes.
						\[
							\begin{tikzcd}[row sep=scriptsize, column sep=1.5cm]
								I_R(a,b) \dar{\sigma^1_{a a' b}} \rar{\sigma^2_{abc}} & I_R(a,b') \otimes_R I_{2,R}(b',b) \dar{\sigma^1_{a a' b'} \otimes_R \id} \\
								I_{1,R}(a,a') \otimes_R I_R(a',b) \rar{\id \otimes_R \sigma^2_{a' b' b}} & I_{1,R}(a,a') \otimes_R I_{R}(a',b') \otimes_R I_{2,R}(b',b)
							\end{tikzcd}
						\]
					\item[Compatibility with $\rho_R$] The maps $\sigma^1_{abc}$ and $\sigma^2_{abc}$ are compatible with the maps $\rho_{1,R}^{abc}$ and $\rho_{2,R}^{abc}$ defined in \eqref{eq:index_flow_cat_cmpat} in the sense that the following diagrams commute as isomorphisms of $R$-line bundles over $\sM_1(a,a') \times \sM_1(a',a'') \times \sN(a'',b)$ and $\sN(a,b'') \times \sM_2(b'',b') \times \sM_2(b',b)$, respectively.
					\[
						\begin{tikzcd}[row sep=scriptsize, column sep=1.5cm]
							I_R(a,b) \rar{\sigma^1_{a a'' b}} \dar{\sigma^1_{a a' b}} & I_{1,R}(a,a'') \otimes_R I_R(a'',b) \dar{\rho^{a a' a''}_{1,R}\otimes_R \id} \\
							I_{1,R}(a,a') \otimes_R I_R(a',b) \rar{\id \otimes_R \sigma^1_{a' a'' b}} & I_{1,R}(a,a') \otimes_R I_{1,R}(a',a'') \otimes_R I_R(a'',b)
						\end{tikzcd}
					\]
					\[
						\begin{tikzcd}[row sep=scriptsize, column sep=1.5cm]
							I_R(a,b) \rar{\sigma^2_{a b' b}} \dar{\sigma^2_{a b'' b}} & I_R(a,b') \otimes_R I_{2,R}(b',b) \dar{\sigma^2_{a b'' b} \otimes_R \id} \\
							I_R(a,b'') \otimes_R I_{2,R}(b'',b) \rar{\id \otimes_R \rho_{2,R}^{b'' b' b}} & I_R(a,b'') \otimes_R I_{2,R}(b'',b') \otimes_R I_{2,R}(b',b)
						\end{tikzcd}
					\]
				\end{description}
            \qed
			\end{lem}
   
			\begin{defn}[$R$-orientation on a flow bimodule]\label{dfn:r_ori_flow_bimod}
				Let $(\sM_1,\fo_1)$ and $(\sM_2,\fo_2)$ be two $R$-oriented flow categories. An \emph{$R$-orientation} on a flow bimodule $\sN \colon \sM_1 \to \sM_2$ consists of isomorphisms
				\[
					 \fm(a,b) \colon \fo_1(a) \overset{\simeq}{\longrightarrow} I_R(a,b) \otimes_R \fo_2(b) 
				\]
				of $R$-line bundles over $\sN(a,b)$ such that for any $a,a' \in \Ob(\mathcal M_1)$ and $b,b' \in \Ob(\mathcal M_2)$, the following diagrams commute as isomorphisms of $R$-line bundles over $\sM_1(a,a') \times \sN(a',b)$ and $\sN(a,b') \times \sM_2(b',b)$, respectively.
				\begin{equation}\label{eq:bimod_ext_ori1}
					\begin{tikzcd}[row sep=scriptsize, column sep=1.5cm]
						\fo_1(a) \rar{\fm(a,b)} \dar{\fo_1(a,a')} & I_R(a,b) \otimes_R \fo_2(b) \dar{\sigma^1_{a a' b} \otimes_R \id} \\
						I_{1,R}(a,a') \otimes_R \fo_1(a') \rar{\id \otimes_R \fm(a',b)} & I_{1,R}(a,a') \otimes_R I_R(a',b) \otimes_R \fo_2(b)
					\end{tikzcd}
                \end{equation}
				\begin{equation}\label{eq:bimod_ext_ori2}
					\begin{tikzcd}[row sep=scriptsize, column sep=1.5cm]
					\fo_1(a) \rar{\fm(a,b)} \dar{\fm(a,b')} & I_R(a,b) \otimes_R \fo_2(b) \dar{\sigma^2_{a b' b} \otimes_R \id} \\
					I_R(a,b') \otimes \fo_2(b') \rar{\id \otimes_R \fo_2(b',b)} & I_R(a,b') \otimes_R I_{2,R}(b',b) \otimes_R \fo_2(b)
					\end{tikzcd}.
                \end{equation}
                We denote the choice of this data by $\fm$, and call the tuple $(\sN,\fm)$ an \emph{$R$-oriented flow bimodule}.
			\end{defn}

			Note that the action maps on index bundles $\sigma^1$ and $\sigma^2$ defined in \eqref{eq:index_bun_map1} and \eqref{eq:index_bun_map2} induce maps
			\begin{align}\label{eq:thombimod1}
				\mathcal N^{-\sigma^1_{a a' b}} \colon \sM_1^{-I_{1}}(a,a') \wedge \sN^{-I}(a',b) &\longrightarrow \sN^{-I}(a,b) \\ \label{eq:thombimod2}
				\mathcal N^{-\sigma^2_{a b' b}} \colon \sN^{-I}(a,b') \wedge \sM_2^{-I_2}(b',b) &\longrightarrow \sN^{-I}(a,b).
			\end{align}
			These maps satisfy suitable compatibility conditions with the maps $\mathcal M_1^{-\rho_1}$ and $\mathcal M_2^{-\rho_2}$ as a consequence of \cref{lem:index_maps_compat}.
			\begin{defn}
				Given a flow bimodule $\sN \colon \sM_1 \to \sM_2$, denote by $\sN^{-I}$ the spectrally enriched $(\sM_1^{-I_1},\sM_2^{-I_2})$-bimodule
				\[ \sN^{-I} \colon (\sM_1^{-I_1})^{\mathrm{op}} \wedge \sM_2^{-I_2} \longrightarrow \Spec \]
				defined on objects by $(a,b) \mapsto \sN^{-I(a,b)}(a,b)$ and on morphisms by the maps \eqref{eq:thombimod1} and \eqref{eq:thombimod2}.
			\end{defn}

			Now, let $(\sM_1,\fo_1)$ and $(\sM_2,\fo_2)$ be two $R$-oriented flow categories. These give rise to a $(\sM_1^{-I_1},\sM_2^{-I_2})$-bimodule
			\[\sF_{\fo_1,\fo_2} \colon (\sM_1^{-I_1})^{\mathrm{op}} \wedge \sM_2^{-I_2} \longrightarrow \Spec\]
			given on objects by $(a,b) \mapsto F_R(\fo_1(a), \fo_2(b))$, where $F_R(-,-)$ denotes the mapping spectrum of maps of $R$-modules.

            Recall from \cref{lma:Rori_to_spec_system} that an $R$-orientation on a flow category yields an $R$-module system.
			\begin{defn}[Bimodule extension of $R$-module systems]\label{dfn:bimod_ext_mod_syst}
				Let $\fo_1$ and $\fo_2$ be two $R$-module systems on the flow categories $\sM_1$ and $\sM_2$, respectively. An \emph{extension} of the $R$-module systems $\fo_1$ and $\fo_2$ along a flow bimodule $\sN \colon \sM_1 \to \sM_2$ is a natural transformation between bimodules
				\[ \sN^{-I} \Longrightarrow \sF_{\fo_1,\fo_2}.\]
				More explicitly, this consists of a collection of maps of $R$-modules
				\[ \mathcal N^{-\fm(a,b)} \colon \fo_1(a) \wedge \sN(a,b)^{-I(a,b)} \longrightarrow \fo_2(b),\]
				compatible with the action maps of $\sN$ in the sense that the following diagrams commute for every $a,a' \in \Ob(\mathcal M_1)$ and $b,b' \in \Ob(\mathcal M_2)$
				\[
					\begin{tikzcd}[row sep=scriptsize, column sep=1.75cm]
						\fo_1(a) \wedge \sM_1^{-I_1}(a,a') \wedge \sN^{-I}(a',b)\rar{\id \wedge \mathcal N^{-\sigma^1_{a a' b}}} \dar{\mathcal M_1^{-\fo_1(a,a')} \wedge \id} &\fo_1(a) \wedge \sN^{-I} (a,b) \dar{\mathcal N^{-\fm(a,b)}} \\
						\fo_1(a') \wedge \sN^{-I}(a',b)\rar{\mathcal N^{-\fm(a',b)}} & \fo_2(b).
					\end{tikzcd}
				\]
				\[
					\begin{tikzcd}[row sep=scriptsize, column sep=1.75cm]
						\fo_1(a) \wedge \sN^{-I}(a,b') \wedge \sM_2^{-I_2}(b',b) \rar{\id \wedge\sigma^2_{a b' b}} \dar{\mathcal N^{-\fm(a,b')}\wedge \id} & \fo_1(a) \wedge \sN^{-I} (a,b) \dar{\mathcal N^{-\fm(a,b)}} \\
						\fo_2(b') \wedge \sM_2^{-I_2}(b',b) \rar{\mathcal M_2^{-\fo_2(b',b)}} & \fo_2(b).
                    \end{tikzcd}
				\]
			\end{defn}
			An $R$-oriented flow bimodule between $R$-oriented flow categories induces a bimodule extension of the associated $R$-module systems on the flow categories. This is similar to the proof of \cref{lma:Rori_to_spec_system}.

			\begin{defn}[Composition of flow bimodules]\label{dfn:composition_flow_bimods}
				Let $\sM_1$, $\sM_2$ and $\sM_3$ be three flow categories and let $\sN_1 \colon \mathcal M_1 \to \mathcal M_2$ and $\sN_2 \colon \mathcal M_2 \to \mathcal M_3$ be two flow bimodules. We define the composition of $\mathcal N_1$ and $\mathcal N_2$ to be the flow bimodule
				\[
					\mathcal N_2 \circ \mathcal N_1 \colon \mathcal M_1 \longrightarrow \mathcal M_3,
				\]
				by declaring $(\mathcal N_2 \circ \mathcal N_1)(a,c)$ to be the colimit of the following diagram
				\[
					\begin{tikzcd}[row sep=scriptsize, column sep=1.25cm]
							\displaystyle\bigsqcup_{b,b' \in \Ob(\mathcal M_2)} (\mathcal N_1(a,b) \times \mathcal M_2(b,b') \times \mathcal N_2(b',c)) \rar[shift left]{\id \times \nu^2_{b b' c}} \rar[shift right,swap]{\nu^2_{a b b'} \times \id} & \displaystyle\bigsqcup_{b \in \Ob(\mathcal M_2)} (\mathcal N_1(a,b) \times \mathcal N_2(b,c)).
					\end{tikzcd}
				\]
			\end{defn}
            \begin{lem}\label{lem:r_ori_composition_bimods}
                Let $(\sM_i,\fo_i)$ be three $R$-oriented flow categories. Let $(\sN_1,\fm_1) \colon \sM_1 \to \sM_2$ and $(\sN_2,\fm_2) \colon \sM_2 \to \sM_3$ be two $R$-oriented flow bimodules. The composition $\sN_{12} \coloneqq \sN_2 \circ \sN_1$ carries a naturally induced $R$-orientation.
            \end{lem}
            \begin{proof}
                 For any tuple $(a,b,c) \in \Ob(\sM_1) \times \Ob(\sM_2) \times \Ob(\sM_3)$ we have isomorphisms of $R$-line bundles over $\sN_1(a,b) \times \sN_2(b,c)$
                \begin{equation}\label{eq:compos_bimod_ori}
                    \fo_1(a) \xrightarrow{\fm_1(a,b)} I_R^{\sN_1}(a,b) \otimes_R \fo_2(b) \xrightarrow{\id \otimes_R \fm_2(b,c)} I_R^{\sN_1}(a,b) \otimes_R I_R^{\sN_2}(b,c) \otimes_R \fo_3(c),
                \end{equation}
                For any $b,b' \in \Ob(\sM_2)$ we have 
                \begin{equation}\label{eq:compos_bimod_ori2}
                    (\id \times \nu^2_{bb'c})^\ast\text{\eqref{eq:compos_bimod_ori}} = (\nu^2_{abb'} \times \id)^\ast \text{\eqref{eq:compos_bimod_ori}}.
                \end{equation}
                Therefore the composition \eqref{eq:compos_bimod_ori} descends to an isomorphism of $R$-line bundles on the colimit $(\sN_2\circ \sN_1)(a,c)$
                \[
                    \fm_{12}(a,c) \colon \fo_1(a) \longrightarrow I^{\sN_{12}}_R(a,c) \otimes_R \fo_3(c),
                \]
                such that the following two diagrams commute:
                \[
                    \begin{tikzcd}[row sep=scriptsize,column sep=1.5cm]
                        \fo_1(a) \rar{\fm_{12}(a,c)} \dar{\fo_1(a,a')} & I^{\sN_{12}}_R(a,c) \otimes_R \fo_3(c) \dar{\sigma^1_{aa'c} \otimes_R \id} \\
                        I^{\sM_1}_R(a,a') \otimes_R \fo_1(a') \rar{\id \otimes_R \fm_{12}(a',c)} & I^{\sM_1}_R(a,a') \otimes_R I_R^{\sN_{12}}(a',c) \otimes_R \fo_3(c)
                    \end{tikzcd}
                \]
                \[
                    \begin{tikzcd}[row sep=scriptsize,column sep=1.5cm]
                        \fo_1(a) \rar{\fm_{12}(a,c)} \dar{\fm_{12}(a,c')} & I^{\sN_{12}}_R(a,c) \otimes_R \fo_3(c) \dar{\sigma^3_{ac'c} \otimes_R \id} \\
                        I^{\sN_{12}}_R(a,c') \otimes_R \fo_3(c') \rar{\id \otimes_R \fo_3(c',c)} & I_R^{\sN_{12}}(a,c') \otimes_R I_R^{\sN_3}(c',c) \otimes_R \fo_3(c)
                    \end{tikzcd}.
                \]
            \end{proof}
   
    \subsection{Bordism of flow bimodules}\label{sec:bordism_of_flow_bimods}
            The notion of a homotopy between maps is captured on the flow categorical level by bordisms.
			\begin{defn}[Bordism of flow bimodules]\label{dfn:bordism_of_flow_bimods}
				Suppose that $\mathcal N_1, \mathcal N_2 \colon \mathcal M_1 \to \mathcal M_2$ are two flow bimodules of two flow categories $\mathcal M_1$ and $\mathcal M_2$. A \emph{bordism from $\mathcal N_1$ to $\mathcal N_2$}
				\[
					\begin{tikzcd}[row sep=scriptsize, column sep=scriptsize]
						\mathcal M_1 \ar[r, bend left, "\mathcal N_1", ""{name=U,inner sep=1pt,below}] \ar[r, bend right, "\mathcal N_2"{below}, ""{name=D,inner sep=1pt}] & \mathcal M_2 \ar[Rightarrow, from=U, to=D, "\mathcal B"]
					\end{tikzcd}
				\]
				is an assignment of a $(\mu_1(a) - \mu_2(b)+1)$-dimensional smooth manifold with corners $\mathcal B(a,b)$ for each tuple $(a,b) \in \Ob(\sM_1) \times \Ob(\sM_2)$ such that for any $a\in \Ob(\sM_1)$, $\bigsqcup_{b\in \Ob(\sM_2)} \sB(a,b)$ is compact. A bordism us furthermore equipped with action maps
				\begin{align*}
					\beta^1_{a a' b} \colon \sM_1(a,a') \times \mathcal B(a', b) &\longrightarrow \mathcal B(a,b), \quad \forall (a,a',b) \in \Ob(\sM_1)^2 \times \Ob(\sM_2) \\
					\beta^2_{a b' b} \colon \mathcal B(a, b') \times \sM_2(b', b) &\longrightarrow \mathcal B(a,b), \quad \forall (a,b',b) \in \Ob(\sM_1) \times \Ob(\sM_2)^2 \\
                \end{align*}
                and inclusions
                \begin{align*}
                    \alpha^1_{a b} \colon \mathcal N_1(a,b) &\longrightarrow \mathcal B(a,b), \quad \forall (a,b) \in \Ob(\sM_1) \times \Ob(\sM_2)\\
                    \alpha^2_{a b} \colon \mathcal N_2(a,b) &\longrightarrow \mathcal B(a,b), \quad \forall (a,b) \in \Ob(\sM_1) \times \Ob(\sM_2)
				\end{align*}
				each of which is a diffeomorphism onto a face of $\mathcal B(a,b)$. The action maps are required to satisfy similar compatibility and commutativity relations as flow bimodules (see \cref{dfn:flow_bimodule}). Moreover, they are required to be compatible with the action maps of $\mathcal N_i$ under the inclusions. Define
                \[
                    \partial^\circ_i \mathcal B(a,b) \coloneqq \bigsqcup_{\substack{a' \in \Ob(\sM_1) \\ \mu_1(a)-\mu_1(a')  = i}} (\mathcal M_1(a,a') \times \mathcal B(a',b)) \sqcup \bigsqcup_{\substack{b' \in \Ob(\sM_2) \\ \mu_1(a) - \mu_2(b') + 2 = i}} (\mathcal B(a,b') \times \mathcal M_2(b',b)).
                \]
                We moreover require that $\sB(a,b)$ is a $\ang{\mu_1(a)-\mu_2(b)+2}$-manifold with faces given by
				\[
					\partial_i \mathcal B(a,b) \coloneqq 
						\begin{cases}
							\sN_1(a,b) \sqcup \sN_2(a,b), & i = 1 \\ 
							\partial^\circ_{i-1} \mathcal B(a,b), & i > 1.
						\end{cases}
				\]
			\end{defn}
			The index bundle of a bordism $\sB \colon \mathcal N_1 \Rightarrow \mathcal N_2$ is the virtual vector bundle
			\[ I(a,b) \coloneqq T\sB(a,b) - \underline{\IR},\]
			for any $(a,b) \in \Ob(\sM_1) \times \Ob(\sM_2)$. We denote the associated $R$-line bundle by $I_R(a,b)$. 
   
			In the following we denote the corresponding index bundles of $\mathcal M_1$ and $\mathcal M_2$ by $I_{\mathcal M_1}$ and $I_{\mathcal M_2}$, respectively. There are induced isomorphisms of vector bundles over $\sM_1(a,b) \times \sB(b,c)$, $\sB(a,b) \times \sM_2(b,c)$, $\sN_1(a,b)$, and $\sN_2(a,b)$, respectively
			\begin{align}\label{eq:bordism_iso1}
				(\beta^1_{abc})^\ast I(a,c) &\overset{\cong}{\longrightarrow} I_{\mathcal M_1}(a,b) \oplus I(b,c) \\ \label{eq:bordism_iso2}
				(\beta^2_{abc})^\ast I(a,c)&\overset{\cong}{\longrightarrow} I(a,b) \oplus I_{\mathcal M_2}(b,c) \\ \label{eq:bordism_iso3}
				(\alpha^1_{ab})^\ast I(a,b) &\overset{\cong}{\longrightarrow} I_{\mathcal N_1}(a,b) \\ \label{eq:bordism_iso4}
				(\alpha^2_{ab})^\ast I(a,b) &\overset{\cong}{\longrightarrow} I_{\mathcal N_2}(a,b).
			\end{align}
			\begin{lem}
                The isomorphisms of vector bundles (\ref{eq:bordism_iso1}--\ref{eq:bordism_iso4}) induce isomorphisms of the associated $R$-line bundles that satisfy similar compatibility diagrams as those in \cref{lem:index_maps_compat}.
            \qed
			\end{lem}
			\begin{defn}[$R$-orientation on a bordism]\label{dfn:r_ori_flow_bord}
            Let $(\sM_1,\fo_1)$ and $(\sM_2,\fo_2)$ be two $R$-oriented flow categories and let $(\sN_i,\fm_i) \colon \sM_1 \to \sM_2$ be two $R$-oriented flow bimodules for $i\in \{1,2\}$. An \emph{$R$-orientation} on a bordism $\sB \colon \sN_1 \Rightarrow \sN_2$ consists of isomorphisms 
            \[\fn(a,b) \colon \fo_1(a) \overset{\simeq}{\longrightarrow} I_R(a,b) \otimes_R \fo_2(b)  \]
            of $R$-line bundles over $\sB(a,b)$ which satisfy compatibilities similar to those in \cref{dfn:r_ori_flow_bimod} and moreover which agrees with the $R$-orientations $\fm_i$ on $\sN_i(a,b)$ under the isomorphisms (\ref{eq:bordism_iso3}--\ref{eq:bordism_iso4}). We denote the choice of an $R$-orientation on $\sB$ by $\fn$, and call the tuple $(\sB,\fn)$ an \emph{$R$-oriented flow bordism}.
			\end{defn}

			From the action maps on a bordism $\sB \colon \sN_1 \Rightarrow \sN_2$, we get induced maps
			\begin{align*}
				\sM_1^{-I_{\mathcal M_1}}(a,a') \wedge \sB^{-I}(a',b') \wedge \sM_2^{-I_{\mathcal M_2}}(b',b) &\longrightarrow \sB^{-I}(a,b) \\
				\sN_i^{-I_{\sN_i}}(a,b) &\longrightarrow \sB^{-I}(a,b),
			\end{align*}
			that are appropriately compatible with the composition data, similar to the diagrams in \cref{dfn:bimod_ext_mod_syst}.  These allow us to define a $(\sM_1^{-I_{\mathcal M_1}},\sM_2^{-I_{\mathcal M_2}})$-bimodule $\sB^{-I}$ that is given on objects by
			\[  (a,b) \longmapsto \sB^{-I}(a,b), \]
			which comes equipped with natural transformations for $i\in \{1,2\}$
			\[ \iota_i \colon \sN_i^{-I_{\sN_i}} \Longrightarrow \sB^{-I}.\]

			\begin{defn}[Bordism extension of $R$-module systems]\label{dfn:bordism_extension_rmod_system}
			Let $\fo_1$ and $\fo_2$ be two $R$-module systems on the flow categories $\sM_1$ and $\sM_2$, respectively. Let $\fm_1$ and $\fm_2$ be extensions of these $R$-module systems along the flow bimodules $\sN_1$ and $\sN_2$, respectively. An \emph{extension} of the $R$-module system extension $\fm_1$ and $\fm_2$ along a bordism $\sB \colon \sN_1 \Rightarrow \sN_2$ is a natural transformation
				\[ \fn \colon \sB^{-I} \Longrightarrow \sF_{\fo_1,\fo_2},\]
				such that $\fn \circ \iota_i = \fm_i$ for $i\in \{1,2\}$.
			\end{defn}
			A bordism extension of $R$-orientations induces a bordism extension of the associated $R$-module systems.
        
		\subsection{Flow multimodules}
            The flow categorical notion corresponding to operations that allow for ``multiple inputs'' is that of a flow multimodule.
            \begin{defn}[Flow multimodule]\label{dfn:flow_multimodule}
                Let $\mathcal M_0, \ldots, \mathcal M_n$ be flow categories. Let
                \[
                d \coloneqq \sum_{i=1}^n \mu_i(a_i)-\mu_0(a_0).
                \]
                A \emph{flow multimodule} $\mathcal N\colon \mathcal M_1, \ldots, \mathcal M_n \to \mathcal M_0$ is an assignment of a $d$-dimensional smooth manifold with corners $\mathcal N(\vec a;a_0)$ for each tuple $(a_0,\vec a) \coloneqq (a_0,a_1,\ldots,a_n) \in \prod_{i=0}^n \Ob(\mathcal M_i)$, such that for any $\vec a \in \prod_{i=1}^n \Ob(\sM_i)$, $\bigsqcup_{a_0\in \Ob(\sM_0)} \sN(\vec a;a_0)$ is compact. A flow multimodule is furthermore equipped with action maps
                \begin{align*}
                    \mathcal M_i(a_i,a_i') \times \mathcal N(a_1,\ldots,a_i',\ldots,a_n;a_0) &\longrightarrow \mathcal N(\vec a;a_0), \quad i\in \left\{1,\ldots,n\right\} \\
                    \mathcal N(\vec a;a_0') \times \mathcal M_0(a_0',a_0) &\longrightarrow \mathcal N(\vec a;a_0),
                \end{align*}
                each of which is a diffeomorphism onto a face of $\sN(\vec a;a_0)$. The action maps are required to satisfy similar compatibility and commutativity relations as flow bimodules (see \cref{dfn:flow_bimodule}). In addition, the actions of $\mathcal M_i, i \in \{1,\dots,n\}$ are required to commute with each other.

                We moreover require that $\sN(\vec a;a_0)$ is a $\ang d$-manifold with faces given by
                \begin{align*}
                    \partial_i \sN(\vec a;a_0) \coloneqq &\bigsqcup_{j=1}^n \left(\bigsqcup_{\substack{a_j' \in \Ob(\sM_j) \\ \mu_j(a_j) - \mu_j(a_j') = i}} (\sM_j(a_j,a_j') \times \sN(a_1,\ldots,a_j',\ldots,a_n;a_0))\right) \\
                    &\qquad \sqcup \bigsqcup_{\substack{a_0' \in \Ob(\sM_0) \\ \sum_{\ell=1}^n \mu_\ell(a_\ell) -\mu_0(a_0') + 1= i}} (\sN(\vec a;a_0') \times \sM_0(a_0',a_0)).
                \end{align*}
            \end{defn}
            
            \begin{defn}[$R$-orientation on a flow multimodule]\label{dfn:r_ori_flow_multimod}
                Let $(\sM_i,\fo_i)$ be $R$-oriented flow categories for $i\in \left\{0,\ldots,n\right\}$. An \emph{$R$-orientation} on a flow multimodule $\sN \colon \sM_1,\ldots,\sM_n \to \sM_0$ consists of isomorphisms
                \[ \fm(\vec a;a_0) \colon \fo_1(a_1) \otimes_R \cdots \otimes_R \fo_n(a_n) \overset{\simeq}{\longrightarrow} I_R(\vec a;a_0) \otimes_R \fo_0(a_0)\]
                of $R$-line bundles over $\sN(\vec a;a_0)$ which are compatible with the action and composition maps of $\sM_i$, in the sense that there are commutative diagrams similar to \eqref{eq:bimod_ext_ori1} and \eqref{eq:bimod_ext_ori2}.
            \end{defn}

            The action maps of a flow multimodule induce maps
            \begin{equation}\label{eq:thommulti}
                \bigwedge_{j=1}^n \sM_j^{-I_j}(a_j,a_j') \wedge \sN^{-I}(a_1',\ldots,a_n';a_0') \wedge \sM_0^{-I_0}(a_0',a_0) \longrightarrow \sN^{-I}(\vec a;a_0).
            \end{equation}
            \begin{defn}
                Given a flow multimodule $\sN \colon \sM_1,\ldots,\sM_n \to \sM_0$, denote by $\sN^{-I}$ the $\mod{R}$-enriched $(\sM_1^{-I_1}, \ldots, \sM_n^{-I_n}; \sM_0^{-I_0})$-multimodule
                \[ \sN^{-I} \colon \left(\bigwedge_{j=1}^n \sM_j^{-I_j}\right)^{\mathrm{op}} \wedge \sM_0^{-I_0} \longrightarrow \Spec \]
                defined on objects by $(\vec a;a_0) \mapsto \sN^{-I}(\vec a;a_0)$ and on morphisms by the maps \eqref{eq:thommulti}.
            \end{defn}

            Now, let $\fo_i$ be a $R$-module system on the flow category $\sM_i$ for $i\in \left\{0,\ldots,n\right\}$. These give rise to a multimodule
            \[\sF_{\vec \fo;\fo_0} \colon \left(\bigwedge_{j=1}^n \mathcal M_j^{-I_j}\right)^{\mathrm{op}} \wedge \sM_0^{-I_0} \longrightarrow \Spec \]
            given on objects by $(\vec a;a_0) \mapsto F_R(\fo_1(a_1) \wedge_R \cdots \wedge_R \fo_n(a_n), \fo_0(a_0))$.
            \begin{defn}[Multimodule extension of $R$-module systems]\label{dfn:multimodule_ext}
                Let $\fo_i$ be an $R$-module system on the flow category $\sM_i$ for $i \in \left\{0,\ldots,n\right\}$. An \emph{extension} of the $R$-module systems $\fo_i$ along a flow multimodule $\sN \colon \sM_1,\ldots,\sM_n \to \sM_0$ is a natural transformation between multimodules
                \[ \sN^{-I} \Longrightarrow \sF_{\fo_1,\ldots,\fo_n;\fo_0}.\]
            \end{defn}
            Similar to the case of flow bimodules, an $R$-orientation on a flow multimodule between $R$-oriented flow categories induces a multimodule extension of the associated $R$-module systems.
            \begin{defn}[Composition of flow multimodules]\label{dfn:compos_multimodules}
                Let
                \[
                    \sN_1 \colon \mathcal M_1,\ldots, \mathcal M_p \longrightarrow \mathcal M_0, \quad \sN_2 \colon \mathcal M_1',\ldots, \mathcal M_q' \longrightarrow \mathcal M'_0
                \]
                be two flow multimodules. Assume that $\mathcal M_0 = \mathcal M_i'$. \emph{The $i$-th composition}, denoted by $\mathcal N_2 \circ_i \mathcal N_1$, is the flow multimodule
                \[
                    \mathcal N_2 \circ_i \mathcal N_1 \colon \mathcal M_1',\ldots,\mathcal M_{i-1}',\mathcal M_1,\ldots, \mathcal M_p, \mathcal M_{i+1}',\ldots,\mathcal M_q' \longrightarrow \mathcal M'_0,
                \]
                such that for any $(a_0,\ldots,a_p) \in \prod_{i=0}^p \Ob(\mathcal M_i)$ and $(b_0,\ldots,b_q) \in \prod_{i=0}^q \Ob(\mathcal M'_i)$ with $a_0 = b_i$,
                \[
                    (\mathcal N_2 \circ_i \mathcal N_1)(b_1,\ldots,b_{i-1},a_1,\ldots,a_p,b_{i+1},\ldots,b_q;b_0)
                \]
                is declared to be the colimit of the following diagram
                \[
                \begin{tikzcd}[row sep=scriptsize, column sep=small]
                    \displaystyle\bigsqcup_{a'_0,b'_0\in \Ob(\sM_0)} (\mathcal N_1(\vec a; a'_0) \times \mathcal M_0(a'_0,b'_0) \times \mathcal N_2(\vec b_{i,b'_0};b_0)) \rar[shift left]{} \rar[shift right,swap]{} & \displaystyle\bigsqcup_{b'_0\in \Ob(\sM_0)} (\mathcal N_1(\vec a;b'_0) \times \mathcal N_2(\vec b_{i,b'_0};b_0))
                \end{tikzcd}
                \]
                where
                \[
                \vec b_{i,x} \coloneqq (b_1,\ldots,b_{i-1},x,b_{i+1},\ldots,b_q).
                \]
                Given that $\sN_1$ and $\sN_2$ are $R$-oriented flow multimodules between $R$-oriented flow categories, there exists a natural $R$-orientation on $\sN_2 \circ_i \sN_1$ that is defined in a similar way as in \cref{lem:r_ori_composition_bimods}.
            \end{defn}

        \subsection{Bordism of flow multimodules}
            \begin{defn}[Bordism of flow multimodules]\label{def:multi_bord}
                Suppose that $\mathcal N_1, \mathcal N_2 \colon \mathcal M_1,\ldots, \mathcal M_n \to \mathcal M_0$ are two flow multimodules of the flow categories $\mathcal M_0,\ldots, \mathcal M_n$. Let $d \coloneqq \sum_{i=1}^n \mu_i(a_i) - \mu_0(a_0)$. A \emph{bordism} from $\mathcal N_1$ to $\mathcal N_2$
                \[
                    \begin{tikzcd}[row sep=scriptsize, column sep=scriptsize]
                        \mathcal M_1,\ldots,\mathcal M_n \ar[r, bend left, "\mathcal N_1", ""{name=U,inner sep=1pt,below}] \ar[r, bend right, "\mathcal N_2"{below}, ""{name=D,inner sep=1pt}] & \mathcal M_0 \ar[Rightarrow, from=U, to=D, "\mathcal B"]
                    \end{tikzcd},
                \]
                is an assignment of a $\left(d + 1\right)$-dimensional smooth manifold with corners $\mathcal B(\vec a;a_0)$ for each tuple $(a_0,\vec a) \coloneqq (a_0,a_1,\ldots,a_n) \in \prod_{i=0}^n \Ob(\sM_i)$, such that for any $\vec a \in \prod_{i=1}^n \Ob(\sM_i)$, $\bigsqcup_{a_0\in \Ob(\sM_0)} \sB(\vec a;a_0)$ is compact. A bordism is furthermore equipped with action maps
                \begin{align*}
                    \mathcal M_i(a_i,a_i') \times \mathcal B(a_1,\ldots,a_i',\ldots,a_n;a_0) &\longrightarrow \mathcal B(\vec a;a_0), \quad i\in \left\{1,\ldots,n\right\} \\
                    \mathcal B(\vec a;a_0') \times \mathcal M_0(a_0',a_0) &\longrightarrow \mathcal B(\vec a;a_0) \\
                    \sigma_j \colon \mathcal N_j(\vec a;a_0) &\longrightarrow \mathcal B(\vec a;a_0), \quad j\in \{1,2\},
                \end{align*}
                each of which is a diffeomorphism onto a face of $\mathcal B(\vec a;a_0)$. The action maps are required to satisfy similar compatibility and commutativity relations as flow multimodules (see \cref{dfn:flow_multimodule}). 

                Let $D \coloneqq \sum_{\ell=1}^n \mu_\ell(a_\ell) -\mu_0(a_0') + 1$ and let
                \begin{align*}
                    \partial_i^\circ \mathcal B(\vec a;a_0) \coloneqq &\bigsqcup_{j=1}^n \left(\bigsqcup_{\substack{a_j' \in \Ob(\sM_j) \\ \mu_j(a_j) - \mu_j(a_j') = i}} (\sM_j(a_j,a_j') \times \mathcal B(a_1,\ldots,a_j',\ldots,a_n;a_0))\right) \\
                    &\qquad \sqcup \bigsqcup_{\substack{a_0' \in \Ob(\sM_0) \\ \sum_{\ell=1}^n \mu_\ell(a_\ell) -\mu_0(a_0') + 2 = i}} (\mathcal B(\vec a;a_0') \times \sM_0(a_0',a_0)).
                \end{align*}
                We moreover require that $\sB(\vec a;a_0)$ is a $\ang{D+1}$-manifold with faces given by
                \[
                    \partial_i \mathcal B(\vec a; a_0) \coloneqq \begin{cases}
                        \mathcal N_1(\vec a;a_0) \sqcup \mathcal N_2(\vec a;a_0), & i = 1 \\
                        \partial_{i-1}^\circ \mathcal B(\vec a;a_0), & i > 1
                    \end{cases}.
                \]
            \end{defn}
            The index bundle of a bordism $\sB \colon \sN_1 \Rightarrow \sN_2$ is the virtual vector bundle 
            \[
                I(a,b) \coloneqq T\sB(\vec a;a_0) - \underline{\IR},
            \]
            for any $(a_0,\vec a) \in \prod_{i=0}^n \Ob(\sM_i)$. We denote the associated $R$-line bundle by $I_R(\vec a;a_0)$.
            \begin{defn}[$R$-oriented flow bordism]\label{dfn:r_ori_flow_multibord}
                Let $(\sM_i,\fo_i)$ be $R$-oriented flow categories for $i\in\{0,\ldots,p\}$ and let
                \[
                (\sN_1,\fm_1),(\sN_2,\fm_2) \colon \sM_1,\ldots,\sM_p \longrightarrow \sM_0
                \]
                be two $R$-oriented flow multimodules. An \emph{$R$-orientation} on a bordism $\sB \colon \sN_1 \Rightarrow \sN_2$ consists of isomorphisms
                \[
                \fn(\vec a;a_0) \colon \fo_1(a_1) \otimes_R \cdots \otimes_R \fo_n(a_n) \overset{\simeq}{\longrightarrow} I_R(\vec a;a_0) \otimes_R \fo_0(a_0),
                \]
                of $R$-line bundles over $\sB(\vec a;a_0)$ which are compatible with the action and composition maps of $\sM_i$, in the sense that there are commutative diagrams of isomorphisms of $R$-line bundles similar to \eqref{eq:bimod_ext_ori1} and \eqref{eq:bimod_ext_ori2}. We denote the choice of an $R$-orientation on $\sB$ by $\fn$, and call the tuple $(\sB,\fn)$ an $R$-oriented bordism.
            \end{defn}

        \subsection{Orientations induced from bordisms}
        We prove here two technical lemmas needed in the construction of orientations in \cref{sec:the_oc_map}.
        \begin{lem}\label{lem:restr_last_strata_ori}
            Let $(\sN_1,\fm_1) \colon (\sM_1,\fo_1) \to (\sM_2,\fo_2)$ be an $R$-oriented flow bimodule. Let $\sN_2 \colon \sM_1 \to \sM_2$ be a flow bimodule and $\sB \colon \sN_1 \Rightarrow \sN_2$ be a flow bordism such that for every tuple $(a,b) \in \Ob(\sM_1) \times \Ob(\sM_2)$ there exist isomorphisms of $R$-line bundles
            \[ \fn(a,b) \colon \fo_1(a) \longrightarrow I_R^{\sB}(a,b) \otimes_R \fo_2 (b)\]
            satisfying the following conditions:
            \begin{enumerate}
                \item the pullback of $\fn(a,b)$ along the inclusion $\sN_1(a,b) \hookrightarrow \partial \sB(a,b)$ coincides with the isomorphism
                \[ \fm_1(a,b) \colon \fo_1(a) \longrightarrow I^{\sN_1}_R(a,b)\otimes_R \fo_2(b) .\]
                \item the pullback of $\fn(a,b)$ along the inclusion $\sM_1(a,a') \times \sB(a',b) \hookrightarrow \partial\sB(a,b)$ coincides with the composition
                \[ \fo_1(a)  \xrightarrow{\fo_1(a,a')} I^{\sM_1}_R(a,a') \otimes_R \fo_1(a') \xrightarrow{\id \otimes_R \fn(a,b)} I^{\sM_1}_R(a,a') \otimes_R I^{\sB}_R(a',b) \otimes_R \fo_2(b).\]
                \item the pullback of $\fn(a,b)$ along the inclusion $\sB(a,b') \times \sM_2(b',b) \hookrightarrow \partial\sB(a,b)$ coincides with the composition
                \[ \fo_1(a) \xrightarrow{\fn(a,b')} I_R^{\sB}(a,b')\otimes_R \fo_2(b') \xrightarrow{\id \otimes_R \fo_2(b',b)} I_R^{\sB}(a,b') \otimes_R I_R^{\sM_2}(b',b) \otimes_R \fo_2(b).\]
            \end{enumerate}
            Then, equipping $\sN_2$ with the isomorphisms $\fm_2(a,b)$ defined by the pullback of $\fn(a,b)$ along the inclusion $\sN_2(a,b) \hookrightarrow \partial \sB(a,b)$ defines an $R$-orientation on $\sN_2$ so that 
            \[ (\sB,\fn) \colon (\sN_1,\fm_1) \Longrightarrow (\sN_2,\fm_2)\]
            is an $R$-oriented flow bordism.
        \end{lem}
        \begin{proof}
            The action maps and inclusion of boundary strata in $\sB$ and $\sN_2$ fit into two commutative diagrams
            \begin{equation}\label{eq:compat_dia_bdry}
                \begin{tikzcd}[row sep=scriptsize, column sep=scriptsize]
                \sM_1(a,a') \times \sN_2(a',b) \dar \rar & \sN_2(a,b) \dar \\
                \sM_1(a,a') \times \sB(a',b) \rar & \sB(a,b)
            \end{tikzcd}
            \quad
            \begin{tikzcd}[row sep=scriptsize, column sep=scriptsize]
                \sN_2(a,b') \times \sM_2(b',b) \dar \rar & \sN_2(a,b) \dar \\
                \sB(a,b') \times \sM_2(b',b) \rar & \sB(a,b)
            \end{tikzcd}.
            \end{equation}
            There is a corresponding diagram of isomorphisms of $R$-line bundles over $\sM_1(a,a') \times \sN_2(a',b)$ as follows:
            \[
            \begin{tikzcd}[row sep=scriptsize, column sep=0pt]
                 \fo_1(a) \ar[dr, "\id"] \ar[rr,"{\fo_1(a,a')}"] \ar[dd,"{\fm_2(a,b)}"] & {} & I_R^{\sM_1}(a,a') \otimes_R \fo_1(a') \ar[dd,"{\id \otimes_R \fm_2(a',b)}",near end] \ar[dr,"\id"] & {} \\
                {} &  \fo_1(a) \ar[rr,crossing over,"{\fo_1(a,a')}",near start] & {} & I_R^{\sM_1}(a,a') \otimes_R \fo_1(a') \ar[dd,"{\id \otimes_R \fn(a',b)}"] \\
                I_R^{\sN_2}(a,b) \otimes_R \fo_2(b) \ar[rr] \ar[dr] & {} & \mathrel{\parbox{3cm}{$I_R^{\sM_1}(a,a') \otimes_R \\ I_R^{\sN_2}(a',b) \otimes_R \fo_2(b)$}} \ar[dr] & {} \\
                {} &  I_R^{\sB}(a,b) \otimes_R \fo_2(b) \ar[rr] \ar[from=uu,crossing over,"{\fn(a,b)}", near end] & {} & \mathrel{\parbox{3cm}{$I_R^{\sM_1}(a,a') \otimes_R \\ I_R^{\sB}(a',b) \otimes_R \fo_2(b)$}}
            \end{tikzcd}.
            \]
            The bottom face commutes because the right diagram in \eqref{eq:compat_dia_bdry} commutes. Both the right and left faces commutes by the definition of $\fm_2$, and the top face commutes trivially. The front face commutes by item (ii). Hence the back face is also commutative. There is a similar diagram involving the action maps $\sN_2(a,b') \times \sM_2(b',b) \to \sN_2(a,b)$, which together with the above diagram show that $\fm_2$ defines an $R$-orientation on $\sN_2$. Finally it is clear by definition that $\fn$ defines an $R$-orientation on $\sB$.
        \end{proof}
        \begin{lem}\label{lem:comp_two_out_of_three}
            Let $(\sM_1,\fo_1), (\sM_2,\fo_2),$ and $(\sM_3,\fo_3)$ be three $R$-oriented flow categories and consider two flow bimodules
            \[ \sM_1 \overset{\sN_{12}}{\longrightarrow} \sM_2 \overset{\sN_{23}}{\longrightarrow} \sM_3,\]
            satisfying the condition that there exists an object $a\in \Ob(\sM_1)$ such that $\sN_{12}(a,b) \neq \varnothing$, for every $b \in \Ob(\sM_2)$. Then, given $R$-orientations $\fm_{12}$ and $\fm_{13}$ on $\sN_{12}$ and $\sN_{23} \circ \sN_{12}$, respectively, there exists an $R$-orientation $\fm_{23}$ on $\sN_{23}$ that is compatible with those on $\sN_{12}$ and $\sN_{23} \circ \sN_{12}$.
        \end{lem}
        \begin{proof}
            Let $c\in \Ob(\sM_3)$. Since $\sN_{13} \coloneqq \sN_{23} \circ \sN_{12}$ is $R$-oriented we have isomorphisms of $R$-line bundles over $\sN_{13}(a,c)$
        	\[
        		\fm_{13}(a,c) \colon \fo_1(a) \overset{\simeq}{\longrightarrow} I_R^{\sN_{13}}(a,c) \otimes_R \fo_3(c),
        	\]
            by \cref{lem:r_ori_composition_bimods}. Let $q \colon \bigsqcup_{b \in \Ob(\sM_2)} (\sN_{12}(a,b) \times \sN_{23}(b,c)) \to \sN_{13}(a,c)$ denote the induced map and denote by $q_b$ the restriction of $q$ corresponding to $b\in \Ob(\sM_2)$. By definition we have $q_b^\ast I_R^{\sN_{13}}(a,c) \simeq I^{\sN_{12}}_R(a,b) \otimes_R I^{\sN_{23}}_R(b,c)$. Thus we have that the isomorphisms $\fm_{13}(a,c)$ is equivalent to the data of isomorphisms of $R$-line bundles over $\sN_{12}(a,b) \times \sN_{23}(b,c)$ for every $b\in \Ob(\sM_2)$
            \[
            q_b^\ast \fm_{13}(a,c) \colon \fo_1(a) \overset{\simeq}{\longrightarrow} I^{\sN_{12}}_R(a,b) \otimes_R I^{\sN_{23}}_R(b,c) \otimes_R \fo_3(c),
            \]
            that are compatible in the sense that for any $b,b'\in \Ob(\sM_2)$, the two isomorphisms $q_b^\ast \fm_{13}(a,c)$ and $q_{b'}^\ast \fm_{13}(a,c)$ agree along the common boundary strata $\sN_{12}(a,b) \times \sM_2(b,b') \times \sN_{23}(b',c)$.
            
        	Assuming that $\sN_{12}$ is $R$-oriented we have isomorphisms of $R$-line bundles over $\sN_{12}(a,b)$
        	\[
        		\fm_{12}(a,b) \colon \fo_1(a) \overset{\simeq}{\longrightarrow} I_R^{\sN_{12}}(a,b) \otimes_R \fo_2(b),
        	\]
        	for every $b \in \Ob(\sM_2)$. This equivalently yields an isomorphism of $R$-line bundles over $\sN_{12}(a,b)$
        	\[
        		\fm'_{12}(b,a) \colon \fo_2(b) \overset{\simeq}{\longrightarrow} (-I^{\sN_{12}}(a,b))_R \otimes_R \fo_1(a).
        	\]
            Let $\pi_{12}^{ab}$ and $\pi_{23}^{bc}$ denote the projections from $\sN_{12}(a,b) \times \sN_{23}(b,c)$ to the first and second factor, respectively. For every $b\in \Ob(\sM_2)$ and $c\in \Ob(\sM_3)$ fix some $x_{ab} \in \sN_{12}(a,b)$ (recall $\sN_{12}(a,b) \neq \varnothing$) and define
            \begin{align*}
                s_{ab} \colon \sN_{23}(b,c) &\longrightarrow \sN_{12}(a,b) \times \sN_{23}(b,c) \\ 
                x &\longmapsto (x_{ab},x).
            \end{align*}
            We get isomorphisms of $R$-line bundles over $\sN_{23}(b,c)$
            \begin{align*}
                \widetilde{\fm}_{12}(b,a) \coloneqq (\pi_{12} \circ s_{ab})^\ast \fm_{12}'(b,a) \colon \fo_2(b) &\overset{\simeq}{\longrightarrow} (-I^{\sN_{12}}(a,b))_R \otimes_R \fo_1(a) \\
                \widetilde{\fm}_{13}^b(a,c) \coloneqq (q_b \circ s_{ab})^\ast \fm_{13}(a,c) \colon \fo_1(a) &\overset{\simeq}{\longrightarrow} I^{\sN_{12}}_R(a,b) \otimes_R I^{\sN_{23}}_R(b,c) \otimes_R \fo_3(c),
            \end{align*}
            $b\in \Ob(\sM_2)$ and $c\in \Ob(\sM_3)$. Denote by $\nu_1$ and $\nu_2$ the two maps
            \begin{align*}
                \nu_1 \colon \sN_{12}(a,b) \times \sM_2(b,b') \times \sN_{23}(b',c) &\longrightarrow \sN_{12}(a,b') \times \sN_{23}(b',c)\\
                \nu_2 \colon \sN_{12}(a,b) \times \sM_2(b,b') \times \sN_{23}(b',c) &\longrightarrow \sN_{12}(a,b) \times \sN_{23}(b,c)
            \end{align*}
            and note that for any $b,b' \in \Ob(\sM_2)$ we have
            \begin{align}\label{eq:change_b}
                (\pi_{23}^{bc} \circ \nu_2)^\ast \widetilde{\fm}^b_{13}(a,c) &= (q_b \circ s_{ab} \circ \pi_{23}^{bc} \circ \nu_2)^\ast \fm_{13}(a,c) \\ \nonumber
                &= (q_b \circ \nu_2)^\ast \fm_{13}(a,c) \overset{\text{\eqref{eq:compos_bimod_ori2}}}{=} (q_{b'} \circ \nu_1)^\ast \fm_{13}(a,c) \\ \nonumber
                &= (q_{b'} \circ s_{ab'} \circ \pi^{b'c}_2 \circ \nu_1)^\ast \fm_{13}(a,c) \\ \nonumber
                &= (\pi^{b'c}_2 \circ \nu_1)^\ast \widetilde{\fm}^{b'}_{13}(a,c).
            \end{align}
            Given $b\in \Ob(\sM_2)$ and $c\in \Ob(\sM_3)$ define $\fm_{23}(b,c) \coloneqq (\id \otimes_R \widetilde{\fm}^b_{13}(a,c)) \circ \widetilde{\fm}_{12}(b,a)$.
            
            We finally show that $\fm_{23}(b,c)$ satisfies the required compatibilities. First, we have the diagram
            \[
                \begin{tikzcd}[row sep=scriptsize, column sep=1cm]
                    \fo_2(b) \rar{\widetilde{\fm}^b_{12}(b,a)} & (-I^{\sN_{12}}(a,b))_R \otimes_R \fo_1(a) \dar{\id\otimes_R \widetilde{\fm}^b_{13}(a,c')} \rar{\id \otimes_R \widetilde{\fm}^b_{13}(a,c)} & I^{\sN_{23}}_R(b,c) \otimes_R \fo_3(c) \dar{\sigma^3_{bc'c}\otimes_R \id} \\
                    {} & I^{\sN_{23}}_R(b,c') \otimes_R \fo_3(c') \rar{\id\otimes_R \fo_3(c',c)} & I^{\sN_{23}}_R(b,c') \otimes_R I^{\sM_{3}}_R(c',c) \otimes_R \fo_3(c)
                \end{tikzcd},
            \]
            which commutes due to the compatibilities satisfied by $\widetilde{\fm}^b_{13}$. Next, we consider the diagram
            \[
                \begin{tikzcd}[row sep=scriptsize, column sep=1.5cm]
                    \fo_2(b) \rar{\widetilde{\fm}_{12}(b,a)} \dar{\fo_2(b,b')} & (-I^{\sN_{12}}(a,b))_R \otimes_R \fo_1(a) \rar{\id \otimes_R \widetilde{\fm}^b_{13}(a,c)} \dar{\sigma^1_{abb'}\otimes_R \id} & I^{\sN_{23}}_R(b,c) \otimes_R \fo_3(c) \dar{\sigma^3_{bb'c}\otimes_R \id} \\
                    I^{\sM_2}_R(b,b')\otimes_R \fo_2(b') \rar{\id\otimes_R \widetilde{\fm}_{12}(b',a)} & \mathrel{\parbox{2.9cm}{$I^{\sM_2}_R(b,b') \\ \otimes_R (-I^{\sN_{12}}(a,b'))_R \\ \otimes_R \fo_1(a)$}} \rar{\id^2\otimes_R \widetilde{\fm}^{b'}_{13}(a,c)} & \mathrel{\parbox{2.9cm}{$I^{\sM_2}_R(b,b') \\ \otimes_R I^{\sN_{23}}_R(b',c) \\ \otimes_R \fo_3(c)$}}
                \end{tikzcd},
            \]
            where the right square commutes due to the compatibility \eqref{eq:change_b} satisfied by $\widetilde{\fm}^b_{13}$. The left square commutes due to the compatibilities satisfied by $\widetilde{\fm}_{12}$, viz.\@, after tensoring with $I^{\sN_{12}}_R(a,b)$ the (inverse of the) left bottom horizontal arrow factors as follows:
            \begin{align*}
                &(-I^{\sN_{12}}(a,b'))_R \otimes_R I^{\sN_{12}}_R(a,b) \otimes_R I^{\sM_2}_R(b,b') \otimes_R \fo_1(a) \\
                & \qquad \longrightarrow I^{\sN_{12}}_R(a,b') \otimes_R \fo_2(b') \xrightarrow{\sigma^2_{abb'}\otimes \id} I^{\sN_{12}}_R(a,b) \otimes_R I^{\sM_2}(b,b') \otimes_R \fo_2(b').
            \end{align*}
            Hence the outer square commutes, and $\fm_{23}$ defines an $R$-orientation on the flow bimodule $\sN_{23}$. This $R$-orientation is by construction compatible with the $R$-orientations on $\sN_{12}$ and $\sN_{13}$ in the sense that the natural $R$-orientation induced on the composition $\sN_{13}$ by these $R$-orientations on $\sN_{12}$ and $\sN_{13}$ via \cref{lem:r_ori_composition_bimods}, is the given one.
        \end{proof}
 
\subsection{Flow categories with local systems}\label{sec:flow_cat_local_system}
\begin{defn}
   Let $A$ be an $R$-algebra and $\sM$ be a flow category. A \emph{local system of $A$-modules on $\sM$} is a functor 
   \[ \sE \colon \sM \longrightarrow \mod{A},\]
   where $\sM$ is regarded as a spectrally enriched category by taking the suspension spectrum of its morphism spaces.
\end{defn}
Let $\sE \colon \sM \to \mod{A}$ be a local system of $A$-modules on an $R$-oriented flow category $(\sM,\fo)$. For any $x,y \in \Ob(\sM)$, the structure maps
\[ \sE(x,y) \colon \sE(x) \wedge \varSigma^\infty_+\sM(x,y) \longrightarrow \sE(y)\]
and the $R$-orientation $\fo$ induce maps
\[ \sE(x) \wedge_R \fo(x) \wedge \sM^{-I}(x,y) \longrightarrow \sE(y) \wedge_R \fo(y).\]
The assignment
\[ Z_{\sM,\fo,\sE}(m) \coloneqq \bigvee_{x\in \mu^{-1}(m)} (\sE(x) \wedge_R \fo(x)), \quad n \in \IZ \]
defines an $A$-linear $\sJ$-module $Z_{\sM,\fo,\sE} \colon \sJ \to \mod{A}$ similarly as in \cref{dfn:rlin_jmod_flow} via the maps \eqref{eq:flowtoj}.
\begin{defn}
    Let $(\sM,\fo)$ be an $R$-oriented flow category satisfying \cref{asmp:flowcjs}. Suppose that $\sE \colon \sM \to \mod A$ is a local system of $A$-modules. The \emph{CJS realization} of $(\sM,\fo,\sE)$, denoted by $|\sM,\fo,\sE|$, is defined as the geometric realization of the $A$-linear $\sJ$-module $Z_{\sM,\fo,\sE}$, see \cref{defn:geometric_realization_J-mod}.
\end{defn}
\begin{rem}\label{rem:trivial_local_sys}
    Given any flow category $\sM$ and an $A$-module $M$, the constant functor $\sM \to \mod{A}$ taking the value $M$ defines a local system on $\sM$ which we denote by $\underline M$. Note that since $|\sM,\fo,\underline M| \simeq |\sM,\fo| \wedge_R M$ with the trivial $A$-action. In particular, we have an equivalence of $R$-modules $|\sM,\fo,\underline R| \simeq |\sM,\fo|$.
\end{rem}

\begin{defn}\label{notn:full_flow_subcat}
   Let $\sM$ be a flow category and let $\sM'$ be a full subcategory of $\sM$ satisfying the condition that if $\sM(x,y) \neq \varnothing$ for $x \in \Ob(\sM')$ and $y \in \Ob(\sM)$, then $y \in \Ob(\sM')$. Define $\sM/\sM'$ to be the full subcategory of $\sM$ with object set $\Ob(\sM) \smallsetminus \Ob(\sM')$.
   
    The restriction of any $R$-orientation $\fo$ on $\sM$ defines $R$-orientations on the flow categories $\sM'$ and $\sM/\sM'$. Similarly, the restriction of a local system $\sE$ of $A$-modules on $\sM$ defines a local system of $A$-modules on $\sM'$ and $\sM/\sM'$.
\end{defn}
\begin{lem}\label{lem:realization_cofiber}
    Let $\sM$ be a flow category and let $\sM'$ be a full subcategory of $\sM$ satisfying the condition that if $\sM(x,y) \neq \varnothing$ for $x \in \Ob(\sM')$ and $y \in \Ob(\sM)$, then $y \in \Ob(\sM')$. There are induced cofiber sequences 
    \[ Z_{\fo|_{\sM'},\sE|_{\sM'}} \Longrightarrow Z_{\fo,\sE} \Longrightarrow Z_{\fo|_{\sM/\sM'}, \sE|_{\sM/\sM'}},\]
    and 
    \[ |\sM', \fo|_{\sM'}, \sE|_{\sM'}| \longrightarrow |\sM, \fo, \sE| \longrightarrow |\sM/\sM', \fo|_{\sM/\sM'}, \sE|_{\sM/\sM'}|.\]
    \qed
\end{lem}
\begin{lem}\label{lem:realiation_point}
    Let $\sM_\ast[d]$ be the point flow category shifted by degree $d$ (see \cref{dfn:point_flow_cat}). Let $\fo$ be an $R$-orientation and $\sE$ a local system of $A$-modules on $\sM_\ast[d]$. Then
    \[ |\sM_\ast[d],\fo,\sE| \simeq \sE(p) \wedge_R \fo(p).\]
\end{lem}
\begin{proof}
    This is almost identical to the argument in \cref{exmp:cjspt}.
\end{proof}

\begin{defn}[Flow multimodule with a local system]
    Let $(\sM_i,\fo_i)$ for $i \in \{0,\dots,n\}$ be $R$-oriented flow categories. Let $\sE_1$ and $\sE_0$ be local systems of $A$-modules on $\sM_1$ and $\sM_0$, respectively. An \emph{$R$-oriented flow multimodule with a local system}
    \[ (\sN,\fm,\sE) \colon (\sM_1,\fo_1,\sE_1),(\sM_2,\fo_2),\dots,(\sM_n,\fo_n) \longrightarrow (\sM_0,\fo_0,\sE_0)\]
    is an $R$-oriented flow multimodule $(\sN,\fm) \colon (\sM_1,\fo_1),\ldots, (\sM_n,\fo_n) \to (\sM_0,\fo_0)$ equipped with a natural transformation
    \[ \sE \colon \sN \Longrightarrow F_{A}\left(\sE_1,\sE_0\right),\]
    of spectrally enriched functors $\sM_1^{\text{op}} \wedge \dots \wedge \sM_n^{\text{op}} \wedge \sM_0 \to \mod{A}$, i.e., the data of maps
    \[
    \sE(a_1,\ldots,a_n;a_0) \colon \sE_1(a_1) \wedge \varSigma^\infty_+\sN(a_1,\ldots,a_n;a_0) \longrightarrow \sE_0(a_0),
    \]
    that are compatible with the action maps in $\sN$.
\end{defn}

Given an $R$-oriented flow multimodule with a local system
    \[ (\sN,\fm,\sE) \colon (\sM_1,\fo_1,\sE_1), (\sM_2,\fo_2),\dots,(\sM_n,\fo_n) \longrightarrow (\sM_0,\fo_0,\sE_0).\]
    the natural transformation $\sE$ and the multimodule extension of $R$-module systems (see \cref{dfn:multimodule_ext}) induces maps
    \[
    \sN^{-I}(a_1,\dots,a_n;a_0) \Longrightarrow F_A\left( Z_{\sM_1,\fo_1,\sE_1}(\mu_1(a_1)) \wedge_R  \bigwedge_{i=2}^nZ_{\sM_i,\fo_i,\sE_i}(\mu_i(a_i)), Z_{\sM_0,\fo_0,\sE_0}(\mu_0(a_0))\right),
    \]
    which allows us to define a $\sJ$-module multimorphism $W_{\sN,\fm,\sE}$ similar to \eqref{eq:J-mod_multimorphism_flow_multimod}. The CJS realization of $(\sN,\fm,\sE)$ is defined as the geometric realization of $W_{\sN,\fm,\sE}$, and similar to \cref{lma:functoriality_cjs}, there is an induced map
    \[ |\sN,\fm,\sE| \colon |\sM_1,\fo_1,\sE_1| \wedge_R |\sM_2,\fo_2| \wedge_R \dots \wedge_R |\sM_n,\fo_n| \longrightarrow |\sM_0,\fo_0,\sE_0|.\]

\begin{defn}[Composition local system]\label{dfn:composition_local_system}
    Consider two $R$-oriented flow bimodules with local systems
    \begin{align*}
        (\sN_{12},\fm_{12},\sE_{12}) \colon (\sM_1,\fo_1,\sE_1) &\longrightarrow (\sM_2,\fo_2,\sE_2) \\
        (\sN_{23},\fm_{23},\sE_{23}) \colon (\sM_2,\fo_2,\sE_2) &\longrightarrow (\sM_3,\fo_3,\sE_3).
    \end{align*}
    The composition $(\sN_{23},\fm_{23}) \circ (\sN_{12},\fm_{12})$ (see \cref{lem:r_ori_composition_bimods} and \cref{dfn:composition_flow_bimods}) is equipped with the \emph{composition local system} $\sE_{23}\circ \sE_{12}$ that is defined via the maps
    \[
    \begin{tikzcd}[row sep=scriptsize,column sep=1.5cm]
        \sE_1(a) \wedge \varSigma^\infty_+\sN_{12}(a,b) \wedge \varSigma^\infty_+\sN_{23}(b,c) \rar{\sE_{12}(a,b)\wedge \id} & \sE_{2}(b) \wedge \varSigma^\infty_+\sN_{23}(b,c) \rar{\sE_{23}(b,c)} & \sE_3(c),
    \end{tikzcd}
    \]
    for every $a\in \Ob(\sM_1)$, $b\in \Ob(\sM_2)$ and $c\in \Ob(\sM_3)$.
\end{defn}
\begin{rem}
    \Cref{dfn:composition_local_system} is extended to compositions of $R$-oriented flow multimodules with local systems in the obvious way.
\end{rem}

\begin{defn}[Flow bordism with a local system]
    Let
    \[
    (\sN,\fm,\sE), (\sN',\fm',\sE') \colon (\sM_1,\fo_1,\sE_1), (\sM_2,\fo_2),\dots,(\sM_n,\fo_n) \longrightarrow (\sM_0,\fo_0,\sE_n)
    \]
    be two $R$-oriented flow multimodules with local systems of $A$-modules. An \emph{$R$-oriented flow bordism with a local system} $(\sB,\fn,\sF) \colon (\sN,\fm,\sE) \Rightarrow (\sN',\fm',\sE')$ is an $R$-oriented flow bordism $(\sB,\fn) \colon (\sN,\fm) \Rightarrow (\sN',\fm')$ equipped with a natural transformation
    \[
    \sF \colon \sB \Longrightarrow F_A(\sE,\sE'),
    \]
    of spectrally enriched functors $\sM_1^{\text{op}} \wedge \dots \wedge \sM_n^{\text{op}} \wedge \sM_0 \to \mod{A}$ that are compatible with the action maps in $\sB$, and pulls back to $\sE$ and $\sE'$ along the action maps $\sN(\vec a;a_0) \to \sB(\vec a;a_0)$ and $\sN'(\vec a;a_0) \to \sB(\vec a;a_0)$, respectively.
\end{defn}

Given an $R$-oriented flow bordism with a local system $(\sB,\fn,\sF) \colon (\sN,\fm,\sE) \Rightarrow (\sN',\fm',\sE')$ between two $R$-oriented flow multimodules with local systems, similar to \cref{lma:bordism_gives_homotopy_of_maps} its CJS realization defines a homotopy $|\sN,\fm,\sE| \simeq |\sN',\fm',\sE'|$.
           
\section{Cohen--Jones--Segal geometric realization}\label{sec:cjs}
    In this section we define the geometric realization functor which assigns to an $R$-oriented flow category an $R$-module spectrum. The construction is a generalization of the Pontryagin--Thom construction and is based on the construction introduced by Cohen--Jones--Segal \cite{cohen1995floer,cohen2009floer}. We shall define a version that uses coherent systems of collars as opposed to the original approach of Cohen--Jones--Segal based on neat embeddings of $\left\langle k\right\rangle$-manifolds.  The approach taken here is also the one taken in e.g.\@ \cite{abouzaid2021arnold,large2021spectral}. 
    
    We start with some auxiliary constructions. For $m,n \in \IZ$ with $m \geq n$, define
    \[ J(m,n) \coloneqq [0,\infty]^{\{m>x>n\}}.\]
    There are natural inclusion maps 
    \begin{align}\label{eq:jincl}
        i_{mpn} \colon J(m,p) \times J(p,n) &\hooklongrightarrow J(m,n) \\ \nonumber
        (x_1,x_2)  &\longmapsto (x_1, \infty, x_2), 
    \end{align}
    for any $m > p > n$. These inclusions come with natural collars, i.e., embeddings
    \begin{equation}\label{eq:jcol}
        c_{mpn} \colon J(m,p) \times J(p,n) \times (0,\infty]^{\{p\}} \hooklongrightarrow J(m,n).
    \end{equation}
    Moreover these collars are associative in the sense that the following diagram commutes
    \[
        \begin{tikzcd}[row sep=scriptsize, column sep=1.5cm]
            J(m,p) \times J(p,q) \times J(q,n) \times (0,\infty]^{\{p,q\}} \rar{c_{mpq} \times \id \times \id} \dar{\id \times c_{pqn}} & J(m,q) \times J(q,n) \times (0,\infty]^{\{q\}} \dar{c_{mqn}}\\
            J(m,p) \times J(p,n) \times (0,\infty]^{\{p\}} \rar{c_{mpn}} & J(m,n).
        \end{tikzcd}
    \]
    In particular, we have embeddings
    \begin{equation}\label{eq:embJ}
        \iota_{mpn} \colon J(m,p) \times J(p,n) = J(m,p) \times J(p,n) \times \{1\} \hooklongrightarrow J(m,n),
    \end{equation}
    for any $m > p > n$, which are also associative. Below whenever we refer to an embedding of $J(m,p) \times J(p,n)$ inside $J(m,n)$, we mean $\iota_{mpn}$ unless specified otherwise. Note that the normal bundle of these embeddings are given by the trivial vector bundle $\un \IR^{\{p\}}$.

    Let $\nu_{m,n}$ denote the trivial vector bundle $J(m,n) \oplus \un \IR \to J(m,n)$ and let 
    \[J^\nu(m,n) \coloneqq J(m,n)^{\nu_{m,n}}.\] 
    Notice that we have
    \begin{equation}\label{eq:jnures}
        \nu_{m,n}|_{J(m,p) \times J(p,n)} \oplus \un \IR^{\{p\}} \simeq \nu_{m,p} \oplus \nu_{p,n} .
    \end{equation}
    Thus there are maps
    \[ J^\nu(m,n) \longrightarrow J^\nu(m,p) \wedge J^\nu(p,n).\]
    \begin{defn}\label{dfn:bdry}
        Let $\partial J(m,n) \subset J(m,n)$ be the set of points where at least one coordinate is equal to $\infty$.
    \end{defn}
    \begin{rem}\label{rem:bdry}
        Note that $\partial J(m,n)$ in \cref{dfn:bdry} is different from the boundary locus of $J(m,n)$ considered as a manifold with corners. 
    \end{rem}
    We see that
    \[ (\partial J(m,p) \times J(p,n)) \cup (J(m,p) \times \partial J(p,n)) \subset \partial J(m,n).\]
    Define $J^\nu_{\text{rel }\partial}(m,n) \coloneqq J^\nu(m,n)/ \partial J^\nu(m,n)$, where $\partial J^\nu(m,n) \coloneqq (\partial J(m,n))^{\nu_{m,n}|_{\partial J(m,n)}}$. Thus we have induced maps
    \begin{equation}\label{eq:jincl1}
        \iota_{mpn}^\nu \colon J^\nu_{\text{rel }\partial}(m,n) \longrightarrow J^\nu_{\text{rel }\partial}(m,p) \wedge J^\nu_{\text{rel }\partial}(p,n).
    \end{equation}
    Moreover, these maps are coassociative, in the sense that the following diagram is commutative for any $m > p > q > n$.
    \begin{equation}\label{eq:coassoc_J}
        \begin{tikzcd}[row sep=scriptsize, column sep=1.5cm]
            J^\nu_{\text{rel }\partial}(m,n) \rar{\iota_{mpn}^\nu} \dar{\iota_{mqn}^\nu} & J^\nu_{\text{rel }\partial}(m,p) \wedge J^\nu_{\text{rel }\partial}(p,n) \dar{\id \wedge \iota_{pqn}^\nu} \\
            J^\nu_{\text{rel }\partial}(m,q) \wedge J^\nu_{\text{rel }\partial}(q,n) \rar{\iota_{mpq}^\nu \wedge \id} & J^\nu_{\text{rel }\partial}(m,p) \wedge J^\nu_{\text{rel }\partial}(p,q) \wedge J^\nu_{\text{rel }\partial}(q,n)
        \end{tikzcd}
    \end{equation}
    \begin{defn}\label{dfn:cat_J}
        Let $\sJ$ denote the spectrally enriched category defined by
        \begin{enumerate}
            \item $\Ob(\sJ) = \IZ$
            \item For $m,n \in \IZ$, define 
                \[ \sJ(m,n) \coloneqq \begin{cases}
                    \ast, &\text{if } m <n\\
                    \IS &\text{if } m=n\\
                    D(J^\nu_{\text{rel }\partial}(m,n)), &\text{otherwise.}
                \end{cases} \]
            where $D(-) = F(-,\IS)$ denotes the Spanier--Whitehead dual.
            \item The composition is given by dualizing the maps $\iota^\nu_{mpn}$ in \eqref{eq:jincl1}.
        \end{enumerate}
    \end{defn}
    \begin{defn}\label{dfn:r_linear_J-mod}
        Let $R$ denote a commutative ring spectrum. An \emph{$R$-linear $\sJ$-module} is a spectrally enriched functor
            \[ Z \colon \sJ \longrightarrow \mod{R}.\]
    \end{defn}
    More explicitly, the data of such a functor consists of an $R$-module $Z(i)$ for each $i\in \IZ$ and maps of $R$-modules $Z(i) \wedge \sJ(i,j) \to Z(j)$ for all $i,j\in \IZ$. The composition $Z(i) \simeq Z(i) \wedge \IS = Z(i) \wedge \sJ(i,i) \to Z(i)$ is the identity.
    \begin{defn}[Two sided bar construction] Given functors $Z \colon \sJ \to \mod{R}$ and $Y \colon \sJ^{\mathrm{op}} \to \Spec$, define the \emph{two sided bar construction} $B(Z,\sJ,Y)$  to be the geometric realization of the simplicial $R$-module $B_\bullet(Z,\sJ,Y)$ that is defined on objects by
        \[ [k] \longmapsto \bigvee_{a_0 \geq a_1 \geq \cdots \geq a_k} (Z(a_0) \wedge \sJ(a_0, a_1) \wedge \cdots \wedge \sJ(a_{k-1},a_k) \wedge Y(a_k)) .\]
    \end{defn}
    \begin{rem}\label{rem:two_sided_bar_constr}
        \begin{enumerate}
            \item The two sided bar construction $B(X,\sJ,Y)$ is a model for the derived tensor product $X \otimes_\sJ Y$, see \cite[Lemma A.105 and Proposition A.106]{abouzaid2021arnold} (cf.\@ \cite[Propositions IV.7.5 and IV.7.6]{elmendorf1997rings}). 
            \item Note that for a given $\ell \in \IZ$ the compositions in $\sJ$ yields a functor $\sJ^{\mathrm{op}} \to \Spec$ defined on objects by $k\mapsto \sJ(k,\ell)$.
        \end{enumerate}
    \end{rem}
    For $m\geq q$, define 
    \[ V_q(m) = [0,\infty]^{\{m > x \geq q\}}\]
    and denote by $\partial V_q(m)$ the subset of $V_q(m)$ where at least one of the coordinates is equal to $\infty$; define $\partial \mathrm{pt} = \varnothing$. Similar to the maps \eqref{eq:jincl} there are composition maps
    \begin{equation}\label{eq:v_j_compos}
        J(m,n) \times V_q(n) \longrightarrow V_q(m).
    \end{equation}
    Denote by $\sV_q(m)$ the Spanier--Whitehead dual $D(\varSigma^\infty V_q(m)/\partial V_q(m))$.
    Note that 
    \[ \sV_q(m) = \begin{cases} \IS & \text{if } m = q\\
    \ast & \text{if } m > q\end{cases}.\]
    Note that \eqref{eq:v_j_compos} preserve boundary components. Thus Spanier--Whitehead dual of the Thom collapse along \eqref{eq:v_j_compos} gives maps
    \begin{equation}\label{eq:v}
        \sJ(m,n) \wedge \sV_q(n) \longrightarrow \sV_q(m).
    \end{equation}
    For later reference, we also define the space $V^q(m) = [0,\infty]^{\{q \geq x > m\}}$, with associated boundary $\partial V^q(m)$ and the spectrum $\sV^q(m)$ defined analogously to $\sV_q(m)$.
    \begin{defn} For any $q\in \IZ$, define $\sV_q \colon \sJ^{\mathrm{op}} \to \Spec$ to be the spectrally enriched functor given on objects by 
    \[ 
        m \longmapsto \begin{cases}  \ast &\text{if } m<q,\\
        \sV_q(m) &\text{otherwise,}
        \end{cases}
    \]
    and on morphisms via the maps \eqref{eq:v}.
    \end{defn}
    \begin{defn}[$q$-realization of an $R$-linear $\sJ$-module]\label{defn:q-realization}
        Let $q\in \IZ$. Given an $R$-linear $\sJ$-module $Z\colon \sJ \to \mod{R}$, the \emph{$q$-realization} of $Z$ is defined to be the two sided bar construction
        \[
        |Z|_q \coloneqq B(Z, \sJ, \sV_q).
        \]
    \end{defn}

    For $q > p$, and any $m \geq q$ there is a map $\sV_p(m) \to \sV_q(m)$ induced by the map $V_q(m) \to V_p(m)$ given by extension by $0$. Hence we get induced maps on bar constructions 
    \[ |Z|_p \longrightarrow |Z|_q \]
    which are well-defined up to homotopy. Moreover, for $q > p > r$, the following diagram is homotopy commutative
    \[
        \begin{tikzcd}[row sep=scriptsize, column sep=scriptsize]
            \vert Z \vert_r \rar \drar & \vert Z \vert_p \dar \\
            {} & \vert Z \vert_q
        \end{tikzcd}.
    \]
    \begin{defn}[Geometric realization of an $R$-linear $\sJ$-module]\label{defn:geometric_realization_J-mod}
        Suppose that $Z \colon \sJ \to \mod R$ is an $R$-linear $\sJ$-module. The \emph{geometric realization} of $Z$ is defined as follows
        \[
        |Z| \coloneqq \holim_q |Z|_q.
        \]
    \end{defn}
    \begin{rem}\label{rem:pro-spec}
        We say that an $R$-linear $\sJ$-module $Z$ is \emph{bounded below} if $Z(q)= \ast$ for all $q\ll 0$. In this case $|Z|_{q-1} \to |Z|_q$ is a weak equivalence for all $q\ll 0$ and hence the sequence of $R$-modules $|Z|_q$ stabilizes as $q \to -\infty$. Thus, if $Z$ is bounded below we have
        \[ |Z| = |Z|_q, \quad \text{for} \  q \ll 0.\]
    \end{rem}
    
    \begin{exmp}\label{exmp:cjspt}
    Let $Z_\ast[d] \colon \sJ \to \mod{R}$ be the $R$-linear $\sJ$-module given on objects by 
    \[  Z_\ast[d](n) \coloneqq \begin{cases} R  & \text{if } n=d\\ \ast &\text{otherwise}\end{cases}.\]
    From \cref{defn:geometric_realization_J-mod} it follows that 
        \[ |Z_\ast[d]| \simeq R .\]
    \end{exmp}
    \begin{rem}\label{rem:CJS_shift}
        More generally, for any $\sJ$-module $Z$ we have $|Z[d]| \simeq |Z|$, where $Z[d]$ is the $\sJ$-module defined on objects by $Z[d](n) \coloneqq Z(n-d)$.
    \end{rem}

    \begin{defn}[Point flow category]\label{dfn:point_flow_cat}
        Let $d \in \IZ$. The \emph{point flow category} concentrated in degree $d$ is the flow category $\sM_\ast[d]$ that is the $R$-oriented flow category defined as follows:
        \begin{itemize}
            \item The set of objects is $\Ob(\sM_\ast[d]) \coloneqq \{p\}$, and $\mu(p) = d$.
            \item The morphisms are $\sM_\ast[d](p,p) \coloneqq \varnothing$.
            \item The $R$-orientation is given by $\fo_\ast(p) \coloneqq R$.
        \end{itemize}
        We use the notation $\sM_\ast \coloneqq \sM_\ast[0]$.
    \end{defn}
    \begin{rem}
        By \cref{exmp:cjspt} it follows that the CJS realization of the point flow category $\sM_\ast[d]$ is equal to $R$.
    \end{rem}

    \subsection{Functoriality of the geometric realization functor}\label{sec:functoriality_geom_realiz}
    We now discuss functoriality properties of the geometric realization functor with respect to a certain notion of morphism. The construction in this section is entirely analogous to \cite[Section 3]{large2021spectral} with the slight difference being that we work with a notion of multimodules which allows for multiple inputs and a single output, whereas loc.\@ cit.\@ works with multimodules that have a single input and multiple outputs.

    \begin{notn}\label{notn:k_index}
        Let $\vec m \coloneqq (m_1,\ldots,m_p) \in \IZ^p$ and $m_0,m_i' \in \IZ$ for some $i \in \{1,\ldots,p\}$.
        \begin{enumerate}
            \item Let 
             \[
                \vec m_i \coloneqq (m_1,\ldots,m_{i-1},m_i',m_{i+1},\ldots,m_p).
             \]
             \item Let
             \[
                S(\vec m; m_0) \coloneqq \left\{ (r_0, \ldots, r_p) \mid  m_0 \leq r_0,\; r_i \leq m_i, \; \sum_{i=1}^p r_i \geq r_0 \right\}.
             \]
        \end{enumerate}
    \end{notn}
    
    Let $w = m_1+ \dots +m_p - m_0 \geq 0$. Define 
    \[ K(\vec m;m_0) \subset \prod_{i=1}^p[0,\infty]^{\{m_i> x \geq m_i - w\}}  \times [0,\infty]^{\{w+m_0 \geq x > m_0\}}\]
    to be the subset 
    \begin{align}\label{eq:K_abstract_cube}
        K(\vec m;m_0) \coloneqq &\bigcup_{ (r_0, \ldots, r_p) \in S(\vec m;m_0)}  \left( \prod_{i=1}^p V_{r_i}(m_i)  \right) \times V^{r_0}(m_0)\\ \nonumber
        &= \bigcup_{ (r_0, \ldots, r_p) \in S(\vec m;m_0)} \left( \prod_{i=1}^p [0,\infty]^{\{m_i > x \geq r_i\}} \right) \times [0,\infty]^{\{r_0\geq x > m_0\}}. 
    \end{align}
    \begin{rem} 
    $K(\vec m;m_0)$ coincides with $\mathcal R(w)$ as defined in \cite[Definition 3.10]{large2021spectral} with $\mathcal R=\text{pt}$.
    \end{rem}
    Similar to the inclusion maps \eqref{eq:jincl} we have composition maps
    \begin{align}\label{eq:Rincl}
        K(\vec m;m_0') \times J(m_0',m_0) & \hooklongrightarrow K(\vec m;m_0) \\ \label{eq:Rincl2}
        J(m_i,m_i') \times K(\vec m_i;m_0) & \hooklongrightarrow K(\vec m;m_0), \quad i\in \{1,\ldots,p\}
    \end{align}
    that commute with each other are compatible with the inclusions \eqref{eq:jincl}. Moreover, similar to the collars \eqref{eq:jcol}, there are collars
    \begin{align*}
        K(\vec m;m_0') \times J(m_0',m_0) \times (0,\infty]^{\{m_0'\}} &\hooklongrightarrow K(\vec m;m_0)\\
        J(m_i,m_i') \times K(\vec m_i;m_0) \times (0,\infty]^{\{m_i'\}} &\hooklongrightarrow K(\vec m;m_0), \quad i\in \{1,\ldots,p\}
    \end{align*}
    that are associative and compatible with the inclusions \eqref{eq:jincl}, \eqref{eq:Rincl}, \eqref{eq:Rincl2}, and the collars \eqref{eq:jcol}. These induce embeddings 
    \begin{align}\label{eq:Remb}
        K(\vec m;m_0') \times J(m_0',m_0) \times \{1\} &\hooklongrightarrow K(\vec m;m_0) \\ \label{eq:Remb2}
        J(m_i,m_i') \times K(\vec m_i;m_0) \times \{1\} &\hooklongrightarrow K(\vec m;m_0), \quad i\in \{1,\ldots,p\}
    \end{align}
    that are associative and compatible with the embeddings \eqref{eq:embJ}. Below, whenever we refer to embeddings of $K(\vec m;m_0') \times J(m_0',m_0) $ or $J(m_i,m_i') \times K(\vec m_i;m_0)$ into $K(\vec m; m_0)$, we mean the embeddings \eqref{eq:Remb} or \eqref{eq:Remb2}, unless otherwise specified.

    The normal bundles of these embeddings coincide with the pullbacks of the bundles $\nu_{m_0',m_0}$ and $\nu_{m_i,m_i'}$ as in \eqref{eq:jnures} along the projection onto the $J(-,-)$ factors. From these we get maps
    \begin{align*}
       K(\vec m;m_0) &\longrightarrow J^\nu(m_i,m_i') \wedge K(\vec m_i;m_i), \quad i\in \{1,\ldots,p\}\\
       K(\vec m;m_0) &\longrightarrow K(\vec m;m_0') \wedge J^\nu(m_0',m_0).
    \end{align*}
    Similar to \cref{dfn:bdry} we have the following.
    \begin{defn}\label{dfn:rbdry}
        Define $\partial K(\vec m;m_0) \subset K(\vec m; m_0)$ to be the set of points where at least one coordinate is equal to $\infty$.
    \end{defn}
    We see that
    \begin{align*} 
        &(\partial K(\vec m;p) \times J(p,n)) \cup (K(\vec m;p) \times \partial J(p,n)) \\
        &\qquad \cup \bigcup_{i=1}^p((\partial J(m_i,m_i') \times K(\vec m_i;n)) \cup (J(m_i,m_i') \times \partial K(\vec m_i,n)))\subset \partial K(\vec m;n).
  \end{align*}
      Define $K_{\text{rel } \partial}(\vec m;n) \coloneqq K(\vec m;n)/ \partial K(\vec m;n)$. Via the Pontryagin--Thom collapse map we obtain induced maps
    \begin{align}\label{eq:rspecbac1}
        K_{\text{rel } \partial}(\vec m;m_0) &\longrightarrow J^\nu_{\text{rel }\partial}(m_i;m_i') \wedge K_{\text{rel } \partial}(\vec m_i;m_0), \quad i\in \{1,\ldots,p\} \\
        \label{eq:rspecbac2}
        K_{\text{rel } \partial}(\vec m;m_0) &\longrightarrow K_{\text{rel } \partial}(\vec m; m_0') \wedge J^\nu_{\text{rel }\partial}(m_0',m_0) 
    \end{align}
    These maps are coassociative with respect to the maps \eqref{eq:jincl1}. Recall the definition of the spectrally enriched category $\sJ$ from \cref{dfn:cat_J}. Let $\sJ \wedge \sJ$ denote the spectrally enriched category with $\Ob(\sJ \wedge \sJ) \coloneqq \Ob(\sJ) \times \Ob(\sJ)$, and morphisms $(\sJ \wedge \sJ)((m_1,n_1),(m_2,n_2)) \coloneqq \sJ(m_1,n_1) \wedge \sJ(m_2,n_2)$. In analogy with \cref{dfn:r_linear_J-mod}, the above allows us to form an $\IS$-linear $\sJ$-multimodule, meaning a spectrally enriched functor
    \[\sK \colon \left(\bigwedge_{i=1}^p\sJ\right)^{\mathrm{op}} \wedge \sJ \longrightarrow \Spec\]
    given on objects by
    \[ \sK(\vec m;m_0) \coloneqq D(K_{\text{rel } \partial}(\vec m;m_0)),\]
    where $D$ denotes the Spanier--Whitehead dual. The dual of the maps \eqref{eq:rspecbac1} and \eqref{eq:rspecbac2} yield maps
    \[
        \left(\bigwedge_{i=1}^p \sJ(m_i,m_i')\right) \wedge \sK(\vec m';m_0') \wedge \sJ(m_0',m_0) \longrightarrow \sK(\vec m;m_0),
    \]
    which by adjunction yields a map that defines $\sK$ on morphisms.
    \begin{rem}\label{rem:K_and_R}
        $\sK(\vec m;m_0)$ coincides with $\mathscr R(w)$ for $\mathcal R = \text{pt}$ as defined in \cite[Definition 3.13]{large2021spectral}.
    \end{rem}

    \begin{notn}
        Given a functor $F \colon (\IZ^{\mathrm{op}})^p \to \Spec$ and an integer $q\in \IZ$, we denote by 
        \[ \lim_{q_1+\dots+q_p \geq q} F \quad \text{and} \quad \colim_{q_1+\dots+q_p \geq q} F,\]
        the limit and colimit of $F$, respectively, indexed by the overcategory $(+) \downarrow q$, where $(+) \colon (\IZ^{\mathrm{op}})^p \to \IZ^{\mathrm{op}}$ is the functor $(q_1,\dots,q_p) \mapsto q_1 + \dots + q_p$.
    \end{notn}
     
    \begin{prop}\label{prop:rcon}
        For each $\vec m \in \IZ^p$ and $q\in \IZ$, there is a weak equivalence of spectra
        \[B(\sK(\vec m;-), \sJ, \sV_q) \overset{\simeq}{\longrightarrow} \lim_{q_1+\dots+ q_p \geq q} \bigwedge_{i=1}^p \sV_{q_i}(m_i),\] 
    \end{prop}
    \begin{proof} 
        The proof follows a similar outline as \cite[Proposition 3.16]{large2021spectral}.

        Define an intermediate space 
        \[ KV_q(\vec m) \coloneqq \bigcup_{ (r_0, \ldots, r_p) \in S(\vec m;q)} \left( \prod_{i=1}^p V_{r_i}(m_i) \right) \times \partial [0,\infty]^{\{r_0\geq x \geq q\}},
        \]
        inside $\prod_{i=1}^p[0,\infty]^{\{m_i> x \geq m_i - w\}}  \times \partial [0,\infty]^{\{w+q \geq x \geq q\}}$, where $S(\vec m;q)$ is as in \cref{notn:k_index} and $w \coloneqq m_1 + \dots + m_p - q$. Set 
        \[\partial KV_q(\vec m) \coloneqq \bigcup_{ (r_0, \ldots, r_p) \in S(\vec m;q)} \partial \left( \prod_{i=1}^p V_{r_i}(m_i) \right) \times \partial[0,\infty]^{\{r_0\geq x \geq q\}}, \]
        where $\partial \left( \prod_{i=1}^p V_{r_i}(m_i) \right) \coloneqq \bigcup_{i=1}^p \left(V_{r_1}(m_1) \times \dots \times \partial V_{r_i}(m_i) \times \dots \times V_{r_p}(m_p) \right)$, and define
        \[ \sKV_q(\vec m) \coloneqq D(KV_q(\vec m) / \partial KV_q(\vec m)).\] 
        There is a map \begin{equation}\label{eq:kv_comparison2}
            \colim_{q_1 + \dots + q_p \geq q} \prod_{i=1}^p V_{q_i}(m_i) \longrightarrow KV_q(\vec m)
        \end{equation}
        defined by the inclusions
        $\prod_{i=1}^p V_{q_i}(m_i) \simeq \prod_{i=1}^p V_{q_i}(m_i) \times \{0\}^{\{r_0 \geq x > q\}} \times \{\infty \} \hookrightarrow KV_q(\vec m)$. Note that the colimit in the domain of \eqref{eq:kv_comparison2} maps homeomorphically onto its image which is given by 
        \begin{equation}\label{eq:kv_comparison3}
            \bigcup_{q_1+\dots+q_p = q} \prod_{i=1}^p V_{q_i}(m_i) \times \{0\}^{\{r_0 \geq x > q\}} \times \{\infty\} \subset KV_q(\vec m).
        \end{equation}
        The map \eqref{eq:kv_comparison2} preserves boundaries and hence there is an induced map on relative cochains
        \begin{equation}\label{eq:kv_comparison}
            \sKV_q(\vec m) \longrightarrow \lim_{q_1+\dots+q_p \geq q} \bigwedge_{i=1}^p \sV_{q_i}(m_i).
        \end{equation}
        We show that \eqref{eq:kv_comparison} is a weak equivalence. 
        
        To see this, first notice that the inclusion \eqref{eq:kv_comparison3} admits a deformation retraction, given by retracting the coordinates of $\partial[0,\infty]^{\{r_0 \geq x \geq q\}}$ that are not equal to $\infty$, to zero, meaning that \eqref{eq:kv_comparison2} is a weak equivalence. Moreover, we note that this retract restricts to a deformation retract of 
        \begin{equation}\label{eq:kv_comparison4}
        \colim_{q_1+\dots+q_p \geq q} \partial \left(\prod_{i=1}^p V_{q_i}(m_i) \right) \longrightarrow \partial KV_q(\vec m).
        \end{equation}
        Using the fact that the maps \eqref{eq:kv_comparison2} and \eqref{eq:kv_comparison4} are weak equivalences it follows that \eqref{eq:kv_comparison} is a weak equivalence, as claimed.

        To finish the proof it remains to prove there is a weak equivalence
        \[\sKV_q(\vec m) \simeq B(\sK(\vec m;-),\sJ,\sV_q) .\]
        Note that there is a composition map $V^m(n) \times V_q(n) \to [0,\infty]^{\{m \geq x \geq q\}}$ defined by $(x_1,x_2) \mapsto (x_1,\infty,x_2)$, that together with \eqref{eq:v_j_compos} and \eqref{eq:Rincl} gives a natural map
        \[ \varPhi\colon K(\vec m;q_0) \times J(q_0,q_1) \times \dots \times J(q_{n-1},q_n) \times V_q(q_n) \longrightarrow KV_q(\vec m).\]
        The Pontryagin--Thom collapse induces maps
        \begin{equation}\label{eq:kv_comparison8}
            KV_q(\vec m) \longrightarrow  \left( K(\vec m,q_0) \times J(q_0,q_1) \times \dots \times J(q_{n-1},q_n) \times V_q(q_n) \right)^{\nu_{\varPhi}},
        \end{equation}
        where $\nu_\varPhi$ is the normal bundle of the embedding $\varPhi$, which is the trivial bundle of rank $n$. These maps moreover preserve boundaries. Taking Spanier--Whitehead duals, we get maps
        \[ \sK(\vec m;q_0) \wedge \sJ(q_0,q_1) \wedge \dots \wedge \sJ(q_{n-1},q_n) \wedge \sV_q(q_n) \longrightarrow \sKV_q(\vec m).\]
        It is not difficult to verify that these maps are compatible with simplicial face maps, giving an induced map on the geometric realization        \begin{equation}\label{eq:kv_comparison5} 
            B(\sK(\vec m;-),\sJ,\sV_q) \longrightarrow \sKV_q(\vec m),
        \end{equation}
        where we consider the target as the geometric realization of the constant simplicial object. 

        Note that $B(\sK(\vec m;-),\sJ,\sV_q)$ carries an increasing filtration given by 
        \[ F^\ell B(\sK(\vec m;-),\sJ,\sV_q) \coloneqq B(\sK(\vec m;-)|_{\sJ^{\leq \ell}},\sJ^{\leq \ell},\sV_q|_{(\sJ^{\leq \ell})^{\mathrm{op}}})\]
        where $\sJ^{\leq \ell}$ is the full subcategory of $\sJ$ spanned by $\IZ_{\leq \ell}$. The associated graded in degree $\ell$ is given by 
        \begin{equation}\label{eq:kv_comparison7}
            \Gr^\ell B(\sK(\vec m;-),\sJ,\sV_q) \simeq \begin{cases} \sK(\vec m; \ell) &\text{if } \ell \geq q\\
            \ast &\text{if } \ell  < q\end{cases}
        \end{equation}
        On the other hand, we define an increasing exhaustive filtration on $\sKV_q(\vec m)$ as follows. First define a decreasing exhaustive filtration on $KV_q(\vec m)$ by letting $F_\ell KV_q(\vec m)$ (for any $\ell \in \IZ$) be the subset of $KV_q(\vec m)$ such that
        \begin{itemize}
            \item $r_0 \geq \ell+1 \geq q$, and
            \item there exists some $i$ with $r_0 \geq i \geq \ell+1$ such that the $i$-th coordinate in the factor $\partial [0,\infty]^{\{r_0 \geq x \geq q\}}$ is equal to $\infty$.
        \end{itemize}
        Next, for any $\ell \in \IZ$ let
        \[
        F_{\ell, \text{rel }\partial} KV_q(\vec m) \coloneqq F_\ell KV_q(\vec m)/ (\partial KV_q(\vec m) \cap F_\ell KV_q(\vec m)),
        \]
        and define
        \[ F^\ell_{\text{rel }\partial} KV_q(\vec m) \coloneqq KV_{q,\text{rel }\partial}(\vec m)/F_{\ell, \text{rel }\partial} KV_q(\vec m).\]
        For any $\ell \in \IZ$, define
        \[ F^\ell \sKV_q(\vec m) = D(\varSigma^\infty F^\ell_{\text{rel }\partial} KV_q(\vec m)).\]
        This defines an increasing exhaustive filtration on $\sKV_q(\vec m)$. The associated graded in degree $\ell$ is given by
        \begin{align*}
            \Gr^\ell \sKV_q(\vec m) &= \hocofib(F^{\ell-1}\sK\sV_q(\vec m) \to F^{\ell}\sK\sV_q(\vec m)) \\
            &\simeq D\left( \hofib( \varSigma^\infty F^{\ell}_{\text{rel }\partial} KV_q(\vec m) \to \varSigma^\infty F^{\ell-1}_{\text{rel }\partial}KV_q(\vec m) ) \right)\\
            &\simeq D\left( \hocofib( \varSigma^\infty F_{\ell, \text{rel }\partial} KV_q(\vec m) \to \varSigma^\infty F_{\ell-1, \text{rel }\partial} KV_q(\vec m) ) \right)
        \end{align*}
        It can be observed that
        \[
        \hocofib(\varSigma^\infty F_{\ell-1, \text{rel }\partial} KV_q(\vec m) \to \varSigma^\infty F_{\ell, \text{rel }\partial} KV_q(\vec m)) \simeq \begin{cases}
            \varSigma^\infty K_{\text{rel }\partial}(\vec m,\ell) & \text{if }\ell \geq q \\
            \ast & \text{if }\ell < q
        \end{cases},
        \]
        implying that 
        \begin{equation}\label{eq:kv_comparison6}
            \Gr^\ell \sKV_q(\vec m) \simeq \begin{cases} \sK(\vec m; \ell) &\text{if } \ell \geq q\\
            \ast &\text{if } \ell  < q\end{cases}.
        \end{equation}
        Further, for $\ell \geq q$, observe that the Pontryagin--Thom collapse map \eqref{eq:kv_comparison8} factor as
        \[
            \begin{tikzcd}
                KV_{q,\text{rel }\partial}(\vec m) \rar \dar &\parbox{7cm}{$K_{\text{rel }\partial}(\vec m,q_0) \times J_{\text{rel }\partial}^\nu(q_0,q_1) \\ \quad \times \dots \times J_{\text{rel }\partial}^\nu(q_{n-1},q_n)\times V_{q,\text{rel }\partial}(q_n)$}\\
                KV_{q,\text{rel }\partial}(\vec m)/ F_{\ell, \text{rel }\partial} KV_q(\vec m) \ar{ru} & 
            \end{tikzcd}.
        \]
        From this it follows that the map \eqref{eq:kv_comparison5} preserves the filtrations. It can further be verified that it induces identity on the associated graded under the identifications \eqref{eq:kv_comparison7} and \eqref{eq:kv_comparison6}. 

        This completes the proof of the fact that there is an weak equivalence
        \[ \sKV_q(\vec m) \overset{\simeq}{\longrightarrow} B(\sK(\vec m;-),\sJ,\sV_q),\]
        in turn completing the proof of the proposition.
    \end{proof}
    Now, given $R$-linear $\sJ$-modules $Z_0, \ldots, Z_p$ we consider the spectrally enriched functor 
    \begin{align*}
        \sF_{Z_1,\ldots,Z_p;Z_0} \colon \left(\bigwedge_{i=1}^p\sJ\right)^{\mathrm{op}} \wedge \sJ &\longrightarrow \Spec\\
        (\vec m; m_0) &\longmapsto F_R(Z_1(m_1) \wedge \cdots \wedge Z_p(m_p);Z_0(m_0))
    \end{align*}
    \begin{defn}\label{defn:jmodmor}
        Let $Z_0, \ldots,Z_p$ be $R$-linear $\sJ$-modules. A \emph{$\sJ$-module multimorphism} is a natural transformation $\sK \Rightarrow \sF_{Z_1,\ldots, Z_p;Z_0}$.
    \end{defn}
    \begin{lem}\label{lem:functoriality_geom_realiz}
        Let $Z_0,\ldots,Z_p$ be $R$-linear $\sJ$-modules and let $W \colon \sK \Rightarrow \sF_{Z_1,\ldots, Z_p;Z_0}$ be a $\sJ$-module multimorphism. Then there is an induced map of geometric realizations
        \[ |W| \colon |Z_1| \wedge_R \cdots \wedge_R |Z_p| \longrightarrow |Z_0|.\]
    \end{lem}
    \begin{proof}
        We have maps
        \[ \bigwedge_{i=1}^p |Z_i| = \bigwedge_{i=1}^p \holim_q |Z_i|_q \longrightarrow \holim_q \lim_{q_1+\dots+q_p \geq q} \bigwedge_{i=1}^p |Z_i|_{q_i} .\]
        Using the definition of the product structure on simplicial objects, and the fact that geometric realizations preserve products, it follows that
        \[ \lim_{q_1+\dots+q_p \geq q} \bigwedge_{i=1}^p |Z_i|_{q_i} \simeq B\left(\bigwedge_{i=1}^p Z_i,  \sJ^{\wedge p},\lim_{q_1+\dots+q_p \geq q} \bigwedge_{i=1}^p \sV_{q_i}\right).\]
        Since the two sided bar construction preserves weak equivalences (see \cite[Section A.3.4]{abouzaid2021arnold}), using \cref{prop:rcon}, we obtain a weak equivalence
        \[  B\left(\bigwedge_{i=1}^p Z_i, \sJ^{\wedge p}, B(\sK,\sJ,\sV_q)\right) \overset{\simeq}{\longrightarrow} B\left(\bigwedge_{i=1}^p Z_i,  \sJ^{\wedge p},\lim_{q_1+\dots+q_p \geq q} \bigwedge_{i=1}^p \sV_{q_i}\right).\]
        We then have
        \[ B\left(\bigwedge_{i=1}^p Z_i, \sJ^{\wedge p}, B(\sK,\sJ,\sV_q)\right) \simeq B\left(B\left(\bigwedge_{i=1}^p Z_i,\sJ^{\wedge p},\sK\right), \sJ, \sV_q\right) .\]
        Finally, the $\sJ$-module multimorphism structure of $W$ induces a natural transformation
        \[ B\left(\bigwedge_{i=1}^p Z_i, \sJ^{\wedge p},\sK \right) \Longrightarrow Z_0,\]
        thus completing the proof.
    \end{proof}
    
    \subsection{CJS realization of a flow category}\label{subsec:flowtojmod}
    Let $(\sM,\fo)$ be an $R$-oriented flow category (see \cref{defn:rorcoh}) with the induced $R$-module system $\fo \colon \sM^{-I} \to \mod{R}$ (see \cref{def:kmodsys}). In this section we explain how to define the CJS realization of $(\sM,\fo)$.
    \begin{defn}\label{dfn:flow_cat_coherent_collars}
        A \emph{coherent system of collars} on a flow category $\sM$ consists of a smooth embedding
        \begin{equation}\label{eq:M_coh_collars}
            \chi_{abc} \colon \sM(a,b) \times \sM(b,c) \times (0,\infty]^{\{b\}} \longrightarrow \sM(a,c)
        \end{equation}
        for each $a,b,c \in \Ob(\sM)$ such that for any $a,b,c,d \in \Ob(\sM)$, the following diagram commutes
        \begin{equation}\label{eq:collar_comp}
            \begin{tikzcd}[row sep=scriptsize, column sep=1cm]
                \sM(a,b) \times \sM(b,c) \times \sM(c,d) \times (0,\infty]^{\{b, c\}} \rar{\chi_{abc} \times \id} \dar{\id \times \chi_{bcd}} & \sM(a,c) \times \sM(c,d) \times (0,\infty]^{\{c\}} \dar{\chi_{acd}} \\
                \sM(a,b) \times \sM(b,d) \times (0,\infty]^{\{b\}} \rar{\chi_{abd}} & \sM(a,d)
            \end{tikzcd}.
        \end{equation}
        Moreover, we require that the collars $\chi_{abd}$ and $\chi_{acd}$ are disjoint unless $\sM(b,c)$ is non-empty, in which case their intersection coincides with the image of the compositions along sides of the square \eqref{eq:collar_comp}.
    \end{defn}
    \begin{asmpt}\label{asmp:flowcjs}
            The flow category $\sM$ is equipped with a coherent system of collars.
    \end{asmpt}
    The coherent system of collars induces a collection of embeddings  
    \begin{equation}\label{eq:collar_emb} 
        i_{abc} \colon \sM(a,b) \times \sM(b,c) \cong \sM(a,b) \times \sM(b,c) \times \{1\} \longrightarrow \sM(a,c) 
    \end{equation}
    and it follows from \eqref{eq:collar_comp} that these embeddings are associative, i.e.\@ that the following diagram is commutative.
    \[
        \begin{tikzcd}[row sep=scriptsize, column sep=1.5cm]
            \sM(a,b) \times \sM(b,c) \times \sM(c,d) \rar{i_{abc} \times \id} \dar{\id \times i_{bcd}} & \sM(a,c) \times \sM(c,d) \dar{i_{acd}} \\
            \sM(a,b) \times \sM(b,d) \rar{i_{abd}} & \sM(a,d)
        \end{tikzcd}
    \]
    Note that the normal bundle of $\sM(a,b) \times \sM(b,c)$ inside $\sM(a,b)$ is given by the trivial bundle with fiber $\IR^{\{b\}}$.

    Let $\nu_{a,b}$ denote the vector bundle $\sM(a,b) \oplus \un \IR \to \sM(a,b)$. Similar to \eqref{eq:jincl1}, we define $\mathcal M^\nu(a,b) \coloneqq \mathcal M(a,b)^{\nu_{a,b}}$ and $\mathcal M^\nu_{\text{rel }\partial}(a,b) \coloneqq \mathcal M^\nu(a,b)/ \partial \mathcal M^\nu (a,b)$ where $\partial \mathcal M^\nu (a,b) \coloneqq (\partial \mathcal M(a,b))^{\nu_{a,b}|_{\partial \sM(a,b)}}$. Thus we have induced maps
    \begin{equation}\label{eq:embmod}
        \mathcal M^\nu_{\text{rel }\partial}(a,b) \longrightarrow \mathcal M^\nu_{\text{rel }\partial}(a,c) \wedge \mathcal M^\nu_{\text{rel }\partial}(c,b),
    \end{equation}
    that are coassociative in the sense that a diagram similar to that of \eqref{eq:coassoc_J} commutes. 
    \begin{rem} 
    Note that unlike the notation used in \cref{dfn:bdry}, $\partial \sM(a,b)$ here refers the manifold theoretic boundary of $\sM(a,b)$ (cf.\@ \cref{rem:bdry}).
    \end{rem}
   
    Now, notice that the coherent collars on $\sM$ induce maps
    \begin{equation}\label{eq:modcol}
        \gamma_{ab} \colon \sM(a,b) \longrightarrow J(\mu(a),\mu(b)),
    \end{equation}
    such that for $\mu(a) > p > \mu(b)$, the $p$-th component of this map is given by 
    \begin{itemize}
        \item the collar parameter inside collars of boundary strata of the form $\sM(a,c) \times \sM(c,b)$ with $\mu(c)=p$, and
        \item $0$ on points outside these collars.
    \end{itemize}
    The maps \eqref{eq:modcol} are compatible with the coherent system of collars on $\sM$ and $\sJ$ in the sense that the following diagram commutes
    \[
        \begin{tikzcd}[row sep=scriptsize, column sep=1.5cm]
            \sM(a,b) \times \sM(b,c) \times (0,\infty]^{\{b\}} \rar{\chi_{abc}} \dar{\gamma_{ac} \times \gamma_{cb} \times \id} & \sM(a,c) \dar{\gamma_{ac}} \\
            J(\mu(a),\mu(b)) \times J(\mu(b), \mu(c)) \times (0,\infty]^{\{\mu(b)\}} \rar{c_{\mu(a) \mu(b) \mu(c)}} & J(\mu(a),\mu(c))
        \end{tikzcd}.
    \]
    There are now induced maps
    \[\pi_{ab} \colon \mathcal M^\nu_{\text{rel }\partial}(a,b) \longrightarrow J^\nu_{\text{rel }\partial}(\mu(a),\mu(b)).\]
    Since the maps \eqref{eq:modcol} intertwine the embeddings of the lower strata $\chi_{abc}$ \eqref{eq:M_coh_collars} and $c_{\mu(a)\mu(b) \mu(c)}$ \eqref{eq:jcol}, it follows that $\pi_{ab}$ is intertwines the cocomposition maps on $\sM^\nu_{\text{rel }\partial}(a,b)$ and $J^\nu_{\text{rel }\partial}(\mu(a),\mu(b))$ defined in \eqref{eq:embmod} and \eqref{eq:jincl1}, respectively.

    Consider the dual maps
    \[ D(\pi) \colon D(J^\nu_{\text{rel }\partial}(\mu(a),\mu(b))) \longrightarrow D(\sM^\nu_{\text{rel }\partial}(a,b)).\]
    By Atiyah duality, $D(\sM^\nu_{\text{rel }\partial}(a,b)) \simeq \sM^{-TM - \nu}(a,b) = \sM^{-I}(a,b)$, see \cref{dfn:Thom_spec_flow_cat}. Thus we have maps
    \begin{equation}\label{eq:flowtoj}
        \sJ(\mu(a),\mu(b)) \longrightarrow \sM^{-I}(a,b)
    \end{equation}
    which are compatible with composition on both sides.

    With these preliminaries we are now in position to define the CJS realization of a flow category.
    \begin{defn}\label{dfn:rlin_jmod_flow}
        Let $(\sM,\fo)$ be an $R$-oriented flow category satisfying \cref{asmp:flowcjs}. Define the $R$-linear $\sJ$-module $Z_{\sM,\fo} \colon \sJ \to \mod R$ on objects by
        \[
        Z_{\sM,\fo} (m) \coloneqq \bigvee_{a \in \mu^{-1}(m) \subset \Ob(\sM)} \fo(a),
        \]
        for $m\in \IZ$. The $\sJ$-module structure maps are defined via the compositions
        \[
        \fo(a) \wedge \sJ(\mu(a),\mu(b)) \longrightarrow \fo(a) \wedge \sM^{-I}(a,b) \xrightarrow{\sM^{-\fo(a,b)}} \fo(b),
        \]
        for any $a,b\in \Ob(\sM)$, where the first arrow is induced by the maps \eqref{eq:flowtoj} and the second arrow is induced by the $R$-orientation on $\sM$ as in \cref{lma:Rori_to_spec_system}. 
    \end{defn}
    
    \begin{defn}[Cohen--Jones--Segal realization of a flow category]\label{dfn:cjs_realization}
        Let $(\sM,\fo)$ be an $R$-oriented flow category satisfying \cref{asmp:flowcjs}. The \emph{Cohen--Jones--Segal (CJS) realization} of $(\sM,\fo)$ is defined as the $R$-module
        \[
            |\sM,\fo| \coloneqq |Z_{\sM,\fo}|,
        \]
        where $Z_{\sM,\fo} \colon \sJ \to \mod{R}$ is the $R$-linear $\sJ$-module defined in \cref{dfn:rlin_jmod_flow}.
    \end{defn}
    
    \subsection{CJS realization of a flow multimodule}
    Let $(\sM_i,\fo_i)$ be an $R$-oriented flow category (see \cref{defn:rorcoh}), satisfying \cref{asmp:flowcjs} for $i\in \{0,\ldots,p\}$. Let
    \[
    (\sN,\fm) \colon \sM_1, \ldots, \sM_p \longrightarrow \sM_0
    \]
    be an $R$-oriented flow multimodule, see \cref{dfn:r_ori_flow_multimod}. In this section we explain how to define the CJS realization of a flow multimodule.
    \begin{notn}
         Let $\vec a \coloneqq (a_1,\ldots,a_p) \in \prod_{i=1}^p \Ob(\sM_i)$.
        \begin{enumerate}
            \item If $a_i' \in \Ob(\mathcal M_i)$ for some $i \in \{1,\ldots,p\}$ we will use the notation
            \[
                \vec a_i \coloneqq (a_1,\ldots,a_{i-1},a_i',a_{i+1},\ldots,a_p).
            \]
            \item We will use the notation
            \[
                \mu(\vec a) \coloneqq (\mu_1(a_1),\ldots,\mu_p(a_p)).
            \]
        \end{enumerate}
    \end{notn}
 
    In the following, we assume that $\sN$ admits a coherent system of collars:
    \begin{asmpt}\label{asmp:bimodcjs}
        The flow multimodule $\sN$ is equipped with a \emph{coherent system of collars}. Namely, for every $a_i,a_i' \in \Ob(\sM_i)$ and $b,b' \in \Ob(\sM_0)$, there exist smooth embeddings
            \begin{align*}
                \sM_i(a_i,a_i') \times \sN(\vec a_i;b) \times (0,\infty]^{\{a_i'\}} &\longrightarrow \sN(\vec a;b), \quad i\in \{1,\ldots,p\}\\
                \sN(\vec a;b') \times \sM_0(b',b) \times (0,\infty]^{\{b'\}} &\longrightarrow \sN(\vec a;b)
            \end{align*}
            which are compatible with each other in the following sense: For any $a_i,a_i' \in \Ob(\sM_i)$ and $b,b' \in \Ob(\sM_0)$, the following diagram commutes
            \begin{equation}\label{eq:collar_comp_bimod}
                \begin{tikzcd}[row sep=scriptsize,column sep=scriptsize]
                    \mathrel{\parbox{5cm}{$\sM_i(a_i,a_i') \times \sN(\vec a_i;b') \\ \phantom{\hspace{7mm}}\times \sM_0(b',b) \times (0,\infty]^{\{a_i', b'\}}$}} \rar \dar & \sN(\vec a;b') \times \sM_0(b',b) \times (0,\infty]^{\{b'\}} \dar \\
                    \sM_i(a_i,a_i') \times \sN(\vec a_i;b) \times (0,\infty]^{\{a'_i\}} \rar & \sN(\vec a;b)
                \end{tikzcd},
            \end{equation}
            and they are compatible with the coherent system of collars on $\sM_i$ and $\sM_0$ in the sense that diagrams similar to \eqref{eq:collar_comp_bimod}, involving compositions of $\sM_i$ and $\sM_0$, commute. 
        \end{asmpt}
    The coherent system of collars induces a collection of embeddings  
    \begin{align*}
           \sM_i(a_i,a_i') \times \sN(\vec a_i;b) \cong \sM_i(a_i,a_i') \times \sN(\vec a_i;b) \times \{1\} &\longrightarrow \sN(\vec a;b), \quad i\in \{1,\ldots,p\}\\
           \sN(\vec a;b') \times \sM_0(b',b) \cong \sN(\vec a;b') \times \sM_0(b',b) \times \{1\} &\longrightarrow \sN(\vec a;b).
    \end{align*}
    It follows from the commutativity of \eqref{eq:collar_comp_bimod} that these embeddings are compatible with each other via the actions of $\sM_i$ and $\sM_0$; this is similar to the compatibilities satisfied by \eqref{eq:collar_emb}. The normal bundles of $\sM_i(a_i,a_i') \times \sN(\vec a_i;b)$ and $\sN(\vec a;b') \times \sM_0(b',b)$ inside $\sN(\vec a;b)$ are given by the rank one trivial bundles with fibers $\IR^{\{a_i'\}}$ and $\IR^{\{b'\}}$, respectively, for $i\in \{1,\ldots,p\}$. 
    Similar to \eqref{eq:embmod}, we then have induced maps 
    \begin{align}\label{eq:multimod_embmod}
        \mathcal N_{\text{rel }\partial}(\vec a;b) &\longrightarrow \mathcal N_{\text{rel }\partial}(\vec a;b') \wedge (\mathcal M_0)^\nu_{\text{rel }\partial}(b',b),\\ \label{eq:multimod_embmod2}
        \mathcal N_{\text{rel }\partial}(\vec a;b) &\longrightarrow (\mathcal M_i)^\nu_{\text{rel }\partial}(a_i,a_i') \wedge \mathcal N_{\text{rel }\partial}(\vec a_i;b'), \quad i\in \{1,\ldots,p\},
    \end{align}
    where $(\sM_i)^\nu_{\text{rel }\partial}$ are as in \eqref{eq:embmod} and
    \begin{equation}\label{eq:bimod_jincl1}
        \mathcal N_{\text{rel }\partial} (\vec a;b) = \mathcal N (\vec a;b) / \partial \mathcal N (\vec a;b).
    \end{equation}
    The maps \eqref{eq:multimod_embmod} and \eqref{eq:multimod_embmod2} moreover commute with each other and are coassociative with respect to the action of $\sM_i$ on $\sN$ for $i\in \{0,\ldots,p\}$. The coherent collars on $\sN$ (see \cref{asmp:bimodcjs}) induce maps
    \begin{equation}\label{eq:multimod_col}
        \gamma_{\vec a;b} \colon \sN(\vec a;b) \longrightarrow K(\mu(\vec a);\mu_0(b)), 
    \end{equation}
    that are defined as follows: for a point contained inside a collar of the boundary stratum
    \[
    \left(\prod_{i=1}^p \sM_i(a_i,a_i') \right) \times \sN(\vec a_i;b') \times \sM_0(b',b) \hooklongrightarrow \sN(\vec a;b),
    \]
    the image of $\gamma_{\vec a;b}$ is contained in
    \[
    \prod_{i=1}^p [0,\infty]^{\{\mu_i(a_i) > x \geq \mu_i(a_i')\}} \times [0,\infty]^{\{\mu_0(b') \geq x > \mu_0(b)\}},
    \]
    with the coordinates given by the collar parameters (extended by $0$ outside the collar) of all collars containing the point.
    
    There are now induced maps
    \[\pi_{\vec a;b} \colon \mathcal N_{\text{rel }\partial}(\vec a;b) \longrightarrow K_{\text{rel }\partial}(\mu(\vec a);\mu_0(b)),\]
    that intertwine the coactions of $(\sM_i)^\nu_{\text{rel }\partial}, (\sM_0)^\nu_{\text{rel }\partial},$ and $J^\nu_{\text{rel }\partial}$ defined in \eqref{eq:multimod_embmod}, \eqref{eq:multimod_embmod2} and \eqref{eq:rspecbac1}, \eqref{eq:rspecbac2}, respectively.

    Consider the Spanier--Whitehead dual maps
    \[ D(\pi) \colon D(K_{\text{rel }\partial}(\mu(\vec a) ;\mu_0(b))) \longrightarrow D(\sN_{\text{rel }\partial}(\vec a;b)).\]
    By Atiyah duality, $D(\sN_{\text{rel }\partial}(a,b)) \simeq \sN^{-T\sN - \nu}(a,b) = \sN^{-I}(a,b)$, see \cref{dfn:Thom_spec_flow_cat}.
    Thus we have maps
    \begin{equation}\label{eq:bimodflowtoj}
        \sK(\mu(\vec a);\mu_0(b)) \longrightarrow \sN^{-I}(\vec a;b)
    \end{equation}
    which are compatible with the action maps on both sides.
    
    We are now in position to define the CJS realization of a flow multimodule
    \[
    \sN \colon \sM_1,\ldots,\sM_p \longrightarrow \sM_0
    \]
    satisfying \cref{asmp:bimodcjs}. Recall that each $R$-oriented flow category $(\sM_i,\fo_i)$ satisfies \cref{asmp:flowcjs}. 

    Our first task is to define a $\sJ$-module multimorphism
    \[
    W_{\sN,\fm} \colon \sK \Longrightarrow \sF_{Z_1,\ldots,Z_p;Z_0}, \quad k\in \IZ_{\geq 1},
    \]
    where $Z_i \coloneqq Z_{\sM_i,\fo_i}$, given by
    \[ Z_i(m_i) = \bigvee_{a_i \in \mu_i^{-1}(m_i) \subset \Ob(\sM_{i})} \fo_i(a_i), \quad k\in \IZ_{\geq 1},\]
    where $m_i \in \IZ$, for each $i\in \{1,\ldots,p\}$, is as in \cref{dfn:rlin_jmod_flow}. The $R$-orientation $\fm$ on $\sN$ induces maps
    \[ \sN^{-I}(\vec a;b)  \longrightarrow F_R\left(\fo_1(a_1) \wedge_R \cdots \wedge_R \fo_p(a_p), \fo_0(b)\right).\]
    For $m_1,\ldots,m_p, m_0 \in \IZ$, we define
    \begin{equation}\label{eq:J-mod_multimorphism_flow_multimod}
        W_{\sN, \fm} (\vec m;m_0) \colon \sK(\vec m;m_0) \longrightarrow F_R\left(Z_1(m_1) \wedge_R \cdots \wedge_R Z_p(m_p), Z_0(m_0)\right), \quad k\in \IZ_{\geq 1},
    \end{equation}
    via the maps \eqref{eq:bimodflowtoj}. From the compatibility of the maps \eqref{eq:bimodflowtoj} with the action maps it follows that this assignment defines a natural transformation of $\IS$-linear $\sJ$-modules and thus a $\sJ$-module multimorphism.

    \begin{defn}[Cohen--Jones--Segal realization of a flow multimodule]
    Let $(\sM_i,\fo_i)$ be an $R$-oriented flow category satisfying \cref{asmp:flowcjs} for $i\in \{0,\ldots,p\}$. Let
    \[
        (\sN,\fm) \colon (\sM_1,\fo_1), \ldots, (\sM_p,\fo_p) \longrightarrow (\sM_0, \fo_0)
    \]
    be an $R$-oriented flow multimodule satisfying \cref{asmp:bimodcjs}. The \emph{Cohen--Jones--Segal (CJS) realization} of $(\sN,\fm)$ is defined by
    \[
    |\sN,\fm| \coloneqq |W_{\sN,\fm}|,
    \]
    where $W_{\sN,\fm}$ is the $\sJ$-module multimorphism defined in \eqref{eq:J-mod_multimorphism_flow_multimod}.
    \end{defn}

    \begin{prop}\label{lma:functoriality_cjs}
        Let $(\sM_i, \fo_i)$ be an $R$-oriented flow category satisfying \cref{asmp:flowcjs} for $i\in\{1,\ldots,p\}$, and let
        \[
            (\sN, \fm) \colon (\sM_1,\fo_1), \ldots, (\sM_p,\fo_p) \longrightarrow (\sM_0,\fo_0)
        \]
        be an $R$-oriented flow multimodule satisfying \cref{asmp:bimodcjs}. Then there is an induced map on the CJS realizations 
        \[ |\sN,\fm| \colon |\sM_1,\fo_1| \wedge_R \cdots \wedge_R |\sM_p,\fo_p| \longrightarrow |\sM_0,\fo_0| .\]
    \end{prop}
    \begin{proof}
        This is an immediate consequence of \cref{lem:functoriality_geom_realiz}.
    \end{proof}

    The next result follows by an argument similar to \cite[Proposition 3.16]{large2021spectral} using the fact that flow bordisms as defined in \cref{def:multi_bord} correspond to the notion of $[0,1]$-parametrized maps of flow categories as defined in \cite[Definition 3.6]{large2021spectral}. We omit the details.
    \begin{prop}\label{lma:bordism_gives_homotopy_of_maps}
        Let $(\mathcal M_i, \mathfrak o_{\mathcal M_i})$ be an $R$-oriented flow category satisfying \cref{asmp:flowcjs} for $i\in \{0,\ldots,p\}$. Let 
        \[
        (\sN_1,\fm_1),(\sN_2,\fm_2) \colon (\sM_1,\fo_1), \ldots, (\sM_p,\fo_p) \longrightarrow (\sM_0,\fo_0)
        \]
        be two $R$-oriented flow multimodules satisfying \cref{asmp:bimodcjs}. If $\mathcal (B,\fn) \colon (\mathcal N_1, \fm_1) \Rightarrow (\mathcal N_2,\fm_2)$ is an $R$-oriented bordism, then there is an induced homotopy $|\mathcal N_1,\mathfrak m_1| \simeq |\mathcal N_2, \mathfrak m_2|$.
        \qed
    \end{prop}
    Now we consider a flow bimodule whose CJS realization is homotopic to the identity map. The naive idea is to let $\sM$ be an $R$-oriented flow category and to consider $\sM(a,b) \times [0,1]$. This does however not define a flow bimodule, because the two composition maps
    \begin{align*}
    	(\sM(a,b) \times [0,1]) \times \sM(b,c) &\longrightarrow \sM(a,c) \\
    	\sM(a,b) \times (\sM(b,c) \times [0,1]) &\longrightarrow \sM(a,c)
    \end{align*}
    are equal. The idea to rectify this problem is to equip the topological manifold $\sM(a,b) \times [0,1]$ with a different smooth structure as a manifold with corners.

    Notice that there is a map
        \[\delta_{m,n} \colon K(m;n) \longrightarrow J(m,n) \]
        defined on the subset $[0,\infty]^{\{m>x\geq r\}} \times [0,\infty]^{\{r\geq x >n\}}$ by
    \begin{align}\label{eq:ktoj}
        [0,\infty]^{\{m>x\geq r\}} \times [0,\infty]^{\{r\geq x >n\}} &\longrightarrow [0,\infty]^{\{m>x>n\}}\\ \nonumber
        (\vec x_1,y_1) \times (y_2, \vec x_2) &\longmapsto (\vec x_1, y_1+y_2, \vec x_2).
    \end{align}
    The fibers of $\delta_{m,n}$ are given by a sequence of concatenation of intervals.
    
    Recall the map $\sM(a,b) \to J(\mu(a),\mu(b))$ from \eqref{eq:modcol}. Consider the smooth manifolds with corners
    \[ \id_\sM(a;b) \coloneqq \sM(a,b) \times_{J(\mu(a),\mu(b))} K(\mu(a);\mu(b)).\]
    There are induced maps
    \begin{align*}
        \sM(a,a') \times \id_{\sM}(a';b) &\longrightarrow \id_{\sM}(a;b) \\
        \id_{\sM}(a;b') \times \sM(b',b) &\longrightarrow \id_{\sM}(a;b) 
    \end{align*}
    such that defining
    \[
        \partial_i \id_{\sM}(a;b) \coloneqq \bigsqcup_{\substack{a' \in \Ob(\sM) \\ \mu(a) - \mu(a') = i}} (\sM(a,a') \times \id_{\sM}(a';b)) \sqcup \bigsqcup_{\substack{b' \in \Ob(\sM) \\ \mu(a) - \mu(b') + 1 = i}} (\id_{\sM}(a;b') \times \sM(b', b)),
    \]  
    defines a $\ang{\mu(a)-\mu(b)}$-manifold structure on $\id_\sM(a;b)$. For every $a,b\in \Ob(\sM)$, we note that there is an identification of stable vector bundles
    \[
        T\id_{\sM}(a;b) \cong T\sM(a,b) \oplus \underline{\IR} = I(a,b).
    \]
    Recall that the $R$-orientation on $\sM$ consists of an isomorphism of $R$-line bundles
    \[
        \fo(a,b) \colon \fo(a) \overset{\simeq}{\longrightarrow} I_R(a,b) \otimes_R \fo(b)
    \]
    for every pair of objects $a,b\in\Ob(\sM)$. Pulling this isomorphism back via the projection $\id_{\sM}(a;b) \to \sM(a,b)$ yields an isomorphism of $R$-line bundles over $\id_{\sM}(a;b)$:
    \begin{equation}\label{eq:ori_diagonal}
        \fo(a) \overset{\simeq}{\longrightarrow} (T\id_{\sM}(a;b))_R \otimes \fo(b).
    \end{equation}
    We have thus shown the following.
    \begin{lem}\label{lem:diagonal_bimod}
        Let $(\sM,\fo)$ be an $R$-oriented flow category. For any $a,b\in \Ob(\sM)$, the assignment $(a,b) \mapsto \id_{\sM}(a;b)$ and the isomorphism \eqref{eq:ori_diagonal} of $R$-line bundles over $\id_{\sM}(a;b)$ defines an $R$-oriented flow bimodule $\id_{\sM} \colon \sM \to \sM$.
        \qed
    \end{lem}
    The flow bimodule $\id_{\sM}$ is corresponds to the diagonal bimodule as defined in \cite[Lemma 6.7]{abouzaid2024foundation}. We note that the map \eqref{eq:multimod_col}, for the bimodule $\id_{\sM}$ coincides with the projection map:
    \[ \id_{\sM}(a;b) = \sM(a,b) \times_{J(\mu(a),\mu(b))} K(\mu(a);\mu(b)) \longrightarrow K(\mu(a);\mu(b)).\]
    Consequently, the image of the induced map
    \[
        \sK(\mu(a);\mu(b)) \longrightarrow \id_{\sM}^{-I}(a;b) \longrightarrow F_R(\fo(a),\fo(b)),
    \]
    coincides with the composition
    \[ 
       \sK(\mu(a);\mu(b)) \longrightarrow \sJ(\mu(a),\mu(b)) \longrightarrow \sM^{-I}(a,b) \longrightarrow F_R(\fo(a),\fo(b))
    \]
    for any $a,b\in \Ob(\sM)$.
    The above shows the following.
    \begin{prop}\label{prop:induced_id}
        Let $(\sM,\fo)$ be an $R$-oriented flow category satisfying \cref{asmp:flowcjs}. The CJS realization of the $R$-oriented flow bimodule $\id_{\sM}$ is $\id_{|\sM,\fo|}$.
        \qed
    \end{prop}
    
\subsection{Morse chain complex of a flow category}\label{subsec:morse_flow_chain}
    Throughout this section we let $k$ denote the discrete ring $\pi_0(R)$. Given an $R$-oriented flow category $(\sM,\fo)$, the homotopy groups 
    \[ \pi_\bullet(|\sM,\fo| \wedge_R Hk) \]
    can be computed as the homology of a certain $k$-linear chain complex constructed using morphisms of dimension zero in $\sM$. When $R=\IS$ and $(\sM,\fo)$ is the Morse--Smale flow category of a Morse function on a manifold, oriented with its canonical $R$-orientation (see \cref{defn:ms_cat} below), the chain complex we describe below coincides with the Morse homology chain complex with coefficients in $\IZ$.

    \begin{defn}[Morse differentials in a flow category]
        Let $(\sM,\fo)$ be an $R$-oriented flow category and let $a,b \in \Ob(\sM)$ be such that 
        \[ \mu(a) - \mu(b)=1,\]
        so that $\sM(a,b)$ is a finite collection of points $\{u_1,\ldots,u_n\}$. For every $i\in \{1,\ldots,n\}$, there are isomorphisms of $R$-line bundles over $\{u_i\}$
        \[ d_{a,b}(i) \colon \fo(a) \xrightarrow{\fo(a,b)|_{\{u_i\}}} I_R(a,b)|_{\{u_i\}} \otimes_R \fo(b) \simeq \fo(b), \quad i\in \{1,\ldots,n\}, \]
        where the first isomorphism comes from the $R$-orientation on $\sM$ \cref{defn:rorcoh} and the second isomorphism comes from the the canonical trivialization of $I(a,b)|_{\{u_i\}}$. Denote the associated map on Thom spectra by $d_{a,b}(i) \in F_R(\fo(a), \fo(b))$ (see \cref{rem:constant_R-line_isos}). Also, define 
        \[ d(a,b) \coloneqq \sum_{i=1}^n d_{a,b}(i).\]
    \end{defn}

    \begin{lem}\label{lma:d-squared_zero_morse_cpx}
        Let $(\sM,\fo)$ be an $R$-oriented flow category and let $a,c \in \Ob(\sM)$ be such that $\mu(a)-\mu(c)=2$. Then,
        \[ \sum_{\substack{b \in \Ob(\mathcal M) \\ \mu(a)-\mu(b)=1}} d(b,c) \circ d(a,b) \]
        is null-homotopic.
    \end{lem}
    \begin{proof}
        Over $\sM(a,c)$ there is an isomorphism of $R$-line bundles $\fo(a) \xrightarrow{\simeq} \fo(c) \otimes_R I_R(a,c)$, which we claim provides the null-homotopy. Namely, $\sM(a,c)$ is a compact smooth $1$-manifold with boundary. Let $\sM_1(a,c),\ldots,\sM_N(a,c)$ enumerate the components of $\sM(a,c)$ that are diffeomorphic to $[0,1]$. For each $k\in \{1,\ldots,N\}$ we have an isomorphism of $R$-line bundles over $\sM_k(a,c)$:
        \[
            \fo(a) \overset{\simeq}{\longrightarrow} \fo(c) \otimes_R I_R(a,c)|_{\sM_k(a,c)}.
        \]
        Since $I_R(a,c)$ is canonically trivialized, there is an induced map \begin{equation}\label{eq:null_htpy_interval}
            \varSigma_+^\infty \sM(a,c) \longrightarrow F_R(\fo(a),\fo(c)),
        \end{equation}
        see \cref{rem:constant_R-line_isos}. Restriction of \eqref{eq:null_htpy_interval} under the two inclusions of the boundary strata
        \[
        \sM(a,b_k) \times \sM(b_k,c) \hooklongrightarrow \sM_k(a,c) \hooklongleftarrow \sM(a,b_k') \times \sM(b_k',c),
        \]
        coincide with maps $d(b_k,c) \circ d(a,b_k)$ and $d(b_k',c) \circ d(a,b_k')$, respectively, showing that \eqref{eq:null_htpy_interval} provides a null-homotopy of $d(b_k,c) \circ d(a,b_k) + d(b_k',c) \circ d(a,b_k')$. Summation over $k$ finishes the proof.
    \end{proof}
    \begin{defn}[Morse complex of a flow category]\label{dfn:morse_complex_flow_cat}
        The \emph{Morse complex} of the $R$-oriented flow category $(M,\fo)$ is defined as the $k$-linear chain complex 
        \[ CM(\sM,\fo) \coloneqq \left( \bigoplus_{n \in \IZ} \left( \bigoplus_{\mu(a)=n} \pi_0( \fo(a) \wedge_R Hk) \right), d_{(M,\fo)} \coloneqq \bigoplus_{\mu(a)-\mu(b)=1} \pi_0(d(a,b) \wedge_R Hk) \right),\]
    \end{defn}

    Recall that $\sJ(m+1,m)\simeq \IS$.
    \begin{lem}
        Let $X \colon \sJ \to \mod{R}$ be an $R$-linear $\sJ$-module. There are induced maps
        \[ \delta_{X,n} \colon X(n) \simeq X(n) \wedge \sJ(n,n-1) \longrightarrow X(n-1)\]
        such that $\delta_{X,n} \circ \delta_{X,n+1}$ is canonically null-homotopic.
    \end{lem}
    \begin{proof}
    This follows from the fact that the composition
    \[ \sJ(m+2,m+1) \wedge \sJ(m+1,m) \longrightarrow \sJ(m+2,m)\]
    is canonically null-homotopic.
    \end{proof}
    \begin{defn}[Morse complex of an $R$-linear $\sJ$-module]
            Let $X \colon \sJ \to \mod{R}$ be an $R$-linear $\sJ$-module. Define the \emph{Morse complex} of $X$ to be the $k$-linear chain complex
            \[ CM_\bullet(X;k) \coloneqq \left(\bigoplus_{n \in \IZ} \pi_0(X(n) \wedge_R Hk), d_X \coloneqq \sum_{n\in \IZ} \pi_0(\delta_{X,n} \wedge_R Hk) \right),\]
    \end{defn}
    \begin{prop}\label{prop:homology_wedge_ring} Assuming that $X \colon \sJ \to \mod{R}$ is a $\sJ$-module that is bounded below, we have
    \[ \pi_\bullet(|X| \wedge_R Hk) \simeq H_\bullet ( CM_\bullet(X;k) ).\]
    \end{prop}
    \begin{proof} This can be deduced from the proof of \cite[Theorem 6]{cohen2009floer}.\end{proof}
    
    \begin{lem}\label{lem:flow_htpy_morse}
        Let $(\sM,\fo)$ be an $R$-oriented flow category that satisfies \cref{asmp:flowcjs}. Then, there is an isomorphism of $k$-linear chain complexes
        \[ CM(\sM,\fo) \cong CM_\bullet(Z_{\sM,\fo};k).\]
        In particular $\pi_\bullet(|\sM,\fo|) \cong H_\bullet (CM_\bullet( Z_{\sM, \fo};k))$.
        \qed
    \end{lem}
    \begin{defn}
        Let $(\sN,\fm) \colon (\sM_1,\fo_1) \to (\sM_2,\fo_2)$ be an $R$-oriented flow bimodule (see \cref{dfn:r_ori_flow_bimod}). Let $a \in \Ob(\sM_1)$ and $b \in \Ob(\sM_2)$ be such that 
        \[ \mu(a) - \mu(b)=0,\]
        so that $\sN(a,b)$ is a finite collection of points $\{u_1,\ldots,u_n\}$. Define
        \[ f(a,b)_i \coloneqq \fm(a,b)|_{\{u_i\}} \colon \fo_1(a) \overset{\simeq}{\longrightarrow} I_R(a,b)|_{\{u_i\}}\otimes_R \fo_2(b) \simeq \fo_2(b), \]
        where the first isomorphism comes from the $R$-orientation $\fm$ on $\sN$ (\cref{dfn:r_ori_flow_bimod}) and the second isomorphism comes from the the canonical trivialization of $I(a,b)|_{\{u_i\}}$. Define
        \[ f(a,b) \coloneqq \sum_{i=1}^n f(a,b)_i.\]
    \end{defn}

    \begin{lem} 
        Let $a \in \Ob(\sM_1)$ and $c \in \Ob(\sM_2)$ be such that $\mu(a)-\mu(b)=1$. Then,
        \[ \sum_{\substack{a' \in \Ob(\sM_1) \\ \mu(a)-\mu(a')=1}} f(a',b) \circ d(a,b) + \sum_{\substack{b' \in \Ob(\sM_1) \\ \mu(b')-\mu(b)=1}} d(b',b) \circ f(a,b') \]
        is null-homotopic.
    \end{lem}
    \begin{proof}
        This is very similar to the proof of \cref{lma:d-squared_zero_morse_cpx}.
    \end{proof}
    \begin{defn}[Morse chain map of a flow bimodule]
        The \emph{Morse chain map} of the $R$-oriented flow bimodule $(\sN,\fm)$ is the $k$-linear chain map 
        \[ CM(\sN,\fm) = \bigoplus_{\mu(a)-\mu(b)=0} \pi_0( f(a,b) \wedge_R Hk) \colon CM(\sM_1,\fo_1) \longrightarrow CM(\sM_2,\fo_2).\]
    \end{defn}

    Recall that $\sK(m;m)= \IS$.
    \begin{lem}
        Let $Z_1,Z_2 \colon \sJ \to \mod{R}$ be two $R$-linear $\sJ$-modules. Let $W \colon \sK \Rightarrow \sF_{Z_1;Z_2}$ be a $\sJ$-module multimorphism (see \cref{defn:jmodmor}) and consider
        \[ \phi_{W,n}\colon Z_1(n) \simeq Z_1(n) \wedge \sK(n;n) \longrightarrow Z_2(n).\]
        Then $\phi_{W,n} \circ \delta_{Z_1,n+1}$ is canonically homotopic to $\delta_{Z_2,n} \circ \phi_{W,n+1}$.
    \end{lem}
    \begin{proof}
    This follows from the fact that the compositions
    \[ \sJ(n+1,n) \wedge \sK(n;n) \longrightarrow \sK(n+1;n)\]
    and
    \[ \sK(n+1;n+1) \wedge \sJ(n+1,n) \longrightarrow \sK(n+1;n)\]
    are canonically homotopic.
    \end{proof}
    \begin{defn}[Morse chain map of an $R$-linear $\sJ$-bimodule]
        Let $Z_1,Z_2 \colon \sJ \to \mod{R}$ be two $R$-linear $\sJ$-modules and $W \colon \sK \Rightarrow \sF_{Z_1;Z_2}$ a $\sJ$-module multimorphism. Define the \emph{Morse chain map} of $W$ to be the chain map 
        \[ CM_\bullet(W;k) = \bigoplus_{n \in \IZ} \pi_0( \phi_{W,n} \wedge_R Hk) \colon CM_\bullet(Z_1;k) \longrightarrow CM_\bullet(Z_2;k).\]
    \end{defn}
    \begin{prop}
        Let $Z_1,Z_2 \colon \sJ \to \mod{R}$ be two $R$-linear $\sJ$-modules that are bounded below and let $W \colon \sK \Rightarrow \sF_{Z_1;Z_2}$ be a $\sJ$-module multimorphism. The following diagram is commutative.
        \[ 
            \begin{tikzcd}[row sep=scriptsize, column sep=scriptsize]
                \pi_\bullet(|Z_1| \wedge_R Hk) \rar{\simeq}\dar{\pi_\bullet(|W| \wedge_R Hk)} & H_\bullet ( CM_\bullet(Z_1;k) ) \dar{ CM_\bullet(W;k)}\\
                \pi_\bullet(|Z_2| \wedge_R Hk) \rar{\simeq} & H_\bullet ( CM_\bullet(Z_2;k) )
           \end{tikzcd}
        \]
    \end{prop}
    \begin{proof} This can be deduced by an analysis similar to the one in \cite[Theorem 6]{cohen2009floer}, via filtrations on the geometric realizations.\end{proof}

    \begin{lem}
        Let $(\sN,\fm) \colon  (\sM_1,\fo_1) \to (\sM_2,\fo_2)$ be an $R$-oriented flow bimodule of two $R$-oriented flow categories satisfying \cref{asmp:flowcjs}. Then, the following diagram commutes:
        \[ 
            \begin{tikzcd}[row sep=scriptsize, column sep=scriptsize]
                CM(\sM_1,\fo_1) \rar{\cong}\dar{CM(\sN,\fm)} & CM_\bullet(Z_{\sM_1,\fo_1}) \dar{CM_\bullet(W_{\sN,\fm})}\\
                CM(\sM_2,\fo_2) \rar & CM_\bullet(Z_{\sM_2,\fo_2}).
            \end{tikzcd}.
        \]
        \qed
    \end{lem}

\subsection{Morse chain complex with local systems}
We now generalize the discussion in \cref{subsec:morse_flow_chain} to flow categories equipped with local systems.
\begin{notn}
    \begin{enumerate}
        \item We denote by $k$ the discrete ring $\pi_0(R)$.
        \item Let $(\sM,\fo)$ be an $Hk$-oriented flow category. For any $x\in \Ob(\sM)$ consider the rank one free $k$-module
        \[ o(x) \coloneqq \pi_\bullet(\fo(x)).\]
    \end{enumerate}
\end{notn}
\begin{notn}\label{notn:j_sys_fun}
	Suppose that $Z$ is a $\sJ$-module. We denote by $- \circ_k -$ the following composition
	\begin{align*}
	    C_\bullet(Z(m);k) \otimes_k C_\bullet(\sJ(m,n);k) \overset{EZ}{\longrightarrow} C_\bullet(Z(m) \wedge \sJ(m,n);k) \longrightarrow C_\bullet(Z(n);k),
	\end{align*}
    where $EZ$ is the Eilenberg--Zilber map, and the second map is induced by the $\sJ$-module structure.
\end{notn}
\begin{lem}\label{lem:j_sys_chains} 
	There exist classes $j_{m,n} \in C_{-1}(\sJ(m,n);k)$ such that 
	\begin{equation}\label{eq:morse_eq3} 
		\partial j_{m,n} = \sum_{m > \ell > n} j_{m,\ell} \circ_k j_{\ell,n}.
	\end{equation}
\end{lem}
\begin{proof}
	Note that 
	\begin{align*}
		C_{-1}(\sJ(m,n);k) &= C_{-1}\left(D(J(m,n)^{\nu_{m,n}}/\partial J(m,n)^{\nu_{m,n}}\right);k)\\
				&\cong C^0(J(m,n),\partial J(m,n);k)\\
				&\cong C_{m-n-1}(J(m,n),\partial_0 J(m,n);k)
	\end{align*}
	where $\partial_0 J(m,n) \subset J(m,n)= [0,\infty]^{m>x>n}$ is the subset of $J(m,n)$ where at least one coordinate is equal to $0$. Here, in the last quasi-isomorphism we have used the Poincar\'e--Lefschetz duality with $k$-coefficients for the manifold with boundary $J(m,n)$ (recall that the manifold-theoretic boundary of $J(m,n)$ is $\partial J(m,n) \cup \partial_0 J(m,n)$).
	The data of chains $j_{m,n}$ satisfying \eqref{eq:morse_eq3} is thus equivalent to choices of chains 
	\[ j'_{m,n} \in C_{m-n-1}(J(m,n),\partial_0 J(m,n);k)\]
	such that 
	\[ \partial j'_{m,n} = \sum_{m > k > n} j'_{m,k} \circ_k j'_{k,n} \in C_{m-n-2}(\partial_0 J(m,n);k).\]
	Such a choice can be made inductively by defining $j'_{m,n} \in C_\bullet(J(m,n),\partial_0 J(m,n);k)$ to be a chain representative for the relative homology class of the submanifold $\partial_0 J(m,n)$. 
\end{proof}
\begin{lem}\label{lem:k_sys_chains}
	There exist classes $k_{m,n} \in C_0(\sK(m;n);k)$ such that 
	\[ \partial k_{m,n} = \sum_{m > m' \geq n} j_{m,m'} \circ_k k_{m',n} - \sum_{m \geq n' > n} k_{m,n'} \circ_k j_{n',n}.\]
\end{lem}
\begin{proof}
	This is proved by an argument similar to \cref{lem:j_sys_chains}.
\end{proof}
\begin{defn}
	Let $Z \colon \sJ \to \Spec$ be a bounded below $\sJ$-module. The \emph{Morse complex} of $Z$ with coefficients in a discrete ring $k$ is the chain complex
	\[ CM(Z;k) \coloneqq \left( \bigoplus_{m \in \IZ} C_\bullet(Z(m) ;k), d \right)\]
	where
	\[ d(\sigma) \coloneqq \partial(\sigma) + \sum_{m,n \in \IZ} \sigma_m \circ_k j_{m,n},\]
	where $\sigma_m \in C_\bullet(Z(m);k)$ is the $m$-th component of $\sigma$, the chains $j_{m,n}$ are as in \cref{lem:j_sys_chains}, and $\partial$ is the induced cellular differential on $\bigoplus_{m\in \IZ}C_\bullet(Z(m);k)$.
\end{defn}
\begin{rem}\label{rem:cm_fil}
    Note that $CM(Z;k)$ carries a natural increasing filtration defined by
    \begin{equation}\label{eq:filtr_cm}
        F^p CM(Z;k) = \bigoplus_{m \leq p} C_\bullet(Z(m);k)
    \end{equation}
    whose associated graded is given by
    \[ \Gr^p CM(Z;k) = (C_\bullet(Z(p);k),\partial).\]
\end{rem}
\begin{defn}
    Let $(\sM,\fo,\sE)$ be an $R$-oriented flow category equipped with a local system of $A$-modules. The \emph{Morse complex of $(\sM,\fo,\sE)$ with coefficients in $k$} is defined as the chain complex
    \[ CM(\sM,\fo,\sE;k) \coloneqq CM(Z_{\sM,\fo,\sE};k).\]
    
\end{defn}
\begin{lem}\label{lem:cm_flow}
    Let $(\sM,\fo,\sE)$ be an $R$-oriented flow category equipped with a local system of $A$-modules.
    \begin{enumerate}
    \item For all $x,y \in \Ob(\sM)$, let $m_{xy} \in C_{\mu(x)-\mu(y)-1}(\sM(x,y);o(x)^\vee \otimes_k o(y))$ denote the image of $j_{\mu(x)\mu(y)}$ under the induced map 
    \begin{align*}
        C_{-1} (\sJ(\mu(x),\mu(y));k) \longrightarrow C_{-1}&(\sM(x,y)^{-I(x,y)};k)\\
        &\cong C_{\mu(x)-\mu(y)-1}(\sM(x,y);o(x)^\vee \otimes_k o(y)).
    \end{align*}
    Then, For all $x,y \in \Ob(\sM)$,
    \[ 
        \partial m_{xy} = \sum_{z} m_{xz} \circ_k m_{zy} \in C_{\mu(x)-\mu(y)-2}( \partial \sM(x,y);o(x)^\vee \otimes_k o(y)).
    \]
    where $- \circ_k m_{xy}$ are the composition maps
    \[ - \circ_k - \colon C_\bullet(\sE(x); o(x)) \otimes_k C_\bullet(\sM(x,y);o(x)^\vee \otimes_k o(y)) \longrightarrow C_\bullet(\sE(y); o(y)).\]
    \item The Morse complex of $(\sM,\fo,\sE)$ can be expressed as
    \[
    CM(\sM,\fo,\sE;k) = \left(\bigoplus_{x \in \Ob(\sM)} C_\bullet(\sE(x);o(x)),d\right),
    \]
    where
    \[
    d (\sigma) \coloneqq \partial(\sigma) + \sum_{y \in \Ob(\sM)} (-1)^{|\sigma|} \sigma \circ_k m_{xy},
    \]
    where $\partial$ is the differential in the complex $C_\bullet(\sE(x);o(x))$.
    \end{enumerate}
\end{lem}
\begin{lem}\label{lem:nat_trans_CM}
    Let $Z_1,Z_2 \colon \sJ \to \Spec$ be two bounded below $\sJ$-modules, and let $F \colon Z_1 \Rightarrow Z_2$ be a natural transformation. 
    \begin{enumerate}
        \item There is an induced chain map 
        \begin{align}\label{eq:ls_morse_map} 
            CM(F;k) \colon CM(Z_1;k) &\longrightarrow CM(Z_2;k)\\
            \sum_{m \in \IZ} x_m &\longmapsto \sum_{m \in \IZ} F(m)_\bullet x_m. \nonumber
        \end{align}
        \item
        Assume that $F$ is a pointwise weak equivalence, i.e.,
        \[ F(m) \colon Z_1(m) \longrightarrow Z_2(m)\]
        is a weak equivalence for each $m \in \IZ$. Then $CM(F;k)$ is a filtered chain map of the underlying chain complexes that induces a quasi-isomorphism of the associated graded. In particular, if $Z_1$ and $Z_2$ are bounded below, i.e., $Z_1(q)=Z_2(q)=\ast$ for $q \ll 0$, then $CM(F;k)$ is a quasi-isomorphism. 
    \end{enumerate}
\end{lem}
\begin{proof}
    \begin{enumerate}
        \item From the natural transformation $F \colon Z_1 \Rightarrow Z_2$ we obtain a commutative diagram
        \[ 
            \begin{tikzcd}[column sep=1.3cm]
                CM(Z_1(m);k) \otimes_k C_\bullet(\sJ(m,n);k) \rar{F(m)_\bullet \otimes_k \id} \dar{- \circ_k -} & CM(Z_1(m);k) \otimes_k CM(\sJ(m,n);k) \dar{- \circ_k -} \\
                CM(Z_1(n);k) \rar{F(n)_\bullet} & CM(Z_2(n);k)
            \end{tikzcd},
        \]
        for any $m,n \in \IZ$. It follows that 
        \[ \sum_{m \in \IZ} \sum_{n \in \IZ} F(n)_\bullet (x_m \circ_k j_{m,n}) = \sum_{m \in \IZ} \sum_{n \in \IZ} F(m)_\bullet(x_m) \circ_k j_{m,n}.\]
        Moreover, 
        \[ \sum_{n \in \IZ} F(n)_\bullet (\partial x_n) = \sum_{n \in \IZ} \partial(F(n)_\bullet x_n).\]
        Since $F(m)_\bullet$ is a chain map for any $m\in \IZ$. This finishes the proof that \eqref{eq:ls_morse_map} defines a chain map.
        \item It is clear that $CM(F;k)$ respects the filtration defined in \eqref{eq:filtr_cm}. Supposing that $F$ is a pointwise weak equivalence implies that the induced maps on the associated graded is a quasi-isomorphism. The assumption that $Z_1$ and $Z_2$ are bounded implies that the filtrations are bounded below, and hence it follows that the chain map $CM(F;k)$ is a quasi-isomorphism.
    \end{enumerate}
\end{proof}
\begin{lem}\label{lem:simpl_replacemnt_J-mod}
    Let $A$ be an $R$-algebra, and let $Z \colon \sJ \to \mod A$ be an $A$-linear $\sJ$-module. There is a natural transformation of $A$-linear $\sJ$-modules
    \[ B(Z,\sJ,\sJ) \Longrightarrow Z\]
    which is a pointwise weak equivalence.
\end{lem}
\begin{proof}
    For all $p,k\in \IZ$ and integers $a_0 \geq \cdots \geq a_k$, there is a map
    \begin{equation}\label{eq:comp_map_simpl}
        Z(a_0) \wedge \sJ(a_0,a_1) \wedge \cdots \wedge \sJ(a_{k-1},a_k) \wedge \sJ(a_k,p) \longrightarrow Z(p),
    \end{equation}
    that is induced by the $\sJ$-module structure on $Z$. These maps are easily seen to be compatible with the face and degeneracy maps in $B_\bullet(Z,\sJ,\sJ(-,p))$ and hence we get an induced map $B(Z,\sJ,\sJ(-,p)) \to Z(p)$ for any $p\in \IZ$. The composition maps in $\sJ$ induce maps
    \[
    B(Z,\sJ,\sJ(-,p)) \wedge \sJ(p,q) \longrightarrow B(Z,\sJ,\sJ(-,q)),
    \]
    for any $p,q\in \IZ$ such that the following diagram commutes
    \[
    \begin{tikzcd}[sep=scriptsize]
        B(Z,\sJ,\sJ(-,p)) \wedge \sJ(p,q) \rar \dar & Z(p) \wedge \sJ(p,q) \dar \\
        B(Z,\sJ,\sJ(,-q)) \rar & Z(q)
    \end{tikzcd},
    \]
    which is to say there is a natural transformation $B(Z,\sJ,\sJ) \Rightarrow Z$. The maps \eqref{eq:comp_map_simpl} induce a map $B(Z,\sJ,\sJ(-,p)) \to Z(p)$ for any $p\in \IZ$, and compositions of the unit map induce a map $Z(p) \to B(Z,\sJ,\sJ(-,p))$ that is a homotopy inverse to \eqref{eq:comp_map_simpl}, see \cite[Lemma 6.3]{blumberg2012localization}.
\end{proof}
\begin{defn}\label{dfn:module_morphism_induced_CM}
	Let $Z_1,Z_2 \colon \sJ \to \Spec$ be two bounded below $\sJ$-modules. Let $W \colon \sK \Rightarrow F(Z_1,Z_2)$ be a $\sJ$-module morphism. Then, define $CM(W;k)$ to be the map 
    \begin{align*}
        CM(W;k) \colon CM(Z_1;k) &\longrightarrow CM(Z_2;k) \\
        \sigma &\longmapsto \sum_{m,n\in \IZ} \sigma_m \circ_k k_{m,n},
    \end{align*}
	where $\sigma_m \in C_\bullet(Z_1(m);k)$ is the $m$-th component of $\sigma$, and the chains $k_{m,n} \in C_0(\sK(m;n);k)$ are as in \cref{lem:k_sys_chains}.
\end{defn}
\begin{defn}
    Let $(\sN,\fm,\sE) \colon (\sM_1,\fo_1,\sE_1) \to (\sM_2,\fo_2,\sE_2)$ be an $R$-oriented flow bimodule with local system. Define
    \[ CM(\sN,\fo,\sE;k) \coloneqq CM(W_{\sN,\fm,\sE}) \colon CM(\sM_1,\fo_1,\sE_1;k) \longrightarrow CM(\sM_2,\fo_2,\sE_2;k)\]
\end{defn}
\begin{lem}\label{lem:cm_mod_flow}
     Let $(\sN,\fm) \colon (\sM_1,\fo_1,\sE_1) \to (\sM_2,\fo_2,\sE_2)$ be an $R$-oriented flow bimodule with local system. Let $\{m_{xx'}\}_{x,x' \in \Ob(\sM_1)}$ and $\{m_{yy'}\}_{y,y' \in \Ob(\sM_2)}$ be as in \cref{lem:cm_flow}.
     \begin{enumerate}
        \item For all $x \in \Ob(\sM_1)$ and $y \in \Ob(\sM_2)$, let $n_{xy} \in C_{\mu(x)-\mu(y)}(\sN(x,y);o(x)^\vee \otimes_k o(y))$ denote the image of $k_{\mu(x)\mu(y)}$ under the induced map 
        \begin{align*}
            C_{-1} (\sK(\mu(x),\mu(y));k) \longrightarrow C_{0}&(\sN(x,y)^{-I(x,y)};k)\\ 
            &\cong C_{\mu(x)-\mu(y)}(\sN(x,y);o(x)^\vee \otimes_k o(y)).
        \end{align*}
        Then, for all $x,y \in \Ob(\sM)$,
        \begin{align*}
            \partial n_{xy} = \sum_{x'} m_{xx'} \circ_k & n_{x'y} + \sum_{y'}n_{xy'} \circ_k m_{y'y}\\
            &\in C_{\mu(x)-\mu(y)-1}( \partial \sN(x,y);o(x)^\vee \otimes_k o(y)),
        \end{align*}
        where $- \circ_k m_{xy}$ are the composition maps
        \[ - \circ_k - \colon C_\bullet(\sE(x); o(x)) \otimes_k C_\bullet(\sM(x,y);o(x)^\vee \otimes_k o(y)) \longrightarrow C_\bullet(\sE(y); o(y)).\]
        \item The Morse chain map associated to $(\sN,\fm,\sE)$ can be expressed as
        \begin{align*}
            CM(\sM,\fo,\sE;k) \colon CM(\sM_2,\fo_2,\sE_2) &\longrightarrow CM(\sM_2,\fo_2,\sE_2)\\
        \sigma &\longmapsto \sum_{xy} \sigma_x \circ_k n_{xy}
        \end{align*}
    \end{enumerate}
    \qed
\end{lem}

Recall the definition of the spaces $V_q(n)$ from \cref{sec:functoriality_geom_realiz} and their associated relative cochains $\sV_q(n)$. Note that there are inclusion maps 
\[
V_q(n) \hooklongrightarrow K(n;q)
\]
that are given by the inclusion
\[
[0,\infty]^{\{n > x \geq q\}} \hooklongrightarrow [0,\infty]^{\{n > x \geq q\}} \times [0,\infty]^{\{q \geq x > q\}} \subset K(n;q).
\]
These maps respect the boundaries and hence induces maps $V_q(n)/\partial V_q(n) \to K(n;q)/\partial K(n;q)$. Taking Spanier--Whitehead duals thus yield maps
\[
    \sK(n;q) \longrightarrow \sV_q(n).
\]
These maps in turn define a natural transformation of $\sJ$-modules
\begin{equation}\label{eq:morse_eq6}
    \sK(-;q) \Longrightarrow \sV_q(-).
\end{equation}

\begin{notn}
    Let $A$ be an $R$-algebra.
	\begin{enumerate}
	    \item For $X \in \mod{A}$ and $q \in \IZ$, denote by $X[q]$ the $A$-linear $\sJ$-module defined by 
		\[	
		      X[q](n) \coloneqq \begin{cases} X, &\text{for } n=q\\
			\ast, &\text{otherwise}
			\end{cases}.
		  \]  
        \item If $Z$ is an $A$-linear $\sJ$-module, and $p\in \IZ$, we denote by $Z^{\leq p}$ the $A$-linear $\sJ$-module defined by
        \[
        Z^{\leq p}(n) \coloneqq \begin{cases}
            Z(n), & n \leq p \\
            \ast, & \text{otherwise}
        \end{cases}.
        \]
	\end{enumerate}
\end{notn}

Note that there is a split cofiber sequence 
\[
\begin{tikzcd}[sep=scriptsize]
    B(Z^{\leq p-1},\sJ,\sJ) \rar & B(Z^{\leq p},\sJ,\sJ) \rar & B(Z(p)[p],\sJ,\sJ),
\end{tikzcd}
\]
that corresponds to the splitting
\[ B(Z^{\leq p},\sJ,\sJ) \simeq B(Z^{\leq p-1},\sJ,\sJ) \vee B(Z(p)[p],\sJ,\sJ), \]
such that the first map and second maps coincides with the inclusion of the first factor and projection onto the second factors, respectively.
\begin{lem}\label{lem:cm_cofib}
	There is a (split) exact sequence of chain complexes
	\[
    \begin{tikzcd}[sep=small]
        0 \rar & CM(B(Z^{\leq p-1},\sJ,\sJ);k) \rar & CM(B(Z^{\leq p},\sJ,\sJ);k) \rar & CM(B(Z(p)[p],\sJ,\sJ);k) \rar & 0
    \end{tikzcd}
    \]
	where the maps are induced from the natural transformations of $\sJ$-modules 
	\[Z^{\leq p-1} \Longrightarrow Z^{\leq p} \Longrightarrow Z(p)[p].\]
    \qed
\end{lem}
\begin{prop}\label{prop:cm_prop}
	Let $Z \colon \sJ \to \mod{A}$ be a bounded below $A$-linear $\sJ$-module. Then, there exists an isomorphism
    \begin{equation}\label{eq:cm_prop}
        H_\bullet \left( CM(Z;k) \right) \cong \pi_\bullet(|Z| \wedge_R Hk).
    \end{equation}
\end{prop}
\begin{proof}

	Fix $q \ll 0$ such that $|Z|_q \simeq |Z|$. Consider the $\sJ$-module $|Z|_q[q]$ and note that  
	\begin{equation}\label{eq:morse_eq1}
	||Z|_q[q]|_q \simeq |Z|_q \quad \text{and} \quad CM(|Z|_q[q]) \cong C_\bullet(|Z|_q).
	\end{equation}
	Moreover, note that $C_\bullet(|Z|_q)$ carries an increasing filtration defined by
	\begin{equation}\label{eq:cm_prop_fil}
		F^p C_\bullet (|Z|_q) \coloneqq C_\bullet(|Z^{\leq p}|_q).
	\end{equation}
	On the other hand, consider the $\sJ$-module $B(Z,\sJ,\sJ)$. Consider the filtration on $CM(B(Z,\sJ,\sJ))$ as defined by
    \[ F^p CM( B(Z,\sJ,\sJ)) \coloneqq CM ( B(Z^{\leq p},\sJ,\sJ) ).\]
    The maps are induced by the natural transformation $B(Z^{\leq p-1},\sJ,\sJ) \Rightarrow B(Z^{\leq p},\sJ,\sJ)$ induced by the natural transformation $Z^{\leq p-1} \Rightarrow Z^{\leq p}$ given pointwise by the inclusion map. By \cref{lem:cm_cofib} we have that
    \[
        \Gr^p CM(B(Z,\sJ,\sJ)) \cong CM(B(Z(p)[p],\sJ,\sJ)).
    \]
    Next, combining \cref{lem:simpl_replacemnt_J-mod} and \cref{lem:nat_trans_CM}(ii) yields a quasi-isomorphism
    \[
        CM(B(Z,\sJ,\sJ)) \overset{\cong}{\longrightarrow} CM(Z).
    \]
    To finish the proof it is therefore sufficient to construct a filtered chain map
	\[
        CM(B(Z,\sJ,\sJ)) \longrightarrow C_\bullet(|Z|_q) \cong CM(|Z|_q[q])
    \]
    such that the induced map on associated gradeds is a quasi-isomorphism.

    To construct such a chain map, we now construct a $\sJ$-module morphism
	\[ W \colon \sK \Longrightarrow F_A(B(Z, \sJ,\sJ),|Z|_q[q]).\]
    Such a morphism is completely described by specifying for every $p \in \IZ$ a map
	\[ W(p,q) \colon B(Z,\sJ,\sJ(-,p)) \wedge \sK(p;q) \longrightarrow |Z|_q,\]
    such that the following diagram commutes for all $p \geq p'$:
	\begin{equation}\label{eq:cm_prop_dig}
        \begin{tikzcd}[sep=scriptsize]
            B(Z,\sJ,\sJ(-,p)) \wedge \sJ(p,p') \wedge \sK(p';q) \rar \dar & B(Z,\sJ,\sJ(-,p')) \wedge \sK(p';q) \dar{W(p',q)} \\
			B(Z,\sJ,\sJ(-,p)) \wedge \sK(p;q) \rar{W(p,q)} & {|Z|_q}
		\end{tikzcd}.
    \end{equation}
	To define the maps $W(p,q)$, first notice that the source and targets of these maps are the geometric realization of the simplicial $A$-modules $B_\bullet(Z,\sJ,\sJ(-,p)) \wedge \sK(p;q)$ and $B_\bullet(Z,\sJ,\sV_q)$, respectively. We define $W(p,q)$ as the geometric realization of the following simplicial maps
	\begin{align}\label{eq:simplicial_maps_W}
		W(p,q)_n \colon &Z(q_0) \wedge \sJ(q_0,q_1) \wedge \dots \wedge \sJ(q_{n-1},q_n) \wedge \sJ(q_n,p) \wedge \sK(p;q) \\ \nonumber
        &\qquad \longrightarrow Z(q_0) \wedge \sJ(q_0,q_1) \wedge \dots \wedge \sJ(q_{n-1},q_n) \wedge \sK(q_n;q) \\ \nonumber
        &\qquad \longrightarrow Z(q_0) \wedge \sJ(q_0,q_1) \wedge \dots \wedge \sJ(q_{n-1},q_n) \wedge \sV_q(q_n).
	\end{align}
	Here the first map is induced by the left $\sJ$-action on $\sK$ and the second map is induced from the natural transformation \eqref{eq:morse_eq6}. The commutativity of the diagram \eqref{eq:cm_prop_dig} follows from the commutativity of the following diagram:
	\[
		\begin{tikzcd}[sep=scriptsize]
			\parbox{4.4cm}{$Z(q_0) \wedge \sJ(q_0,q_1) \wedge \cdots\\ \wedge \sJ(q_n,p) \wedge \sJ(p,p') \wedge \sK(p';q)$} \ar[rr] \ar[dd]& & \parbox{4cm}{$Z(q_0) \wedge \sJ(q_0,q_1) \wedge \cdots\\ \wedge \sJ(q_n,p) \wedge \sK(p;q)$} \ar[d]\\	
	          & & \parbox{4cm}{$Z(q_0) \wedge \sJ(q_0,q_1) \wedge \cdots\\ \wedge \sJ(q_{n-1},q_n) \wedge \sK(q_n;q)$} \ar[d]\\
			\parbox{4cm}{$Z(q_0) \wedge \sJ(q_0,q_1) \wedge \cdots\\ \wedge \sJ(q_n,p') \wedge \sK(p';q)$} \ar[r] & \parbox{4cm}{$Z(q_0) \wedge \sJ(q_0,q_1) \wedge \cdots\\ \wedge \sJ(q_{n-1},q_n) \wedge \sK(q_n;q)$} \ar[r]& \parbox{4cm}{$Z(q_0) \wedge \sJ(q_0,q_1) \wedge \cdots\\ \wedge \sJ(q_{n-1},q_n) \wedge \sV_q(q_n)$} 
		\end{tikzcd}
	\]
	By \cref{dfn:module_morphism_induced_CM}, the $\sJ$-module morphism $W$ defines a map of filtered chain complexes
	\[ CM(B(Z,\sJ,\sJ)) \longrightarrow CM(|Z|_q[q]) \cong C_\bullet(|Z|_q).\]
	Hence it suffices to show that the induced map on associated gradeds
	\[ \Gr^p CM(B(Z,\sJ,\sJ)) \cong CM(B(Z(p)[p],\sJ,\sJ)) \longrightarrow C_\bullet(|Z(p)[p]|_q) \cong \Gr^p CM(|Z|_q[q]) \]
	is a quasi-isomorphism. 

\end{proof}
\begin{cor}\label{prop:lsys_cm_prop}
    Let $(\sM,\fo,\sE)$ be an $R$-oriented flow category equipped with a local system whose grading function $\mu$ is bounded below. Then, there exists an isomorphism
    \begin{equation}\label{eq:lsys_cm_prop}
        H_\bullet \left( CM(\sM,\fo,\sE;k) \right) \cong \pi_\bullet(|\sM,\fo,\sE| \wedge Hk).
    \end{equation}
    \qed
\end{cor}

Recall that $\sK(a;b)= D(K(a;b)/\partial K(a;b))$, where
\begin{align*}
    K(a;b) &= \bigcup_{a \geq r_1 \geq r_0 \geq b} [0,\infty]^{\{a>x\geq r_1\}} \times [0,\infty]^{\{r_0 \geq x > b\}} \\
    &\subset [0,\infty]^{\{a > x \geq b\}} \times [0,\infty]^{\{a \geq x > b\}}.
\end{align*}
Consider the map 
\[\delta_{ab} \times s_{ab} \colon K(a;b) \longrightarrow J(a,b) \times [b,a]\]
defined as follows: The first component $\delta_{ab}$ of the map is defined on the subset $[0,\infty]^{\{a>x\geq r\}} \times [0,\infty]^{\{r \geq x > b\}}$ by
\begin{align*}
	[0,\infty]^{\{a>x\geq r\}} \times [0,\infty]^{\{r \geq x > b\}} &\longrightarrow J(a,b) \\
	(\vec x,x_r) \times (y_r,\vec y) &\longmapsto \begin{cases} 
		\vec y \quad &\text{if } r=a\\ 
		(\vec x,x_r+y_r,\vec y) \quad &\text{if } a>r>b\\
		\vec x  \quad &\text{if } r=b
	\end{cases}.
\end{align*}
as in \eqref{eq:ktoj}. To define the second component $s_{ab}$, we fix an increasing homeomorphism $\varsigma \colon [-\infty,\infty] \to [0,1]$. Then $s_{ab}$ is defined on the subset $[0,\infty]^{\{a>x\geq r\}} \times [0,\infty]^{\{r \geq x > b\}}$ by
\begin{align*}
	[0,\infty]^{\{a>x\geq r\}} &\times [0,\infty]^{\{r \geq x > b\}} \longrightarrow [b,a] \\
	\vec x\times \vec y &\longmapsto a- \left(\sum_{a > i \geq r} \varsigma(x_i+\dots+x_{r}) + \sum_{r \geq j > b} \varsigma(-(y_{r}+\dots+y_j)) \right).
\end{align*}

\begin{lem}
	The map $\delta_{ab} \times s_{ab} \colon K(a;b) \longrightarrow J(a,b) \times [b,a]$ satisfies the following properties:
	\begin{enumerate}
		\item It is a homeomorphism. 
		\item It maps $\partial K(a;b)$ onto $(\partial J(a,b) \times[b,a]) \cup (J(a,b) \times \{b,a\})$.
		\item It is compatible with the action of $J(a,b)$ in the sense that the following diagrams commute
		
		\[
			\begin{tikzcd}[column sep=1.5cm]
				J(a,a') \times K(a';b) \rar{\id \times \delta_{a'b} \times s_{a'b}} \dar{\text{\eqref{eq:Rincl}}} &  J(a,a') \times J(a',b) \times [b,a'] \dar{i_{aa'b}\times \id} \\
				K(a;b) \rar{\delta_{ab} \times s_{ab}} & J(a,b) \times [b,a]
            \end{tikzcd},
        \]
        and
        \[
            \begin{tikzcd}[column sep=1.5cm]
                K(a;b') \times J(b',b) \rar{\delta_{ab'} \times s_{ab'} \times \id} \dar{\text{\eqref{eq:Rincl2}}} &  J(a,b') \times [b',a] \times J(b',b) \dar{i_{ab'b} \times \id}\\
				K(a;b) \rar{\delta_{ab} \times s_{ab}} & J(a,b) \times [b,a]
			\end{tikzcd}.
		\]
	\end{enumerate}
\end{lem}
\begin{proof}
	\begin{enumerate}
	    \item It suffices to show that $\delta_{ab} \times s_{ab}$ is a bijection. For $\vec x \in J(a,b)$, the fiber of $\delta_{ab}$ over $\vec x$ is given by
    	\begin{align*}
    		\delta_{ab}^{-1}(\vec x) = & [0,\infty]^{\{a\}} \times \{\vec x\}\\
            &\cup \{ (x_{a-1},\dots,x_{r+1},t) \times (x_r-t,\dots,x_{b+1}) \ \mid \ a > r > b, \; 0 \leq t \leq x_r \} \\
            &\cup \{\vec x\} \times [0,\infty]^{\{b\}}.
    	\end{align*}
        The first and the third summands map bijectively onto
        \[ \left[a-\sum_{a > j >b}\varsigma ( -(x_{a-1}+\dots+x_j)),a\right] \quad \text{and} \quad \left[b,a-\sum_{a > i >b}\varsigma ((x_i+\dots+x_{b+1}))\right],\]
        while the second summand is in bijection with pairs 
		\[\{(r,t) \mid a > r > b, \; 0 \leq t \leq x_r\}.\]
		The function $s_{ab}$ is strictly increasing with respect to the lexicographical ordering on the pairs $(r,t)$ under this bijection. It then suffices to observe that 
		\begin{align*}
		    s_{ab}(a-1,0) &= a-\sum_{a > j >b}\varsigma ( -(x_{a-1}+\dots+x_j)) \\
            s_{ab}(b+1,\infty) &= a-\sum_{a > i >b}\varsigma ((x_i+\dots+x_{b+1})).
		\end{align*}
        \item This is clear by construction.
        \item We check commutativity in the first diagram by hand on the subset $[0,\infty]^{\{a'>x\geq r\}} \times [0,\infty]^{\{r \geq x > b\}} \subset K(a';b)$. Namely, we have
        \begin{align*}
            &((i_{aa'b} \times \id) \circ (\id \times \delta_{a'b} \times s_{a'b}))(\vec x,((\vec y,y_r),(z_r,\vec z))) \\
            &\qquad = (i_{aa'b} \times \id)(\vec x,(\vec y,y_r+z_r,\vec z),s_{a'b}((\vec y,y_r),(z_r,\vec z))) \\
            & \qquad = ((\vec x,\infty, \vec y, y_r+z_r,\vec z),s_{a'b}((\vec y,y_r),(z_r,\vec z))),
        \end{align*}
        and
        \begin{align*}
            &((\delta_{ab} \times s_{ab}) \circ \text{\eqref{eq:Rincl}})(\vec x,\vec y,y_r,z_r,\vec z) = (\delta_{ab} \times s_{ab})((\vec x,\infty,\vec y,y_r),(z_r,\vec z)) \\
            & \qquad = ((\vec x,\infty,\vec y,y_r+z_r,\vec z),s_{ab}((\vec x,\infty,\vec y,y_r),(z_r,\vec z))).
        \end{align*}
        By definition
        \begin{align*}
            &s_{ab}((\vec x,\infty,\vec y,y_r),(z_r,\vec z)) \\
            &\qquad = a- \sum_{a > i \geq a'} \varsigma(\infty) - \sum_{a' > i \geq r} \varsigma(y_i + \dots + y_r) -\sum_{r \geq j > b} \varsigma(-(z_r+\dots+z_j))\\
            &\qquad = a-(a-a') - \sum_{a' > i \geq r} \varsigma(y_i + \dots + y_r) - \sum_{r \geq j > b} \varsigma(-(z_r+\dots+z_j)) \\
            &\qquad = s_{a'b}((\vec y,y_r),(z_r,\vec z)).
        \end{align*}
        Checking that the second diagram commutes is done similarly.
	\end{enumerate}
\end{proof}
\begin{cor}
	There is a natural transformation of $\sJ$-bimodules
	\begin{equation}\label{eq:sktosj}
		\sK \Longrightarrow \sJ
	\end{equation}
	which is a pointwise weak equivalence.
    \qed
\end{cor}

We start by defining an intermediate space and list its properties which will be used in the proof of \cref{prop:jmod_fun} below. In analogy with the construction of $KV_q$ in the proof of \cref{prop:rcon}, for every $a, b \in \IZ$, we define 
\[ KJ(a,b) \coloneqq \bigcup_{a \geq r_1 \geq r_0 > b} [0,\infty]^{\{a > x \geq r_1\}} \times \partial[0,\infty]^{\{r_0 \geq x > b\}}.\]
The boundary of $KJ(a,b)$ is defined to be
\[\partial KJ(a,b) \coloneqq \bigcup_{a \geq r_1 \geq r_0 > b} \partial[0,\infty]^{\{a > x \geq r_1\}} \times \partial[0,\infty]^{\{r_0 \geq x > b\}}.\]
Denote by $\nu_{ab}$ the trivial rank $1$ vector bundle over $KJ(a,b)$. Define
\[ \sKJ(a,b) \coloneqq D\left( \frac {KJ(a,b)^{\nu_{ab}}}{\partial KJ(a,b)^{\nu_{ab}}} \right).\]
Note that there are composition maps
\[ KJ(a,b) \times V_q(b) \longrightarrow KV_q(a) \]
which induce maps
\[\sKJ(a,b) \wedge \sV_q(b) \longrightarrow \sKV_q(a).\]
Similar to \eqref{eq:Rincl}, $\sKJ(a,b)$ is a $\sJ^{\mathrm{op}}$-module, and the above maps yield a natural transformation of $\sJ^{\mathrm{op}}$-modules
\begin{equation}\label{eq:skjtoskv}
	\sKJ \Longrightarrow \sKV_q.
\end{equation}
The proof of the following result is carried out exactly like the proof of \cref{prop:rcon}.
\begin{prop} There are natural transformations of $\sJ$-bimodules
	\[ B(\sK,\sJ,\sJ) \Longleftarrow \sKJ \Longrightarrow \sJ \]
	which are pointwise weak equivalences.
    \qed
\end{prop}

\begin{prop}\label{prop:jmod_fun}
	Let $Z_1,Z_2 \colon \sJ \to \mod{A}$ be $A$-linear $\sJ$-modules and $W \colon \sK \Rightarrow F_A(Z_1,Z_2)$ be an $A$-linear $\sJ$-module morphism. Then, the following diagram is commutative
	\[
	    \begin{tikzcd}[row sep=scriptsize,column sep=1cm]
	        H_\bullet(CM(Z_1;k)) \rar{CM(W)} \dar{\cong}[swap]{\eqref{eq:cm_prop}} & H_\bullet(CM(Z_2;k)) \dar{\cong}[swap]{\eqref{eq:cm_prop}} \\
	        \pi_\bullet(|Z_1| \wedge Hk) \rar{|W|\wedge Hk} & \pi_\bullet(|Z_2| \wedge Hk)
	    \end{tikzcd}.
    	\]
\end{prop}
\begin{proof}
	We have the commutative diagram of chain complexes
	\[
		\begin{tikzcd}[sep=scriptsize]
			CM(B(Z_1,\sJ,\sJ)) \rar{\text{(i)}} & C_\bullet(|Z_1|_q) \\
			CM(B(Z_1,\sJ,\sKJ)) \rar{\text{(ii)}} \uar{\text{(a)}} \dar[swap]{\text{(b)}}& C_\bullet(B(Z_1,\sJ,\sKV_q)) \uar[swap]{(\text{a}')} \dar{(\text{b}')} \\
			CM(B(Z_1,\sJ,B(\sK,\sJ,\sJ))) \rar{\text{(iii)}} \dar[leftrightarrow, swap]{\text{(c)}} & C_\bullet(B(Z_1,\sJ,B(\sK,\sJ,\sV_q))) \dar[leftrightarrow]{(\text{c}')}\\
			CM(B(B(Z_1,\sJ,\sK),\sJ,\sJ)) \rar{\text{(iv)}} \dar[swap]{\text{(d)}}& C_\bullet(B(B(Z_1,\sJ,\sK),\sJ,\sV_q)) \dar{(\text{d}')}\\
			CM(B(Z_2,\sJ,\sJ)) \rar{\text{(v)}} & C_\bullet(|Z_2|_q)
		\end{tikzcd}
	\]
	where the arrows are defined as follows:
	\begin{itemize}
		\item The horizontal arrows (i), (iv), and (v) are induced via the maps in \eqref{eq:simplicial_maps_W} with $Z = Z_1$, $Z = B(Z_1,\sJ,\sK)$ and $Z = Z_2$, respectively.
		\item The horizontal arrow (ii) is obtained, similarly as in \eqref{eq:simplicial_maps_W}, as the geometric realization of the composition:
	\begin{align*}
            &\sK(p;q_0) \wedge \sJ(q_0,q_1) \wedge \dots \wedge \sJ(q_{n-1},q_n) \wedge \sKJ(q_n,r) \wedge \sK(r;q) \\
            &\qquad \overset{\eqref{eq:morse_eq6}}{\longrightarrow} \sK(p;q_0) \wedge \sJ(q_0,q_1) \wedge \dots \wedge \sKJ(q_n,r) \wedge \sV_q(r) \\
            &\qquad \overset{\eqref{eq:skjtoskv}}{\longrightarrow}\sK(p;q_0) \wedge \sJ(q_0,q_1) \wedge \dots \wedge \sKV_q(q_n).
        \end{align*}
		\item The horizontal arrow (iii) is obtained as follows. Consider the following composition
        \begin{align*}
            &\sK(p;q_0) \wedge \sJ(q_0,q_1) \wedge \dots \wedge \sJ(q_{n-1},q_n) \wedge \sJ(q_n,r) \wedge \sK(r;q) \\
            &\qquad \longrightarrow \sK(p;q_0) \wedge \sJ(q_0,q_1) \wedge \dots \wedge \sJ(q_{n-1},q_n) \wedge \sK(q_n;q) \\
            &\qquad \longrightarrow \sK(p;q_0) \wedge \sJ(q_0,q_1) \wedge \dots \wedge \sJ(q_{n-1},q_n) \wedge \sV_q(q_n),
        \end{align*}
        where the second map is induced by the natural transformation \eqref{eq:morse_eq6}. These maps induce maps 
        \begin{equation}\label{eq:geom_realiz_intermed}
            B(\sK(p;-),\sJ,\sJ(-,r)) \wedge \sK(r;q) \longrightarrow B(\sK(p;-),\sJ,\sV_q).
        \end{equation}
        Then (ii) is induced via the simplicial maps
        \begin{align*}
    		&Z_1(q_0) \wedge \sJ(q_0,q_1) \wedge \dots \wedge \sJ(q_{n-1},q_n) \wedge B(\sK(q_n;-),\sJ,\sJ(-,p)) \wedge \sK(p;q) \\
            &\qquad \overset{\id^{\wedge(n+1)} \wedge \text{\eqref{eq:geom_realiz_intermed}}}{\longrightarrow} Z_1(q_0) \wedge \sJ(q_0,q_1) \wedge \dots \wedge \sJ(q_{n-1},q_n) \wedge B(\sK(q_n;-),\sJ,\sV_q).
    	\end{align*}
\item The vertical arrows (a), ($\text{a}'$) and (b), ($\text{b}'$) are induced, respectively, by natural transformations of $\sJ$-modules 
	\begin{align*}
		B(\sK, \sJ,\sJ) \overset{\text{\eqref{eq:sktosj}}}{\Longleftarrow} B(\sK,\sJ,\sK) &\Longleftarrow \sKJ \Longrightarrow \sJ, \quad \text{and}\\
		B(\sK, \sJ,\sV_q) &\overset{\text{\eqref{eq:kv_comparison5}}}{\Longleftarrow} \sKV_q \Longrightarrow \sV_q,
	\end{align*}
        which are pointwise weak equivalences.
		\item The vertical arrows (c) and ($\text{c}'$) are defined using the equivalences $B(B(Z,\sJ,\sK),\sJ,\sJ) \simeq B(Z,\sJ,B(\sK,\sJ,\sJ))$ and $B(B(Z,\sJ,\sK),\sJ,\sV_q) \simeq B(Z,\sJ,B(\sK,\sJ,\sV_q))$.
		\item The vertical arrows (d) and ($\text{d}'$) are defined using the natural transformation $B(Z_1,\sJ,\sK) \Rightarrow Z_2$ induced by $W$.
	\end{itemize}
    It is clear from the definitions of the various maps that the top and bottom squares are commutative. For the middle square, note that $B(Z_1,\sJ,B(\sK,\sJ,\sJ)) \simeq B(B(Z_1,\sJ,\sK),\sJ,\sJ)$ is the geometric realization of the bisimplicial $A$-module $B_\bullet(Z_1,\sJ,B_\bullet(\sK,\sJ,\sJ))$, which is pointwise equivalent to $B_\bullet(B_\bullet(Z_1,\sJ,\sK),\sJ,\sJ)$. The maps $\id^{\wedge (n+1)} \wedge \text{\eqref{eq:geom_realiz_intermed}}$ and \eqref{eq:simplicial_maps_W} are compatible with this equivalence, showing the commutativity of the middle square.
\end{proof}
\begin{cor}\label{cor:jmod_fun}
    Let $(\sM_1,\fo_1,\sE_1)$ and $(\sM_2,\fo_2,\sE_2)$ be two $R$-oriented flow categories with local systems. Let $(\sN,\fm,\sE) \colon (\sM_1,\fo_1,\sE_1) \to (\sM_2,\fo_2,\sE_2)$ be an $R$-oriented flow bimodule with local system. The following diagram is commutative
    \[
    \begin{tikzcd}[row sep=scriptsize,column sep=1cm]
        H_\bullet(CM(Z_{\sM_1,\fo_1,\sE_1};k)) \rar{CM(\sN;k)} \dar{\cong}[swap]{\eqref{eq:lsys_cm_prop}} & H_\bullet(CM(Z_{\sM_2,\fo_2,\sE_2};k)) \dar{\cong}[swap]{\eqref{eq:lsys_cm_prop}} \\
        \pi_\bullet(|\sM_1,\fo_1,\sE_1|\wedge Hk) \rar{|\sN,\fm,\sE|} & \pi_\bullet(|\sM_2,\fo_2,\sE_2|\wedge Hk)
    \end{tikzcd}.
    \]
    \qed
\end{cor}

\subsection{Morse--Smale flow categories}\label{subsec:morsesmale}
    In this section we define an $R$-oriented flow category associated to a Morse--Smale function on a compact Riemannian manifold with boundary.
    \begin{notn}
        \begin{itemize}
            \item $M$ is a compact smooth manifold with boundary.
            \item $f\colon M \to \IR$ is a smooth function.
            \item $g$ is a Riemannian metric on $M$.
            \item $W^s(-)$ and $W^u(-)$ denote the stable and unstable manifolds of a point under the flow of the negative gradient vector field $-\nabla^gf$, respectively.
        \end{itemize}
    \end{notn}
    \begin{asmpt}\label{asmpt:msm}
        \begin{itemize}
            \item $f$ is a proper Morse function on $M$.
            \item $(f,g)$ is a Morse--Smale pair. Namely, for any $x,y \in \crit(f)$, $W^u(x)$ and $W^s(y)$ intersect each other transversely.
            \item The negative gradient vector field $-\nabla^gf$ points inwards along the boundary $\partial M$. 
        \end{itemize}
    \end{asmpt}
    For a Morse--Smale pair $(f,g)$, the space of unparametrized gradient flow lines from $x$ to $y$
    \[ (W^u(x) \cap W^s(y))/\IR \]
    is a smooth manifold provided that $x\neq y$. It admits a compactification by a smooth manifold with corners consisting of broken gradient trajectories \cite{wehrheim2012smooth}. Denote this compactification by $\sM(x;y)$. Concatenation of broken flow lines induces smooth maps
    \begin{equation}\label{eq:morsecomp} 
        \sM(x;y) \times \sM(y;z) \longrightarrow \sM(x;z). 
    \end{equation}
    \begin{lem}\label{lem:msm}
        Let $(f,g)$ be as in \cref{asmpt:msm}. There is an $R$-oriented flow category $(\sM(f,g),\fo_{(f,g)})$ defined by the following:
        \begin{itemize}
            \item $\Ob(\sM(f,g)) = \crit(f)$.
            \item $\mu(x) = \ind_f(x)$ for $x \in \Ob(\sM(f,g))$, where $\ind_f$ denotes the Morse index.
            \item The morphisms from $x$ to $y$ is given by $\sM(x;y)$, with compositions given by the maps \eqref{eq:morsecomp}.
            \item The $R$-orientation $\fo_{(f,g)}$ is given by 
            \[ x \longmapsto (T_x \overline W^u(x))_R,\]
            and the compatible isomorphisms
            \[
            (T_x\overline W^u(x))_R \longrightarrow I_R(x,y) \otimes_R (T_y \overline W^u(y))_R
            \]
            that are induced by the exact sequences
            \[
            \begin{tikzcd}[sep=scriptsize]
                0 \rar & T_p (W^u(x) \cap W^s(y)) \rar & T_pW^u(x) \rar & N_pW^s(y) \rar & 0
            \end{tikzcd}
            \]
            and
            \[
            \begin{tikzcd}[sep=scriptsize]
                0 \rar & \IR \rar & T_p(W^u(x) \cap W^s(y)) \rar & T_p\sM(x;y) \rar & 0
            \end{tikzcd}.
            \]
        \end{itemize}
        \qed
    \end{lem}
    \begin{defn}[Morse--Smale flow category]\label{defn:ms_cat} Let $(f,g)$ be a pair satisfying \cref{asmpt:msm}. The \emph{Morse--Smale flow category} of $(f,g)$ is defined as the flow category $\sM(f,g)$ in \cref{lem:msm}. We refer to the $R$-orientation $\fo_{(f,g)}$ defined there as the \emph{canonical $R$-orientation} on $\sM(f,g)$.
    \end{defn}
    \begin{prop}\label{prop:morse_homology}
        $|\sM(f,g), \fo_{(f,g)}|\simeq \varSigma^\infty_+ M \wedge R$. In particular, for $k=\pi_0 R$, $H_\bullet(M;k) \simeq H_\bullet (CM(\sM(f,g);k)).$
    \end{prop}
    \begin{proof} This follows from \cite[Proposition 5.1]{cohen1995floer} by observing that $\fo_{(f,g)} \simeq \fo_{(f,g), \IS} \wedge R$, where $\fo_{(f,g),\IS}$ is the canonical $\IS$-orientation on $\sM(f,g)$. The second assertion follows from \cref{lem:flow_htpy_morse}; since $M$ is compact and $f$ is proper, $f$ has a finite number of critical points and hence $Z_{\sM(f,g),\fo_{(f,g)}}$ is bounded below.
    \end{proof}
    
    \begin{asmpt}\label{asmpt:msmseq}
    Let $M_1 \subset M_2 \subset \cdots$ be an increasing sequence of smooth manifolds such that $M_n$ is a compact submanifold-with-boundary of $M_{n+1}$ for each $n \in \IZ_{\geq 1}$. Let $g=\{g_n\}_{n =1}^\infty$ be a sequence of Riemannian metrics on the manifolds $M_n$ and
    \[ f \colon \bigcup_{n =1}^\infty M_n \longrightarrow \IR\]
    be a function such that for every $n \in \IZ_{\geq 1}$
    \begin{enumerate}
        \item $(f|_{M_n},g_n)$ satisfies \cref{asmpt:msm}, 
        \item for $x \in \crit_{f_n}$, $\ind_{f_n} (x) = \ind_{f_{n+1}} (x)$, where $\ind$ denotes the Morse index, and
        \item $g_{n+1}|_{M_n} = g_n$.
    \end{enumerate}
    \end{asmpt}
    Note that we obtain a filtration of flow categories:
    \begin{equation}\label{eq:msincl}
        \sM(f_1,g_1) \subset \sM(f_2,g_2) \subset \cdots .
    \end{equation}
    Moreover the canonical $R$-orientations on these categories are compatible in the sense that $\fo_{(f_{n+1},g_{n+1})}|_{\sM(f_n,g_n)} = \fo_{(f_n,g_n)}$ for each $n \in \IZ_{\geq 1}$.
    \begin{lem}\label{lem:msmsep}
    Let $(f,g)$ be a pair of a Morse function and a Riemannian metric satisfying \cref{asmpt:msmseq}. Then we have the following:
    \begin{enumerate}
        \item There is a flow category $\sM(f,g)$ with 
              \begin{itemize}
                  \item set of objects given by  $\bigcup_{n=1}^\infty \Ob(\sM(f_n,g_n))$,
                  \item $\mu(x) = \ind_{f_n}(x)$ for $x\in \Ob(\sM(f_n,g_n))$, and
                  \item $\sM(f,g) (x,y)$ given by $\sM(f_n,g_n) (x,y)$ for $n \gg 1$.
              \end{itemize}
            We refer to it as the \emph{Morse--Smale flow category} of the pair $(f,g)$.
           \item $\fo_{(f,g)}(x) \coloneqq \fo_{(f_n,g_n)}(x)$, for $n \gg 1$, defines an $R$-orientation on $\sM(f,g)$. We refer to it the \emph{canonical orientation} on $\sM(f,g)$.
           \item There exist natural inclusions $\iota_n \colon \sM(f_n,g_n) \hookrightarrow \sM(f,g)$ of full subcategories, which are compatible with the inclusions \eqref{eq:msincl} and which preserve the canonical orientations.
           \item $|\sM(f,g), \fo_{(f,g)}| \simeq \varSigma^\infty_+ M \wedge R$.
    \end{enumerate}
    \qed
    \end{lem}

    \begin{notn}\label{notn:morse_flowline_compact}
    Let $(M,f,g)$ be a tuple of a compact Riemannian manifold with boundary and a Morse--Smale function satisfying assumption \cref{asmpt:msm}.
    
    For a critical point $x \in \crit(f)$, let $\Wu{x}{}$ ($\overline{W}^s(x)$) denote the the manifolds-with-corners compactification of $W^u(x)$ ($W^s(x)$) obtained by adding broken flow lines of $-\nabla^g f$ starting (ending) at $x$, see \cite[Definition 2.7]{latour1994existence} and also \cite[Section 2.4.6; Theorem 2.3; Section 4.9.b]{barraud2007lagrangian,wehrheim2012smooth,audin2014morse}.
    \end{notn}
    \begin{rem}
        We warn the reader that $\Wu{x}{}$ and $\overline{W}^s(x)$ are in general different than the closures of $W^u(x)$ and $W^s(x)$ inside $M$. In particular, these are manifolds with corners with the interiors given by $W^u(x)$ and $W^s(x)$, and are thus always contractible, see \cite[Section 4.9.c]{audin2014morse}.
    \end{rem}

    Let $X$, $Y$, and $Z$ be smooth manifolds. Choose three pairs $(f_X,g_X)$, $(f_Y,g_Y)$ and $(f_Z,g_Z)$ satisfying \cref{asmpt:msm}. Let $(\sM(f_X,g_X),\fo_{(f_X,g_X)})$, $(\sM(f_Y,g_Y),\fo_{(f_Y,g_Y)})$, and $(\sM(f_Z,g_Z),\fo_{(f_Z,g_Z)})$ denote their $R$-oriented Morse--Smale flow categories.
    \begin{defn}\label{dfn:morse_pontryagin}
        Let $\varphi \colon X \times Y \to Z$ be a smooth map such that $\varphi|_{\overline W^u_{-\nabla f_X}(x) \times \overline W^u_{-\nabla f_Y}(y)} \pitchfork \overline W^s_{-\nabla f_Z}(z)$. We define an $R$-oriented flow multimodule
        \[ \sN_\varphi \colon \sM(f_X,g_X),\sM(f_Y,g_Y) \longrightarrow \sM(f_Z,g_Z) \]
        by
        \[ \sN_\varphi(x,y;z) \coloneqq \left( \overline W^u_{-\nabla f_X}(x) \times \overline W^u_{-\nabla f_Y}(y) \right) \times_\varphi \overline W^s_{-\nabla f_Z}(z). \]
        The $R$-orientation is defined via the following composition of canonical isomorphisms of $R$-line bundles over $\sN_\varphi(x,y;z)$
        \begin{align*}
            (T\sN_\varphi(x,y;z))_R \otimes_R \fo_{(f_Z,g_Z)}(z) &\simeq (T\overline W^u_{-\nabla f_X}(x))_R \otimes_R (T\overline W^u_{-\nabla f_Y}(y))_R \\
            &\quad \otimes_R (T\overline W^s_{-\nabla f_Z}(z))_R \otimes_R (-T_z Z)_R\otimes_R \fo_{(f_Z,g_Z)}(z) \\
            &\simeq (T\overline W^u_{-\nabla f_X}(x))_R \otimes_R (T\overline W^u_{-\nabla f_Y}(y))_R \\
            &\simeq \fo_{(f_X,g_X)}(x) \otimes_R \fo_{(f_Y,g_Y)}(y).
        \end{align*}
    \end{defn}

    \begin{lem}\label{lem:InducedProductMaps}
    The map
    \[ |\sN_\varphi| \colon |\sM(f_X,g_X),\fo_{(f_X,g_X)}| \wedge_R |\sM(f_Y,g_Y),\fo_{(f_Y,g_Y)}| \longrightarrow |\sM(f_Z,g_Z), \fo_{(f_Z,g_Z)}| \]
    is homotopic to the induced map
    \[\varSigma^\infty_+ \varphi \wedge R \colon (\varSigma^\infty_+ X \wedge R) \wedge_R (\varSigma^\infty_+ Y \wedge R) \longrightarrow \varSigma^\infty_+ Z \wedge R.\]
    \hfill \qed
    \end{lem}
    \begin{rem}\label{rem:induced_map}
        Taking $Y = \varnothing$, there are analogs of \cref{dfn:morse_pontryagin} and \cref{lem:InducedProductMaps}. Namely, if $\varphi \colon X \to Z$ is a smooth map such that $\varphi|_{\overline W^u_{-\nabla f_X}(x)} \pitchfork \varphi|_{\overline W^s_{-\nabla f_Z}(z)}$, then
        \[
        |\sN_{\varphi}| \colon |\sM(f_X,g_X),\fo_{(f_X,g_X)}| \longrightarrow |\sM(f_Z,g_Z), \fo_{(f_Z,g_Z)}|
        \]
        is homotopic to the induced map $\varSigma^\infty_+ \varphi \wedge R \colon \varSigma^\infty_+ X \wedge R \to \varSigma^\infty_+ Z \wedge R$.
    \end{rem}

    \begin{defn}[Cohomological Morse--Smale flow category]\label{defn:coh_morse-smale_flow_cat}
        Let $(f,g)$ be as in \cref{asmpt:msm}. Define $\sM^{-\bullet}(f,g)$ to be the $R$-oriented flow category with
        \begin{itemize}
            \item $\Ob(\sM^{-\bullet}(f,g)) = \crit(f)$.
            \item $\mu(x) = -\ind_f(x)$ for $x\in \Ob(\sM^{-\bullet}(f,g))$, where $\ind_f$ denotes the Morse index.
            \item The morphisms from $x$ to $y$ are given by $\sM^{-\bullet}(f,g)(x;y) \coloneqq \sM(f,g)(y;x)$.
            \item The $R$-orientation $\fo^{-\bullet}_{(f,g)}$ is given by 
            \[
            x\longmapsto (-T_x\overline W^u_{-\nabla f}(x))_R \simeq (T_x \overline W^s_{-\nabla f}(x))_R \otimes_R (-T_xM)_R.
            \]
        \end{itemize}
    \end{defn}
    \begin{rem}\label{rmk:morse_cohomology_flow}
        Recall from \cref{prop:morse_homology} that the homology of the Morse complex of $\sM(f,g)$ is isomorphic to Morse homology. Similarly, the homology of the Morse complex of the $R$-oriented flow category $\sM^{-\bullet}(f,g)$ is isomorphic to Morse cohomology with reversed grading. This flow category plays an important role in \cref{sec:construction_OC} when comparing with the flow category for Lagrangians.
    \end{rem}
    \begin{defn}\label{dfn:unit_morse}
        Let $\mathcal M_\ast$ denote the $R$-oriented point flow category (see \cref{dfn:point_flow_cat}). The \emph{unit} in $|\sM^{-\bullet}(f,g),\fo^{-\bullet}_{(f,g)}|$ is the morphism 
        \[ \sN_f \colon R \longrightarrow |\sM^{-\bullet}(f,g),\fo^{-\bullet}_{(f,g)}|, \]
        that is defined as the CJS realization of the $R$-oriented flow bimodule $\sN_f \colon \sM_\ast \to \sM^{-\bullet}(f,g)$ defined by $(p,x) \mapsto \overline W^u_{-\nabla f}(x)$. The $R$-orientation on $\sN_f$ is defined via the identity isomorphism $(T_x\overline W^u_{-\nabla f}(x))_R \otimes_R (-T_x \overline W^u_{-\nabla f}(x))_R \simeq R$.
    \end{defn}
    \begin{rem}
        For $R = Hk$, where $k$ is a discrete ring, the unit $\sN_f$ defined in \cref{dfn:unit_morse} is a cohomological unit for the cup product on Morse cohomology, as expected.
    \end{rem}
    \begin{lem}\label{lma:coh_morse_flow_thom_space}
        Let $(f,g)$ be a pair satisfying \cref{asmpt:msmseq}. Assume that $M$ is closed.
        \begin{enumerate}
            \item $|\sM^{-\bullet}(f,g),\fo^{-\bullet}_{(f,g)}| \simeq M^{-TM} \wedge R$.
            \item Let $M$ be $R$-oriented. The CJS realization of $\sN_f$ in \cref{dfn:unit_morse} coincides with the unit map $R \to M^{-TM} \wedge R$.
        \end{enumerate}
    \end{lem}
    \begin{proof}
        For item (i), consider the disk bundle of the normal bundle $\pi \colon D \to M$ of an embedding $M \hookrightarrow \IR^N$ for some large $N$. Equip it with a Morse function $-\pi^\ast f + G$, where $G(x,v) = -|v|^2$ in local coordinates $x \in M$ and $v\in D_x$. The homological Morse flow category of this Morse function coincides with a shift of $\sM^{-\bullet}(f,g)$. The result then follows from \cref{rem:CJS_shift} by an analog of \cref{prop:morse_homology} for Morse functions with the negative gradient pointing outwards along the boundary. For item (ii), note that the unit is defined by the desuspension of the Pontryagin--Thom map of the embedding $M \hookrightarrow \IR^N \subset S^N$. The assertion then follows by an analog of \cref{lem:InducedProductMaps} (in particular \cref{rem:induced_map}) for functions with the negative gradient pointing outwards along the boundary.
    \end{proof}
    \begin{lem}\label{lem:morse_pd}
        Let $(f,g)$ and $(f',g')$ be two pairs on $M$ and $N$, respectively, satisfying \cref{asmpt:msmseq}. Assume that $M$ is $R$-oriented and that there is a smooth map $\varphi \colon M \to N$ such that $\varphi|_{\overline W^s_{-\nabla f}(x)} \pitchfork \overline W^s_{-\nabla f'}(y)$ for every $x \in \crit f$ and $y \in \crit f'$. Then the assignment
        \[
        \sI \colon \sM^{-\bullet}(f,g) \longrightarrow \sM(f',g'), \quad \sI(x;y) \coloneqq \overline W^s_{-\nabla f}(x) \times_\varphi \overline W^s_{-\nabla f'}(y)
        \]
        defines an $R$-oriented degree $n$ flow bimodule, where $n = \dim M$.
    \end{lem}
    \begin{proof}
        By construction $\sI(x;y)$ is a smooth manifold with corners of dimension $\dim \sI(x;y) = n + \mu^{-\bullet}(x) - \mu(y)$. By analyzing broken flow lines we see that there are maps
        \begin{align*}
            \sM^{-\bullet}(f,g)(x;x') \times \sI(x';y) &\longrightarrow \sI(x;y) \\
            \sI(x;y') \times \sM(f',g')(y';y) &\longrightarrow \sI(x;y)
        \end{align*}
        that are diffeomorphisms onto faces of $\sI(x;y)$ such that defining
        \begin{align*}
            \partial_i \sI(x;y) &\coloneqq \bigsqcup_{\substack{x'\in \crit(f) \\ \mu^{-\bullet}(x) - \mu^{-\bullet}(x') = i}} (\sM^{-\bullet}(f,g)(x;x') \times \sI(x';y)) \\
            &\qquad \sqcup \bigsqcup_{\substack{y'\in \crit(f') \\ n + \mu^{-\bullet}(x) - \mu(y) +1 = i}} (\sI(x;y') \times \sM(f',g')(y';y)),
        \end{align*}
        endows $\sI(x;y)$ with the structure of a $\ang{n+\mu^{-\bullet}(x)-\mu(y)}$-manifold (see \cref{def:k_mfd}). Letting $p \in \sI(x;y)$ We have canonical isomorphisms of $R$-line bundles over $\sI(x;y)$
        \begin{align*}
            (T_p \sI(x;y))_R &\simeq (T_p \overline W^s_{-\nabla f}(x))_R \otimes_R (T_p \overline W^s_{-\nabla f'}(y))_R \otimes_R (-T_pM)_R \\
            &\simeq (T_p \overline W^s_{-\nabla f}(x))_R \otimes_R (-T_p \overline W^u_{-\nabla f'}(y)).
        \end{align*}
        The $R$-orientation on the flow bimodule $\sI$ is given by the composition of isomorphisms of $R$-line bundles over $\sI(x;y)$
        \begin{align*}
            (T_p\sI(x;y))_R \otimes_R \fo_{(f',g')}(y) &\simeq (T_p \overline W^s_{-\nabla f}(x))_R \simeq (T_x \overline W^s_{-\nabla f}(x))_R \\
            &\simeq (-T_x \overline W^u_{-\nabla f}(x))_R = \fo_{(f,g)}(x),
        \end{align*}
        where the two first isomorphisms are canonical and the last isomorphism uses the $R$-orientation on $M$.
    \end{proof}
    \begin{lem}\label{lma:pd_thom_class}
        Assume that $M$ is a closed $R$-oriented manifold of dimension $n$. The CJS realization of the flow bimodule $\sI$ defined in \cref{lem:morse_pd} coincides with the composition 
        \[
            M^{-TM} \wedge R \longrightarrow \varSigma^{\infty-n}_+ M \wedge R \longrightarrow \varSigma^{\infty-n}_+ N \wedge R
        \]
        given by the Thom isomorphism (\cref{thm:ori_thom}(ii)) followed by the map of $R$-modules induced by the smooth map $\varphi \colon M \to N$.
    \end{lem}
    \begin{proof}
        It suffices to consider the case when $\varphi$ is identity. The flow category $\sM^{-\bullet}(-f,g)$, ignoring orientations, coincides with $\sM(f,g)$ up to a shift. Moreover, the trivialization of $(TM)_R$ induced by the $R$-orientation further identifies the orientations on these flow categories. Therefore, we have an identification of $R$-oriented flow categories
        \begin{equation}\label{eq:pd_thom_class}
            (\sM^{-\bullet}(-f,g),\fo^{-\bullet}_{-f,g}) \cong (\sM(f,g),\fo_{f,g}).
        \end{equation} 
        This induces an identification of CJS realizations $M^{-TM}\wedge R \simeq \varSigma^{\infty-n}_+M \wedge R$, and it can be verified that this equivalence coincides with the Thom isomorphism corresponding to the $R$-orientation on $M$.
        
        Fix an embedding $M \hookrightarrow \IR^N$ for some $N$ and denote the disk bundle of the normal bundle by $D$. Up to a shift, both $\sM^{-\bullet}(-f,g)$ and $\sM^{-\bullet}(f,g)$ are homological Morse--Smale categories on $D$ (with negative gradient pointing outwards along boundary, as in the proof of \cref{lma:coh_morse_flow_thom_space}). To finish the proof, observe that the flow bimodule $\sI$ as defined in \cref{lma:pd_thom_class} coincides with the flow bimodule $\sN_{\id_D} \colon \sM^{-\bullet}(f,g) \to \sM^{-\bullet}(-f,g)$ under the identification of $R$-oriented flow categories \eqref{eq:pd_thom_class}, whose geometric realization coincides with the identity map $\id_{M^{-TM}} \colon M^{-TM} \wedge R \to M^{-TM} \wedge R$ composed with the identification $M^{-TM}\wedge R \simeq \varSigma^{\infty-n}_+M \wedge R$ above.
    \end{proof}

\subsection{Morse--Smale flow categories with local systems}
\begin{notn}
    Let $M$ be a closed manifold and let $\xi \colon M_+ \to \BGL_1(R)$ be an $R$-line bundle on $M$. For any $x \in M$ we denote by $\xi_x$ the rank one free $R$-module
    \[ \xi_x \coloneqq \{x\}_+ \hooklongrightarrow M_+ \longrightarrow \BGL_1(R).\]
    For any path $\gamma \colon [0,\ell] \to M$, we denote the induced isomorphism $\xi_x \to \xi_y$ by $\xi_\gamma$. Note that the concatenation of paths corresponds to composition of the induced isomorphisms.
\end{notn}
Given a closed manifold $M$ equipped with a Morse--Smale pair $(f,g)$, recall the construction of the Morse--Smale flow category in \cref{defn:ms_cat}.

\begin{defn}[Morse--Smale orientation twisted by an $R$-line bundle]
Let $M$ be a closed manifold equipped with an $R$-line bundle $\xi \colon M_+ \to \BGL_1(R)$ and $(f,g)$ a Morse--Smale pair on $M$ satisfying \cref{asmpt:msm}. The \emph{$\xi$-twisted $R$-orientation on the flow category $\sM(f,g)$} is the $R$-orientation $\fo_\xi$ defined by the assignment 
\[ \fo_\xi(x) \coloneqq (T_x\overline W^u(x))_R \otimes_R \xi_x,\]
and the compatible isomorphisms
\[
(T_x \overline W^u(x))_R \otimes_R \xi_x \longrightarrow I_R(x,y) \otimes_R (T_y\overline W^u(y))_R \otimes_R \xi_y,
\]
defined similarly to the canonical $R$-orientation on the Morse--Smale flow category in \cref{lem:msm}.
\end{defn}

\begin{notn}
    Let $M$ be a closed manifold and let $(f,g)$ be a Morse--Smale pair on $M$. Fix a basepoint $m_0$ on $M$. Note that using the Riemannian metric $g$, we get evaluation maps into the space of Moore paths
    \[ \ev \colon \sM(f,g)(x;y) \longrightarrow \sP_{xy} M.\]
    Denote by 
    \[\sP M  \colon \sM(f,g) \longrightarrow \mod{(\varSigma^\infty_+ \varOmega M \wedge R)}\]
    the local system of $(\varSigma^\infty_+ \varOmega M \wedge R)$-modules defined on objects by
    \[x \longmapsto \varSigma^\infty_+\sP_{m_0x} M \wedge R\]
    and on morphisms via the composition of maps
    \begin{align*}
        (\varSigma^\infty_+ \sP_{m_0x} M \wedge R) \wedge \varSigma^\infty_+\sM(f,g)(x;y) &\xrightarrow{\id \wedge \ev} (\varSigma^\infty_+\sP_{m_0x} M \wedge R) \wedge \varSigma^\infty_+\sP_{xy} M \\
        &\overset{P}{\longrightarrow} \varSigma^\infty_+\sP_{m_0y} M \wedge R,
    \end{align*}
    where the second map is the concatenation of Moore paths.
\end{notn}
    
    Let $\sP_{m_0}M$ denote the space of Moore paths based at $m_0$ and consider the projection $\pi \colon \sP_{m_0} M \to M$ defined on a Moore path $\gamma \colon [0,\ell] \to M$ by $\pi(\gamma) = \gamma(\ell)$. Abusing the notation, we denote the pullback of $\xi$ under $\pi$ by $\xi$ as well. Since $\sP_{m_0}M$ admits a deformation retract onto $\{m_0\}$, the inclusion $\xi_{m_0} \hookrightarrow (\sP_{m_0} M)^\xi$ is a homotopy equivalence of $R$-modules with homotopy inverse
    \begin{align*}
     \psi \colon (\sP_{m_0} M)^\xi &\overset{\simeq}{\longrightarrow} \xi_{m_0}.
    \end{align*}
    Moreover, the $(\varSigma^\infty_+\varOmega M\wedge R)$-action on $\varSigma^\infty_+\sP_{m_0} M\wedge R$ (given by concatenation) is compatible with the projection to $M$, in the sense that the following diagram commutes
    \[
    \begin{tikzcd}[sep=scriptsize]
        (\varSigma^\infty_+\varOmega M\wedge R) \wedge_R (\varSigma^\infty_+\sP_{m_0} M\wedge R) \rar{P} \drar[swap]{\pi} & \varSigma^\infty_+\sP_{m_0} M\wedge R \dar{\pi}\\
        & \varSigma^\infty_+M \wedge R
    \end{tikzcd},
    \]
    where $P$ is given by concatenation of Moore paths.
    It follows that the $(\varSigma^\infty_+\varOmega M\wedge R)$-action on $\varSigma^\infty_+\sP_{m_0} M\wedge R$ induces an action on the Thom spectrum
    \[ (\varSigma^\infty_+\varOmega M\wedge R) \wedge_R (\sP_{m_0} M)^\xi \longrightarrow (\sP_{m_0} M)^\xi,\] 
    and the following diagram is commutative
    \[
    \begin{tikzcd}[sep=scriptsize]
        (\varSigma^\infty_+\varOmega M\wedge R)\wedge_R (\sP_{m_0}M)^\xi \ar[rr,"P"] \dar{\id \wedge_R \psi} & & (\sP_{m_0}M)^\xi \dar{\psi}\\
        (\varSigma^\infty_+\varOmega M\wedge R) \wedge_R \xi_{m_0} \rar[hook] & (\varSigma^\infty_+\sP_{m_0}M \wedge R) \wedge_R \xi_{m_0} \rar{\psi} & \xi_{m_0}
    \end{tikzcd},
    \]
    where the lower horizontal composition coincides with the $(\varSigma^\infty_+\varOmega M\wedge R)$-module structure on $\xi_{m_0}$ induced by the map $\varOmega \xi \colon (\varOmega M)_+ \to \GL_1(R)$. 

    \begin{notn}
        Let $R_\xi$ denote $R$ considered as an $(\varSigma^\infty_+\varOmega M\wedge R)$-module via the map $\varOmega \xi \colon (\varOmega M)_+ \to \GL_1(R)$.
    \end{notn}
    
    Let $\xi \colon M_+ \to \BGL_1(R)$ be an $R$-line bundle and recall the definition of the point flow category $\sM_\ast$ from \cref{dfn:point_flow_cat}. Let $\fo_\xi(p) \coloneqq \xi_{m_0}$, and denote by $\sP_{m_0} \colon \sM_\ast \to \mod{(\varSigma^\infty_+\varOmega M\wedge R)}$ the local system defined by $\sP_{m_0}(p) \coloneqq \varSigma^\infty_+\sP_{m_0}M \wedge R$.

    Let
    \[
    \sN_\xi \colon (\sM(f,g),\fo_\xi,\sP M) \longrightarrow (\sM_\ast,\fo_\xi,\sP_{m_0})
    \]
    denote the $R$-oriented flow bimodule with local system defined by the assignment
    \[ (x,p) \longmapsto \overline W^u(x).\]
    The $R$-orientation is defined by an identification
    \[
    (T_x\overline W^u(x))_R \otimes_R \xi_x \overset{\simeq}{\longrightarrow} (T\overline W^u(x))_R \otimes_R \xi_{m_0},
    \]
    for every $x\in \Ob(\sM(f,g))$, obtained using isomorphisms between $\xi_x$ and $\xi_{m_0}$ given by elements of $\sP_{m_0x}M$. The local system on $\sN_\xi$ is a natural transformation 
    \[\sP\sN_\xi \colon \sN_\xi \Longrightarrow F_{\varSigma^\infty_+\varOmega M\wedge R}(\sP M,\sP_{m_0}),\]
    that is defined as the following composition
    \[ (\varSigma^\infty_+ \sP_{m_0x} M \wedge R) \wedge \varSigma^\infty_+\sN_\xi(x,p) \xrightarrow{\id \wedge \ev} (\varSigma^\infty_+ \sP_{m_0x} M \wedge R) \wedge \varSigma^\infty_+ \sP_x M \overset{P}{\longrightarrow} \varSigma^\infty_+ \sP_{m_0} M \wedge R,\]
    where $\ev \colon \sN_\xi(x,p) = \overline W^u(x) \to \sP_{x} M$ is the evaluation defined using the Riemannian metric $g$ on $M$.
    
\begin{lem}\label{lem:cjs_twisted_pt}
    The CJS realization of $(\sM_\ast,\fo_\xi,\sP_{m_0})$ is given by the $(\varSigma^\infty_+ \varOmega M \wedge R)$-module $R_\xi$.
\end{lem}
\begin{proof}     
    The $(\varSigma^\infty_+ \varOmega M \wedge R)$-linear $\sJ$-module $Z_{\fo_{\xi},\sP_{m_0}} \colon \sJ \to \mod{(\varSigma^\infty_+ \varOmega M \wedge R)}$ associated to $(\sM_\ast,\fo_\xi,\sP_{m_0})$ is given on objects by
    \[
    Z_{\fo_{\xi},\sP_{m_0}}(n) = \begin{cases}
        (\varSigma^\infty_+ \sP_{m_0}M \wedge R) \wedge_R \xi_{m_0} \simeq (\sP_{m_0}M)^{\xi}, & n = 0 \\
        0, & n \neq 0
    \end{cases}.
    \]
    This gives that we have $|\sM_\ast,\fo_\xi,\sP_{m_0}| \simeq R$ as an $R$-module. Next, the $(\varSigma^\infty_+ \varOmega M \wedge R)$-module structure is given by the composition
    \[
    (\varSigma^\infty_+ \varOmega M \wedge R) \wedge_R (\varSigma^\infty_+ \sP_{m_0}M \wedge R) \wedge_R \xi_{m_0} \xrightarrow{P \wedge_R \id} (\varSigma^\infty_+ \sP_{m_0}M \wedge R) \wedge_R \xi_{m_0},
    \]
    which finishes the proof.
\end{proof}
\begin{prop}\label{prop:morse_flow_path}
    There is a weak equivalence of $(\varSigma^\infty_+ \varOmega M \wedge R)$-modules
    \[ |\sM(f,g),\fo_\xi, \sP M| \simeq R_\xi.\]
\end{prop}
\begin{proof}

    The CJS realization of the $R$-oriented flow bimodule with a local system $(\sN_\xi,\fo_\xi,\sP \sN_\xi)$ and \cref{lem:cjs_twisted_pt} yields a map of $(\varSigma^\infty_+ \varOmega M \wedge R)$-modules
    \begin{equation}\label{eq:CJS_loc}
        |\sN_\xi,\fo_\xi,\sP \sN_\xi| \colon |\sM(f,g),\fo_\xi,\sP M| \longrightarrow R_\xi,
    \end{equation}
    and the proof of \cref{lem:cjs_twisted_pt} shows $R_\xi \simeq (\sP_{m_0}M)^\xi$ as $\varOmega M$-modules. We now verify that this map is a weak equivalence.
    
    Note that both $|\sM(f,g),\fo_\xi,\sP M|$ and $(\sP_{m_0} M)^\xi$ are filtered objects: for any $p \in \IR$, let
    \begin{align*}
        M^{\leq p} &\coloneqq f^{-1}((-\infty, p]) \\
        \sM(f,g)^{\leq p} &\coloneqq \sM(f|_{M^{\leq p}}, g|_{M^{\leq p}}) \\
        \fo_\xi^{\leq p} &\coloneqq \fo_\xi|_{\sM(f,g)^{\leq p}} \\
        \sP M^{\leq p} &\coloneqq \sP M|_{\sM(f,g)^{\leq p}}.
    \end{align*}
    Assume without loss of generality that the critical points of $f$ have distinct critical values and that $m_0$ belongs to a neighborhood of the global minimum of $f$. Letting
    \[ F^p |\sM(f,g),\fo_\xi,\sP M| \coloneqq |\sM(f,g)^{\leq p}, \fo_\xi^{\leq p}, \sP M^{\leq p}|\]
    defines an increasing filtration on $|\sM(f,g),\fo_\xi,\sP M|$. On the other hand, pulling back the filtration $M^{\leq p}$ on $M$ along the projection $\pi \colon \sP_{m_0}M \to M$, we obtain the following filtration on $(\sP_{m_0} M)^\xi$:
    \[ F^p (\sP_{m_0} M)^\xi \coloneqq (\pi^{-1}(M^{\leq p}))^{\xi|_{\pi^{-1}(M^{\leq p})}}.\]
    It is clear from the construction that the map \eqref{eq:CJS_loc} respects these filtrations.
    
    When $p$ is not a critical value, the inclusions
    \[ F^{p-\varepsilon} |\sM(f,g),\fo_\xi,\sP M| \hooklongrightarrow F^p |\sM(f,g),\fo_\xi,\sP M| \quad \text{and} \quad F^{p-\varepsilon} (\sP_{m_0} M)^\xi \hooklongrightarrow F^p(\sP_{m_0} M)^\xi\]
    are weak equivalences for $\varepsilon > 0$ small enough. To show that \eqref{eq:CJS_loc} is a weak equivalence, it suffices to show that for a critical value $p$, the induced maps on associated graded
    \begin{equation}\label{eq:CJS_loc_filtered}
        F^p|\sM(f,g),\fo_\xi,\sP M| / F^{p-\varepsilon}|\sM(f,g),\fo_\xi,\sP M| \simeq F^p (\sP_{m_0} M)^\xi / F^{p-\varepsilon} (\sP_{m_0}M)^\xi
    \end{equation}
    is an equivalence for $\varepsilon > 0$ small enough. First, by \cref{lem:realization_cofiber} and \cref{lem:realiation_point} we have
    \begin{align*}
        &F^p |\sM(f,g),\fo_\xi,\sP M|/F^{p-\varepsilon} |\sM(f,g),\fo_\xi,\sP M| \\
        &\qquad \simeq |(\sM(f,g)^{\leq p},\fo_\xi^{\leq p},\sP M^{\leq p})/(\sM(f,g)^{\leq p-\varepsilon},\fo_\xi^{\leq p-\varepsilon},\sP M^{\leq p-\varepsilon})| \\
        &\qquad \simeq (\varSigma^\infty_+ \sP_{m_0x} M \wedge R) \wedge_R \fo_{\xi}(x),
    \end{align*}
    since the flow category $\sM(f,g)^{\leq p}/\sM(f,g)^{\leq p-\varepsilon}$ is the point flow category concentrated in degree $\mu(x)$. Using the fact that $\pi \colon \sP_{m_0} M \to M$ is a fibration, a standard Morse-theoretic argument gives 
    \[ F^{p}(\sP_{m_0} M) /F^{p-\varepsilon}(\sP_{m_0} M) \simeq (\varSigma^\infty_+ \sP_{m_0x} M \wedge R) \wedge_R (T_x \overline W^u (x))_R,\]
    and therefore
    \[ F^{p}(\sP_{m_0} M)^\xi /F^{p-\varepsilon}(\sP_{m_0} M)^\xi \simeq (\varSigma^\infty_+ \sP_{m_0x} M \wedge R) \wedge_R \xi_x.\]
    Moreover, under these identifications, the construction of the map \eqref{eq:CJS_loc} implies that the map \eqref{eq:CJS_loc_filtered} is given by the identity, thus completing the proof.
\end{proof}
We now discuss a cohomological version of \cref{prop:morse_flow_path}. Namely, recall the definition of the cohomological Morse--Smale flow category $(\sM^{-\bullet}(f,g),\fo_{(f,g)}^{-\bullet})$ in \cref{defn:coh_morse-smale_flow_cat}. We twist its orientation by the $R$-line bundle $\xi$, and define
\[
    \fo_{\xi}^{-\bullet}(x) \coloneqq (-T_x \overline W^u_{-\nabla f}(x))_R \otimes_R \xi_x.
\]
We furthermore equip this $R$-oriented flow category with the following local system:
\begin{align}\label{eq:local_syst_cohomological_path}
    \sP^{-\bullet}M \colon \sM^{-\bullet}(f,g) &\longrightarrow \mod{(\varSigma^\infty_+ \varOmega M \wedge R)} \\
    x &\longmapsto \varSigma^\infty_+ \sP_{xm_0}M \wedge R. \nonumber
\end{align}
We consider the point flow category $\sM_\ast$ with twisted $R$-orientation $\fo_\xi$ as in \cref{lem:cjs_twisted_pt}, with the reversed local system defined by $\sP^{-\bullet}_{m_0}(p) = \varSigma^\infty_+ \sP_{-,m_0}M\wedge R$. Consider the the $R$-oriented flow bimodule $\sN_f \colon \sM_\ast \to \sM^{-\bullet}(f,g)$ defined in \cref{dfn:unit_morse}, and equip it with the local system
\[
    \sP \sN_f \colon \sN_f \Longrightarrow F_{\varSigma^\infty_+ \varOmega M \wedge R}(\sP_{m_0},\sP^{-\bullet}M),
\]
defined as follows. By abuse of notation, denote both the projection maps $\sP_xM \to M$ and $\sP_{-,m_0}M \to M$ by $\pi$, and consider the induced map
\[
\varPi \colon \sP_{-,m_0} M \times \sP_{x}M \longrightarrow \sP_{-,m_0} M \times_{\pi \times \pi} \sP_xM.
\]
The local system $\sP \sN_f$ is now defined via the composition
\begin{align*}
    (\varSigma^\infty_+\sP_{-,m_0}M \wedge R) \wedge \varSigma^\infty_+ \sN_f(p,x) &\xrightarrow{\id \wedge \ev} (\varSigma^\infty_+ \sP_{-,m_0}M \wedge R) \wedge \varSigma^\infty_+ \sP_{x}M \\
    &\overset{\varPi}{\longrightarrow} \varSigma^\infty_+(\sP_{-,m_0} M \times_{\pi \times \pi} \sP_{x}M) \wedge R\\
    &\overset{P}{\longrightarrow} \varSigma^\infty_+ \sP_{xm_0}M \wedge R,
\end{align*}
where $\ev \colon \sN_\xi^{-\bullet}(x,p) = \overline W^u_{-\nabla f}(x) \to \sP_{x}M$ is the evaluation map. The proof of the following result is completely analogous to \cref{prop:morse_flow_path}.
\begin{prop}\label{prop:morse_flow_path_coh}
    There are weak equivalences of $(\varSigma^\infty_+ \varOmega M \wedge R)$-modules
    \[
        |\sM^{-\bullet}(f,g),\fo_\xi^{-\bullet},\sP^{-\bullet}M| \simeq (\sP_{-,m_0}M)^{\xi} \simeq R_\xi.
    \]
    \qed
\end{prop}
\begin{proof}
    This is completely analogous to the proof of \cref{prop:morse_flow_path}, using the $R$-oriented flow bimodule with local system $(\sN_f,\fo_f,\sP\sN_f)$, instead of $(\sN_\xi,\fo_\xi,\sP\sN_\xi)$.
\end{proof}

\section{Abstract branes and spaces of abstract cappings}\label{sec:AbstractBranes}

Let $R$ be a commutative ring spectrum. In this section, we define the space of abstract $R$-branes and study spaces of abstract caps that is needed for constructing canonical orientations on moduli spaces of strips with Lagrangian boundary conditions discussed in \cref{sec:canonical_orientations}.

\subsection{$R$-orientations}
The unit map $\eta \colon \IS \to R$ determines a map ${\B}^n{\eta} \colon {\B}^n{\GL_1}(\IS) \to {\B}^n{\GL_1}(R)$ for any $n \geq 1$.  Recall from \cref{sec:background_spectra} that given a (stable) vector bundle over a space $M$, there is an $R$-line bundle associated to it given by the composition
\[ M \longrightarrow \BO \overset{{\B}J}{\longrightarrow} \BGL_1(\IS) \overset{{\B}\eta}{\longrightarrow} \BGL_1(R), \]
where $J \colon \OO \to \GL_1(\IS)$ is the $J$-homomorphism.

We recall the following two facts from \cref{dfn:ori_vb} and \cref{rem:ori_fb_torsor}. A vector bundle is \emph{$R$-orientable} if the above associated $R$-line bundle is null-homotopic and an \emph{$R$-orientation} is a choice of such a null-homotopy.  Consequently, given an $R$-orientation for an $R$-orientable vector bundle over a $M$, the set of $R$-orientations is a torsor over $[M,\GL_1(R)]$.

\subsection{Abstract $R$-branes and spaces of disks}\label{sec:abst_r-branes_spaces_of_disks}

\begin{defn}[Abstract Lagrangian $R$-brane]\label{def:absrbrane}
The space of \emph{abstract Lagrangian $R$-branes} is
\[ (\lag)^\# \coloneqq \hofib (\lag \longrightarrow {\B}^2\OO \overset{{\B}^2J}{\longrightarrow} {\B}^2{\GL_1}(\IS) \overset{{\B}^2\eta}{\longrightarrow} {\B}^2{\GL_1}(R)), \]
where $\lag \to {\B}^2{\OO}$ is the composition of the delooping of the Bott map $\varOmega(\lag) \simeq \IZ \times \BO$ and the projection to $\BO$. 
\end{defn}

The complexification map $c \colon \BO \to \BU$ gives the bundle of stable Lagrangian distributions over $\BU$ and its homotopy fiber is $\lag$.  Define the space
\[ \sD = \left\{ (\alpha,u) \in \varOmega \BO \times \maps(\ID ^2,\BO) \mid c \circ \alpha = c \circ u|_{\partial \ID^2} \right\}. \]
There are two evaluation maps
\[ S^1 \times \sD \longrightarrow \BO \]
given by $(t,\alpha,u) \mapsto \alpha(t)$ and $(t,\alpha,u) \mapsto u|_{\partial \ID^2}(t)$, respectively.  Since $c \circ \alpha(t) = c \circ u|_{\partial \ID^2}(t)$, this pair of evaluation maps determines a unique map (up to homotopy) $S^1 \times \sD \to \lag$.  Since all of the maps above are based, by adjunction, this determines a map $\rho \colon \sD \to \varOmega(\lag)$.  This gives the vector bundle
\begin{equation}\label{eq:dind}
    \sD \overset{\rho}{\longrightarrow} \varOmega(\lag) \overset{\ind}{\longrightarrow} \BO.
\end{equation}

\begin{rem}
Every element $(\alpha,u) \in \sD$ determines a pair $(\alpha,c \circ u)$, which is a complex vector bundle over the disk along with a choice of Lagrangian distribution on its boundary. Let $D_{\alpha,u}$ denote the corresponding Cauchy--Riemann operator acting on the space of sections. The vector bundle classified by \eqref{eq:dind} is the vector bundle over $\sD$ whose fiber over $(\alpha, u)$ is given by
\[ \ind(D_{\alpha,u}) = \mathrm{ker}(D_{\alpha,u}) \otimes  (-\mathrm{coker}(D_{\alpha,u})).\]
\end{rem}

\begin{defn}
Let $\sD^\#$ denote the homotopy pullback of the looping of $(\lag)^\# \to \lag$ and $\rho \colon \sD \to \varOmega(\lag)$. The two induced maps are denoted by $\rho^\# \colon \sD^\# \to \varOmega (\lag)^\#$ and $\pi \colon \sD^\# \to \sD$.
\end{defn}

\begin{lem}\label{lem:DHashNullity}
The composition
\[ \sD^\# \longrightarrow \sD \longrightarrow \BO \longrightarrow \BGL_1(R) \]
is canonically null-homotopic.
\end{lem}

\begin{proof}
We have a commutative diagram
\[
	\begin{tikzcd}[row sep=scriptsize, column sep=scriptsize]
		\sD^\# \rar \dar & \sD \dar & {} & {} \\
		\varOmega(\lag)^\# \rar & \varOmega(\lag) \rar & \BO \rar & \BGL_1(R)
	\end{tikzcd}
\]
Since the lower horizontal chain of compositions is canonically null-homotopic and since the square is a homotopy pullback square, we obtain a canonical null-homotopy of the desired composition.
\end{proof}

\subsection{Strip caps and ends}

We begin by fixing some notation.

\begin{notn}
Let $D_\pm \coloneqq \ID^2 \smallsetminus \{\mp1\}$ and fix the standard complex structure on the plane.  Define
\[ \Ends(D_+) \coloneqq \left\{ \varepsilon \colon (-\infty,0] \times [0,1] \hookrightarrow D_+ \left| \begin{matrix} \varepsilon \mbox{ is a biholomorphism onto its image} \\ \lim_{s \to -\infty} \varepsilon(s,t) = -1 \end{matrix} \right.\right\}. \]
Equip $\Ends(D_+)$ with the subspace topology in the space of smooth maps.  Similarly define $\Ends(D_-)$.  Notice that $\Ends(D_\pm)$ is contractible.
\end{notn}

 Fixing an orientation preserving diffeomorphism $(-\infty,0] \to (-1,0]$ endows $\{-\infty\} \cup (-\infty,0]$ with a smooth topology that it is diffeomorphic to $[-1,0]$.  With this, every $\varepsilon \in \Ends(D_+)$ determines a compactification of $D_+$, denoted by $\overline D_+$, that is diffeomorphic to a disk.

 \begin{notn}
    Given $\varepsilon \in \Ends(D_+)$, we pick a decomposition of $\partial \overline D_+ = \partial_1 \overline D_+ \cup \partial_2 \overline D_+$ where $\partial_1 \overline D_+$ is a (closed) neighborhood of $\{1\} \subset \partial \overline D_+$ and $\partial_2 \overline D_+$ is the closure of the complement of $\partial_1 \overline D_+$, in $\partial \overline D_+$. 
 \end{notn}
 
\begin{defn}
Given any $\varepsilon \in \Ends(D_\pm)$, we call the space $\overline D_\pm$ a \emph{positive/negative strip cap}.
\end{defn}

\subsection{Strip caps with Lagrangian boundary conditions}\label{subsec:StripCapsSpace}

We give a reworking of the definition of $\sD$ and $\sD^\#$ in terms of strip caps as opposed to disks.

\begin{defn}\label{dfn:abstract_strip_caps}
    Suppose $\gamma \colon [0,1] \to \BO$ is a continuous map. Define the space
    \[ \sD_+(\gamma) = \left\{ (\varepsilon, \alpha,u) \in \Ends(D_+) \times \sP\BO \times \maps\left(\overline D_+,\BO\right) \left| \begin{matrix} c \circ \alpha = c \circ u|_{\partial_1 \overline D_+} \\ \gamma = u|_{\partial_2 \overline D_+} \end{matrix} \right. \right\}, \]
    where we recall that $\varepsilon$ determines the compactification of $D_+$. We define $\sD_-(\gamma)$ analogously.
\end{defn}
\begin{rem}
    As with $\sD$ in \cref{sec:abst_r-branes_spaces_of_disks}, we obtain a unique map (up to homotopy)
    \[ \rho \colon \sD_+(\gamma) \longrightarrow \sP_{x_0,x_1}(\lag),\]
    via the two evaluation maps $(t,\varepsilon,\alpha,u) \mapsto \alpha(t)$ and $(t,\varepsilon,\alpha,u) \mapsto u|_{\partial_1 \overline D_+}(t)$.
\end{rem}
\begin{notn}
    For any $x \in \lag$, denote a choice of lift to $(\lag)^\#$ by $x^\#$.
\end{notn}
\begin{defn}\label{dfn:pullback_path}
    Let $\sD_+^\#(\gamma)$ denote the space such that the following diagram is a homotopy pullback: 
    \[
    \begin{tikzcd}[row sep=scriptsize,column sep=scriptsize]
        \sD_+^\#(\gamma) \rar{\pi} \dar[swap]{\rho^\#} & \sD_+(\gamma) \dar{\rho} \\
        \sP_{x_0^\#,x_1^\#}(\lag)^\# \rar & \sP_{x_0,x_1}(\lag)
    \end{tikzcd}.
    \]
    We define $\sD^\#_-(\gamma)$ analogously.
\end{defn}

\begin{lem}
Suppose $\gamma \colon [0,1] \to \BO$ is a continuous map. Consider the space $\sD$ based at $\gamma(0) \in \BO$ as defined in \cref{sec:abst_r-branes_spaces_of_disks}.  There are continuous gluing maps that make the following diagram commute:
\[
    \begin{tikzcd}[row sep=scriptsize, column sep=scriptsize]
        \sD_+^\#(\gamma) \times \sD_-^\#(\gamma) \rar \dar & \sD^\# \dar \\ 
        \sD_+(\gamma) \times \sD_-(\gamma) \rar  & \sD.  \\ 
    \end{tikzcd}
\]
\qed
\end{lem}
\section{Brane structures}\label{sec:Polarizations}

In this section, we recall the definition of a stably polarized symplectic manifold and the stable Lagrangian Gauss map of a Lagrangian inside such a symplectic manifold.

\subsection{Polarizations}

We recall the definition of a polarization.

\begin{defn}
A (stable) complex vector bundle represented by a map $\eta \colon M \to \BU$ is called \emph{stably polarized} if it is equipped with a lift of $\eta$ to a map $\sE \colon M \to \BO$ such that $c \circ \sE = \eta$, where $c \colon \BO \to \BU$ is the complexification map.  In other words, $\eta$ is stably the complexification of a real vector bundle.
\end{defn}

\begin{defn}
An almost complex manifold $M$ is \emph{stably polarized} if its tangent bundle is stably polarized.
\end{defn}

\begin{rem}
    \begin{enumerate}
        \item A stable polarization is equivalent to a stable Lagrangian distribution in $TX$, globally over $X$.
        \item An almost complex manifold being stably polarized is a somewhat restrictive condition.  For example, this implies that all odd Chern classes are $2$-torsion.
    \end{enumerate}
\end{rem}

\begin{exmp}\label{exmp:cotpol}
The cotangent bundle $T^*Q$ of any smooth manifold $Q$ is stably polarized.  In fact, it is polarized: the cotangent fibers give a global Lagrangian distribution. Unless specified otherwise, we always assume this choice of polarization for cotangent bundles.
\end{exmp}

\subsection{The stable Lagrangian Gauss map}\label{subsec:stabg}
In this section, we fix a symplectic manifold $X$ with a stable polarization $\sE$ and a Lagrangian submanifold $L \subset X$. Comparing the Lagrangian distribution determined by $L$ to the global Lagrangian distribution $\sE$ in $TX$ gives rise to a Lagrangian Gauss map of $L$.

First, view $\BO$ as the bundle of Lagrangian distributions over $\BU$ via the fiber sequence
\[ \lag \longrightarrow \BO \longrightarrow \BU,\]
where the final map $c \colon \BO \to \BU$ is given by complexification.  The Lagrangian distribution of $L$ determines a section $L \to \BO$ that covers the map $TX|_L \colon L \to \BU$.  Similarly, the stable polarization of $X$ restricted to $L$ gives a section $\sE|_L \colon L \to \BO$ that covers the map $TX|_L \colon L \to \BU$.  This determines a unique (up to homotopy) lift $\sG_L \colon L \to \lag$ so that the composition $L \to \lag \to \BO$ represents the virtual (stable) vector bundle $TL- \sE|_L$.

\begin{defn}\label{defn:LagrangianGauss}
The \emph{stable Lagrangian Gauss map} of $L$ with respect to the polarization $\sE$ is the map $\sG_L \colon L \to \lag$.
\end{defn}

\begin{rem}
There is a more explicit way of describing the lift $\sG_L$.  First, stabilize so that $TX \oplus \un{\IC}^K = \sE \otimes \IC$ for the real vector bundle $\sE \to X$.  Fix a fiberwise linear embedding $i \colon \sE \to \IR^N$ and write $i^*\IR^N = \sE \oplus \nu$.  We have that 
\[ (\un{\IC}^N \to X) = i^*\IR^N \otimes \IC = (TX \oplus \un{\IC}^K) \oplus (\nu \otimes \IC). \]
For each $x \in L$, the subspace $T_xL \oplus \IR^K \oplus \nu \subset \un{\IC}^N$ is Lagrangian.  So the map $x \mapsto T_xL \oplus \IR^K \oplus \nu$ determines the map $\sG_L \colon L \to \lag$.
\end{rem}

\subsection{$R$-branes}

Let $X$ be a stably polarized symplectic manifold and let $L$ be a Lagrangian submanifold. Let $R$ be a commutative ring spectrum.

\begin{defn}\label{defn:R-ori_data}
A choice of \emph{$R$-orientation data} on $L$ is a choice of a lift of the stable Lagrangian Gauss map $\sG_L \colon L \to \lag$ to a map $\sG_L^\# \colon L \to (\lag)^\#$, where $(\lag)^\#$ is as in \cref{def:absrbrane}. 
\end{defn}

\begin{rem}\label{rem:brane}
\begin{enumerate}
	\item Given $R$-orientation data on $L$, the choices of $R$-orientation data on $L$ is a torsor over $[L,\BGL_1(R)]$. That is, the choices of such correspond to choices of $R$-line bundles on $L$.
	\item The composition
	\[ \varOmega L \longrightarrow \varOmega (\lag)^\# \longrightarrow \varOmega(\lag) \longrightarrow \BO \longrightarrow \BGL_1(R) \]
	is canonically null-homotopic.  Hence, the choice of $R$-orientation data on $L$ canonically determines an $R$-orientation of the $R$-line bundle on $\varOmega L$ associated to the vector bundle obtained by pulling back the index bundle from $\varOmega(\lag)$.
	\item Any contractible Lagrangian inside a stably polarized Liouville sector canonically admits $R$-orientation data.  In particular, the cotangent fiber in a cotangent bundle canonically admits $R$-orientation data.
\end{enumerate}
\end{rem}

\begin{rem}\label{rem:equiv_Rbrane_loopsinfRbrane}
Consider the connective cover $R_{\geq 0} \to R$.  A choice of $R$-orientation data is equivalent to a choice of $R_{\geq 0}$-orientation data. This follows from the fact map $\GL_1(R_{\geq 0}) \to \GL_1(R)$ is a weak equivalence, since $R_{\geq 0} \simeq \varOmega^\infty R$ (see \cref{defn:space_of_units} and \cref{rem:conn_cov_cnstr}).
\end{rem}

\begin{defn}\label{dfn:gradingData}
A \emph{grading} associated to a Lagrangian $L \subset X$ is a lift of $\sG \colon L \to \lag$ to the universal cover $\widetilde{\lag}$.
\end{defn}
\begin{rem}
    The existence of a lift of $\sG$ to the universal cover of $\lag$ guarantees that we obtain a well-defined $\IZ$-grading. The obstruction to the existence of such a lift is given by the Maslov class.
\end{rem}

\begin{defn}\label{defn:Brane}
A \emph{Lagrangian $R$-brane} in a stably polarized Liouville sector $X$ is an exact conical Lagrangian that stays away from the symplectic boundary of $X$ (see \cref{sec:geom_background}) together with a choice of grading and $R$-orientation data.
\end{defn}

\subsection{Integer coefficients}\label{sec:integer_coeffs}
    We now show that the data of a Lagrangian $R$-brane in the special case of $R = H\IZ$ corresponds to a choice of relative pin structure, where $H\IZ$ denotes the Eilenberg--MacLane spectrum of the integers. Let $X$ be a stably polarized Liouville sector and let $L \subset X$ denote an exact conical Lagrangian that stays away from the symplectic boundary of $X$.

    Fix a background class $b \in H^2(X;\IZ/2)$.
    \begin{asmpt}\label{asmpt:sw_lag}
        The second Stiefel--Whitney class of $L$ agrees with the pullback of $b$ under the inclusion $i \colon L \hookrightarrow X$.
    \end{asmpt}
    Fix a triangulation on $X$ such that $L$ is a subcomplex. Denote the $3$-skeletons of $X$ and $L$ determined by this triangulation by $X[3]$ and $L[3]$, respectively. Let $E_b$ be an orientable vector bundle on $X[3]$, such that $w_2(E_b) = b|_{X[3]}$; there is a unique such vector bundle up stable isomorphism. We fix a null-homotopy of $w_2(E_b|_{L[3]}) - b|_{L[3]}$ viewed as a map $L[3] \to K(\IZ/2,2)$. The following definition is standard, see e.g.\@ \cite{solomon2006intersection,abouzaid2012wrapped} and cf.\@ \cite[Chapter 8]{fukaya2009lagrangian}.
    \begin{defn}[Relative pin structure]\label{dfn:rel_pin}
        A \emph{relative pin structure} on $L \subset X$ is a choice of pin structure on the vector bundle $E_b|_{L[3]} \oplus TL|_{L[3]}$.
    \end{defn}
    \begin{rem}
        The obstruction to the existence of a pin structure is the second Stiefel--Whitney class. It vanishes for the vector bundle $E_b|_{L[3]} \oplus TL|_{L[3]}$ because of the orientability of $E_b$ and by \cref{asmpt:sw_lag}.
    \end{rem}
    \begin{defn}[$\IZ$-brane structure]
        Let $L \subset X$ be a Lagrangian satisfying \cref{asmpt:sw_lag}. A \emph{$\IZ$-brane structure} on $L$ is a choice of grading (see \cref{dfn:gradingData}) and a choice of relative pin structure.
    \end{defn}
    \begin{lem}\label{lem:hz_ori_data}
        Let $\sE \colon X \to \BO$ be a stable polarization such that $w_2(\sE) = b$. A choice of $H\IZ$-brane structure on $L$ in the sense of \cref{defn:Brane} is equivalent to a choice of $\IZ$-brane structure on $L$.
    \end{lem}
    \begin{proof}
        Consider the following diagram
       \begin{equation}\label{eq:dia_w2}
           \begin{tikzcd}[row sep=scriptsize, column sep=scriptsize]
               (\lag)^\# \rar \dar & \lag \dar \rar & K(\IZ/2,2) \\
               \BPin \rar & \BO \urar[bend right,swap]{w_2} & {} 
           \end{tikzcd},
       \end{equation}
       where both rows are homotopy fiber sequences, the square is a homotopy pullback, and the triangle commutes by \cite[Lemma 11.7]{seidel2008fukaya}.

       The choice of a grading of $L$ determines a choice of null-homotopy of $w_1(TL - \sE|_L)$. Since $w_1(E_b)=0$ it follows that $w_2(\sE|_{X[3]}-E_b)$ is also null-homotopic. Fix such a null-homotopy and note that
       \[ \sE|_{L[3]}-E_b|_{L[3]} \cong (TL|_{L[3]} - E_b|_{L[3]}) - (TL|_{L[3]}-\sE|_{L[3]}).\]
       Using the Whitney sum formula, the choices of null-homotopies of $w_1(TL - \sE|_L)$ and $w_2(\sE|_{X[3]}-E_b)$ above yields
       \[
            w_2(TL|_{L[3]} - E_b|_{L[3]}) = w_2(TL|_{L[3]}-\sE|_{L[3]}).
       \]
       The orientation on $E_b$ determines a null-homotopy of $w_2(E_b^{\oplus 2})$ and hence we get a one-to-one correspondence between null-homotopies of $w_2(TL|_{L[3]} - E_b|_{L[3]})$ and $w_2(TL|_{L[3]} \oplus E_b|_{L[3]})$. To conclude we get a one-to-one correspondence between null-homotopies of $w_2(TL|_{L[3]} \oplus E_b|_{L[3]})$ and $w_2(TL|_{L[3]}-\sE|_{L[3]})$. Since a pin structure is equivalent to the choice of a null-homotopy of the second Stiefel--Whitney class, this completes the proof of the lemma.
    \end{proof}

\section{Moduli spaces}\label{sec:ModuliSpaces}

In this section, we define moduli spaces of $J$-holomorphic disks with Lagrangian boundary conditions.  These moduli spaces and their associated index bundles will give rise to the requisite flow category-type data that is needed to define the wrapped Donaldson--Fukaya category with coefficients in a ring spectrum $R$.  

\begin{asmpt}
\begin{enumerate}
\item Let $R$ be a commutative ring spectrum.
\item Let $X$ be a stably polarized Liouville sector.
\item Let $L \subset X$ be a Lagrangian $R$-brane.
\end{enumerate}
\end{asmpt}

Since $X$ is stably polarized, any Lagrangian $L \subset X$ admits a stable Lagrangian Gauss map $\sG_L \colon L \to \lag$, see \cref{defn:LagrangianGauss}. 
\begin{notn}
    Throughout this section, $\mu$ denotes Maslov index, see \cite{robbin1993maslov}.
\end{notn}

Often we will abuse the notation and refer to the Lagrangian $R$-brane simply by the underlying Lagrangian submanifold itself. In such instances, the grading and $R$-orientation data are assumed.

\subsection{Admissible Floer data}
Let $X$ be a Liouville sector that corresponds to the Liouville pair $(\overline X,F)$ (see \cref{sec:geom_background}). If $H \in C^\infty(\overline X)$ is a Hamiltonian, we denote by $X_H$ the Hamiltonian vector field associated to, which is defined by $\omega(X_H,-) = -dH$. We consider the following space of admissible Floer data:

\begin{defn}[{\cite[Definition 2.13]{sylvan2019on}}]\label{dfn:compatible_H}
A Hamiltonian $H \in C^\infty(\overline X)$ is \emph{compatible with the Liouville pair $(\overline X,F)$} if the following holds.
\begin{enumerate}
\item $H$ is strictly positive and $H(r,-) = r^2$ in the conical end $[0,\infty)_r \times \partial_\infty\overline X$ of $\overline X$.
\item $X_H$ is tangent to $\im \sigma|_{F \times \{0\}}$ where $\sigma$ is defined in \eqref{eq:stop_proper_emb}.
\item $d \theta(X_H)$ is nowhere negative on a neighborhood of $F$, where $\theta$ is an angular coordinate on $\IC_{\mathrm{Re} \leq \varepsilon}$
\end{enumerate}
Let $\mathcal H(\overline X,F)$ be the space of Hamiltonians on $\overline X$ that are compatible with $(\overline X,F)$.
\end{defn}

\begin{defn}\label{dfn:compatible_J}
An $\omega$-compatible almost complex structure $J$ on $\overline X$ is \emph{compatible with the Liouville pair $(\overline X,F)$} if the following holds.
\begin{enumerate}
\item There is some $c > 0$ such that $d(r^2) \circ J = c r^2 \cdot \lambda$ in the conical end.
\item $J$ is invariant under the Liouville flow outside a compact set.
\item The projection $\pi_{\IC}\colon F \times \IC_{\mathrm{Re} < \varepsilon} \to \IC_{\mathrm{Re} < \varepsilon}$ is holomorphic in a neighborhood of $\left\{\mathrm{Re} = 0\right\}$.
\end{enumerate}
Denote by $\mathcal J(\overline X,F)$ the space of almost complex structures on $\overline X$ compatible with $(\overline X,F)$.
\end{defn}

\begin{rem}
The holomorphicity assumption in \cref{dfn:compatible_J}(iii) is a standard assumption used to prevent $J$-holomorphic curves from passing through the stop (the core of the horizontal boundary of the corresponding Liouville sector). The space $\mathcal J(\overline X,F)$ is non-empty and contractible, see \cite[Lemma 3.2]{sylvan2019on}.
\end{rem}

\begin{defn}\label{defn:AdmissibleFloerData}
Let $\sH\sJ(\overline{X},F) \coloneqq \sH(\overline{X},F) \times \sJ(\overline{X},F)$ denote the product space.  We say that a pair $(H,J) \in \sH\sJ(\overline{X},F)$ is \emph{admissible}.
\end{defn}
\begin{notn}
    \begin{enumerate}
        \item Let $\overline{\mathcal{R}}_k$ denote the compactified Deligne--Mumford moduli space of disks with $k+1$ boundary punctures $\xi_0,\ldots,\xi_k$.
        \item Let $\psi^\tau$ denote the time-$\tau$ Liouville flow on $\overline{X}$.
        \item Let $J_t$ denote a map $[0,1] \to \sJ(\overline X,F)$.
        \item Let $H \in \sH(\overline X,F)$.
    \end{enumerate}
\end{notn}
\begin{defn}[{\cite[Definition 4.1]{abouzaid2010a}}]\label{dfn:Floer_datum}
    A \emph{Floer datum} on a stable disk $S \in \overline{\mathcal{R}}_k$ with one negative end and $k$ positive ends consists of the following choices on each component:
    \begin{description}
        \item[Strip-like ends] Biholomorphisms $\varepsilon^-_0 \colon (-\infty,0] \times [0,1] \to N(\xi_0)$ and $\varepsilon^+_i \colon [0,\infty) \times [0,1] \to N(\xi_i)$ for $i\in \{1,\ldots,k\}$, where $N$ denotes an unspecified but fixed neighborhood. We let $\boldsymbol{\varepsilon}$ denote the tuple $(\varepsilon_0^-,\varepsilon_1^+,\ldots,\varepsilon_k^+)$.
        \item[Time-shifting map] A map $\rho_S \colon \partial \overline{S} \to [1,\infty)$ that is constantly equal to $w_{k,S}$ near the $k$-th end.
        \item[Basic $1$-form and Hamiltonian perturbations] A closed $1$-form $\alpha_S$ with $\alpha_S|_{\partial S} = 0$ and a map $H_S \colon S \to \mathcal{H}(\overline{X},F)$ defining a Hamiltonian vector field $X_S$ such that
        \[
            (\varepsilon_0^-)^\ast(X_S \otimes \alpha_S) = X_{\frac{1}{w_{0,S}}(\psi^{w_{0,S}})^\ast H} \otimes dt, \quad (\varepsilon_i^+)^\ast(X_S \otimes \alpha_S) = X_{\frac{1}{w_{i,S}}(\psi^{w_{i,S}})^\ast H} \otimes dt,
        \]
        for $i\in\{1,\ldots,k\}$.
        \item[Almost complex structures] A map $J_S \colon S \to \mathcal{J}(\overline{X},F)$ such that
        \[
        (\varepsilon_0^-)^\ast J_S = (\psi^{w_{0,S}})^\ast J_t, \quad (\varepsilon_i^+)^\ast J_S = (\psi^{w_{i,S}})^\ast J_t,
        \]
        for $i\in \{1,\ldots,k\}$.
    \end{description}
\end{defn}
\begin{rem}
    The condition on the basic $1$-forms and Hamiltonian perturbations implies that $\alpha_S$ pulls back via $\varepsilon^-_0$ and $\varepsilon^+_i$ to the $1$-forms $w_{0,S}dt$ and $w_{i,S}dt$ for $i\in \{1,\ldots,k\}$, respectively. In particular this means that $w_{0,S} = \sum_{i=1}^k w_{i,S}$.
\end{rem}
For every lower dimensional stratum in $\osr_k$, choose strip-like ends of each element that vary smoothly over the strata. If $\sigma \subset \partial \osr_k$ is a $p$-codimensional stratum of the boundary, any element is represented by a stable disk with $p$ nodes. For each node and for each real number $R$ we obtain an element of $\osr_k$ by removing the images of $(-\infty,-R] \times [0,1]$ and $[R,\infty) \times [0,1]$ in the two strip-like ends at the two sides of the node, and gluing the complements. In total this yields a chart on $\osr_k$
\begin{equation}\label{eq:gluing_corner}
    (0,\infty]^p \times \sigma \longrightarrow \osr_k,
\end{equation}
whose image is an open neighborhood of $\sigma$ such that each coordinate corresponds to a gluing parameter.
\begin{defn}\label{dfn:conformally_equiv_floer_data}
    Two choices of Floer data $(\boldsymbol{\varepsilon}^i,\rho_S^i,\alpha_S^i,H_S^i,J_S^i)$ for $i \in \{1,2\}$ on a stable disk $S \in \osr_k$ are \emph{conformally equivalent} if there is a constant $C$ so that $w^2_{S} = C w^1_{S}$, $\alpha^2_{S} = C \alpha^1_{S}$, $(\psi^C)^\ast J^2_{S} = J^1_{S}$ and $H^2_{S} = \frac{1}{C^2} (\psi^C)^\ast H^1_{S}$.
\end{defn}
\begin{defn}[{\cite[Definition 4.2]{abouzaid2010a}}]\label{dfn:univ_Floer_datum}
    A \emph{universal and conformally consistent choice of Floer data} is a choice of Floer datum for every integer $k \geq 2$ and every (representative) of $S\in \osr_k$ which varies smoothly over $\osr_k$ such that in the coordinates \eqref{eq:gluing_corner}, Floer data agree to infinite order at the boundary stratum with the Floer data obtained by gluing.
\end{defn}
\begin{lem}[{\cite[Lemma 3.8]{sylvan2019on}}]
    Universal and conformally consistent choices of Floer data exist. Any Floer datum on $S \in \osr_k$ can be extended to a universal and conformally consistent choice of Floer data that agrees with a given such choice on $\osr_m$ for all $m < k$.
    \qed
\end{lem}

\subsection{Moduli space for the differential}

\begin{notn}
\begin{enumerate}
\item Let $L_0$ and $L_1$ be two Lagrangian $R$-branes in $\overline X$.
\item Let $H \in \sH(\overline X,F)$ and $J_t \colon [0,1] \to \sJ(\overline X,F)$. By abuse of notation we will write $(H,J_t) \in \sH\sJ(\overline X,F)$ to denote such a choice.
\end{enumerate}
\end{notn}

\begin{defn}\label{defn:ham_chords}
Define
\[ \sX(L_0,L_1;H) \coloneqq \left\{ x \colon [0,1] \to \overline{X} \mid \dot{x}(t) = X_H(x(t)), \,x(0) \in L_0, \,x(1) \in L_1 \right\}.\]
The gradings of the Lagrangians give a well-defined grading function on $\mathcal X(L_0,L_1;H)$ via the Maslov index \cite{robbin1993maslov}, which we denote by $\mu$.
\end{defn}

\begin{defn}\label{dfn:floer_strips_boudary_asymp}
Given $a,b \in \mathcal X(L_0,L_1)$, define $\widetilde{\sR}(a,b;H,J_t)$ to be the set of maps 
\begin{equation}\label{eq:floer_strips_boudary_asymp}
\widetilde{\sR}(a;b;H,J_t) \coloneqq \left\{ u\colon \IR \times [0,1] \to \overline{X} \,\left|\, \begin{matrix}
\partial_s u + J_{t}\left(\partial_t u - X_{H}\right) = 0 \\
u(\IR \times \left\{0\right\}) \subset L_0 \\
u(\IR \times \left\{1\right\}) \subset L_1 \\
\lim_{s\to -\infty} u(s,t) = b(t) \\
\lim_{s\to \infty} u(s,t) = a(t) \end{matrix} \right. \right\}.
\end{equation}
\end{defn}

There is a natural $\IR$-action given by translation of the $s$-coordinate in the domain. Define $\sR(a;b;H,J_t) \coloneqq \widetilde{\sR}(a;b;H,J_t)/\IR$ whenever this action is free and declare it to be empty otherwise.

\begin{defn}
    Define $\osr(a;b;H,J_t)$ to be the Gromov compactification of $\sR(a;b;H,J_t)$ (see \cite[(9l)]{seidel2008fukaya}, \cite[Lemma 3.9]{sylvan2019on}).
\end{defn}

\begin{notn}
Often times, we will omit the data of $H$ and $J_t$ from the notation of $\osr(a;b;H,J_t)$, simply writing $\osr(a;b)$. We will continue this abuse of notation in the following subsections of this section.
\end{notn}

Let $I(a;b) \coloneqq \underline{\IR} \oplus T\osr(a;b) \to \osr(a;b)$ denote a copy of the trivial line bundle summed with the tangent bundle of $\osr(a;b)$.  As in \cref{dfn:assoc_rline}, let $I_R(a;b)$ denote the $R$-line bundle associated to $I(a;b)$.

\begin{lem}\label{lem:transversality_for_strips}
There exists a comeager set of $(H,J_t) \in \sH\sJ(\overline{X},F)$ such that the moduli space $\osr(a;b)$ is a smooth manifold with corners of dimension $\mu(b)-\mu(a)-1$.  Moreover, the following conditions are satisfied:
\begin{enumerate}
    \item For $a,b,c \in \sX(L_0,L_1)$, there is a map
    \[ \mu_{abc}\colon \osr(a;b) \times \osr(b;c) \longrightarrow \osr(a;c) \]
    that is a diffeomorphsim onto a face of $\osr(a;c)$ such that defining
    \[ \partial_i \osr(a;c) \coloneqq \bigsqcup_{\substack{b \in \sX(L_0,L_1) \\ \mu(b)-\mu(a) = i}} \osr(a;b) \times \osr(b;c),\]
    endows $\osr(a;c)$ with the structure of a $\ang{\mu(c)-\mu(a)-1}$-manifold (see \cref{def:k_mfd}).
    \item There exist rank one free $R$-modules $\check{\fo}(a)$ for all $a \in \sX(L_0,L_1)$ and isomorphisms:
    \[ \psi_{ab} \colon \check{\fo}(b) \overset{\simeq}{\longrightarrow} \check{\fo}(a)  \otimes_RI_R(a;b)\]
    that are compatible with the maps $\mu_{abc}$ in the sense that there are commutative diagrams similar to \eqref{eq:rorcoh}.
\end{enumerate}
\end{lem}

\begin{proof}
The existence of a smooth structure away from the corners is standard in the literature, see e.g.\@ \cite[(9k)]{seidel2008fukaya} and \cite{floer1995transversality}. The existence of a smooth structure respecting the corner structure and the maps $\mu_{abc}$ are due to Large \cite[Section 6]{large2021spectral} and Fukaya--Oh--Ohta--Ono \cite{fukaya2016exponential}.  This is the proof of item (i).  We defer the proof of item (ii) to \cref{sec:gluing_abs_caps}.
\end{proof}

\subsection{Moduli spaces for continuations}

\begin{notn}\label{notn:tau_continuation}
\begin{enumerate}
    \item Fix two Lagrangian $R$-branes $L_0$ and $L_1$ of $X$.
    \item Let $\tau\colon \IR \to [0,1]$ denote a smooth function such that $\tau(s) \equiv 1$ near $-\infty$ and $\tau(s) \equiv 0$ near $\infty$. 
\end{enumerate}
\end{notn}
\begin{asmpt}\label{asmpt:one-param_floer_data}
    Assume that $(H^{r},J^{r}_t)_{r \in [0,1]} \subset \sH\sJ(\overline X,F)$ is a $1$-parameter family such that $(H^0,J_t^0)$ and $(H^1,J_t^1)$ are admissible in the sense that \cref{lem:transversality_for_strips} holds.
\end{asmpt}
\begin{defn}\label{dfn:moduli_maps_cont}
Let $a \in \mathcal X(L_0,L_1;H^0)$ and $b \in \mathcal X(L_0,L_1;H^1)$.  Define $\osr^{\tau}(a;b)$ to be the Gromov compactification of the set of maps
\[ \sR^\tau(a;b) \coloneqq \left\{ u \colon \IR \times [0,1] \to \overline{X} \, \left| \, \begin{matrix} \partial_su + J^{\tau(s)}_t(\partial_t u - X_{H^{\tau(s)}_t}) = 0 \\ u(\IR \times \{0\}) \subset L_0 \\ u(\IR \times \{1\}) \subset L_1 \\ \lim_{s \to -\infty} u(s,t) = b(t) \\ \lim_{s \to \infty} u(s,t) = a(t) \end{matrix} \right. \right\}. \]
\end{defn}

Endow $\osr^\tau(a;b)$ with the Gromov topology.  With this topology it is a compact metric space (see \cite[Lemma 3.9]{sylvan2019on}).

Let $I^\tau(a;b) \to \osr^\tau(a;b)$ denote the tangent bundle (the index bundle) of $\osr^\tau(a;b)$, and let $I^\tau_R(a;b)$ denote its associated $R$-line bundle (see \cref{dfn:assoc_rline}).

\begin{lem}\label{lem:moduli_cont}
There exists a comeager set of $(H^r,J^r_t)_{r \in [0,1]} \subset \sH\sJ(\overline X,F)$ satisfying \cref{asmpt:one-param_floer_data} such that the moduli space $\osr^\tau(a;b)$ is a smooth manifold with corners of dimension $\mu(b)- \mu(a)$.  Moreover, the following conditions are satisfied:
\begin{enumerate}
    \item For $a,a' \in \sX(L_0,L_1;H^0)$ and $b,b' \in \sX(L_0,L_1;H^1)$, there are maps
    \begin{align*}
        \osr(a;a';H^0,J^0_t) \times \osr^\tau(a';b) \longrightarrow \osr^\tau(a;b) \\
        \osr^\tau(a;b') \times \osr(b';b;H^1,J^1_t) \longrightarrow \osr^\tau(a;b)
    \end{align*}
    that are diffeomorphisms onto faces of $\osr^\tau(a;b)$ such that defining
    \begin{align*}
    \partial_i \overline{\mathcal R}^\tau(a;b) &\coloneqq \bigsqcup_{\substack{a' \in \sX(L_0,L_1;H^0) \\ \mu(a') - \mu(a)= i}} (\overline{\mathcal R}(a;a';H^0,J^0_t) \times \overline{\mathcal R}^\tau(a';c))  \\
    &\qquad \sqcup \bigsqcup_{\substack{b' \in \sX(L_0,L_1;H^1) \\ \mu(b') - \mu(a) + 1= i}} (\overline{\mathcal R}^\tau(a;b') \times \osr(b';b;H^1,J^1_t)),
    \end{align*}
    endows $\overline{\mathcal R}^\tau(a;b)$ with the structure of a $\ang{\mu(b)-\mu(a)}$-manifold (see \cref{def:k_mfd}).
    \item Let $\check{\fo}_0(a)$ and $\check{\fo}_1(b)$ be the rank one free $R$-modules as in \cref{lem:transversality_for_strips}(ii) associated to $(H^0,J^0_t)$ and $(H^1,J^1_t)$, respectively. There are isomorphisms
    \[ \check{\fo}_1(b) \longrightarrow \check{\fo}_0(a) \otimes_R I^\tau_R(a;b)\]
    that extend and are compatible with the isomorphisms in \cref{lem:transversality_for_strips}(ii), and the above maps in the sense that diagrams similar to \eqref{eq:bimod_ext_ori1} and \eqref{eq:bimod_ext_ori2} commute.
\end{enumerate}
\end{lem}
\begin{proof}
For details on item (i), see the proof of \cref{lem:transversality_for_strips}.  We defer the proof of item (ii) to \cref{sec:gluing_abs_caps}.
\end{proof}

\begin{notn}
    \begin{itemize}
        \item Let $\sigma \colon \IR^2 \to [0,1]^2$ denote the smooth function $\sigma = (\tau,\tau)$ where $\tau$ is as in \cref{notn:tau_continuation}(ii).
        \item Let $\sigma_1 \coloneqq (\tau,1)$ and $\sigma_0 \coloneqq (\tau,0)$.
    \end{itemize}
\end{notn}
\begin{asmpt}\label{asmpt:two-param_floer_data}
    \begin{enumerate}
        \item Let $(H^{0,r},J^{0,r}_t)_{r\in [0,1]}$ and $(H^{1,r},J^{1,r}_t)_{r\in [0,1]}$ be two $1$-parameter families of elements of $\sH\sJ(\overline X,F)$, each one being regular in the sense that both \cref{asmpt:one-param_floer_data} and \cref{lem:moduli_cont} hold.
        \item Let $(H^{q,r},J^{q,r}_t)_{(q,r)\in [0,1]^2} \subset \sH\sJ(\overline X,F)$ be a $2$-parameter family extending the $1$-parameter families $(H^{0,r},J^{0,r}_t)_{r\in [0,1]}$ and $(H^{1,r},J^{1,r}_t)_{r\in [0,1]}$ and moreover is constant in $q$ for $r\in \{0,1\}$. We use the notation $(H^{0},J^{0}_t) \coloneqq (H^{0,r},J^{0,r}_t)$ and $(H^{1},J^{1}_t) \coloneqq (H^{1,r},J^{1,r}_t)$, respectively, for any $q\in [0,1]$.
    \end{enumerate}
\end{asmpt}

\begin{defn}
Let $a \in \sX(L_0,L_1;H^0)$ and $b \in \sX(L_0,L_1;H^1)$. Define $\osr^\sigma(a;b)$ to be the Gromov compactification of the set of maps:
\[ \sR^\sigma(a;b) \coloneqq \left\{ \begin{matrix} s_0 \in \IR, \\ u \colon \IR \times [0,1] \to \overline{X}\end{matrix} \, \left| \, \begin{matrix} \partial_su + J^{\sigma(s,s_0)}_t(\partial_t u - X_{H^{\sigma(s,s_0)}_t}) = 0 \\ u(\IR \times \{0\}) \subset L_0 \\ u(\IR \times \{1\}) \subset L_1 \\ \lim_{s \to -\infty} u(s,t) = b(t) \\ \lim_{s \to \infty} u(s,t) = a(t) \end{matrix} \right. \right\}. \]
Endow $\osr^\sigma(a;b)$ with the Gromov topology.  With this topology it is a compact metric space (see \cite[Lemma 3.9]{sylvan2019on}).
\end{defn}

Let $I^\sigma(a;b) \to \osr^\sigma(a;b)$ denote the index bundle of $\osr^\sigma(a;b)$ which is given by the virtual vector bundle $T\osr^\sigma(a;b) - \underline \IR$, and let $I^\sigma_R(a;b)$ denote its associated $R$-line bundle (see \cref{dfn:assoc_rline}).

\begin{lem}\label{lem:two_param_moduli}
There exists a comeager set of $(H^{q,r},J_t^{q,r})_{(q,r) \in [0,1]^2} \subset \sH\sJ(\overline X,F)$ satisfying \cref{asmpt:two-param_floer_data}(ii) such that the moduli space $\osr^\sigma(a;b)$ is a smooth manifold with corners of dimension $\mu(b)-\mu(a)+1$.  Moreover, the following conditions are satisfied:
\begin{enumerate}
    \item For $a,a' \in \sX(L_0,L_1;H^0)$ and $b,b' \in \sX(L_0,L_1;H^1)$, there are maps
    \begin{align*}
    \osr(a;a';H^0,J^0_t) \times \osr^\sigma(a';b) \longrightarrow \osr^\sigma(a;b) \\
    \osr^\sigma(a;b') \times \osr(b;b'; H^1, J^1_t) \longrightarrow \osr^\sigma(a;b)\\
    \osr^{\sigma_0}(a;b) \longrightarrow \osr^\sigma(a;b)\\
    \osr^{\sigma_1}(a;b) \longrightarrow \osr^\sigma(a;b)
    \end{align*}
    that are diffeomorphisms onto faces of $\osr^\sigma(a;b)$ such that defining
    \begin{align*}
        \partial_i^\circ \osr^\sigma(a;b) \coloneqq &\bigsqcup_{\substack{a' \in \sX(L_0,L_1;H^0) \\ \mu(a')-\mu(a)= i}} (\osr(a;a';H^0,J^0_t) \times \osr^\sigma(a';b)) \\
        &\qquad \sqcup \bigsqcup_{\substack{b' \in \sX(L_0,L_1;H^1) \\ \mu(b')-\mu(a)+2=i}} (\osr^\sigma(a;b') \times \osr(b';b; H^1, J^1_t))
    \end{align*}
    and
    \[
        \partial_i \osr^\sigma(a;b) \coloneqq \begin{cases}
            \osr^{\sigma_0}(a;b) \sqcup \osr^{\sigma_1}(a;b), & i = 0 \\
            \partial_{i-1}^\circ \osr^\sigma(a;b), & i > 1
        \end{cases},
    \]
    endows $\osr^\sigma(a;b)$ with the structure of a $\ang{\mu(b)-\mu(a)+2}$-manifold (see \cref{def:k_mfd}).
    \item Let $\check{\fo}_0(a)$ and $\check{\fo}_1(b)$ be the rank one free $R$-modules as in \cref{lem:transversality_for_strips}, associated to $(H^0,J^0_t)$ and $(H^1,J^1_t)$, respectively. There are isomorphisms
    \[ \check{\fo}_2(b) \longrightarrow \check{\fo}_0(a) \otimes_R I^\sigma_R(a;b) \]
    that extend and are compatible with the isomorphisms from \cref{lem:transversality_for_strips}(ii) and \cref{lem:moduli_cont}(ii) and the above maps in the sense that diagrams similar to \eqref{eq:bimod_ext_ori1} and \eqref{eq:bimod_ext_ori2} commute.
\end{enumerate}
\qed
\end{lem}

\begin{asmpt}
    Let $(H^r,J^r_t)_{r\in [0,2]} \subset \sH\sJ(\overline X,F)$ be so that $(H^0,J^0_t)$, $(H^1,J^1_t)$ and $(H^2,J^2_t)$ are regular in the sense that \cref{lem:transversality_for_strips} holds.
\end{asmpt}
\begin{notn}
    \begin{itemize}
        \item Let $\tau_0 \colon \IR \to [0,1]$ be a smooth function such that $\tau(s) \equiv 1$ near $-\infty$ and $\tau(s) \equiv 0$ near $\infty$.
        \item Let $\tau_1 \colon \IR \to [1,2]$ be a smooth function such that $\tau(s) \equiv 2$ near $-\infty$ and $\tau(s) \equiv 1$ near $\infty$.
        \item Let $\tau_2 \colon \IR \to [0,2]$ be a smooth function such that $\tau(s) \equiv 2$ near $-\infty$ and $\tau(s) \equiv 0$ near $\infty$.
    \end{itemize}
\end{notn}
\begin{defn}\label{dfn:2-parameter_composition}
    Let $R \in \IR$ and define $\tau^R \colon \IR \to [0,2]$ to be a smooth family of smooth functions such that
    \begin{enumerate}
        \item $\tau^R \equiv \tau_2$ near $R = -\infty$.
        \item $\tau^R(s) = \begin{cases}
            \tau_0(s-R), & s \geq 0 \\
            \tau_1(s+R), & s \leq 0
        \end{cases}$ near $R = \infty$.
    \end{enumerate}
\end{defn}
\begin{rem}
    Note that for $R$ near $\infty$ we have
    \[
    (H^{\tau^R(s)},J^{\tau^R(s)}_t) = \begin{cases}
        (H^{\tau_0(s-R)},J^{\tau_0(s-R)}_t), & s \geq 0 \\
        (H^{\tau_1(s+R)},J^{\tau_1(s+R)}_t), & s \leq 0
    \end{cases},
    \]
    and for $R$ near $-\infty$ we have $(H^{\tau^R(s)},J^{\tau^R(s)}_t) = (H^{\tau_2(s)},J^{\tau_2(s)}_t)$ for all $s\in \IR$.
\end{rem}

\begin{lem}\label{lma:gluing_floer_data}
There exists $R_0>0$ sufficiently large such that $(H^{\tau^R(s)},J^{\tau^R(s)}_t)_{s \in \IR}$ is a regular $1$-parameter family in the sense that \cref{lem:moduli_cont} holds for every $R>R_0$.
\qed
\end{lem}
\begin{lem}\label{lem:interpolate_compos}
    Let $\tau \colon \IR \to [0,1]$ be a smooth function such that $\tau(s) \equiv 1$ near $-\infty$ and $\tau(s) \equiv 0$ near $\infty$. Defining
    \begin{align*}
        \sigma_c \colon \IR^2 &\longrightarrow [0,1] \times [0,2] \\
        (s,s') &\longmapsto (\tau(s),\tau^s(s'))
    \end{align*}
    yields a $2$-parameter family of admissible Floer data $(H^{\sigma_c(s,s')},J^{\sigma_c(s,s')}_t)_{(s,s') \in \IR^2}$ that satisfies \cref{asmpt:two-param_floer_data}(ii).
    \qed
\end{lem}

For $a \in \sX(L_0,L_1;H^0)$ and $c \in \sX(L_0,L_1;H^2)$ define $\osr^{\tau_1\circ\tau_0}(a;c)$ to be the colimit of the following diagram
\[
    \begin{tikzcd}[row sep=scriptsize, column sep=scriptsize]
        \displaystyle \bigsqcup_{b'} (\osr^{\tau_0}(a;b) \times \osr(b;b';H^1,J_t^1) \times \osr^{\tau_1}(b';c)) \rar[shift right,swap]{} \rar[shift left]{} & \displaystyle \bigsqcup_{b} (\osr^{\tau_0}(a;b) \times \osr^{\tau_1}(b;c))
    \end{tikzcd},
\]
where the disjoint unions are taken over $b,b' \in \sX(L_0,L_1;H^1,J_t^1)$. The following is a consequence of \cref{lma:gluing_floer_data} and \cref{lem:two_param_moduli}.
\begin{lem}\label{lma:2-param_cont_moduli_composition}
    There exists a comeager set of $(H^{q,r},J^{q,r}_t)_{(q,r) \in [0,1]\times[0,2]} \subset \sH\sJ(\overline X,F)$ satisfying \cref{asmpt:two-param_floer_data}(ii) such that the moduli space $\osr^{\sigma_c}(a;c)$ is a smooth manifold with corners of dimension $\mu(c)-\mu(a)+1$.  Moreover, the following conditions are satisfied:
    \begin{enumerate}
        \item For $a,a' \in \sX(L_0,L_1;H^0)$ and $c,c' \in \sX(L_0,L_1;H^2)$, there are maps
        \begin{align*}
            \osr(a;a';H^0,J^0_t) \times \osr^{\sigma_c}(a';c) \longrightarrow \osr^{\sigma_c}(a;c) \\
            \osr^{\sigma_c}(a;c') \times \osr(c';c;H^2,J^2_t) \longrightarrow \osr^{\sigma_c}(a;c)\\
            \osr^{\tau_1\circ\tau_0}(a;c) \longrightarrow \osr^{\sigma_c}(a;c)\\
            \osr^{\tau_2}(a;c) \longrightarrow \osr^{\sigma_c}(a;c)
        \end{align*}
        that are diffeomorphisms onto faces of $\osr^{\sigma_c}(a;c)$ such that defining
        \begin{align*}
            \partial_i^\circ \osr^{\sigma_c}(a;c) \coloneqq &\bigsqcup_{\substack{a' \in \sX(L_0,L_1;H^0) \\ \mu(a')-\mu(a) = i}} (\osr(a;a';H^0,J^0_t) \times \osr^{\sigma_c}(a;c)) \\
            &\qquad \sqcup \bigsqcup_{\substack{c' \in \sX(L_0,L_1;H^2) \\ \mu(c')-\mu(a)+2=i}} (\osr^{\sigma_c}(a;c') \times \osr(c';c; H^2, J^2_t))
        \end{align*}
        and
        \[
            \partial_i \osr^{\sigma_c}(a;c) \coloneqq \begin{cases}
            \osr^{\tau_1\circ \tau_0}(a;c) \sqcup  \osr^{\tau_2}(a;c), & i = 1 \\
            \partial_{i-1}^\circ \osr^{\sigma_c}(a;c), & i > 1
            \end{cases},
        \]
        endows $\osr^{\sigma_c}(a;c)$ with the structure of a $\ang{\mu(c)-\mu(a)+2}$-manifold (see \cref{def:k_mfd}).
        \item Let $\check{\fo}_0(a)$ and $\check{\fo}_2(c)$ be the rank one free $R$-modules as in \cref{lem:transversality_for_strips}, associated to $(H^0,J^0_t)$ and $(H^2,J^2_t)$, respectively. There are isomorphisms
        \[ \check{\fo}_2(c) \longrightarrow \check{\fo}_0(a) \otimes_R I^{\sigma_c}_R(a;c) \]
        that extend and are compatible with the isomorphisms from \cref{lem:transversality_for_strips}(ii) and \cref{lem:moduli_cont}(ii) and the above maps in the sense that diagrams similar to \eqref{eq:bimod_ext_ori1} and \eqref{eq:bimod_ext_ori2} commute.
    \end{enumerate}
    \qed
\end{lem}

\subsection{Moduli space for the product}\label{subsec:moduli_product}

\begin{notn}
    Let $L_0$, $L_1$ and $L_2$ be Lagrangian $R$-branes in $X$.
\end{notn}

Let $S \coloneqq \ID^2 \smallsetminus \left\{\zeta_0,\zeta_1,\zeta_2\right\}$ be the unit disk in $\IC$ with its standard orientation and three boundary punctures removed. We equip $S$ with the choice of a Floer datum, with Hamiltonian $H_S \colon S \to \sH(\overline X,F)$ and almost complex structure $J_S \colon S \to \sJ(\overline X,F)$, see \cref{dfn:Floer_datum}. By abus

\begin{defn}\label{dfn:moduli_product}
Let $a \in \mathcal X(L_0,L_1;H)$, $b \in \mathcal X(L_1,L_2;H)$ and $c \in \mathcal X\left(L_0,L_2;\frac{1}{2}(\psi^2)^\ast H\right)$.  Define $\osr(a,b;c)$ to be the Gromov compactification of the set of maps 
\[
\sR(a,b;c) \coloneqq \left\{ u \colon S \to \overline X \left| \begin{matrix} u \mbox{ satisfies Floer's equation} \\ u(z) \in \psi^{\rho_S(z)}L_i, \;\; z \in (\partial S)_i \\ \lim_{s \to -\infty}u(\varepsilon_0^-(s,t)) = \psi^{2}(c(t)) \\ \lim_{s \to \infty}u(\varepsilon_1^+(s,t)) = a(t),
\\ \lim_{s \to \infty}u(\varepsilon_2^+(s,t)) = b(t)
\end{matrix} \right. \right\}, \]
where $(\partial S)_i$ denotes the boundary arc immediately after the boundary puncture $\zeta_i \subset \partial S$, according to the boundary orientation of the disk.  Endow $\osr(a,b;c)$ with the Gromov topology.  With this topology, it is a compact metric space (see \cite[Lemma 3.9]{sylvan2019on}).
\end{defn}

Let $I(a,b;c) \to \osr(a,b;c)$ denote the tangent bundle (the index bundle) of $\osr(a,b;c)$.  Let $I_R(a,b;c)$ denote the associated $R$-line bundle.

\begin{lem}\label{lem:moduli_product}
There exists a comeager set of $H_S \colon S \to \sH(\overline X,F)$ and $J_S \colon S \to \sJ(\overline X,F)$ such that moduli space $\osr(a,b;c)$ is a smooth manifold with corners of dimension $\mu(c) - \mu(a) - \mu(b)$. Moreover, the following conditions are satisfied:
\begin{enumerate}
    \item For $a,a' \in \sX(L_0,L_1)$, $b,b' \in \sX(L_1,L_2)$, and $c,c' \in \sX(L_0,L_2)$, there are maps
    \begin{align*}
        c^1_{a a' bc} \colon \osr(a;a') \times \osr(a',b;c) &\longrightarrow \osr(a,b;c)\\
        c^2_{a b b'c} \colon \osr(b;b') \times \osr(a,b';c) &\longrightarrow \osr(a,b;c)\\
        c^0_{a bc'c} \colon \osr(a,b;c') \times \osr(c';c) &\longrightarrow \osr(a,b;c)
    \end{align*}
    that satisfy the analogous compatibility conditions to those spelled out in \cref{dfn:flow_bimodule} and are diffeomorphisms onto faces of their targets such that defining
    \begin{align*}
    \partial_i \osr(a,b;c) &\coloneqq \bigsqcup_{\substack{a' \\ \mu(a') - \mu(a) = i}} (\osr(a;a') \times \osr(a',b;c))\\
    &\qquad \sqcup \bigsqcup_{\substack{b' \\ \mu(b')-\mu(b) = i}} (\osr(b;b') \times \osr(a,b';c)) \\
    &\qquad \sqcup \bigsqcup_{\substack{c' \\ \mu(c')-\mu(b)-\mu(a)+1 = i}} (\osr(a,b;c') \times \osr(c;c')),
    \end{align*}
    endows $\osr(a,b;c)$ with the structure of a $\ang{\mu(c)-\mu(a)-\mu(b)+1}$-manifold (see \cref{def:k_mfd}).
    \item Let $\check{\fo}_0(a)$, $\check{\fo}_1(b)$, and $\check{\fo}_2(c)$ be rank one free $R$-modules as in \cref{lem:transversality_for_strips}(ii). There are isomorphisms
    \[ \check{\fo}_2(c) \overset{\simeq}{\longrightarrow} \check{\fo}_0(a) \otimes_R \check{\fo}_1(b) \otimes_R I_R(a,b;c)  \]
    that are compatible with the isomorphisms from \cref{lem:transversality_for_strips}(ii) and the above maps in the sense that diagrams similar to \eqref{eq:bimod_ext_ori1} and \eqref{eq:bimod_ext_ori2} commute.
\end{enumerate}
\end{lem}
\begin{proof}
    For item (i), see the proof of \cref{lem:transversality_for_strips}.  We defer the proof of item (ii) to \cref{sec:gluing_abs_caps}.
\end{proof}

\begin{rem}
    Two conformally equivalent Floer data (see \cref{dfn:conformally_equiv_floer_data}) on $S$ in \cref{dfn:moduli_product} yields canonically identified moduli spaces.
\end{rem}

\subsection{Moduli space for associativity}

In order to show associativity of the product in the wrapped Donaldson--Fukaya category, we need to consider moduli spaces with three input punctures and one output puncture.

\begin{notn}
Let $L_0$, $L_1$, $L_2$, and $L_3$ be Lagrangian $R$-branes.
\end{notn}

Recall that $\osr_3$ denotes the Deligne--Mumford compactification of the moduli space of disks with a negative boundary puncture $\zeta_0$, and positive boundary punctures $\zeta_1, \zeta_2, \zeta_3$, labeled in counterclockwise order. We have equipped $\osr_3$ with a universal and conformally consistent choice of Floer data with Hamiltonians $H_S \colon S \to \sH(\overline X,F)$ and almost complex structures $J_S \colon S \to \sJ(\overline X,F)$, see \cref{dfn:univ_Floer_datum}.

\begin{defn}\label{dfn:moduli_assoc}
Let $a \in \mathcal X(L_0,L_1;H)$, $b \in \mathcal X(L_1,L_2;H)$, $c \in \mathcal X(L_2,L_3;H)$ and $d \in \mathcal X(L_0,L_3;\frac{1}{3}(\psi^3)^\ast H)$.  Define $\osr(a,b,c;d)$ to be the Gromov compactification of the set of maps
\[ \mathcal R(a,b,c;d) \coloneqq \left\{ \begin{matrix} S' \in \osr_3, \\ u \colon S' \to \overline X \end{matrix}\, \left| \, \begin{matrix} u \mbox{ solves Floer's equation} \\ u(z) \in \psi^{\rho_{S'}(z)}L_i, \;\; z \in (\partial S')_i \\ \lim_{s\to -\infty} u(\varepsilon_0^-(s,t)) = \psi^3(d(t))
\\ \lim_{s\to \infty} u(\varepsilon_1^+(s,t)) = a(t),
\\ \lim_{s\to \infty} u(\varepsilon_2^+(s,t)) = b(t),
\\ \lim_{s\to \infty} u(\varepsilon_3^+(s,t)) = c(t),\end{matrix} \right. \right\},\]
where above $(\partial S')_i$ denotes the boundary arc immediately after the boundary puncture $\zeta_i \subset \partial S'$ according to the boundary orientation of the disk. $\osr(a,b,c;d)$ is a compact metric space endowed with the Gromov topology (see \cite[Lemma 3.9]{sylvan2019on}).
\end{defn}
In the following we let $\overline{\mathcal R}_1(a,b,c;d)$ denote the colimit of the following diagram
\[
    \begin{tikzcd}[row sep=scriptsize, column sep=1.25cm]
        \displaystyle \bigsqcup_{e,e' \in \sX(L_1,L_3)} (\overline{\mathcal R}(b,c;e) \times \overline{\mathcal R}(e;e') \times \overline{\mathcal R}(a,e';d)) \rar[shift right,swap]{c^0_{bc e e'} \times \id} \rar[shift left]{\id \times c^2_{a e e' d}} & \displaystyle \bigsqcup_{e \in \sX(L_1,L_3)} (\overline{\mathcal R}(b,c;e) \times \overline{\mathcal R}(a,e;d))
    \end{tikzcd}.
\]
Similarly, we let $\overline{\mathcal R}_2(a,b,c;d)$ denote the colimit of the following diagram
\[
    \begin{tikzcd}[row sep=scriptsize, column sep=1.25cm]
        \displaystyle \bigsqcup_{f,f' \in \sX(L_0,L_2)} (\overline{\mathcal R}(a,b;f) \times \overline{\mathcal R}(f;f') \times \overline{\mathcal R}(f',c;d)) \rar[shift right,swap]{c^0_{ab f f'} \times \id} \rar[shift left]{\id \times c^1_{f f' cd }} & \displaystyle \bigsqcup_{f \in \sX(L_0,L_2)} (\overline{\mathcal R}(a,b;f) \times \overline{\mathcal R}(f,c;d))
    \end{tikzcd},
\]
where the maps $c^j$ are those in \cref{lem:moduli_product}(i).
\begin{lem}\label{lem:moduli_associativity}
There exists a comeager set of $H_S \colon S \to \sH(\overline X,F)$ and $J_S \colon S \to \sJ(\overline X,F)$ such that the moduli space $\osr(a,b,c;d)$ is a smooth manifold with corners of dimension $D \coloneqq \mu(d)-\mu(a)-\mu(b)-\mu(c)$. Moreover, the following conditions are satisfied:
\begin{enumerate}
    \item For $a,a' \in \sX(L_0,L_1)$, $b,b' \in \sX(L_1,L_2)$, $c,c' \in \sX(L_2,L_3)$, $d,d' \in \sX(L_0,L_3)$, there are maps
    \begin{align*}
        \osr(a;a') \times \osr(a',b,c;d) &\longrightarrow \osr(a,b,c;d)\\
        \osr(b;b') \times \osr(a,b',c;d) &\longrightarrow \osr(a,b,c;d)\\
        \osr(c;c') \times \osr(a,b,c';d) &\longrightarrow \osr(a,b,c;d)\\
        \osr(a,b,c';d) \times \osr(d;d') &\longrightarrow \osr(a,b,c;d)\\
        \overline{\mathcal R}_1(a,b,c;d) &\longrightarrow \osr(a,b,c;d) \\
        \overline{\mathcal R}_2(a,b,c;d) &\longrightarrow \osr(a,b,c;d)
    \end{align*}
    that satisfy the analogous compatibility conditions spelled out in \cref{dfn:flow_multimodule} and are diffeomorphisms onto faces of their targets. Let
    \begin{align*}
        \partial^\circ_i \osr(a,b,c;d) \coloneqq &\bigsqcup_{\substack{a' \in \sX(L_0,L_1) \\ \mu(a')-\mu(a) + 1= i}} (\osr(a;a') \times \osr(a',b,c;d)) \\
        &\qquad \sqcup \bigsqcup_{\substack{b' \in \sX(L_1,L_2) \\ \mu(b')-\mu(b) +1 = i}} (\osr(b;b') \times \osr(a,b',c;d)) \\
        &\qquad \sqcup \bigsqcup_{\substack{c' \in \sX(L_2,L_3) \\ \mu(c')-\mu(c) +1 = i}} (\osr(c;c') \times \osr(a,b,c';d)) \\
        &\qquad \sqcup \bigsqcup_{\substack{d' \in \sX(L_0,L_3) \\ D + 1 = i}} (\osr(a,b,c;d') \times \osr(d';d)).
    \end{align*}
    Setting
    \[
        \partial_i \osr(a,b,c;d) \coloneqq \begin{cases}
                \overline{\mathcal R}_1(a,b,c;d)\sqcup \overline{\mathcal R}_2(a,b,c;d), & i = 1 \\ \partial^\circ_{i-1} \osr(a,b,c;d), & i > 1
            \end{cases},
    \]
    endows $\osr(a,b,c;d)$ with the structure of a $\ang{D+1}$-manifold (see \cref{def:k_mfd}).
    \item Let $\check{\fo}_0(a)$, $\check{\fo}_1(b)$, $\check{\fo}_2(c)$, and $\check{\fo}_3(d)$ be rank one free $R$-modules as in \cref{lem:transversality_for_strips}(ii). There are isomorphisms
    \[ \check{\fo}_3(d) \overset{\simeq}{\longrightarrow} \check{\fo}_0(a) \otimes_R \check{\fo}_1(b) \otimes_R \check{\fo}_2(c) \otimes_R I_R(a,b,c;d)\]
    that are compatible with the isomorphisms from \cref{lem:transversality_for_strips}(ii) and the above maps in the sense that diagrams similar to \eqref{eq:bimod_ext_ori1} and \eqref{eq:bimod_ext_ori2} commute.
\end{enumerate}
\qed
\end{lem}
\begin{rem}
    A conformal equivalence (see \cref{dfn:conformally_equiv_floer_data}) of any of the Floer data in the universal and conformally consistent choice of Floer data on $\osr_3$ in \cref{dfn:moduli_assoc} induces a canonical identification between the two corresponding moduli spaces $\osr(a,b,c;d)$.
\end{rem}

\subsection{Moduli spaces for units}

\begin{notn}
Let $L$ be a Lagrangian $R$-brane.
\end{notn}

Consider $S \coloneqq \ID^2 \smallsetminus \{-1\}$ be the unit disk in $\IC$ with its standard orientation and a negative puncture at $-1$.  We equip $S$ with a choice of Floer datum, with Hamiltonian $H_S \colon S \to \sH(\overline X,F)$ and almost complex structure $J_S \colon S \to \sJ(\overline X,F)$, see \cref{dfn:Floer_datum} (note that even though $S$ is not stable this definition still makes sense).

\begin{defn}\label{dfn:moduli_floer_unit}
Let $a \in \sX(L,L)$.  Define $\osr_L(a)$ to be the Gromov compactification of the set of maps
\[ \sR_L(a) \coloneqq \left.\left\{ u \colon S \to \overline{X} \, \left| \, \begin{matrix} u \mbox{ solves the Floer equation} \\ u(z) \in \psi^{\rho_{S}(z)}L \; \; z \in \partial S \\ \lim_{s \to -\infty} u(\varepsilon^-(s,t)) = a(t) \end{matrix} \right. \right\}\right. \]
Endow $\osr_L(a)$ with the Gromov topology.  With this topology, it is a compact metric space (see \cite[Lemma 3.9]{sylvan2019on}).
\end{defn}

\begin{lem}\label{lem:moduli_units}
There exists a comeager set of $H_S \colon S \to \sH(\overline X,F)$ and $J_S \colon S \to \sJ(\overline X,F)$ such that the moduli space $\osr_L(b)$ is a smooth manifold with corners of dimension $\mu(b)$.  Moreover, the following conditions are satisfied:
\begin{enumerate}
    \item For $a,b \in \sX(L,L)$, there are maps
    \[ \osr_L(a) \times \osr(a;b) \longrightarrow \osr(b) \]
    that are diffeomorphisms onto faces of $\osr(b)$ such that defining
    \[ \partial_i \osr_L(b) \coloneqq \bigsqcup_{\substack{a \in \sX(L,L) \\ \mu(a) + 1 = i}} \osr_L(a) \times \osr(a;b), \]
    endows $\osr_L(b)$ with the structure of a $\ang{\mu(b)}$-manifold (see \cref{def:k_mfd}).
    \item Let $\check{\fo}(a)$ be the rank one free $R$-module as in \cref{lem:transversality_for_strips}(ii). There are isomorphisms
    \[ \check{\fo}(b) \overset{\simeq}{\longrightarrow} I_R(b) \]
    that are compatible with the isomorphisms from \cref{lem:transversality_for_strips}(ii) in the sense that diagrams similar to \eqref{eq:bimod_ext_ori1} and \eqref{eq:bimod_ext_ori2} commute.
\end{enumerate}
\qed
\end{lem}

For some $\rho \in \IR_{\geq 0}$, let $Z_+^\rho \coloneqq [\rho,\infty) \times [0,1]$ and $Z_-^{\rho} \coloneqq (-\infty,-\rho] \times [0,1]$. Let $S_1$ denote the domain of an element of $\overline{\mathcal R}(a,b;c)$ that is equipped with the positive strip-like end $\varepsilon_1 \colon Z_+^0 \to S_1$ near the boundary puncture $\zeta_1$, and let $S_1^\rho \coloneqq S_1 \smallsetminus \varepsilon_1(Z_+^\rho)$. Similarly let $S_2$ denote the domain of an element of $\overline{\mathcal R}_L(x_1)$ equipped with the negative strip-like end $\varepsilon_L \colon Z_-^0 \to S_2$, and let $S^2_\rho \coloneqq S_2 \smallsetminus \varepsilon_L(Z_-^\rho)$. Define $S^\rho \coloneqq S_1 \cup_{\varphi_\rho} S_2$ where $\varphi_\rho$ is the following composition
\[
    \varepsilon_1(Z_+^\rho) \overset{\varepsilon_1^{-1}}{\longrightarrow} Z_+^\rho \xrightarrow{(\cdot (-1), \id)} Z_-^\rho \overset{\varepsilon_L}{\longrightarrow} \varepsilon_L(Z_-^\rho).
\]
We compactify the set $\{S^\rho \mid \rho \geq 0\}$ as follows. Let $S^\infty$ denote the nodal curve with components $S_1$ and $S_2$ glued along $\zeta_1 \in S_1$ and $-1 \in S_2$, and let $S^{-1}$ denote $S_2$ after removing the boundary puncture $\zeta_1$. We let $\{S^\rho\}_{-1 \leq \rho \leq 0}$ be a smooth deformation between $S_{-1}$ and $S_0$. Furthermore equip $S^\rho$ with Floer data with Hamiltonian $H_{S^\rho} \colon S^\rho \to \sH(\overline X,F)$ and almost complex structure $J_{S^\rho} \colon S^\rho \to \sJ(\overline X,F)$ inherited from $S_1^\rho$ and $S_2^\rho$ as well as with strip-like ends $\varepsilon_0$ and $\varepsilon_2$ inherited from $S_2$ at positive puncture $\zeta_0$ and negative puncture $\zeta_2$.
\begin{defn}\label{dfn:conn_sum_dom_unit}
Let $a \in \sX(L,L)$ and $b,c\in \sX(L,K)$.  Define $\overline{\mathcal B}_L(a;b)$ to be the Gromov compactification of the set of maps
\[ \sB_L(a;b) \coloneqq \left.\left\{ \begin{matrix} \rho \in [-1,\infty] ,\\ u \colon S^\rho \to \overline{X} \end{matrix} \, \left| \, \begin{matrix} u \mbox{ solves the Floer equation} \\ u((\partial S_1)_0) \subset L \\ u((\partial S_1)_1) \subset L \\ u((\partial S_1)_2) \subset K \\ u(\partial S_2) \subset L \\ \lim_{s \to -\infty} u(\varepsilon_0^-(s,t)) = b(t) \\ \lim_{s \to \infty} u(\varepsilon_2^+(s,t)) = a(t) \end{matrix} \right. \right\}\right. \]
Endow $\sB_L(a;b)$ with the Gromov topology.  With this topology, it is a compact metric space (see \cite[Lemma 3.9]{sylvan2019on}).
\end{defn}

In the following let $\overline{\mathcal B}_1(a;b)$ denote the colimit of the following diagram
\[
    \begin{tikzcd}[row sep=scriptsize, column sep=1.25cm]
            \displaystyle \bigsqcup_{a,a' \in \sX(L,L)} (\overline{\mathcal R}_L(a) \times \overline{\mathcal R}(a;a') \times \overline{\mathcal R}(a',b;c)) \rar[shift right,swap]{} \rar[shift left]{} & \displaystyle \bigsqcup_{a \in \sX(L,L)} (\overline{\mathcal R}_L(a) \times \overline{\mathcal R}(a,b;c))
        \end{tikzcd},
\]
where the two arrows are the action maps from \cref{lem:moduli_units} and \cref{lem:moduli_product}, respectively. Let $\tau_0 \colon \IR \to [0,1]$ be defined by $\tau_0 \equiv 0$. Recall the definition of $\osr^{\tau_0}$ in \cref{dfn:moduli_maps_cont}.
\begin{lem}\label{lem:moduli_unit_bordism}
    There exists a comeager set of $H_{S^\rho} \colon S^\rho \to \sH(\overline X,F)$ and $J_{S^\rho} \colon S^\rho \to \sJ(\overline X,F)$ such that the moduli space $\overline{\mathcal B}_L(a;b)$ is a smooth manifold with corners of dimension $\mu(b)-\mu(a)+1$.  Moreover, the following conditions are satisfied:
    \begin{enumerate}
        \item For $a,a' \in \sX(L,L)$ and $b,b' \in \sX(L,K)$ there are maps
        \begin{align*}
            \overline{\mathcal R}(a;a') \times \overline{\mathcal B}_L(a';b) &\longrightarrow \overline{\mathcal B}_L(a;b) \\
            \overline{\mathcal B}_L(a;b') \times \overline{\mathcal R}(b';b) &\longrightarrow \overline{\mathcal B}_L(a;b) \\
            \overline{\mathcal B}_1(a;b) &\longrightarrow \overline{\mathcal B}_L(a;b) \\
            \osr^{\tau_0}(a;b) &\longrightarrow \overline{\mathcal B}_L(a;b),
        \end{align*}
        that satisfy the analogous compatibility conditions spelled out in \cref{dfn:flow_multimodule} are diffeomorphisms onto faces of their targets such that setting
        \begin{align*}
            \partial_i^\circ \overline{\mathcal B}_L(a;b) \coloneqq &\bigsqcup_{\substack{a' \in \sX(L,L) \\ \mu(a')-\mu(a) = i}} (\overline{\mathcal R}(a;a') \times \overline{\mathcal B}_L(a';b)) \\
            &\qquad \sqcup \bigsqcup_{\substack{b' \in \sX(L,K) \\ \mu(b')-\mu(a)+2 = i}} (\overline{\mathcal B}_L(a;b') \times \overline{\mathcal R}(b';b))
        \end{align*}
        and defining
        \[
            \partial_i \overline{\mathcal B}_L(a;b) \coloneqq \begin{cases}
                \overline{\mathcal B}_1(a;b) \sqcup \osr^{\tau_0}(a;b), & i = 1 \\
                \partial_{i-1}^\circ \overline{\mathcal B}_L(a;b), & i > 1 \\
            \end{cases},
        \]
        endows $\overline{\mathcal B}_L(a;b)$ with the structure of a $\ang{\mu(b)-\mu(a)+2}$-manifold (see \cref{def:k_mfd}).
        \item Let $\check{\fo}_0(a)$ and $\check{\fo}_1(b)$ be rank one free $R$-modules as in \cref{lem:transversality_for_strips}(ii). There are isomorphisms
        \[ \check{\fo}_1(b) \overset{\simeq}{\longrightarrow} \check{\fo}_0(a) \otimes_R I_R(a;b)\]
        that are compatible with the isomorphisms from \cref{lem:transversality_for_strips}(ii) and the above maps in the sense that diagrams similar to \eqref{eq:bimod_ext_ori1} and \eqref{eq:bimod_ext_ori2} commute.
    \end{enumerate}
    \qed
\end{lem}
\section{Canonical orientations}\label{sec:canonical_orientations}

In this section, we construct canonical $R$-orientations for moduli spaces of $J$-holomorphic disks.  For ease of exposition, we restrict to the case of stably polarized Liouville sectors although our construction of canonical orientations can be carried out (with additional work) for weaker assumptions on the Liouville sectors.

\begin{notn}
    \begin{enumerate}
        \item Let $X$ be a Liouville sector with a choice of stable polarization $\sE \colon X \to \BO$.
        \item Let $L_0,L_1 \subset X$ be Lagrangian $R$-branes
        \item Let $H_t \colon X \to \IR$ be a time-dependent non-degenerate Hamiltonian and let $J_t$ be a time dependent almost complex structure on $X$ such that for each fixed $t$ and each time-$1$ chord $\gamma$ of $H$ from $L_0$ to $L_1$, $J_t$ is constant in a Darboux ball containing $\gamma(t)$.  We assume that $(H_t,J_t) \in \sH\sJ(\overline X,F)$, see \cref{defn:AdmissibleFloerData}.
    \end{enumerate}
\end{notn}

\subsection{Local operators}

Associated to a time-$1$ chord $\gamma \in \sX(L_0,L_1;H)$, we have a local operator obtained from the linearized Floer equation at the constant solution with image equal to $\gamma$.

\begin{defn}
    Given a time-$1$ chord $\gamma$ of $H$ from $L_0$ to $L_1$, the \emph{local operator} at $\gamma$ is the translation invariant operator
    \[ Y_\gamma \in \varOmega^{0,1}_{\IR \times [0,1]} \otimes_J \End(\gamma^*TX) \]
    that is determined as follows.  The linearization of the Floer equation at the constant solution $u(s,t) = \gamma(t)$ determines an invertible operator:
    \[ D_\gamma \colon W^{1,2}\left(\IR \times [0,1], (\gamma^*TX, L_0|_{\gamma(0)}, L_1|_{\gamma(1)})\right) \longrightarrow L^2\left(\IR \times [0,1], \varOmega^{0,1}_{\IR \times [0,1]} \otimes_J \gamma^*TX\right), \]
    where $W^{1,2}\left(\IR \times [0,1], (\gamma^*TX, L_0|_{\gamma(0)}, L_1|_{\gamma(1)})\right)$ denotes the space of sections of $\gamma^*TX$ over $\IR \times [0,1]$ which map $\IR \times \{0\}$ to $TL_0|_{\gamma(0)}$ and $\IR \times \{1\}$ to $TL_1|_{\gamma(1)}$.  It is of the form
    \[ D_\gamma(\xi) = \overline{\partial}_{J_t}(\xi) + Y_t(\xi), \]
    where $Y_t = Y_\gamma$.
\end{defn}

\subsection{The space of abstract Floer strip caps}

Recall the definition of $\sD(\gamma)$ from \cref{dfn:abstract_strip_caps} and that the following diagram is a homotopy pullback by \cref{dfn:pullback_path}
\[
\begin{tikzcd}[row sep=scriptsize,column sep=scriptsize]
    \sD_+^\#(\gamma) \rar{\pi} \dar[swap]{\rho^\#} & \sD_+(\gamma) \dar{\rho} \\
    \sP_{x_0^\#,x_1^\#}(\lag)^\# \rar & \sP_{x_0,x_1}(\lag)
\end{tikzcd}.
\]
We now assume that $L_0$ and $L_1$ are equipped with $R$-brane structures.

\begin{defn}\label{dfn:abstract_floer_strip_caps}
The set of \emph{positive abstract Floer strip caps} associated to a time-$1$ chord $\gamma$ of $H$ from $L_0$ to $L_1$ is the set of tuples $(v, L, J, Y, g)$ with $(\varepsilon,\alpha,u) \coloneqq \pi(v)$ that satisfy the following:
\begin{enumerate}
	\item $v \in \sD_+^\#(\gamma^\ast \sE)$ such that $\rho^\# (v)(i) = \sG_{L_i}^\# (\gamma(i))$ for $i \in \{0,1\}$.
	\item $L \in \IR_{>0}$.
	\item $J$ is an almost complex structure on the vector bundle over $D_+$ classified by $c \circ u$ such that $(\varepsilon|_{(-\infty,-L) \times [0,1]})^*J$ agrees with the given almost complex structure $J_t$ along $\gamma$.
	\item $Y \in \varOmega^{0,1}_{D_+} \otimes_J \End(c \circ u)$ such that $Y_\gamma = (\varepsilon|_{(-\infty,-L) \times [0,1]})^*Y$.
	\item $g$ is a metric on $D_+$ such that $(\varepsilon|_{(-\infty,-L) \times [0,1]})^*g$ agrees with the standard metric on $\IR \times [0,1]$.

\end{enumerate}
Let $\sC_+^\#(\gamma)$ denote the set of such tuples and endow it with the subspace topology of the product topology.  We define $\sC_-^\#(\gamma)$ and $\sC^\#$ analogously.
\end{defn}
\begin{rem}
    Note that $\sC^\#_\pm(\gamma)$ depends on the two Lagrangian $R$-branes $L_1$ and $L_2$. Sometimes we make the choice of Lagrangian $R$-branes explicit and use the notation $\sC^\#_{\pm;L_1,L_2}(\gamma)$.
\end{rem}
We define spaces $\sC_\pm(\gamma)$ analogously to \cref{dfn:abstract_floer_strip_caps} and note that there are projection maps
\begin{align}\label{eq:projection_strip_caps}
    p_\pm \colon \sC_\pm^\#(\gamma) &\longrightarrow \sC_\pm(\gamma) \\
    (v,L,J,Y,g) &\longmapsto (\pi(v),L,J,Y,g). \nonumber
\end{align}

\begin{lem}\label{lma:forgetful_maps_caps}
There are forgetful maps

\begin{center}
    \begin{minipage}{0.3\textwidth}
	\begin{align*}
		\sC^\#_\pm(\gamma) &\longrightarrow \sD^\#_\pm(\gamma^\ast \sE) \\
		(v,L,J,Y,g) &\longmapsto v
	\end{align*}
\end{minipage}
\begin{minipage}{0.3\textwidth}
	\begin{align*}
		\sC_\pm(\gamma) &\longrightarrow \sD_\pm(\gamma^\ast \sE) \\
		(v,L,J,Y,g) &\longmapsto v
	\end{align*}
\end{minipage}

\end{center}
that are continuous homotopy equivalences which moreover commute with the projections $p_\pm$. 
\end{lem}

\begin{proof}
Continuity and the condition that they commute with the projections $p_\pm$ are clear.  The statement about homotopy equivalence follows from the fact that the set of tuples $(L,J,Y,g)$ satisfying items (ii)--(v) in \cref{dfn:abstract_floer_strip_caps} is contractible.
\end{proof}

\subsection{Index bundles of abstract strip caps}\label{sec:index_bundle_abstract_strip_caps}

An element $(v,L,J,Y,g) \in \sC_\pm^\#(\gamma)$ with $(\varepsilon,\alpha,u) \coloneqq \pi(v)$ determines the linear operator
\begin{align*}
    D_v \colon W^{1,2}\left(\overline D_\pm, (c \circ u, \alpha(0), \alpha(1))\right) &\longrightarrow L^2\left(\overline D_\pm, \varOmega^{0,1}_{\overline D_\pm} \otimes_J (c \circ u)\right) \\
    \xi &\longmapsto \overline{\partial}_J(\xi) + Y(\xi).
\end{align*}

\begin{defn}\label{dfn:index_bundle_caps}
The \emph{index bundle} associated to $\sC_\pm^\#(\gamma)$ is the bundle $\sV_\pm(\gamma) \to \sC_\pm^\#(\gamma)$ with fiber over $(v,L,J,Y,g)$ given by $\ind(D_v)$.
\end{defn}

\begin{lem}\label{lem:CanonicalOrientation}
The index bundle $\sV_\pm(\gamma) \to \sC^\#_\pm(\gamma)$ is $R$-orientable.
\end{lem}

\begin{proof}
Let $\sV \to \sD^\#$ denote the vector bundle that is classified by the composition
\[
\ind \circ \rho^\# \colon \sD^\# \longrightarrow \BO.
\]
Similarly, $\sV_\pm(\gamma) \to \sD^\#_\pm(\gamma)$ is the vector bundle classified by the map
\[
\sD^\#_\pm(\gamma) \overset{\rho^\#}{\longrightarrow} \sP_{x_0^\#,x_1^\#}(\lag)^\# \longrightarrow \sP_{x_0,x_1}(\lag) \overset{\ind}{\longrightarrow} \BO.
\]
Similarly as in \cref{lem:cap_gluing} below, following \cite[Section 8]{large2021spectral} there is a continuous gluing map that is covered my maps of index bundles which are fiberwise isomorphisms:
\[
    \begin{tikzcd}[row sep=scriptsize, column sep=scriptsize]
        \sV_+(\gamma) \times \sV_-(\gamma) \rar \dar & \sV \dar \\ 
        \sD_+^\#(\gamma) \times \sD_-^\#(\gamma) \rar  & \sD^\#
        .  \\ 
    \end{tikzcd}
\]
By \cref{lem:DHashNullity}, $\sV$ is $R$-orientable.  It follows that $\sV_\pm(\gamma)$ is $R$-orientable. Since the index bundle $\sV_\pm(\gamma) \to \sC^\#_\pm(\gamma)$ is the pullback of the corresponding index bundle over $\sD^\#_\pm(\gamma)$, the result follows.
\end{proof}

\subsection{Gluing abstract strip caps}\label{sec:gluing_abs_caps}
We now tie in our discussion back to our moduli spaces of $J$-holomorphic disks with Lagrangian boundary conditions.  To reduce clutter below, we denote $\sC_+^\#$ and $\sV_+$ by $\sC^\#$ and $\sV$, respectively.

\begin{lem}\label{lem:cap_gluing}
Consider moduli spaces of $J$-holomorphic disks as in \cref{lem:transversality_for_strips,lem:moduli_cont,lem:two_param_moduli,lma:2-param_cont_moduli_composition,lem:moduli_product,lem:moduli_associativity,lem:moduli_units}.  There are continuous gluing maps
\begin{center}
    \begin{minipage}{0.45\textwidth}
        \begin{align*}
            \sC^\#(a) \times \osr(a;b) &\longrightarrow \sC^\#(b) \\
            \sC^\#(a) \times \osr^{\tau}(a;b) &\longrightarrow \sC^\#(b) \\
            \sC^\#(a) \times \osr^\sigma(a;b) &\longrightarrow \sC^\#(b) \\
            \sC^\#(a) \times \osr^{\sigma_c}(a;b) &\longrightarrow \sC^\#(b)
        \end{align*}
    \end{minipage}
    \begin{minipage}{0.45\textwidth}
        \begin{align*}
            \sC^\#(a) \times \sC^\#(b) \times \osr(a,b;c) &\longrightarrow \sC^\#(c) \\
            \sC^\#(a) \times \sC^\#(b) \times \sC^\#(c) \times \osr(a,b,c;d) &\longrightarrow \sC^\#(d) \\
            \osr_L(a) &\longrightarrow \sC^\#(a)
        \end{align*}
    \end{minipage}
\end{center}
that are covered by maps of index bundles which are fiberwise isomorphisms, so that the following diagrams commute

\[
	\begin{tikzcd}[row sep=scriptsize, column sep=scriptsize]
		\sV(a) \times I(a;b) \rar \dar & \sV(b) \dar \\ \sC^\#(a) \times \osr(a;b) \rar & \sC^\#(b)
	\end{tikzcd}, \qquad
    \begin{tikzcd}[row sep=scriptsize, column sep=scriptsize]
		\sV(a) \times I^\tau(a;b) \rar \dar & \sV(b) \dar \\ \sC^\#(a) \times \osr^{\tau}(a;b) \rar & \sC^\#(b)
	\end{tikzcd}
\]
\[
	\begin{tikzcd}[row sep=scriptsize, column sep=scriptsize]
		\sV(a) \times I^\sigma(a;b) \ar[r] \ar[d] & \sV(b) \ar[d] \\ \sC^\#(a) \times \osr^\sigma(a;b) \ar[r] & \sC^\#(b)
	\end{tikzcd}, \qquad
    \begin{tikzcd}[row sep=scriptsize, column sep=scriptsize]
		\sV(a) \times I^{\sigma_c}(a;b) \ar[r] \ar[d] & \sV(b) \ar[d] \\ \sC^\#(a) \times \osr^{\sigma_c}(a;b) \ar[r] & \sC^\#(b)
	\end{tikzcd}
\]
\[
	\begin{tikzcd}[row sep=scriptsize, column sep=tiny]
		\sV(a) \times \sV(b) \times I(a,b;c) \ar[r] \ar[d] & \sV(c) \ar[d] \\\sC^\#(a) \times \sC^\#(b) \times \osr(a,b;c) \ar[r] & \sC^\#(c)
	\end{tikzcd}, \qquad
    \begin{tikzcd}[row sep=scriptsize, column sep=tiny]
		\sV(a) \times \sV(b) \times \sV(c) \times  I(a,b,c;d) \ar[r] \ar[d] & \sV(d) \ar[d] \\ \sC^\#(a) \times \sC^\#(b) \times \sC^\#(c) \times \osr(a,b,c;d) \ar[r] & \sC^\#(d)
	\end{tikzcd}
\]
\[
	\begin{tikzcd}[row sep=scriptsize, column sep=scriptsize]
		I_L(a) \ar[r] \ar[d] & \sV(a) \ar[d] \\ \osr_L(a) \ar[r] & \sC^\#(a)
	\end{tikzcd}.
\]

Moreover, these gluing maps are associative in the appropriate sense.
\end{lem}

\begin{proof}
The proof of this lemma is essentially carried out by Large in \cite[Section 8.2]{large2021spectral}.  Even though Large is dealing with a particular (and different) notion of Lagrangian brane, the geometric aspect of his argument carry over and apply directly in our context.  We will focus on the case of $\osr(a;b)$ as the other cases (except the final case) are analogous.  In short, the data specified in the definition of $\sC^\#(a)$ may be matched up with the analogous data determined by the linearization of the Floer equation at an element of $\sR(a;b)$.  Following Large's argument in \cite[Section 8.2]{large2021spectral}, one produces a coherent gluing map
\[ G \colon \sC^\#(a) \times \overline{\sR}(a;b) \longrightarrow \sC(b). \]
The stable polarization of $X$ and the $R$-brane structures on the Lagrangians determine a unique lift of this gluing map to $\sC^\#(b)$, giving the desired map.

Finally, for the last instance of $\osr_L(a)$.  Notice that there is a natural map $\osr_L(a) \hookrightarrow \sC^\#(a)$.  The index bundle on $\osr_L(a)$ is simply the pullback of the index bundle on $\sC^\#(a)$.  This yields the final gluing map.
\end{proof}

\begin{proof}[Proof of item (ii) in \cref{lem:transversality_for_strips,lem:moduli_cont,lem:two_param_moduli,lma:2-param_cont_moduli_composition,lem:moduli_product,lem:moduli_associativity,lem:moduli_units}]
This follow immediately from \cref{lem:CanonicalOrientation} and \cref{lem:cap_gluing}.  Namely, by \cref{lem:CanonicalOrientation} the index bundle $\sV(a)$ admits an $R$-orientation. Hence the associated $R$-line bundle $\sV_R(a)$ over $\sC^\#(a)$ is trivializable, and we fix such a trivialization. Abusing the notation we will also denote the associated rank one free $R$-module by $\sV_R(a)$ given by the generic fiber.  Setting $\check{\fo}(a) \coloneqq \sV_R(a)$ and using the compatibility conditions above yield the desired data.
\end{proof}

\section{The spectral wrapped Donaldson--Fukaya category}\label{sec:donaldson-fukaya}

The purpose of this section is to accumulate the work from \cref{sec:ModuliSpaces} and \cref{sec:canonical_orientations} to give the construction of the wrapped Donaldson--Fukaya category of a stably polarized Liouville sector with coefficients in a commutative ring spectrum $R$.  It is a category enriched in the homotopy category of $R$-modules. The objects are Lagrangian $R$-branes and the morphism space associated to two Lagrangian $R$-branes is the $R$-module spectrum obtained as the CJS realization of a flow category associated to the two Lagrangian $R$-branes.

\subsection{Flow category associated to a pair of Lagrangian $R$-branes}\label{subsec:flow_lag}

\begin{notn}
\begin{enumerate}
    \item Let $R$ be a commutative ring spectrum.
    \item Let $X$ be a stably polarized Liouville sector.
    \item Let $L_0$ and $L_1$ denote Lagrangian $R$-branes.
    \item Let $(H,J_t) \in \sH\sJ(\overline{X},F)$ be a regular pair in the sense that \cref{lem:transversality_for_strips} holds.
    \item Let $\sX(L_0,L_1;H)$ denote the set of Hamiltonian time-$1$ chords between $L_0$ and $L_1$ from \cref{defn:ham_chords}.
\end{enumerate}
\end{notn}

\begin{defn}[Flow category associated to Lagrangian $R$-branes]\label{dfn:flow_cat_lag}
The \emph{flow category associated to $L_0$, $L_1$, and $(H,J_t) \in \sH\sJ(\overline X,F)$} is denoted by $\sC\sW(L_0,L_1;H,J_t)$ and is defined as follows:
\begin{itemize}
    \item The set of objects is given by $\Ob(\sC\sW(L_0,L_1;H,J_t)) \coloneqq \mathcal X(L_0,L_1,H)$.
    \item The grading function is given by $-\mu$ (i.e.\@ \emph{negative} of the Maslov index).
    \item The morphisms from $a$ to $b$ are given by $\sC\sW(L_0,L_1;H,J_t)(a;b) \coloneqq \overline{\sR}(a;b;H,J_t)$.
    \item The $R$-orientation is given by $\fo_{(L_0,L_1;H)}$ is given by $\fo(a) \coloneqq (-\sV(a))_R$ and the isomorphism of $R$-line bundles
    \[
        \psi_{ab} \colon \fo(a) \overset{\simeq}{\longrightarrow} I_R(x;y) \otimes_R \fo(b),
    \]
    given by \cref{lem:transversality_for_strips}(ii). Here $\sV(a) = \sV_+(a)$ is as in \cref{sec:index_bundle_abstract_strip_caps}, and $(-\sV(a))_R$ denotes by abuse of notation the generic fiber of the associated $R$-line bundle $(-\sV(a))_R$, 
\end{itemize}
\end{defn}

\begin{rem}
    \begin{enumerate}
        \item By \cref{lem:transversality_for_strips} it is clear that $\sC\sW(L_0,L_1;H,J_t)$ defines an $R$-oriented flow category (see \cref{dfn:flow_cat}).
        \item The flow category $\sC\sW(L_0,L_1;H,J_t)$ can also be equipped with a coherent system of collars as defined in \cref{dfn:flow_cat_coherent_collars}.  These collars are determined by the gluing parameters for the $J$-holomorphic disks, see \cite[Section 8.2]{large2021spectral} for details. In the following we assume implicitly that we have chosen a coherent system of collars for $\sC\sW(L_0,L_1;H,J_t)$.
    \end{enumerate}
\end{rem}

\begin{notn}
    \begin{enumerate}
        \item Let $c \in \sX(L_0,L_1;H)$. Denote the action functional of $c$ associated to Floer's equation by
        \[
            \sA(c) \coloneqq \int c^\ast \lambda - \int H(c(t)) - (h_{1}(c(1)) - h_{0}(c(0))),
        \]
        where $h_0 \colon L_0 \to \IR$ and $h_1 \colon L_1 \to \IR$ are smooth functions satisfying $\lambda|_{L_0} = dh_0$ and $\lambda|_{L_1} = dh_1$, respectively.
        \item For any $A \in \IR$, set
        \[
            \sX^{\leq A}(L_0,L_1;H) \coloneqq \left\{c \in \sX(L_0,L_1;H) \mid \sA(c) \leq A\right\}.
        \]
    \end{enumerate}
\end{notn}
\begin{defn}\label{defn:action_fil}
    Let $A \in \IR$ and define $\sC\sW^{\leq A}(L_0,L_1;H,J_t)$ to be the $R$-oriented flow category that is defined in the same way as $\sC\sW(L_0,L_1;H,J_t)$ (see \cref{dfn:flow_cat_lag}) except that the set of objects is
    \[
        \Ob(\sC\sW^{\leq A}(L_0,L_1;H,J_t)) \coloneqq \mathcal X^{\leq A}(L_0,L_1;H).
    \]
    We equip it with the restriction of the $R$-orientation to the action filtered pieces, which we denote by $\fo^{\leq A}_{(L_0,L_1;H)}$.
\end{defn}
By \cref{defn:action_fil}, it is clear that if $A \leq B$, then $\sC\sW^{\leq A}(L_0,L_1;H,J_t)$ is a full subcategory of $\sC\sW^{\leq B}(L_0,L_1;H,J_t)$.

\begin{defn}\label{dfn:wrapped_floer_cohomology}
Let $(H,J_t) \in \sH\sJ(\overline X,F)$. Let $L_0$ and $L_1$ be two Lagrangian $R$-branes in $X$. The \emph{wrapped Floer cohomology with coefficients in $R$} of $L_0$ and $L_1$ is given by the following homotopy colimit of CJS realizations (see \cref{dfn:cjs_realization})
\[ HW(L_0,L_1;H,J_t) \coloneqq \hocolim_{k\to \infty} |\sC\sW^{\leq A_k}(L_0,L_1;H,J_t), \fo^{\leq A_k}_{(L_0,L_1;H)}|,\]
where $\{A_k\}_{k=1}^\infty \subset \IR$ is any strictly increasing sequence diverging to $\infty$. For any $A \in \IR$ we use the notation
\[
HW^{\leq A}(L_0,L_1;H,J_t) \coloneqq |\sC\sW^{\leq A}(L_0,L_1;H,J_t), \fo^{\leq A}_{(L_0,L_1;H)}|.
\]
\end{defn}
\begin{rem}\label{rem:Novikov}
    \begin{enumerate}
        \item Since any sequence $\{A_k\}_{k=1}^\infty$ as in \cref{dfn:wrapped_floer_cohomology} is cofinal in $\IR$ with its standard total order, it follows that $HW(L_0,L_1;H,J_t)$ is independent of the choice of the sequence $\{A_k\}_{k=1}^\infty$.
        \item If the grading function $\mu$ on the flow category $\sC\sW(L_0,L_1;H,J_t)$ is bounded below (e.g.\@ if either $L_0$ or $L_1$ is compact), we have
        \[
        HW(L_0,L_1;H,J_t) = |\sC\sW(L_0,L_1;H,J_t), \fo_{(L_0,L_1;H)}|.
        \]
        \item In general, the homotopy colimits over action filtration is necessary in order to get a homotopy type lifting the usual wrapped Floer cohomology. Our definition of CJS realization corresponds to the one appearing in \cite{abouzaid2021arnold}. Recall from \cref{dfn:cjs_realization} that
        \[
        HW(L_0,L_1;H,J_t) = \hocolim_{k\to \infty}\holim_q |\sC\sW^{\leq A_k}(L_0,L_1;H,J_t), \fo^{\leq A_k}_{(L_0,L_1;H)}|_q.
        \]
        Swapping the order of the homotopy colimit and the homotopy limit yields the homotopy type $|\sC\sW(L_0,L_1;H,J_t), \fo_{(L_0,L_1;H)}|$, which lifts a Novikov-type completion of wrapped Floer cohomology (cf.\@ \cite{venkatesh2018rabinowitz}).
    \end{enumerate} 
\end{rem}

We justify our definition. Let $HW^{\bullet}(L_0,L_1;\IZ)$ denote the classical wrapped Floer cohomology with integer coefficients, defined using pin structures relative to the background class $w_2(\mathcal E)$ (see \cite[Section 3.3]{abouzaid2012wrapped}), where $\sE \colon X \to BO$ is the stable polarization.
\begin{lem}\label{lem:wfuk_discrete_coeff}
    Let $R = H\IZ$. There is an isomorphism of $\IZ$-modules
    \[
    \pi_\bullet(HW(L_0,L_1;H,J_t)) \cong HW^{-\bullet}(L_0,L_1;\IZ).
    \]
\end{lem}
\begin{proof}
    Pick a strictly increasing sequence $\{A_k\}_{k=1}^\infty$ diverging to $\infty$. By \cref{lem:flow_htpy_morse} the homotopy groups of $HW^{\leq A_k}(L_0,L_1;H,J_t)$ are isomorphic to the homology of the Morse complex associated with the underlying $R$-oriented flow category $(\sC\sW^{\leq A_k}(L_0,L_1;H,J_t),\fo^{\leq A_k}_{(L_0,L_1;H)})$ (see \cref{dfn:morse_complex_flow_cat}). It is clear by construction that this complex coincides with the complex $CW^{-\bullet}(L_0,L_1;\IZ)$ as defined in \cite[Section 3.3]{abouzaid2012wrapped}, when ignoring signs and orientations (after taking the action-filtered piece of generators of action $\leq A_k$). Viz., the moduli spaces of $J$-holomorphic disks are equal, and the grading function in $\sC\sW^{\leq A_k}(L_0,L_1;H,J_t)$ is the negative of the Maslov index. To compare signs and orientations, let us recall the discussion in \cref{sec:integer_coeffs}. Namely, by \cref{lem:hz_ori_data}, the choices of $H\IZ$-orientation data on $L_0$ and $L_1$ are equivalent to choices of pin structures on $L_0$ and $L_1$ relative to the background class $w_2(\sE)$. Moreover, the Cauchy--Riemann operator at a Hamiltonian chord $y$ considered on \cite[p.\@ 67]{abouzaid2012wrapped} in particular gives an instance of a positive abstract Floer strip cap at $y$ in the sense of \cref{dfn:abstract_floer_strip_caps}. In particular, the trivialization of $\sV_R(y)$ selected in proof of \cref{lem:transversality_for_strips}(ii) at the end of \cref{sec:gluing_abs_caps} fixes an identification of the determinant line $\fo_y$ from \cite[Definition 6.1]{abouzaid2012wrapped} with $\check{\fo}(y)$ as in \cref{lem:transversality_for_strips} defined via \cref{lem:cap_gluing}. Note that this trivialization corresponds to the choice of relative pin structure on $\varLambda_y$ described on \cite[p.\@ 68]{abouzaid2012wrapped}. Tracing through these identifications, it can be checked that the isomorphisms between index bundles obtained form \cite[Lemma 6.1]{abouzaid2012wrapped} coincide with the isomorphisms obtained in \cref{lem:transversality_for_strips} via \cref{lem:cap_gluing}. Finally taking the homotopy colimit as $k\to \infty$ and using that homotopy groups commute with filtered (in particular sequential) homotopy colimits yields the result.
\end{proof}
\begin{rem}\label{rem:wfuk_discrete_coeff_k}
    \Cref{lem:wfuk_discrete_coeff} also holds true for $R = Hk$ where $k$ is any commutative discrete ring $k$.
\end{rem}
\begin{rem}
    We remark that \cref{lem:wfuk_discrete_coeff} stands in contrast with \cite[Section 5.2]{large2021spectral}, where the homotopy groups of the CJS realization of the corresponding flow category yields wrapped Floer \emph{homology}.
\end{rem}
We now show invariance of the isomorphism class of the $R$-module $HW(L_0,L_1;H;J_t)$ under changes of the Floer data $(H,J_t) \in \sH\sJ(\overline X,F)$.
\begin{lem}\label{lem:IndOfData}
Given two choices of admissible Floer data $(H^0,J_t^0), (H^1,J_t^1) \in \sH\sJ(\overline X,F)$ in the sense that \cref{lem:transversality_for_strips} holds, there is an equivalence of $R$-modules
\[
HW(L_0,L_1;H^0;J_t^0) \simeq HW(L_0,L_1;H^1,J_t^1).
\]
\end{lem}

\begin{proof}
This follows the same line of argument as the standard proof of invariance for Lagrangian Floer cohomology, cf.\@ \cite[Section 3.4]{salamon1999lectures} in the closed string case.

Consider an admissible $1$-parameter family $(H^r,J^r_t)_{r\in [0,1]} \subset \sH\sJ(\overline X,F)$ in the sense that \cref{lem:moduli_cont} holds. Pick a strictly increasing sequence $\{A_k\}_{k=1}^\infty$ diverging to $\infty$. By \cref{lem:moduli_cont}, the moduli spaces $\osr^\tau(a;b)$ thus determine an $R$-oriented flow bimodule
\[\osr^\tau\colon \sC\sW^{\leq A_k}(L_0,L_1;H^0,J^0_t) \longrightarrow \sC\sW^{\leq A_k}(L_0,L_1;H^1,J^1_t).\] 
Similarly, by reversing the admissible $1$-parameter family $(H^r,J^r_t)_{r\in [0,1]}$ by $r \mapsto 1-r$ we have that $\osr^{-\tau}$ defines an $R$-oriented flow bimodule
\[\osr^{-\tau} \colon \sC\sW^{\leq A_k}(L_0,L_1;H^1,J^1_t) \longrightarrow \sC\sW^{\leq A_k}(L_0,L_1;H^0,J^0_t).\]
Consider a smooth family of functions $\tau^R$ as defined in \cref{dfn:2-parameter_composition} with $\tau_0(s) = \tau(s)$, $\tau_1(s) = \tau(-s)$ and $\tau_2(s) = 0$ for all $s \in \IR$. Let $\sigma_c$ be as in \cref{lem:interpolate_compos} and choose an admissible $2$-parameter family $(H^{q,r},J^{q,r}_t)_{(q,r) \in [0,1]^2} \subset \sH\sJ(\overline X,F)$ in the sense that \cref{lma:2-param_cont_moduli_composition} holds. The flow bimodule $\osr^{\tau_2}$ coincides with the identity bimodule $\id_{\sC\sW^{\leq A}(L_0,L_1;H^0,J^0_t)}$ and hence by \cref{prop:induced_id} its CJS realization of $\osr^{\tau_2}$ is homotopic to the identity map on $HW^{\leq A_k}(L_0,L_1;H^0;J_t^0)$.

Secondly, it follows from \cref{lma:2-param_cont_moduli_composition} that the moduli spaces of $J$-holomorphic disks $\osr^{\sigma_c}(a;b)$ define an $R$-oriented flow bordism $\osr^{-\tau} \circ \osr^\tau \Rightarrow \osr^{\tau_2}$. Therefore \cref{lma:bordism_gives_homotopy_of_maps} shows that the two morphisms $|{\osr^{\tau_2}}| \simeq \id_{HW^{\leq A_k}(L_0,L_1;H^0;J_t^0)}$ and $|{\osr^{-\tau} \circ \osr^{\tau}}| = |{\osr^{-\tau}}| \circ |{\osr^{\tau}}|$ are homotopic.  Repeating the argument with the roles of $\tau$ and $-\tau$ swapped, we conclude that
\[ |{\osr^\tau}| \colon HW^{\leq A_k}(L_0,L_1;H^0;J_t^0) \longrightarrow HW^{\leq A_k}(L_0,L_1;H^1,J_t^1) \]
is a homotopy equivalence. Taking the homotopy colimit as $k\to \infty$ finishes the proof.
\end{proof}

\begin{notn}
In light of \cref{lem:IndOfData}, we write
\[ HW(L_0,L_1) = HW(L_0,L_1;H,J_t) \]
for some choice of admissible Floer data $(H,J_t) \in \sH\sJ(\overline X, F)$.
\end{notn}

To obtain a well-defined wrapped Donaldson--Fukaya category in our setting with quadratic Hamiltonians, we need to recall a standard rescaling trick, see \cite[Section 3.2]{fukaya2008symplectic}, \cite[Section 3.2]{abouzaid2011a} and \cite[Section 3.2]{sylvan2019on}.

Namely, let $\psi^\tau$ denote the time-$\tau$ flow of the Liouville vector field of $X$. If $(H,J_t) \in \sH\sJ(\overline{X},F)$, then $\left(\frac{1}{\tau}(\psi^\tau)^\ast H, (\psi^\tau)^\ast J_t\right) \in \sH\sJ(\overline{X},F)$, and there is a canonical bijection
\begin{align*}
    \mathcal{X}(L_0,L_1;H) &\overset{\simeq}{\longrightarrow} \mathcal{X}\left(\psi^\tau L_0,\psi^\tau L_1;\frac{1}{\tau} (\psi^\tau)^\ast H\right)\\
    a &\longmapsto a_\tau
\end{align*}
We have a canonical identification
\[
    \overline{\mathcal{R}}(a;b;H;J_t) \simeq \overline{\mathcal{R}}\left(a_\tau;b_\tau;\frac{1}{\tau}(\psi^\tau)^\ast H, (\psi^\tau)^\ast J_t\right),
\]
which induces a canonical identification of $R$-oriented flow categories
\begin{equation}\label{eq:ident_conformal_rescaling}
    \mathcal{CW}(L_0,L_1;H,J_t) \cong \mathcal{CW}\left(\psi^\tau L_0, \psi^\tau L_1; \frac{1}{\tau} (\psi^\tau)^\ast H, (\psi^\tau)^\ast J_t\right),
\end{equation}
and on their action-filtered versions. It follows that there is a canonical identification of $R$-modules
\[
    HW(L_0,L_1;H,J_t) \simeq HW\left(\psi^\tau L_0, \psi^\tau L_1; \frac{1}{\tau} (\psi^\tau)^\ast H, (\psi^\tau)^\ast J_t\right).
\]

\subsection{The triangle product}
Let $L_0, L_1, L_2$ be Lagrangian $R$-branes. Pick strictly increasing sequences $\{A_k\}_{k=1}^\infty$ and $\{B_\ell\}_{\ell=1}^\infty$ diverging to $\infty$. It is a consequence of \cref{lem:moduli_product} and Stokes' theorem that there is an $R$-oriented flow multimodule
\[
    \sC\sW^{\leq A_k}(L_0,L_1;H,J_t),\sC\sW^{\leq B_\ell}(L_1,L_2;H,J_t) \to \sC\sW^{\leq A_k+B_\ell}\left(\psi^2L_0,\psi^2L_2;\frac 12 (\psi^2)^\ast H,(\psi^2)^\ast J_t\right).
\]
Via the identification in \eqref{eq:ident_conformal_rescaling} yields an $R$-oriented flow multimodule
\begin{equation}\label{eq:multimodule_mu2}
    \osr^{k,\ell} \colon \sC\sW^{\leq A_k}(L_0,L_1),\sC\sW^{\leq B_\ell}(L_1,L_2) \longrightarrow \sC\sW^{\leq A_k+B_\ell}(L_0,L_2).
\end{equation}

\begin{defn}\label{defn:triangle_product}
The $R$-oriented flow multimodule $\osr^{k,\ell}$ defined in \eqref{eq:multimodule_mu2} induces via \cref{lma:functoriality_cjs} a morphism on the CJS realizations, and taking homotopy colimit as $k,\ell\to \infty$ yields an $R$-module map
\begin{equation}\label{eq:triangle_product}
    \mu^2 \colon HW(L_0,L_1) \wedge_{R} HW(L_1,L_2) \longrightarrow HW(L_0,L_2).
\end{equation}
\end{defn}
\begin{rem} By construction of a flow bordism similar to that in the proof of \cref{lem:IndOfData} it follows that \eqref{eq:triangle_product} is independent of the choice of Hamiltonian $H_S \colon S \to \sH(\overline X,F)$ and almost complex structure $J_S \colon S \to \sJ(\overline X,F)$.
\end{rem}

This product structure is appropriately associative:

\begin{lem}\label{lma:associativity_comp_wfuk}
The following diagram is homotopy commutative.
\[
    \begin{tikzcd}[row sep=scriptsize, column sep=1cm]
        HW(L_0,L_1) \wedge_{R} HW(L_1,L_2) \wedge_{R} HW(L_2,L_3) \rar{\mu^2 \wedge_{R} \id} \dar{\id \wedge_{R} \mu^2} & HW(L_0,L_2) \wedge_{R} HW(L_2,L_3) \dar{\mu^2} \\
        HW(L_0,L_1) \wedge_{R} HW(L_1,L_3) \rar{\mu^2} & HW(L_0,L_3)
    \end{tikzcd}
\]
\end{lem}

\begin{proof}
    It follows from \cref{lem:moduli_associativity} that the moduli spaces of $J$-holomorpic disks with three positive boundary punctures and one negative boundary puncture define an $R$-oriented flow bordism between the two $R$-oriented flow multimodules
    \[
    \osr \circ_1 \osr,\; \osr \circ_2\osr \colon \sC\sW^{\leq A_k}(L_0,L_1),\sC\sW^{\leq B_\ell}(L_1,L_2),\sC\sW^{\leq C_m}(L_2,L_3) \to \sC\sW^{\leq A_k+B_\ell+C_m}(L_0,L_3),
    \]
    whose CJS realizations are (the action-filtered versions of) $\mu^2 \circ (\mu^2 \wedge_R \id)$ and $\mu^2 \circ (\id \wedge_R \mu^2)$, respectively. Applying \cref{lma:bordism_gives_homotopy_of_maps} and taking homotopy colimits as $k,\ell,m\to\infty$ therefore finishes the proof.
\end{proof}

\subsection{Unit map}

Recall the definition of the $R$-oriented point flow category $\mathcal M_\ast$ from \cref{dfn:point_flow_cat}, and that its CJS realization is $R$. Pick a strictly increasing sequence $\{A_k\}_{k=1}^\infty$ diverging to $\infty$. We note that \cref{dfn:moduli_floer_unit} implies that there is an $R$-oriented flow bimodule
\begin{align*}
    \osr_L \colon \sM_\ast &\longrightarrow \sC\sW^{\leq 0}(L,L) \\
    (p,x) &\longmapsto \osr_L(x).
\end{align*}

\begin{defn}\label{dfn:unit_hw}
The \emph{unit} in $HW(L,L)$ is the composition 
\[ \eta_L \colon R \overset{|\osr_L|}{\longrightarrow} HW^{\leq 0}(L,L) \longrightarrow HW(L,L), \]
\end{defn}

We now have that $\eta_L$ is a unit in the following sense.
\begin{lem}\label{lem:unit_in_fuk}
The following composition is homotopic to the identity
\begin{align*}
    HW(L_0,L_1) \simeq R \wedge_{R} HW(L_0,L_1) &\xrightarrow{\eta_{L_0} \wedge_{R} \id} HW(L_0,L_0) \wedge_{R} HW(L_0,L_1) \\
    &\;\;\overset{\mu^2}{\longrightarrow} HW(L_0,L_1)
\end{align*}
\end{lem}

\begin{proof}
Pick a strictly increasing sequence $\{A_k\}_{k=1}^\infty$ diverging to $\infty$. Consider the $R$-oriented flow multimodule
\[
\osr^{k} \circ_1 \osr_{L_0} \colon \sM_\ast,\sC\sW^{\leq A_k}(L_0,L_1) \longrightarrow \sC\sW^{\leq A_k}(L_0,L_1),
\]
and consider its associated $R$-oriented flow bimodule
\[
\sU^{k} \coloneqq (\osr^{k} \circ_1 \osr_{L_0})(p,-) \colon \sC\sW^{\leq A_k}(L_0,L_1) \longrightarrow \sC\sW^{\leq A_k}(L_0,L_1).
\]
By \cref{lem:moduli_unit_bordism} we may define an $R$-oriented flow bordism
\[
    \mathcal B_{L_0} \colon \sU^{k} \Longrightarrow \osr^{\tau_0},
\]
where $\osr^{\tau_0} \colon \sC\sW^{\leq A_k}(L_0,L_1) \to \sC\sW^{\leq A_k}(L_0,L_1)$ is the $R$-oriented flow bimodule defined in the proof of \cref{lem:IndOfData}, and $\tau_0 \colon \IR \to [0,1]$ is defined by $\tau_0 \equiv 0$. We note that $\osr^{\tau_0}$ is flow bordant to the identity bimodule $\id_{\sC\sW^{\leq A_k}(L_0,L_1)}$ and hence by \cref{prop:induced_id} the CJS realization of $\osr^{\tau_0}$ is homotopic to the identity map on $\sC\sW^{\leq A_k}(L_0,L_1)$, and finally applying \cref{lma:bordism_gives_homotopy_of_maps} shows $|\sU^{k}| \simeq \id_{HW^{\leq A_k}(L_0,L_1)}$ which finishes the proof after passing to the homotopy colimit as $k \to \infty$, since $\hocolim_{k \to \infty} |\sU^{k}| = \mu^2 \circ (\eta_{L_0} \wedge_R \id)$.
\end{proof}

\subsection{Definition of the wrapped Donaldson--Fukaya category}

\begin{defn}\label{dfn:wfuk_K}
The \emph{wrapped Donaldson--Fukaya category with coefficients in $R$} of $X$ is the category $\mathcal W(X;R)$ enriched over the homotopy category of $R$-modules whose objects are given by the set of Lagrangian $R$-branes in $X$, the morphisms from $L_0$ to $L_1$ are given by
\[ \mathcal W(X;R)(L_0,L_1) \coloneqq HW(L_0,L_1), \]
and the composition maps are given by $\mu^2$ in \cref{defn:triangle_product}.
\end{defn}

\begin{rem}
    To recover the wrapped Donaldson--Fukaya category with coefficients in a discrete ring $k$ we consider $R = H k$, where $H k$ is the Eilenberg--MacLane spectrum of $k$, and take the homotopy groups of its morphism spectra (see \cref{lem:wfuk_discrete_coeff} and \cref{rem:wfuk_discrete_coeff_k}). 
\end{rem}
\begin{rem}\label{rem:relation_ps}
    Porcelli--Smith \cite{porcelli2024bordism,porcelli2024spectral} have considered a version of $\sW(X;R)$ defined via bordism groups of flow categories (cf.\@ Abouzaid--Blumberg \cite{abouzaid2024foundation}).

    The Donaldson--Fukaya category $\sF(X;(\varTheta,\varPhi))$ considered in \cite{porcelli2024spectral} is defined for a tangential pair $\varTheta \to \varPhi$ lying over $\BO \xrightarrow{c} \BU$, and our category $\sW(X;R)$ is related to theirs in case $(\varTheta,\varPhi) = ((\lag)^\#,\BO)$. Namely,
    \[
    \Ob(\sW(X;R)) = \Ob(\sF(X;((\lag)^\#,\BO))),
    \]
    and for any two Lagrangian $R$-branes $L_0,L_1 \in \Ob(\sW(X;R))$, we expect $\pi_\bullet(\sW(X;R)(L_0,L_1))$ to be related to $\sF(X;((\lag)^\#,\BO))(L_0,L_1)$ via a certain base change as follows: The tangential pair $((\lag)^\#,\BO)$ determines a pair of Thom ring spectra $(R_\varTheta,\varSigma^\infty_+\varOmega {\OO})$ defined in \cite[Section 3.10]{porcelli2024spectral}, see \cite[Example 2.8]{porcelli2024spectral}. There are ring maps $\varSigma^\infty_+\varOmega {\OO} \to \IS$, and $R_\varTheta \to R$, which implies that under \cite[Conjectures 5.10 and 5.17]{porcelli2024spectral} we expect there to be an isomorphism (up to taking duals, identifying Floer homology and Floer cohomology)
    \[
    \pi_\bullet(F \wedge_{R_\varTheta} R) \overset{\cong}{\longrightarrow} \pi_\bullet(\sW(X;R)(L_0,L_1)),
    \]
    for any $L_0,L_1 \in \Ob(\sW(X;R))$, where $\sF(X;((\lag)^\#,\BO))(L_0,L_1) \cong \pi_\bullet(F)$ for a certain $R_\varTheta$-module $F$ obtained via \cite[Conjectures 5.10 and 5.17]{porcelli2024spectral}.
\end{rem}

The following is one result for wrapped (Donaldson--)Fukaya categories that continues to hold with spectral coefficients without any additional work.
\begin{thm}[Cf.\@ {\cite[Section 3.6]{ganatra2020covariantly}}]\label{thm:inclusion_pushforward}
    Any inclusion of stably polarized Liouville sectors $X \hookrightarrow X'$ (see \cref{dfn:inclusion_sectors}) induces a canonical pushforward functor $\sW(X;R) \to \sW(X';R)$.
\end{thm}
\begin{proof}
    This follows from the discussion in \cite[Section 3.6]{ganatra2020covariantly} as the construction of the pushforward functor is purely geometric. The key point is that the assumption on compatible almost complex structures in \cref{dfn:compatible_J}(iii) ensures that $J$-holomorphic curves with Lagrangian boundary conditions in $X$ are contained in $X$.
\end{proof}
\begin{rem}
    An equivalent formulation in view of convexification is that an inclusion of stops $F' \hookrightarrow F$ induces a canonical pushforward functor $\sW(\overline X, F;R) \to \sW(\overline X,F';R)$ (cf.\@ \cite[(4.14)]{sylvan2019on}).
\end{rem}

\subsection{Equivalences and connective coefficients}\label{subsec:DF_conn_cover}
Recall the definition of the connective cover $R_{\geq 0} \to R$ of $R$ from \cref{sec:background_connective}.

From \cref{rem:equiv_Rbrane_loopsinfRbrane} and the construction of $\sW(X;-)$, it follows that the set of objects of $\sW(X,R)$ and $\sW(X;R_{\geq 0})$ coincide.  In fact, there is a functor of categories enriched over $\Ho \mod{R_{\geq 0}}$,
\[ -\wedge_{R_{\geq 0}} R \colon \sW(X;R_{\geq 0}) \longrightarrow \sW(X;R)\]
which is the identity on objects and is given on morphisms by
\[ HW(L,K;R_{\geq 0}) \longmapsto HW(L,K;R_{\geq 0}) \wedge_{R_{\geq 0}} R, \ \text{for all} \  L,K \in \Ob(\sW(X;R_{\geq 0})).\]

Consequently, if Lagrangian $R$-branes $L$ and $K$ are isomorphic as objects in $\sW(X;R_{\geq 0})$, then they are isomorphic in $\sW(X;R)$ as well. This functor in fact partially reflects isomorphisms:

\begin{lem}\label{cor:lifting_equivalences_to_conn_covers}
Suppose that the homotopy groups of $HW(L,L;R_{\geq 0})$, $HW(K,K;R_{\geq 0})$, $HW(L,K;R_{\geq 0})$, and $HW(K,L;R_{\geq 0})$ are all bounded from below. If $L \cong K$ in $\sW(X;R)$, then $L\cong K$ in $\sW(X;R_{\geq 0})$.
\end{lem}

\begin{proof}
    We may assume without loss of generality that $HW(L,K;R_{\geq 0})$ and $HW(K,L;R_{\geq 0})$ are $q$-connective for some $q \in \mathbb Z$. 
    
    Since $L \cong K$ in $\sW(X;R)$, there exist morphisms $\alpha \colon R \to HW(L,K;R)$ and $\beta \colon R \to HW(K,L;R)$ and homotopies $\mu^2 \circ(\beta \wedge_R \alpha) \simeq \eta_{K;R}$ and $\mu^2\circ(\alpha \wedge_R \beta) \simeq \eta_{L;R}$, where $\eta_{K;R} \colon R \to HW(K,K;R)$ and $\eta_{L;R} \colon R \to HW(L,L;R)$ denote the unit maps, see \cref{dfn:unit_hw}. By construction of $\sW(X;-)$ and \cref{lem:conn_smash_conn_is_id}, we have equivalences
    \[ HW(L,K;R)_{\geq 0} \simeq (HW(L,K;R_{\geq 0}) \wedge_{R_{\geq 0}} R)_{\geq 0} \simeq HW(L,K;R_{\geq 0})_{\geq 0}. \]
    In particular, from \cref{lem:obstruction_lifting_conn_covers}, it follows that the natural map
    \[HW(L,K;R_{\geq 0}) \longrightarrow HW(L,K;R)\]
    induces a bijection
    \[ [R_{\geq 0},HW(L,K;R)_{\geq 0}]_{R_{\geq 0}}  \simeq [R_{\geq 0}, HW(L,K;R)]_{R_{\geq 0}}, \]
     and, in particular, $\alpha$ admits a unique lift to a morphism $\widetilde{\alpha} \colon R_{\geq 0} \to HW(L,K;R_{\geq 0})$.  Similarly, the morphism $\beta$ admits a unique lift to a morphisms $\widetilde{\beta} \colon R_{\geq 0} \to HW(K,L;R_{\geq 0})$.  By uniqueness, we have homotopies $\mu^2\circ(\widetilde{\beta} \wedge_{R_\geq 0} \widetilde{\alpha}) \simeq \eta_{K;R_{\geq 0}}$ and $\mu^2\circ(\widetilde{\alpha} \wedge_{R_\geq 0} \widetilde{\beta}) \simeq \eta_{L;R_{\geq 0}}$. This gives the isomorphism between $L$ and $K$ in $\sW(X;R_{\geq 0})$.
\end{proof}

\begin{rem}\label{rem:finite_mor_space_for_compact}
    If $L$ and $K$ are compact Lagrangian $R$-branes, then the assumption in \cref{cor:lifting_equivalences_to_conn_covers} holds.  This follows from the finiteness of the set of Hamiltonian chords and the definition of the CJS realization (see \cref{dfn:cjs_realization}).
\end{rem}

\subsection{Triviality for subcritical Weinstein sectors}
    We now show that $\sW(X;R)$ is trivial for subcritical Weinstein sectors (see \cref{dfn:subcrit_weinstein_sector}).
    \begin{thm}\label{thm:subcrit_trivial}
        Suppose that $X$ is a stably polarized subcritical Weinstein sector. For any $L_0,L_1 \in \Ob(\sW(X;R))$ we have that $HW(L_0,L_1)$ is the zero $R$-module.
    \end{thm}
    \begin{proof}
        The Floer theoretic proof in the setting of discrete coefficients is essentially geometric, and it carries over to the present situation almost immediately, see e.g.\@ \cite[Section 11.2]{chantraine2017geometric}. Namely, for dimension reasons we without loss of generality assume that $L_0$ and $L_1$ both are disjoint from the core of $X$, and there is a Hamiltonian isotopy that displaces $L_1$ (say) from every compact subset of $X$, see \cite{cieliebak2002subcritical}. Hence we can find a one-parameter family $(H^r,J_t^r)_{r \in \IR_{\geq 0}} \subset \sH\sJ(\overline X,F)$ such that any object of $\sC\sW(L_0,L_1;H^r,J_t^r)$ has action bounded from below by $Ce^{2r}$ for any $r \gg 0$ large enough, see \cite[Lemma 11.2]{chantraine2017geometric}. Consequently for any $A > 0$ there exists $r \gg 0$ such that $\sC\sW^{\leq A}(L_0,L_1;H^{r},J_t^{r})$ is the empty category, whose CJS realization is zero. Finally pick a strictly increasing sequence of positive real numbers $\{A_k\}_{k=1}^\infty$ such that $A_k \to \infty$ as $k \to \infty$, and an accompanying sequence of real numbers $\{r_k\}_{k=1}^\infty$ such that $\sC\sW^{\leq A_k}(L_0,L_1;H^{r_k},J^{r_k}_t)$ is the empty category, for every $k \in \IZ_{\geq 1}$. Therefore we obtain equivalences
        \begin{align*}
            HW(L_0,L_1) &\simeq \hocolim_{k\to\infty} |\sC\sW^{\leq A_k}(L_0,L_1),\fo^{\leq A_k}_{(L_0,L_1)}| \\
            &\hspace{-6mm}\overset{\text{\cref{lem:IndOfData}}}{\simeq}\hocolim_{k\to\infty} |\sC\sW^{\leq A_k}(L_0,L_1;H^{r_k},J_t^{r_k}),\fo^{\leq A_k}_{(L_0,L_1;H^{r_k})}| \simeq 0.
        \end{align*}
    \end{proof}
\section{Open-closed maps and $R$-equivalent Lagrangians}\label{sec:construction_OC}

\begin{notn}\label{asmpt:lagrangian_compact}
\begin{enumerate}
    \item Let $X$ be a stably polarized Liouville sector.
    \item Let $R$ be a commutative ring spectrum.
\end{enumerate}
\end{notn}

\subsection{Morse and Floer theory}

We discuss the relationship between $HW(L,L)$ and the Morse theory of $L$, where $L$ is a compact Lagrangian $R$-brane.

\begin{notn}\label{notn:oc_morse}
\begin{enumerate}
    \item Let $f \colon L \to \IR$ denote a $C^2$-small positive Morse--Smale function on $L$.
    \item Let $U \subset X$ be a neighborhood of $L$ that is identified with a neighborhood of the zero section in $T^*L$, and let $\pi \colon U \to L$ be the associated projection map. We denote by $H_f \colon X \to \IR$ a fixed extension of $f$ which coincides with $\pi^*f$ on $U$.
\end{enumerate}
\end{notn}

Recall the definition of the cohomological Morse--Smale flow category $\sM^{-\bullet}(f,g)$ from \cref{defn:coh_morse-smale_flow_cat}.
\begin{lem}\label{lem:PSS}
There exists a comeager set of compatible almost complex structures $J_t$ on $X$ in the sense of \cref{dfn:compatible_J} such that 
\begin{enumerate}
    \item $(H_f,J_t) \in \sH\sJ(\overline X,F)$ is regular in the sense that \cref{lem:transversality_for_strips} holds,
    \item $f$ is Morse--Smale with respect to the restriction of the associated metric $g \coloneqq \omega(-,J_t-)$ to $L$, and
    \item there is an identification of $\sM^{-\bullet}(f,g)$ and $\sC\sW(L,L;H_f,J_t)$ as $R$-oriented flow categories that is unit preserving.
\end{enumerate}
\end{lem}

\begin{proof}
    Firstly, there is a natural bijection between the object sets of $\sC\sW(L,L;H_f,J_t)$ and $\sM^{-\bullet}(f,g)$. Next, since $f$ is $C^2$-small, we can assume without loss of generality that all solutions of the Floer equation with boundary on $L$ are contained inside $U$ (see e.g.\@ the discussion on p.\@ 526 following the proof of Lemma 3.6 in \cite{floer1988morse}).
    By \cite[Theorem 2]{floer1989witten}, there exists a comeager set of compatible almost complex structures $J_t$ satisfying items (i) and (ii) in the statement (cf.\@ \cite[Theorem 10.1.2]{audin2014morse}); it furthermore yields an identification
    \begin{align}\label{eq:ident_morse_floer}
         \sC\sW(L,L;H_f,J_t)(x;y) = \osr(x;y) &\overset{\cong}{\longrightarrow} \sM^{-\bullet}(f,g)(x;y) \\
         (u_1,\ldots,u_n) &\longmapsto (u_1|_{\IR \times \{0\}},\ldots,u_n|_{\IR \times \{0\}}) \nonumber
    \end{align}
    for any two objects $x,y$. Arguing as in \cite[Proposition 1]{floer1989witten} (cf.\@ \cite[Proposition 10.1.7]{audin2014morse} and \cite[Remark 6.1]{abouzaid2012wrapped}) we obtain isomorphisms
    \[
        (\sV_+(x))_R \simeq (T\overline W^u(x))_R, \quad (\sV_-(x))_R \simeq (T\overline W^s(x))_R
    \]
    for any object $x$, where $\sV_{\pm}(x)$ is the index bundle of the space of abstract Floer strip caps, see \cref{dfn:index_bundle_caps}. This implies $\fo_{(L,L;H_f)}(x) \simeq \fo^{-\bullet}_{(f,g)}(x)$ for any object $x$ and hence that \eqref{eq:ident_morse_floer} is an identification of $R$-oriented flow categories.
    To identify the units, recall the definitions of the unit flow bimodules $\sM_\ast \to \sC\sW(L,L;H_f,J_t)$ and $\sM_\ast \to \sM^{-\bullet}(f,g)$ from \cref{dfn:unit_hw} and \cref{dfn:unit_morse}, respectively. In the former case it is given by the assignment $(p,x) \mapsto \osr_L(x)$ and the isomorphism $(\sV_+(x))_R \overset{\simeq}{\to} (T\osr_L(x))_R$ from \cref{lem:moduli_units}(ii) (which is the identity isomorphism over $\osr_L(x) \hookrightarrow \sC^\#(x)$, see the proof of \cref{lem:cap_gluing}). In the latter case it is given by the assignment $(p,x) \mapsto \overline W^u(x)$ and the identity isomorphism. Pick a Floer datum on $S = \ID^2 \smallsetminus \{-1\} \subset \IC$ consisting of a Hamiltonian $H_S \colon S \to \sH(\overline X,F)$ such that $H_{S}(z)$ extends $f$ in the sense of \cref{notn:oc_morse}(ii), for every $z\in S$. Recall from \cref{notn:morse_flowline_compact} (and in particular \cite[Theorem 2.3]{wehrheim2012smooth}) that $\overline W^u(x)$ denotes the space of half infinite gradient flow lines of $-\nabla f$, compactified by broken flow lines. Let $\gamma \subset \partial S$ denote the half-infinite arc between $-1$ and $-i$. Hence, by a similar argument for the identification \eqref{eq:ident_morse_floer} we obtain an identification via
    \begin{align*}
        \osr_L(x) &\overset{\cong}{\longrightarrow} \overline W^u(x) \\
        (u_1,\ldots,u_n,u_\circ) &\longmapsto (u_1|_{\IR \times \{0\}}, \ldots, u_n|_{\IR \times \{0\}}, u_\circ|_\gamma),
    \end{align*}
    where $u_i \in \osr(y_{i-1},y_i)$ and $u_\circ \in \osr_L(y_{n+1})$ where $x \coloneqq y_0$, which finishes the proof.
\end{proof}

\begin{rem}
    Note that in the case $R = Hk$, \cref{lem:PSS} yields $HW^{-\bullet}(L,L;k) \cong H^{-\bullet}(L;k)$ by working with action-filtered flow categories and after passing to the homology of their Morse complexes (equivalently, homotopy groups of their CJS realizations) see \cref{rmk:morse_cohomology_flow} and \cref{rem:wfuk_discrete_coeff_k}. This is consistent with the classical PSS isomorphism, see e.g.\@ \cite[(12.14)]{seidel2008fukaya}.
\end{rem}

\subsection{The open-closed map}\label{sec:the_oc_map}

Let $S \coloneqq \ID^2 \smallsetminus \{1\}$ be the unit disk in $\IC$, equipped with its standard orientation with a positive puncture at $1$. We equip $S$ with a choice of Floer datum with Hamiltonian $H_S \colon S \to \sH(\overline X,F)$ and almost complex structure $J_S \colon S \to \sJ(\overline X,F)$, see \cref{dfn:Floer_datum}.

\begin{defn}\label{def:oc_pre}
Let $a \in \sX(L,L)$.  Define $\widetilde{\mathcal{OC}}(a)$ to be the Gromov compactification of the set of maps
\[  \left\{ u \colon S \to \overline{X}  \, \left| \, \begin{matrix} u \mbox{ solves Floer's equation} \\ u(z)\in \psi^{\rho_S(z)}L \; \; z\in \partial S \\ \lim_{s \to \infty} u(\varepsilon^+(s,t)) = a(t) \end{matrix} \right. \right\}. \]
Endow $\widetilde{\mathcal{OC}}(a)$ with the Gromov topology.  With this topology, it is a compact metric space.
\end{defn}

\begin{lem}\label{lem:pre_oc_moduli_space}
There exists a comeager set of $H_S \colon S \to \sH(\overline X,F)$ and $J_S \colon S \to \sJ(\overline X,F)$ such that the moduli space $\widetilde{\mathcal{OC}}(a)$ is a smooth manifold with corners of dimension $-\mu^{\mathrm{Maslov}}(a)+n$.  Moreover, the following conditions are satisfied:
\begin{enumerate}
    \item For $a,b \in \sX(L,L)$, there are maps
    \[ \osr(a;b) \times \widetilde{\mathcal{OC}}(b) \longrightarrow \widetilde{\mathcal{OC}}(a) \]
    that are diffeomorphisms onto faces of $\widetilde{\mathcal{OC}}(a)$ such that defining
    \[ \partial_i \widetilde{\mathcal{OC}}(a) \coloneqq \bigsqcup_{\mu^{\mathrm{Maslov}}(b)-\mu^{\mathrm{Maslov}}(a)= i} \osr(a;b) \times \widetilde{\mathcal{OC}}(b),\]
    endows $\widetilde{\mathcal{OC}}(a)$ with the structure of a $\ang{-\mu^{\mathrm{Maslov}}+n}$-manifold (see \cref{def:k_mfd}).
    \item Let $\check{\fo}(a)$ be the rank one free $R$-module as in \cref{lem:transversality_for_strips}. Fixing an $R$-orientation of $L$ determines an isomorphism
    \[ R \overset{\simeq}{\longrightarrow} \check{\fo}(a) \otimes_R I_R^{\widetilde{\mathcal{OC}}}(a)\]
    that is compatible with the isomorphisms from \cref{lem:transversality_for_strips} and the above maps in the sense that diagrams similar to \eqref{eq:bimod_ext_ori1} and \eqref{eq:bimod_ext_ori2} commute.
\end{enumerate}
\end{lem}

\begin{proof}
See \cref{lem:transversality_for_strips} and \cite[Section 5.3]{abouzaid2010a}.  The argument for smoothness follows from a similar discussion as in the proof of \cref{lem:transversality_for_strips}(i).

Turning to $R$-orientations, as in \cref{lem:cap_gluing}, there are analogous gluing maps
\[ \sC_+^\#(a) \times \widetilde{\mathcal{OC}}(a) \longrightarrow \sC_+^\#(a) \times \sC^\#_-(a) \longrightarrow \sD^\#,\]
where the data of $\sD^\#$ is based at $a(0)$. The gluing maps are covered by maps of index bundles.  The $R$-line bundle associated to the index bundle of $\sD^\#$ is canonically trivializable and is given by the associated $R$-line bundle over the constant element. It follows that we have an isomorphism of $R$-line bundles
\[ (\sV_+(a))_R \otimes_R I^{\widetilde{\sO\sC}}_R(a) \simeq (T_{a(0)} L)_R. \]
Analogously, we have an isomorphism of $R$-line bundles
\[ (\sV_+(a))_R \otimes_R I(a;b) \otimes_R I^{\widetilde{\sO\sC}}_R(b) \simeq (T_{b(0)} L)_R. \]
Using the $R$-orientability of $L$, we obtain the desired isomorphism in item (ii).  Similarly the desired compatibilities follow.
\end{proof}

\begin{rem}\label{rem:oc_stable}
    By \cref{lem:pre_oc_moduli_space} it follows that the assignment $(a,p) \mapsto \widetilde{\sO\sC}(a)$ defines an $R$-oriented flow bimodule $\sC\sW(L,L;H_f,J_t) \to \sM_\ast$. Similar to the identification between the $R$-oriented flow bimodules $\sM_\ast \to \sC\sW(L,L;H_f,J_t)$ and $\sM_\ast \to \sM^{-\bullet}(f,g)$ in the proof of \cref{lem:PSS}, there is an identification between $\widetilde{\mathcal{OC}}$ and the assignment $\sM^{-\bullet}(f,g) \to \sM_\ast$, $(x,p) \mapsto \overline W^s(x)$ as $R$-oriented flow bimodules. The $R$-orientation on $\widetilde{\sO\sC}$ induced by \cref{lem:pre_oc_moduli_space}(ii) corresponds under this identification with the isomorphism $(T_x\overline W^u(x))_R \otimes_R (T_x\overline W^s(x))_R \simeq (T_xL)_R \simeq R$ induced by the given $R$-orientation on $L$.
\end{rem}

We shall use $\widetilde{\mathcal{OC}}$ to define an $R$-oriented flow bimodule from $\sC\sW(L,L)$ to the (homological) Morse--Smale flow category associated to the Liouville sector $X$.

\begin{notn}
    \begin{enumerate}
        \item We extend $f \colon L \to \IR$ to a Morse function $F \colon X \to \IR$ as follows:
        Using a bump function, add a positive definite form in the normal direction of $L$ to obtain a function on $X$, that vanishes outside of a neighborhood of $L$.  Then perturb this function away from the neighborhood of $L$ to obtain a Morse function $F \colon X \to \IR$ whose gradient points outwards along $\partial_\infty X$ and points outwards along the horizontal boundary of $X$.  Notice that $\crit f \subset \crit F$ and $W^u_{-\nabla F}(x) = W^u_{-\nabla f} (x)$ for $x \in \crit f$.
        \item Let $J_t \colon [0,1] \to \sJ(\overline X,F)$ be a time-dependent compatible almost complex structure on $X$ and $g = \omega(-,J_t-)$ its associated Riemannian metric so that the conclusion of \cref{lem:PSS} holds and so that $F$ is Morse--Smale with respect to $g$.
    \end{enumerate}
\end{notn}

Consider the evaluation map
\begin{align*}
    \ev \colon \widetilde{\mathcal{OC}}(x) & \longrightarrow X, \\
    u &\longmapsto u(0)
\end{align*}
and note that it is well-defined since we do not quotient by the automorphisms of the domain in the definition of $\widetilde{\mathcal{OC}}(a)$ in \cref{def:oc_pre}.
\begin{defn}[Open-closed flow bimodule]\label{dfn:oc}
Let $L$ be $R$-orientable. The \emph{open-closed flow bimodule} is the $R$-oriented degree $n$ flow bimodule $\mathcal{OC} \colon \sC\sW(L,L;R) \longrightarrow \mathcal M(F,g)$, defined by 
\begin{equation}\label{eq:open-closed_map}
\mathcal{OC}(a;x) \coloneqq \widetilde{\mathcal{OC}}(a) \times_{\ev} \overline{W}^s_{-\nabla F}(x).
\end{equation}
This is a smooth manifold with corners of dimension 
\[\dim \mathcal{OC}(a;x) = \dim \widetilde{\mathcal{OC}}(a) - \codim_X \left(\overline{W}^s_{-\nabla F}(x)\right) = n + \mu(a) - \mu(x).\] 
The $R$-orientation is given by the intersection orientation:
\begin{align*}
    \check{\fo}(a) \otimes_R I_R^{\sO\sC}(a;x) \otimes_R \fo_{(F,g)}(x) \simeq \check{\fo}(a)\otimes_R I_R^{\widetilde{\sO\sC}}(a) \simeq (TL)_R \simeq R,
\end{align*}
where $\check{\fo}(a)$ is the rank one free $R$-module in \cref{lem:pre_oc_moduli_space}(ii).
\end{defn}
\begin{rem}\label{rem:oc_L}
    We can also define an open-closed flow bimodule $\sD \colon \sC\sW(L,L;R) \to \sM(f,g)$ similar to $\sO\sC$ in \cref{dfn:oc}, but using the Morse function $f\colon L \to \IR$ instead of $F$ and evaluation at a boundary point of the domain. It is clear by definition that $|\sO\sC|$ is homotopic to $ |j| \circ |\sD|$, where $j \colon \sM(f,g) \to \sM(F,g)$ is the Morse--Smale flow bimodule whose CJS realization $|j| \colon \varSigma^\infty_+ L \wedge R \to \varSigma^\infty_+ X \wedge R$ is the inclusion.
\end{rem}

\begin{defn}[$R$-fundamental class]
    \begin{enumerate}
        \item Given an $R$-oriented closed manifold $Z$ of dimension $d$, its \emph{$R$-fundamental class} is defined to be the class $[Z]_R \in H_d(Z;R) = [\IS,\varSigma^{\infty-d}_+Z \wedge R] = [R,\varSigma^{\infty-d}_+Z \wedge R]_{R}$ represented by the composition
        \[ R \longrightarrow Z^{-TZ} \wedge R \longrightarrow \varSigma^{\infty-d}_+Z \wedge R, \]
        where the first arrow is the unit of $Z^{-TZ} \wedge R$ and the second arrow is given by the Thom isomorphism and is determined by the $R$-orientation of $Z$ (see \cref{thm:ori_thom}(ii)).
        \item Let $X$ be a manifold and $Z \subset X$ be a closed $R$-oriented submanifold. The \emph{$R$-homology class represented by $Z$} is defined to be the image of $[Z]_R$ under the map $H_d(Z;R) \to H_d(X;R)$ induced by the inclusion $Z \hookrightarrow X$. Abusing the notation we continue to denote this class by $[Z]_R \in H_d(X;R)$. 
    \end{enumerate}
\end{defn}

Pick a strictly increasing sequence $\{A_k\}_{k=1}^\infty$ diverging to $\infty$. Similar to \cref{dfn:wrapped_floer_cohomology}, we also obtain an $R$-oriented degree $n$ flow bimodule
\[
\sO\sC^{k} \colon \sC\sW^{\leq A_k}(L,L;R) \longrightarrow \sM(F,g)
\]
whose CJS realization is a map $|\sO\sC^{k}| \colon HW^{\leq A_k}(L,L;R) \to \varSigma^{\infty-n}_+ X \wedge R$. It induces a map after passing to the homotopy colimit as $k\to \infty$ that we denote by
\[
|\sO\sC| \colon HW(L,L;R) \longrightarrow \varSigma^{\infty-n}_+ X \wedge R.
\]

\begin{lem}\label{lem:oc_fc}
The homotopy class of the composition
\[ R \overset{\eta_L}{\longrightarrow} HW(L,L;R) \overset{|\mathcal {OC}|}{\longrightarrow} \varSigma^{\infty-n}_+ X \wedge R\]
coincides with $[L]_R \in H_n(X;R)$.
\end{lem}

\begin{proof}
    By \cref{lem:PSS} there is an identification between the flow bimodules 
    
    \begin{center}
        \begin{minipage}{0.45\textwidth}
            \begin{align*}
                \osr_L \colon \sM_\ast &\longrightarrow \sC\sW(L,L;H_f,J_t) \\
                (p,a) &\longmapsto \osr_L(a)
            \end{align*}
        \end{minipage}
        and
        \begin{minipage}{0.45\textwidth}
            \begin{align*}
                \sN_f \colon \sM_\ast &\longrightarrow \sM^{-\bullet}(f,g) \\
                (p,x) &\longmapsto \overline W^u(x).
            \end{align*}
        \end{minipage}
    \end{center}
    
    As pointed out in \cref{rem:oc_stable}, there is also an identification between $\widetilde{\mathcal{OC}}$ and the $R$-oriented flow bimodule $\sM^{-\bullet}(f,g) \to \sM_\ast$, $(x,p) \mapsto \overline W^s(x)$ using the same idea as in the proof of \cref{lem:PSS}. It follows that there is an identification between the $R$-oriented flow bimodules $\mathcal{OC}$ and $\sI$ as defined in \cref{lem:morse_pd} (with $\varphi$ being the inclusion $L \hookrightarrow X$). We conclude that there is an identification between the following two compositions of $R$-oriented flow bimodules:
    \[
        \begin{tikzcd}[column sep=scriptsize, row sep=scriptsize]
            \sM_\ast \rar{\osr_L} & \sC\sW(L,L;H_f,J_t) \rar{\mathcal{OC}} & \sM(F,g)
        \end{tikzcd}
    \]
    and
    \[
        \begin{tikzcd}[column sep=scriptsize, row sep=scriptsize]
            \sM_\ast \rar{\sN_f} & \sM^{-\bullet}(f,g) \rar{\sI} & \sM(F,g)
        \end{tikzcd}.
    \]
    We pass to the action filtration on $\sC\sW(L,L;H_f;J_t)$, take CJS realizations, and consider the induced maps on the homotopy colimits over the action. For the latter composition we have
    \begin{enumerate}
        \item $|\sM^{-\bullet}(f,g),\fo^{-\bullet}_{(f,g)}| \simeq L^{-TL} \wedge R$ by \cref{lma:coh_morse_flow_thom_space}(i),
        \item $\eta_f \colon R \to L^{-TL} \wedge R$ is the unit map by \cref{lma:coh_morse_flow_thom_space}(ii), and
        \item $|\sI| \colon L^{-TL} \wedge R \to \varSigma^{\infty-n}_+L \wedge R \to \varSigma^{\infty-n}_+X \wedge R$ is the composition of the Thom isomorphism determined by the $R$-orientation on $L$, followed by the map of $R$-modules induced by the inclusion $L \hookrightarrow X$, see \cref{lma:pd_thom_class}.
    \end{enumerate}
\end{proof}

\begin{lem}\label{lem:oc_unital_wrt_mu2}
    Let $L, K \in \Ob(\mathcal W(X;R))$ be Lagrangian $R$-branes such that $L \cong K$ in $\sW(X;R)$. Let 
    \[ \osr^{\mathrm{sw}} \colon \sC\sW(L,K), \sC\sW(K,L) \longrightarrow \sC\sW(K,K)\]
    denote the $R$-oriented flow bimodule defined by $\osr^{\mathrm{sw}}(a,b;c) \coloneqq \osr(b,a;c)$, where $\osr$ is as in \cref{defn:triangle_product}.  Given an $R$-orientation on $L$, there exists an $R$-oriented flow bordism (see \cref{def:multi_bord}) between the two compositions in the following diagram:
    \[
        \begin{tikzcd}[row sep=scriptsize, column sep=scriptsize]
            \sC\sW(L,K), \sC\sW(K,L) \rar{\osr} \dar{\osr^{\mathrm{sw}}} & \sC\sW(L,L) \dar{\mathcal{OC}_L} \\
            \sC\sW(K,K) \rar{\mathcal{OC}_K^{\tw}} & \sM(F,g)
        \end{tikzcd},
    \]
    where $\OC_K^{tw}$ is an $R$-oriented flow bimodule which coincides with $\OC_K$ when forgetting the $R$-orientation.
\end{lem}

\begin{proof}
    Let $S_r$ denote $\ID^2 \smallsetminus \{\pm 1\} \subset \IC$ equipped with an interior marked point $p_r$ located at $ir \in \IC$ for $r\in (-1,1)$. Fix a Floer datum on $S_r$ with Hamiltonian $H_{S_r} \colon S_r \to \sH(\overline X,F)$, almost complex structure $J_{S_r} \colon S_r \to \sJ(\overline X,F)$, and positive strip-like ends $\varepsilon^{\pm 1}$ near $\pm 1 \in S_r$, respectively, that varies smoothly with $r$, see \cref{dfn:Floer_datum}. Denote the upper and lower boundary arcs of $S_r$ by $(\partial S_r)_{\pm 1}$. We define $S_1$ to be the nodal disk with components $\ID^2 \smallsetminus \{\pm 1, i\}$ and $\ID^2 \smallsetminus \{-1\}$ that are glued along $i$ in the first disk, and $-1$ in the second disk. The first disk is equipped with a negative strip-like end near $i$, and the second disk is equipped with a positive strip-like end near $-1$. We define $S_{-1}$ similarly, except that the first disk now has a boundary puncture at $-i$ (equipped with a negative strip-like end), instead of $i$. As $r \to \pm 1$, we assume that the domain deforms to look like a connected sum, similar to how the domains of the maps in the moduli space defined in \cref{dfn:conn_sum_dom_unit} are constructed. By abuse of notation, we keep denoting this deformed domain by $S_r$ for $r\in [-1,1]$, and the marked point is denoted by $p_r$. For $a \in \Ob(\sC\sW(L,K))$ and $b \in \Ob(\sC\sW(K,L))$, let $\widetilde \sB(a,b)$ denote the Gromov compactification (which is applicable by \cite[Lemma 3.9]{sylvan2019on}) of the moduli space of maps 
    \[ \left\{ \begin{matrix} r \in [-1,1] ,\\ u \colon S_r \to \overline{X} \end{matrix} \, \left| \, \begin{matrix} u \mbox{ solves Floer's equation} \\ u(z) \in \psi^{\rho_{S_r}(z)}L, \; z \in (\partial S_r)_{-1} \\ u(z) \in \psi^{\rho_{S_r}(z)}K, \; z \in (\partial S_r)_1  \\ \lim_{s \to \infty} u(\varepsilon^{1}(s,t)) = a(t) \\ \lim_{s \to \infty} u(\varepsilon^{-1}(s,t)) = b(t)  \end{matrix} \right. \right\}.\]
     Similar to the proof of \cref{lem:pre_oc_moduli_space}, we have gluing maps
    \[
    \sC^\#_+(a) \times \sC^\#_+(b) \times \widetilde{\sB}(a,b) \longrightarrow \sD^\#,
    \]
    where the data of $\sD^\#$ is based at $a(0)$, which are covered by maps of index bundles. The $R$-line bundle associated to the index bundle of $\sD^\#$ is canonically trivialized and is identified with the constant line bundle with fiber the index bundle over the constant element. It follows that we have isomorphisms of $R$-line bundles
    \[
    (\sV_+(a))_R \otimes_R (\sV_+(b))_R \otimes_R I_R^{\widetilde{\sB}}(a,b) \simeq (T_{a(0)}L)_R
    \]
    and
    \begin{equation}\label{eq:bordism_R_ori}
        (\sV_+(a))_R \otimes_R I_R(a;a') \otimes_R (\sV_+(b))_R \otimes_R I_R(b;b') \otimes_R I_R^{\widetilde{\sB}}(a',b') \simeq (T_{a'(0)}L)_R.
    \end{equation}
    By the $R$-orientability of $L$, this defines an $R$-orientation on $\widetilde{\sB}(a,b)$. Consider the evaluation map
    \begin{align*}
        \ev \colon \widetilde \sB(a,b) &\longrightarrow X \\
        (r,u) &\longmapsto u(p_r),
    \end{align*}
    where $p_r$ is the marked point in the domain. Let $x \in \Ob(\sM(F,g))$ and define
    \[
    \sB(a,b;x) \coloneqq \widetilde \sB(a,b) \times_{\ev} \overline{W}^s(x).
    \]
    For a generic choice of $H_{S_r} \colon S_r \to \sH(\overline X,F)$ and $J_{S_r} \colon S_r \to \sJ(\overline X,F)$ we have that $\sB(a,b;x)$ is a smooth manifold with corners of dimension $\mu(a)+\mu(b)-\mu(x)+1$. We note that there are maps 
    \begin{align*}
        \sC\sW(L,K)(a;a') \times \sB(a',b;x) &\longrightarrow \sB(a,b;x) \\
        \sC\sW(K,L)(b;b') \times \sB(a,b';x) &\longrightarrow \sB(a,b;x) \\
        \sB(a,b;x') \times \sM(F,g)(x';x) &\longrightarrow \sB(a,b;x) \\
        (\mathcal{OC}_L \circ \osr)(a,b; x) &\longrightarrow \sB(a,b;x) \\
        (\mathcal{OC}_K \circ \osr^{\mathrm{sw}})(a,b; x) &\longrightarrow \sB(a,b;x),
    \end{align*}
    as unoriented flow bimodules. Let $D \coloneqq \mu(a)+\mu(b)-\mu(x)+1$ and set
    \begin{align*}
    \partial^\circ_i\sB(a,b;x) \coloneqq &\bigsqcup_{\substack{a' \in \Ob(\sC\sW(L,K)) \\ \mu(a)-\mu(a')=i}} (\sC\sW(L,K)(a;a') \times \sB(a',b;x)) \\
    &\quad \sqcup \bigsqcup_{\substack{b' \in \Ob(\sC\sW(K,L)) \\ \mu(b)-\mu(b')=i}} (\sC\sW(K,L)(b;b') \times \sB(a,b';x)) \\
    &\quad \sqcup \bigsqcup_{\substack{x' \in \Ob(\sM(F,g)) \\ \mu(a)+\mu(b)-\mu(x')+2=i}} (\sB(a,b;x') \times \sM(F,g)(x';x)).
    \end{align*}
    Defining
    \[
        \partial_i \sB(a,b;x) \coloneqq \begin{cases}
            (\mathcal{OC}_L \circ \osr)(a,b; x)\sqcup (\mathcal{OC}_K \circ \osr^{\mathrm{sw}})(a,b; x), & i = 1 \\
            \partial^\circ_{i-1} \sB(a,b;x), & i > 1
        \end{cases},
    \]
    endows $\sB(a,b;x)$ with the structure of a $\ang{D+1}$-manifold. Thus we have that the assignment
    \begin{align*}
        \sB \colon \sC\sW(L,K),\sC\sW(K,L) &\longrightarrow \sM(F,g) \\
        (a,b;x) &\longmapsto \sB(a,b;x),
    \end{align*}
    is a flow bordism $\mathcal {OC}_L \circ \osr \Rightarrow \mathcal {OC}_K \circ \osr^{\mathrm{sw}}$.
    
    Finally, we discuss $R$-orientations.  Using the intersection orientation (cf. \eqref{eq:bordism_R_ori}), we have isomorphisms of $R$-line bundles
    \begin{equation}\label{eq:OC_bordism_ori}
        \check{\fo}(a) \otimes_R \check{\fo}(b) \otimes_R I_R^{\sB}(a,b;x) \otimes_R \fo_{(F,g)}(x) \simeq (T_{a(0)}L)_R,
    \end{equation}
    and along a face of the boundary stratum $(\sO\sC_L \circ \osr)(a,b;x) \hookrightarrow \sB(a,b;x)$ this isomorphism restrict to the following composition
    \begin{align*}
        &\check{\fo}(a) \otimes_R \check{\fo}(b) \otimes_R I_R^{\osr}(a,b;a') \otimes_R I_R^{\sO\sC}(a';x) \otimes_R \fo_{(F,g)}(x) \\
        &\quad \simeq \check{\fo}(a') \otimes_R I_R^{\sO\sC}(a';x) \otimes_R \fo_{(F,g)}(x) \simeq (T_{a(0)}L)_R \simeq (T_{a'(0)}L)_R \simeq R,
    \end{align*}
    which is the composition orientation (see \cref{dfn:composition_flow_bimods}).  Compatibilities along the boundary strata of $\sB(a,b;x)$ of the form $\sC\sW(L,K)(a;a') \times \sB(a',b;x)$ and $\sC\sW(L,K)(b;b') \times \sB(a,b';x)$ follows from the $R$-orientability of $L$.  We equip the boundary stratum $(\sO\sC_K \circ \osr^{\mathrm{sw}})(a,b;x)$ with the $R$-orientation coming from the restriction of the isomorphism \eqref{eq:OC_bordism_ori} and the given $R$-orientation on $\osr^{\mathrm{sw}}$ via \cref{lem:moduli_product}(ii). Next, aiming to apply \cref{lem:comp_two_out_of_three}, we need to show that for any $x \in \Ob(\sC\sW(K,K))$, there exists $a\in \Ob(\sC\sW(L,K))$ and $b\in \Ob(\sC\sW(K,L))$ such that $\osr^{\mathrm{sw}}(a,b;x) \neq \varnothing$. Choose a Hamiltonian $H_f$ associated to a Morse--Smale function on $K$ with a unique global maximum $M$. Via the identification in \cref{lem:PSS}(iii), there is some object $x_{\mathrm{max}} \in \Ob(\sC\sW(K,K))$ corresponding to $M$, and it suffices to show $\osr^{\mathrm{sw}}(a,b;x_{\mathrm{max}}) \neq \varnothing$ for some $a\in \Ob(\sC\sW(L,K))$ and $b\in \Ob(\sC\sW(K,L))$. By the assumption that $L \cong K$ in $\sW(X;R)$ there exists $\alpha \colon R \to HW(L,K)$ and $\beta \colon R \to HW(K,L)$ such that $|\osr^{\mathrm{sw}}| \circ (\alpha \wedge_R \beta) \simeq \eta_K$, where $\eta_K$ is the unit in $HW(K,K)$, see \cref{dfn:unit_hw}. After taking smash products with $H\pi_0R$, and passing to homotopy groups (see \cref{prop:homology_wedge_ring} and \cref{lem:wfuk_discrete_coeff}) we have that the following diagram commutes
    \[
    \begin{tikzcd}[row sep=scriptsize,column sep=2cm]
        \pi_0R \rar{|\osr^{\mathrm{sw}}|_\bullet \circ(\alpha_\bullet\otimes \beta_\bullet)} \drar[out=315,in=180][swap]{(\eta_K)_\bullet} & HW^{-\bullet}(K,K;\pi_0R) \dar{\cong} \\
        {}& H^{-\bullet}(K;\pi_0R)
    \end{tikzcd}.
    \]
    If $\osr^{\mathrm{sw}}(a,b;x_{\mathrm{max}}) = \varnothing$ for all $a\in \Ob(\sC\sW(L,K))$ and $b\in \Ob(\sC\sW(K,L))$, it would imply that the maximum $M \in H^{-\bullet}(K;\pi_0R)$ is not a generator, which is a contradiction. We conclude $\osr^{\mathrm{sw}}(a,b;x_{\mathrm{max}}) \neq \varnothing$ for some $a\in \Ob(\sC\sW(L,K))$ and $b\in \Ob(\sC\sW(K,L))$, and hence it follows that for any $x\in \Ob(\sC\sW(K,K))$ we have $\osr^{\mathrm{sw}}(a,b;x) \neq \varnothing$ for some $a\in \Ob(\sC\sW(L,K))$ and $b\in \Ob(\sC\sW(K,L))$.
    
    Therefore \cref{lem:comp_two_out_of_three} can be applied to obtain an induced $R$-orientation on the flow bimodule $\sO\sC_K$, which we denote by $\sO\sC^{\mathrm{tw}}_K$. Lastly, by \cref{lem:restr_last_strata_ori} it follows that $\sB$ defines an $R$-oriented flow bordism $\sO\sC_L \circ \osr \Rightarrow \sO\sC^{\mathrm{tw}}_K\circ \osr^{\mathrm{sw}}$.
\end{proof}

\subsection{Equivalences of Lagrangians under open-closed maps}
In this subsection, we prove the following result:
\begin{thm}\label{thm:R-ori_lags_represent_same_class}
    Let $L$ and $K$ be closed Lagrangian $R$-branes such that $L \cong K$ in $\sW(X;R)$.
    \begin{enumerate}
        \item If $L$ is $R$-orientable, then $K$ is $R$-orientable.
        \item Given an $R$-orientation on $L$, there exists an $R$-orientation on $K$ such that $[L]_R = [K]_R \in H_n(X;R)$.
    \end{enumerate}
\end{thm}

Before giving the proof of \cref{thm:R-ori_lags_represent_same_class}, we first recall some facts about $R$-orientations on closed manifolds.

\begin{notn}
Let $M$ be a closed $n$-dimensional manifold.  For $p \in M$, define the composition
\[ \varepsilon_p \colon H_n(M;R) \longrightarrow H_n(M,M \smallsetminus \{p\};R) \longrightarrow \widetilde{H}_n(S^n;R) \longrightarrow H_0(\mathrm{pt};R) \cong \pi_0(R) \eqqcolon k. \]
\end{notn}

Recall that a choice of $R$-orientation of $TM$ is equivalent to a choice of element $[M]_R \in H_n(M;R)$ such that $\varepsilon_p([M]_R) \in k$ is a unit for all $p \in M$, see \cite[Proposition V.2.2]{rudyak1998manifoldOrientations}.  Using the Hurewicz map (see \eqref{eq:hurewicz_map}), we have the following equivalent formulation of orientability for closed manifolds.

\begin{lem}\label{lem:fundamental_class_detection}
Consider the connective cover $p \colon R_{\geq 0} \to R$ and let $\Hw \colon R_{\geq 0} \to Hk$ denote the Hurewicz map.  Suppose that there exists a class $[M]_{R_{\geq 0}} \in H_n(M;R_{\geq 0})$ such that $\Hw([M]_{R_{\geq 0}}) \in H_n(M;Hk)$ is a unit, then $[M]_R 
 \coloneqq p_\bullet([M]_{R_{\geq 0}}) \in H_n(M;R)$ determines an orientation of $TM$.
\end{lem}

\begin{proof}
Since $p_\bullet$ preserves units, it suffices to show that $[M]_{R_{\geq 0}}$ is a unit.  To see this, notice that by the naturality of the Hurewicz map and the construction of $\varepsilon_p$, we have a commutative diagram:
\[
   \begin{tikzcd}[row sep=scriptsize, column sep=2.5mm]
        H_n(M;R_{\geq 0}) \rar \dar & H_n(M,M \smallsetminus \{p\};R_{\geq 0}) \rar \dar & \widetilde{H}_n(S^n;R_{\geq 0}) \rar \dar & H_0(\mathrm{pt};R_{\geq 0}) \rar \dar & \pi_0(R_{\geq 0}) \dar{\cong} \\
       H_n(M;Hk) \rar & H_n(M,M \smallsetminus \{p\};Hk) \rar & \widetilde{H}_n(S^n;Hk) \rar & H_0(\mathrm{pt};Hk) \rar & \pi_0(Hk). \\
    \end{tikzcd}
\]
The right most vertical arrow is an isomorphism by construction.  Since $\Hw([M]_{R_{\geq 0}})$ is a unit, we have that the bottom horizontal composition is an isomorphism.  From this and the above discussion, the lemma follows.
\end{proof}

\begin{proof}[Proof of \cref{thm:R-ori_lags_represent_same_class}]
Using \cref{lem:oc_unital_wrt_mu2} and \cref{lma:bordism_gives_homotopy_of_maps}, it follows that $|\mathcal{OC}_L| \circ \eta_L$ is homotopic to $|\mathcal{OC}^{\tw}_K| \circ \eta_K$.  By \cref{lem:oc_fc}, $|\mathcal{OC}_L| \circ \eta_L$ coincides with $[L]_R \in H_n(X;R)$.  We will show that $|\OC_K^{\tw}| \circ \eta_K$ coincides with $[K]_R \in H_n(X;R)$ for some choice of $R$-orientation on $K$.

As described in \cref{rem:oc_L} we have $|\sO\sC^{\tw}_K| = |j| \circ |\sD^\tw_K|$, where $\sD_K^{\tw}$ is the open-closed flow bimodule defined using a Morse function on $K$ and $j \colon \sM(f_K,g) \to \sM(F_K,g)$ is the Morse--Smale flow bimodule representing the inclusion. Therefore we have 
\[
|\OC_K^{\tw}| \circ \eta_K = |j| \circ |\sD^{\tw}_K| \circ \eta_K
\]
and thus it suffices to show that $|\sD^{\tw}_K| \circ \eta_K \colon R \to \varSigma^\infty_+K \wedge R$ represents an $R$-fundamental class.

Since $L$ is isomorphic to $K$ in $\sW(X;R)$ and since $L$ and $K$ are compact, by \cref{cor:lifting_equivalences_to_conn_covers} and \cref{rem:finite_mor_space_for_compact}, $L$ is isomorphic to $K$ in $\sW(X; R_{\geq 0})$.  Moreover, this isomorphism is natural with respect to the map $R_{\geq 0} \to R$.  As in \cref{lem:oc_unital_wrt_mu2}, the equivalence in $\sW(X;R_{\geq 0})$ and the choice of $R$-orientation of $L$ yields $R_{\geq 0}$-oriented flow bimodules ${\OC_K^{\tw}}'$ and ${\sD_K^{\tw}}'$ that agree with $\OC_K^{\tw}$ and $\sD_K^{\tw}$ as unoriented flow bimodules.  Moreover, these choices of $R_{\geq 0}$-orientations are compatible with the $R$-orientations on $\OC_K^{\tw}$ and $\sD_K^{\tw}$ in the sense that the induced maps
\begin{align*}
    |\OC_K^{\tw}| \colon HW(K,K;R) &\longrightarrow \varSigma_+^{\infty-n}X\wedge R \\
    |\sD_K^{\tw}| \colon HW(K,K;R) &\longrightarrow \varSigma_+^{\infty-n}K\wedge R
\end{align*}
are induced from the maps
\begin{align*}
    |{\OC_K^{\tw}}'| \colon HW(K,K;R_{\geq 0}) &\longrightarrow \varSigma_+^{\infty-n}X\wedge R_{\geq 0}\\
    |{\sD_K^{\tw}}'| \colon HW(K,K;R_{\geq 0}) &\longrightarrow \varSigma_+^{\infty-n}K\wedge R_{\geq 0}
\end{align*}
by applying the functor $- \wedge_{R_{\geq 0}}R$, where the map $R_{\geq 0} \to R$ realizes $R$ as an $R_{\geq 0}$-module.

So to show that $|\sD^{\tw}_K| \circ \eta_K$ represents an $R$-fundamental class in $H_n(K;R)$, it suffices to show that $|{\sD^{\tw}_K}'| \circ \eta_K$ represents an $R$-fundamental class in $H_n(K;R_{\geq 0})$.  By \cref{lem:fundamental_class_detection}, it further suffices to show that $|{\sD^{\tw}_K}'| \circ \eta_K$ (with ${\sD^{\tw}_K}'$ and $\eta_K$ viewed as $Hk$-oriented flow bimodules) represents a unit in $H_n(K;Hk)$. Without loss of generality we assume that the Morse--Smale function data for $K$ is $C^2$-small so that there is an identification of the $Hk$-oriented flow categories $\sC\sW(K,K;Hk)$ and $\sM^{-\bullet}(f,g)$, by \cref{lem:PSS}.  Moreover, the flow bimodule ${\sD_K^{\mathrm{tw}}}'$ (with the induced $Hk$-orientation) corresponds to an $Hk$-oriented flow bimodule
\[
\sI^{\mathrm{tw}} \colon \sM^{-\bullet}(f,g) \longrightarrow \sM(f,g),
\]
whose CJS realization corresponds (via the quasi-equivalence $\mod{Hk} \simeq \mod k$) to a morphism $\sI_\bullet^{\mathrm{tw}} \colon H^0(K;k) \to H_n(K;k)$, see \cref{lem:morse_pd}. By connectedness and compactness of $K$, we may assume that the Morse function $f$ has a single global minimum $x_{\text{min}}$ and a single global maximum $x_{\text{max}}$. We notice that the only $0$-dimensional piece of $\sI^{\tw}$ is given by $\sI^{\mathrm{tw}}(x_{\text{min}},x_{\text{max}}) = \{x_{\text{max}}\}$.  Hence $\sI_\bullet^{\mathrm{tw}}$ maps the generator $x_{\mathrm{min}} \in H^0(K;k)$ to the generator $x_{\mathrm{max}} \in H_n(K;k)$, and the $Hk$-orientation on $\sI^{\mathrm{tw}}$ determines the sign of this map. It follows that $|{\sD^{\tw}_K}'| \circ \eta_K$ represents a unit in $H_n(K;Hk)$, yielding the result.
\end{proof}

\section{The cotangent fiber and the based loops of the zero section}\label{sec:loops}
    The purpose of this section is to prove an equivalence of $R$-algebras $HW(F,F;R) \simeq \varSigma^\infty_+ \varOmega Q \wedge R$.  We do this by relating the flow category $\sC\sW(F,F;R)$ with the Morse--Smale flow category encoding the stable homotopy type of the based loop space following the standard approach of \cite{milnor1963morse} through finite dimensional approximations.
    \subsection{Morse theory of based loops}
    \begin{notn}
    Throughout this section, we fix the following:
    \begin{enumerate}
    \item Let $R$ be a connective commutative ring spectrum.
    \item Let $Q$ be a closed manifold.
    \item Let $F \coloneqq T^\ast_\xi Q \subset T^\ast Q$ be a fixed cotangent fiber.
    \item Let $\mu$ denote the Maslov index, see \cite{robbin1993maslov}.
    \end{enumerate}
    \end{notn}
    Recall from \cref{exmp:cotpol} that we equip $T^*Q$ with the tautological polarization corresponding to the Lagrangian distribution given by the tangent space of the cotangent fibers. Moreover, by \cref{rem:brane}(iii), $F$ admits a canonical $R$-brane structure. Thus the flow category $\sC\sW(F,F;R)$ admits an $R$-orientation.
    
    Let
    \begin{align*}
        \varOmega Q &\coloneqq \left\{\gamma \colon [0,1] \to Q \mid \text{$\gamma$ piecewise smooth and }\gamma(0) = \gamma(1) = \xi \right\}
    \end{align*}
    For $c > 0$, define
    \[
        \varOmega Q^c \coloneqq \left\{\gamma \in \varOmega Q \mid E(\gamma) \leq c\right\},
    \]
    where $E(\gamma) \coloneqq \int_0^1 | \dot{\gamma}|^2 dt$ is the energy functional. 
 
    \begin{defn}
    We call a finite sequence $\un t = (t_0,\ldots,t_k)$ such that $0 = t_0 < t_1 < \cdots < t_{k-1} < 1 = t_k$ a \emph{subdivision of $[0,1]$}. We denote the length of the sequence $\un t$ by $|\un t|$. We say that $\un t'$ is a \emph{refinement of $\un t$} if $\un t$ is a subsequence of $\un t'$.
    \end{defn}

    \begin{defn}[Broken geodesic] 
        For a subdivision $\un t$ of $[0,1]$, define $\varOmega_{\un t} Q$ to be the space of piecewise geodesics broken at $\un t$, that is, $\gamma \in \varOmega Q$ such that $\gamma|_{[t_i,t_{i+1}]}$ is a smooth geodesic for each $i\in \left\{0,\ldots,k-1\right\}$. For $c > 0$, define 
        \[ \varOmega_{\un t} Q^c = \varOmega Q^c \cap \varOmega_{\un t} Q.\]
    \end{defn}
    If $\un t'$ is a refinement of $\un t$, we have inclusions
        \[ \varOmega_{\un t} Q \hooklongrightarrow \varOmega_{\un t'} Q  \quad \text{and} \quad \varOmega_{\un t} Q^c \hooklongrightarrow \varOmega_{\un t'} Q^c.\]
    The following classical fact allows us to use broken geodesics to provide finite dimensional approximations of loop spaces:
    \begin{lem}\label{lma:fin_dim_approx}
        Let $c > 0$ be such that there exists no geodesic loop with energy $c$.  For any sufficiently fine subdivision $\un t$ of $[0,1]$, the following holds.
        \begin{itemize}
            \item $\varOmega_{\un t}Q^c$ is a deformation retract of $\varOmega Q^c$.
            \item $\varOmega_{\un t}Q^c$ is a smooth finite dimensional manifold that embeds as a codimension $0$ submanifold with boundary in $Q^{|\un{t}|-2}$.
            \item The restriction of the action functional $E$ to $\varOmega_{\un t} Q^c$ defines a Morse function whose critical points are the unbroken geodesic loops of energy less than $c$. Moreover, the Morse index coincides with the index of the Hessian of $E$ at the unbroken geodesic as defined in \cite[Section 13]{milnor1963morse}. 
        \end{itemize}
    \end{lem}
    \begin{proof}
        This follows from \cite[Lemma 16.1 and Theorem 16.2]{milnor1963morse}.
    \end{proof}

    \begin{notn}\label{notn:ColimData}
        \begin{enumerate}
           \item Pick an increasing sequence of positive real numbers $\{c_n\}_{n=1}^\infty$ such that $\lim_{n \to \infty }c_n = \infty$ and a sequence $\{\un t_n\}_{n=1}^\infty$ of refining subdivisions of $[0,1]$ such that
            \[ BQ(n) \coloneqq \varOmega_{\un t_n} Q^{c_n}\]
            satisfies \cref{lma:fin_dim_approx} for each $n\in \IZ_{\geq 1}$. We have inclusions $\iota_n \colon BQ(n) \hookrightarrow BQ(n+1)$ and $BQ(n)$ is a is a submanifold with boundary of $BQ(n+1)$. 
            \item For every $n \in \IZ_{\geq 1}$, choose a Riemannian metric $g_n$ on $BQ(n)$ so that $(E,g_n)$ is a Morse--Smale pair and such that $\iota_n^\ast g_{n+1} = g_n$. 
            \item Define
            \[ BQ = \bigcup_{n=1}^\infty BQ(n).\]
            It follows that $(E,g=\{g_n\})$ satisfies \cref{asmpt:msmseq}.\
            \item The inclusion $BQ \hookrightarrow \varOmega Q$ is a homotopy equivalence. Fix a homotopy inverse 
            \[
            r^{\varOmega Q} \colon \varOmega Q \longrightarrow BQ.
            \]
        \end{enumerate}
    \end{notn}

    \begin{notn}\label{notn:morse_flow_cat_loop_sp}
    We use the notation
    \[
        \sM_{\varOmega Q} \coloneqq \sM(E,g),
    \]
    Where $\sM(E,g)$ is the Morse--Smale flow category as defined in in \cref{lem:msmsep}. We denote the canonical $R$-orientation by $\fo_{\varOmega Q}$. Recall that it is given by the following:
    \begin{itemize}
        \item $\Ob(\sM_{\varOmega Q})$ is given by the set of smooth geodesic loops based at $\xi \in Q$.
        \item The morphism space $\sM_{\varOmega Q}(\gamma,\gamma')$ is given by the compactified moduli space of negative gradient trajectories of $E$ from $\gamma$ to $\gamma'$ inside $BQ$.
        \item $\fo_{\varOmega Q}$ is given by the $R$-line bundle associated to the tangent space of the unstable manifolds of the negative gradient flow.
    \end{itemize}
    \end{notn}
    
    \begin{defn}
    Define the \emph{Moore loop space} by
    \[\varOmega Q^{\mathrm{Moore}} \coloneqq \left\{\gamma \colon [0,\ell] \longrightarrow Q \mid \text{$\ell \in \IR_{>0}$, $\gamma$ piecewise smooth and }\gamma(0) = \gamma(\ell) = \xi \right\}. \]
    \end{defn}
    There is a natural inclusion $\varOmega Q \hookrightarrow \varOmega Q^{\mathrm{Moore}}$ that is a homotopy equivalence. We fix a homotopy inverse which we denote by $r^{\mathrm{Moore}}$.

    \begin{notn}
    Let 
    \begin{align*}
        P \colon \varOmega Q^{\mathrm{Moore}} \times \varOmega Q^{\mathrm{Moore}} &\longrightarrow \varOmega Q^{\mathrm{Moore}} \\
        (\gamma_1,\gamma_2) &\longmapsto \begin{cases}
            \gamma_1(t), & t \in [0,\ell_1] \\
            \gamma_2(t-\ell_1), & t\in [\ell_1, \ell_1+\ell_2]
        \end{cases},
    \end{align*}
    denote the Pontryagin product. It defines a strictly associative product on the Moore loop space. It is defined on $BQ$ and $\varOmega Q$ via homotopy inverses of the inclusions $BQ \hookrightarrow \varOmega Q \hookrightarrow \varOmega Q^{\mathrm{Moore}}$.
    \end{notn}
 \begin{defn}\label{defn:PontryaginFlowMultimodule}
    The \emph{Pontryagin flow multimodule}
    \[
        \sP \colon \sM_{\varOmega Q}, \sM_{\varOmega Q} \longrightarrow \sM_{\varOmega Q},
    \]
    is the $R$-oriented flow multimodule defined by the assignment
     \[ \sP(\gamma_1,\gamma_2;\gamma) \coloneqq (\overline W^u_{-\nabla E|_{BQ(n)}}(\gamma_1) \times \overline W^u_{-\nabla E|_{BQ(n)}}(\gamma_2)) \times_P \overline W^s_{-\nabla E|_{BQ(n)}}(\gamma),\]
     for some $n \gg 0$, where $\overline W^u$ and $\overline W^s$ are as in \cref{notn:morse_flowline_compact}. We equip $\sP$ with the canonical $R$-orientation, see \cref{dfn:morse_pontryagin}.
     \end{defn}

\begin{rem}
    We have implicitly chosen the metrics in \cref{notn:ColimData} generically so that the fiber product $\sP(\gamma_1,\gamma_2;\gamma)$ is a smooth manifold with corners for all choices of $\gamma_1, \gamma_2,$ and $\gamma$.
\end{rem}

\begin{rem}\label{rem:pontryagin_ring_str}
Notice that $\sP$ is a Morse-theoretic lift of the Pontryagin product $P_\bullet$ on $H_\bullet(\varOmega Q; \IZ)$.  Moreover it follows that the CJS realization of the Pontryagin flow multimodule in \cref{defn:PontryaginFlowMultimodule}
\[ |\sP| \colon \left( \varSigma^\infty_+ \varOmega Q \wedge R \right) \wedge_R \left( \varSigma^\infty_+ \varOmega Q \wedge R \right) \longrightarrow \left( \varSigma^\infty_+ \varOmega Q \wedge R \right) \]
induces the coincides with the Pontryagin product on $\varSigma^\infty_+ \varOmega Q \wedge R$ up to homotopy, see \cref{lem:msmsep} and \cref{lem:InducedProductMaps}.
\end{rem}

\subsection{From the cotangent fiber to the zero section}
Equip $T^\ast Q$ with the tautological polarization. This yields canonical choices of $R$-brane structures on the zero section $Q$ and the cotangent fiber $F$.

Let $S \coloneqq \ID^2 \smallsetminus \left\{\zeta_0,\zeta_1,\zeta_2\right\}$ be the unit disk in $\IC$ with its standard orientation and three boundary punctures enumerated along the boundary orientation. Denote by $(\partial S)_i$ the boundary segment between $\zeta_i$ and $\zeta_{i+1}$ for $i\in \left\{0,1,2\right\}$, where $\zeta_3 \coloneqq \zeta_0$. We equip $S$ with the choice of a Floer datum, with Hamiltonian $H_S \colon S \to \sH(T^\ast Q)$ and almost complex structure $J_S \colon S \to \sJ(T^\ast Q)$ such that $X_{H_S}$ vanishes in a neighborhood of $Q$, see \cref{dfn:Floer_datum}.

\begin{defn}
Let $a \in \Ob(\sC\sW(F,F;R))$, and denote by $\overline{\sM}(a)$ the Gromov compactification of the following moduli space of maps
\[
\sM(a) \coloneqq \left\{ u \colon S \to \overline X \left| \begin{matrix} u \mbox{ satisfies Floer's equation} \\ u(z) \in \psi^{\rho_S(z)}F, \;\; z \in (\partial S)_0 \\ u(z) \in \psi^{\rho_S(z)}F, \;\; z \in (\partial S)_1 \\ u(z) \in \psi^{\rho_S(z)}Q, \;\; z \in (\partial S)_2 \\ \lim_{s \to -\infty}u(\varepsilon_0^-(s,t)) = \xi \\ \lim_{s \to \infty}u(\varepsilon_1^+(s,t)) = \xi,
\\ \lim_{s \to \infty}u(\varepsilon_2^+(s,t)) = a(t)
\end{matrix} \right. \right\}. \]
Endow $\overline{\sM}$ with the Gromov topology. With this topology, it is a compact metric space.
\end{defn}

\begin{rem}\label{rmk:half-strip_before_eval}
Notice that $\overline{\sM}(a)$ is defined simply as $\osr(a,\xi;\xi)$ (see \cref{dfn:moduli_product}).  Thus by \cref{lem:moduli_product}, $\overline{\sM}(a)$ admits the structure of a $\ang{-\mu(a)}$-manifold.
\end{rem}
\begin{lem}\label{lem:grading_F_bbd_below}
    The grading function on the flow category $\sC\sW(F,F;R)$ is bounded below. In particular
    \[
    HW(F,F;R) = |\sC\sW(F,F;R);\fo_{(F,F)}|.
    \]
\end{lem}
\begin{proof}
    First, there is a one-to-one grading preserving correspondence between the sets $\Ob(\sC\sW(F,F;R))$ and $\Ob(\sM_{\varOmega Q})$ (see e.g.\@ \cite[Lemma 15]{asplund2021fiber}), which shows that the grading function on $\sC\sW(F,F;R)$ is bounded below. The rest follows from the definition of the CJS realization, see \cref{rem:Novikov}(i).
\end{proof}
Fix a Riemannian metric on $Q$. We define the evaluation map by
\begin{align}\label{eq:ev_loop}
    \ev^{\mathrm{Moore}} \colon \sM(a) &\longrightarrow \varOmega Q^{\mathrm{Moore}} \\
    u &\longmapsto u|_{(\partial S)_2} \nonumber
\end{align}
where $u|_{(\partial S)_2}$ is parameterized by arc-length via the Riemannian metric on $Q$. This evaluation map admits a continuous extension to all of $\overline{\sM}(a)$ (see \cite[Lemma 2.2]{abouzaid2012wrapped} and \cite[Lemma 7]{asplund2021fiber} for details).  Abusing the notation, we also denote this extension by $\ev^{\mathrm{Moore}}$. Finally, define
\begin{equation}\label{eq:ev_loop_geo}
    \ev \coloneqq r^{\varOmega Q} \circ r^{\mathrm{Moore}} \circ \ev^{\mathrm{Moore}} \colon \overline{\sM}(a) \longrightarrow BQ,
\end{equation}
where $r^{\varOmega Q}$ is the fixed homotopy inverse of the inclusion $BQ \hookrightarrow \varOmega Q$ defined in \cref{notn:ColimData}(iv).

\begin{defn}
For $\gamma \in \Ob(\sM_{\varOmega Q})$ and $a \in \Ob(\sC\sW(F,F;R))$ define
\[
    \sN(a;\gamma) \coloneqq \overline{\sM}(a) \times_{\ev} \overline W^s_{- \nabla E|_{BQ(n)}}(\gamma),
\]
for some $n \gg 0$, where $\overline W^s$ is as in \cref{notn:morse_flowline_compact}.
\end{defn}

\begin{rem}
Like for the Pontryagin multimodule we have implicitly made a generic choice of the metrics in \cref{notn:ColimData} so that $\sN(a;\gamma)$ is a smooth manifold with corners.
\end{rem}
\begin{lem}\label{lem:evbimod}
    The assignment $(a,\gamma) \mapsto \sN(a;\gamma)$ defines a flow bimodule
    \[
        \sN \colon \sC\sW(F,F;R) \longrightarrow \sM_{\varOmega Q}.
    \]
    Moreover, this flow bimodule admits an $R$-orientation. We refer to this bimodule as the \emph{evaluation flow bimodule}. 
\end{lem}
\begin{proof}
    Firstly, by \cite[Lemma 2.2]{abouzaid2012wrapped} and \cite[Lemma 2]{asplund2021fiber}, $\dim \overline{\sM}(a) = -\mu(a)$ and
    \[
        \dim \sN(a;\gamma) = -\mu(a) - \ind_{\mathrm{Morse}}(\gamma).
    \]
    For all $a,a' \in \Ob(\sC\sW(F,F))$ and $\gamma, \gamma' \in \Ob(\sM_{\varOmega Q})$, there are maps
    \begin{align*}
        \sC\sW(F,F)(a;a') \times \sN(a';\gamma) &\longrightarrow \sN(a;\gamma) \\
        \sN(a;\gamma') \times \sM_{\varOmega Q}(\gamma';\gamma) &\longrightarrow \sN(a;\gamma),
    \end{align*}
    that are diffeomorphisms onto faces of $\sN(a;\gamma)$ (see \cite[(4.3.35--36)]{abouzaid2015symplectic} for the analogous statement for closed strings) such that defining
    \begin{align*}
        \partial_i \sN(a;\gamma) \coloneqq &\bigsqcup_{\substack{a' \in \Ob(\sC\sW(F,F;R)) \\ \mu(a')-\mu(a)=i}}\sC\sW(F,F)(a;a') \times \sN(a';\gamma) \\
        &\qquad \sqcup \bigsqcup_{\substack{\gamma' \in \Ob(\sM_{\varOmega Q}) \\ -\mu(a)-\ind_{\mathrm{Morse}}(\gamma') + 1= i}} \sN(a;\gamma') \times \sM_{\varOmega Q}(\gamma';\gamma),
    \end{align*}
    endows $\sN(a;\gamma)$ with the structure of a $\ang{-\mu(a)-\ind_{\mathrm{Morse}}(\gamma) + 1}$-manifold. Recall from \cref{dfn:flow_cat_lag} that the grading function on $\sC\sW(F,F)$ is given by $-\mu$.

    We now discuss $R$-orientations. Consider the $R$-line bundles $\fo_{(F,F)}(a)$ and $\fo_{(F,Q)}(\xi)$ associated with $a$ and $\xi$ respectively. The $R$-orientation of $\osr(a,\xi;\xi)$ from \cref{lem:moduli_product} together with \cref{dfn:wrapped_floer_cohomology} induces the following isomorphism of $R$-line bundles over $\overline{\sM}(a)$:
    \[ \fo_{(F,Q)}(\xi) \otimes_R (T\overline{\sM}(a))_R \overset{\simeq}{\longrightarrow} \fo_{(F,F)}(a) \otimes_R \fo_{(F,Q)}(\xi).\]
    It follows that we have isomorphisms of $R$-line bundles over $\overline{\sM}(a)$
    \begin{equation}\label{eq:evlem1} 
        (T\overline{\sM}(a))_R \overset{\simeq}{\longrightarrow} \fo_{(F,F)}(a).
    \end{equation}
    These isomorphisms are moreover compatible with the action of $\overline{\sM}$ on $\overline{\sM}(a) = \osr(a,\xi;\xi)$. Now, since $\sN(a;\gamma) = \overline{\sM}(a) \times_{\ev} \overline W^s_{- \nabla E|_{BQ}}(\gamma)$, we have an isomorphism of $R$-line bundles over $\sN(a;\gamma)$
    \begin{equation}\label{eq:evlem2} 
     (T\sN(a;\gamma))_R \otimes_R (T\overline{W}^u_{- \nabla E|_{BQ}}(\gamma))_R \simeq (T\overline{\sM}(a))_R,
    \end{equation}
    
    As $\overline{W}^u(\gamma)$ is contractible, there is a canonical identification
    \begin{equation}\label{eq:evlem3} 
     (T\overline{W}^u_{- \nabla E|_{BQ}}(\gamma))_R \simeq (T_\gamma \overline{W}^u_{- \nabla E|_{BQ}}(\gamma))_R = \fo_{\varOmega Q}(\gamma).
    \end{equation}
    Combining \eqref{eq:evlem1} and \eqref{eq:evlem3}, we have isomorphisms of $R$-line bundles over $\sN(a;\gamma)$
    \[ \fo_{(F,F)}(a) \overset{\simeq}{\longrightarrow} \fo_{\varOmega Q} (\gamma) \otimes_R (T\sN(a;\gamma))_R \]
\end{proof}
By \cref{lem:grading_F_bbd_below}, the CJS realization of $\sN$ yields is map of $R$-modules
\[
|\sN| \colon HW(F,F;R) \longrightarrow \varSigma^\infty_+ \varOmega Q \wedge R.
\]
Recall from \cref{rem:pontryagin_ring_str} that $\varSigma^\infty_+ \varOmega Q \wedge R$ is an $R$-algebra. The goal of the remaining part of this section is to show that $|\sN|$ defines an equivalence of $R$-algebras.
\begin{lem}\label{lem:ring_struct}
    The following diagram is homotopy commutative
    \[
        \begin{tikzcd}[row sep=scriptsize, column sep=2cm]
            HW(F,F;R) \wedge_{R} HW(F,F;R) \dar{\mu^2} \rar{|\sN| \wedge_{R} |\sN|} & (\varSigma^\infty_+ \varOmega Q \wedge R) \wedge_{R} (\varSigma^\infty_+ \varOmega Q \wedge R) \dar{|\sP|} \\
            HW(F,F;R) \rar{|\sN|} & \varSigma^\infty_+ \varOmega Q \wedge R
        \end{tikzcd}
    \]
\end{lem}
\begin{proof}
    For $a_1,a_2 \in \Ob(\sC\sW(F,F;R))$, denote by $\overline{\sM}(a_1,a_2)$ the moduli space $\osr(a_1,a_2, \xi;\xi)$ as defined in \cref{dfn:moduli_assoc}.  Analogous to the construction of $\sN$, we define an $R$-oriented flow bordism $\sB$ via the assignment
    \begin{equation}\label{eq:ring_struct}
        \sB(a_1,a_2;\gamma) \coloneqq \overline{\sM}(a_1,a_2) \times_{\ev} \overline W^s_{- \nabla E|_{BQ}}(\gamma),
    \end{equation}
    where $\overline W^s$ is as in \cref{notn:morse_flowline_compact}. For all $a_1,a_1',a_2,a_2' \in \Ob(\sC\sW(F,F))$ and $\gamma,\gamma' \in \Ob(\sM_{\varOmega Q})$ there are maps
    \begin{align}
        \sC\sW(F,F)(a_1;a_1') \times \sB(a_1',a_2;\gamma) &\longrightarrow \sB(a_1,a_2;\gamma)\nonumber \\
        \sC\sW(F,F)(a_2;a_2') \times \sB(a_1,a_2';\gamma) &\longrightarrow \sB(a_1,a_2;\gamma)\nonumber \\
        \sB(a_1,a_2;\gamma') \times \sM_{\varOmega Q}(\gamma';\gamma) &\longrightarrow \sB(a_1,a_2;\gamma)\nonumber \\
        \label{eq:boundary_stratum_pontryagin}
        i \colon (\sP \circ (\sN, \sN))(a_1,a_2;\gamma) &\longrightarrow \sB(a_1,a_2;\gamma) \\
        (\sN \circ \osr)(a_1,a_2;\gamma) &\longrightarrow \sB(a_1,a_2;\gamma) \nonumber
    \end{align}
    that are diffeomorphisms onto faces of $\sB(a_1,a_2;\gamma)$ (see \cite[(4.4.30--31)]{abouzaid2015symplectic} for the analogous statement for closed strings), where $\sP \circ (\sN, \sN) \coloneqq (\sP \circ_1 \sN) \circ_2 \sN$ (see \cref{dfn:compos_multimodules}), and where $\osr$ is the triangle product on the flow category $\sC\sW(F,F)$, see \eqref{eq:multimodule_mu2}. Analogously to the formulas in \cref{def:multi_bord} these maps endow $\sB(a_1,a_2;\gamma)$ with the structure of a $\ang{-\mu(a_1)-\mu(a_2)-\ind_{\mathrm{Morse}}(\gamma)+2}$-manifold, showing that the assignment \eqref{eq:ring_struct} indeed defines a flow bordism
    \[
    \sB \colon \sN \circ \osr \Longrightarrow \sP \circ (\sN, \sN).
    \]

    The $R$-orientation on $\sB$ is defined as follows. The isomorphism in \cref{lem:moduli_associativity}(ii) induces an isomorphism of $R$-line bundles over $\overline{\sM}(a_1,a_2)$
    \[
    \fo_{(F,Q)}(\xi) \otimes_R (T \overline{\sM}(a_1,a_2))_R \overset{\simeq}{\longrightarrow} \fo_{(F,F)}(a_1) \otimes_R \fo_{(F,F)}(a_2) \otimes_R \fo_{(F,Q)}(\xi),
    \]
    and hence $T \overline{\sM}(a_1,a_2) \simeq \fo_{(F,F)}(a_1) \otimes_R \fo_{(F,F)}(a_2)$. Now we thus get isomorphisms
    \begin{equation}\label{eq:caps_to_floer}
        (T\sB(a_1,a_2;\gamma))_R \otimes_R (T \overline W^u_{-\nabla E|_{BQ}}(\gamma))_R \simeq (T\overline{\sM}(a_1,a_2))_R \simeq \fo_{(F,F)}(a_1) \otimes_R \fo_{(F,F)}(a_2).
    \end{equation}
    Where $(T\overline W^u_{-\nabla E|_{BQ}}(\gamma))_R = \fo_{\varOmega Q}(\gamma)$ by contractibility of $\overline W^u(\gamma)$ and the definition. There are maps
    \begin{align*}
        j \colon \osr(a_1,a_2;a) \times \overline{\sM}(a) &\longrightarrow \overline{\sM}(a_1,a_2)\\
        k \colon \overline{\sM}(a_1) \times \overline{\sM}(a_2) &\longrightarrow \overline{\sM}(a_1,a_2)
    \end{align*}
    that are diffeomorphisms onto faces of $\overline{\sM}(a_1,a_2)$. It follows from the compatibilities for orientations on moduli spaces $\osr(-,-;-)$ and $\osr(-;-)$ (see also \cref{lem:cap_gluing}) that the following diagrams are commutative
    \[
        \begin{tikzcd}[row sep=scriptsize, column sep=scriptsize]
            (j^\ast T \overline{\sM}(a_1,a_2))_R \dar \drar & {} \\
            (T \osr(a_1,a_2;a))_R \otimes_R (T \overline{\sM}(a))_R \rar & \fo_1(a_1) \otimes_R \fo_2(a_2)
        \end{tikzcd}
    \]
    \[
        \begin{tikzcd}[row sep=scriptsize, column sep=scriptsize]
            (k^\ast T \overline{\sM}(a_1,a_2))_R \dar \drar & {} \\
            (T \overline{\sM}(a_1))_R \otimes_R (T \overline{\sM}(a_2))_R \rar & \fo_1(a_1) \otimes_R \fo_2(a_2)
        \end{tikzcd}.
    \]
    We also have that the first isomorphism in \eqref{eq:caps_to_floer} restricts to the isomorphisms
    \[
    (T \osr(a_1,a_2;a))_R \otimes_R (T\sN(a;\gamma))_R \otimes_R \fo_{\varOmega Q}(\gamma) \simeq (T \osr(a_1,a_2;a))_R \otimes_R (T \overline{\sM}(a))_R
    \]
    and
    \begin{align*}
        &(T \sN(a_1;\gamma_1))_R \otimes_R (T \sN(a_2;\gamma_2))_R \otimes_R (T\sP(\gamma_1,\gamma_2;\gamma))_R \otimes_R \fo_{\varOmega Q}(\gamma)\\
        & \qquad \simeq (T \overline{\sM}(a_1))_R \otimes_R (T \overline{\sM}(a_2))_R
    \end{align*}
    Hence it follows that the $R$-orientation on $\sB(a_1,a_2;\gamma)$ restricts to the given $R$-orientations on the boundary strata that are diffeomorphic to $\sN\circ\osr$ and $\sP \circ (\sN,\sN)$, respectively. By \cref{lem:grading_F_bbd_below}, the CJS realization of $\sB$ yields a homotopy between the compositions $\mu^2\circ |\sN|$ and $(|\sN|\wedge_R |\sN|) \circ |\sP|$, see \cref{lma:functoriality_cjs,lma:bordism_gives_homotopy_of_maps}.
\end{proof}
\begin{lem}\label{lem:floer_loop_integer}
    Let $R = H\IZ$. The $R$-module map $|\sN| \colon HW(F,F;H\IZ) \to \varSigma^\infty_+ \varOmega Q \wedge H\IZ$ coincides with $\sF^1$ from \cite[Lemma 4.5]{abouzaid2012wrapped}, after passing to homotopy groups.
\end{lem}
\begin{proof}
    First, we recall from \cref{exmp:cotpol} the standard polarization on $T^\ast Q$ corresponds to the Lagrangian distribution given by the tangent spaces of the cotangent fibers. This determines the background class $b = \pi^\ast w_2(Q) \in H^2(T^\ast Q;\IZ/2)$, see \cref{sec:integer_coeffs}. By \cref{lem:wfuk_discrete_coeff}, the homotopy groups of $HW(F,F;R)$ coincide with $HW^{-\bullet}_b(F,F;\IZ)$ as defined in \cite[Section 3.3]{abouzaid2012wrapped}. The comparison of $H\IZ$-orientations of the flow bimodule $\sN$ and the orientations discussed in \cite[Section 6.1]{abouzaid2012wrapped} is similar to the discussion in \cref{lem:wfuk_discrete_coeff}.
\end{proof}
\begin{thm}\label{prop:evinv} $|\sN| \colon HW(F,F;R) \to \varSigma^\infty_+ \varOmega Q \wedge R$ is an equivalence of homotopy $R$-algebras.
\end{thm}
\begin{proof}
    By \cref{lem:grading_F_bbd_below} and the definition of the CJS realization (see \cref{dfn:cjs_realization}), it follows that $HW(F,F;R)$ is connective. The ring spectrum $\varSigma^\infty_+ \varOmega Q \wedge R$ is also connective. Let $k\coloneqq \pi_0 R$. It is sufficient to show that 
    \[ \pi_\bullet(|\sN| \wedge_R Hk) \colon \pi_\bullet(HW(F,F;R) \wedge_R Hk) \longrightarrow \pi_\bullet((\varSigma^\infty_+ \varOmega Q \wedge R) \wedge_R Hk)\]
    is an isomorphism of $k$-modules, since \cref{thm:whitehead} implies that $|\sN|$ is an equivalence of $R$-modules, and \cref{lem:ring_struct} shows that $|\sN|$ is an equivalence of homotopy $R$-algebras. 
    
    From \cref{prop:homology_wedge_ring}, it follows that
    \[ \pi_\bullet(HW(F,F;R) \wedge_R Hk) \cong HW^{-\bullet}(F,F;k)\]
    and
    \[ \pi_\bullet((\varSigma^\infty_+ \varOmega Q \wedge R) \wedge_R Hk) \cong H_\bullet(\varOmega Q;k) \]
    are computed by the homology of the Morse complexes of these flow categories and
    \[ \pi_\bullet(|\sN| \wedge_R Hk) = H_\bullet(\sN;k) \] 
    by the Morse chain map. By \cref{lem:floer_loop_integer} (see also \cref{rem:wfuk_discrete_coeff_k}), the associated Morse chain map of $\sN$ coincides with map $\mathcal F^1$ constructed in \cite[Lemma 4.5]{abouzaid2012wrapped} in homology, which is a homology isomorphism. This finishes the proof.
\end{proof}
We define an $R$-algebra structure on $HW(F,F;R)$ by pulling back the $R$-algebra structure on $\varSigma^\infty_+ \varOmega Q \wedge R$ via $|\sN|$ from \cref{prop:evinv}, which we by abuse of notation also denote by $HW(F,F;R)$. It is immediate that this $R$-algebra structure on $HW(F,F;R)$ is a lift of the given homotopy $R$-algebra structure and that $|\sN|$ is an equivalence of $R$-algebras.
\begin{cor}\label{cor:hw_loops_alg}
    The map $|\sN| \colon HW(F,F;R) \to \varSigma^\infty_+ \varOmega Q \wedge R$ is an equivalence of $R$-algebras.
    \qed
\end{cor}

\section{Structured lift of loop space modules}\label{sec:str_lift}
Using the wrapped Donaldson--Fukaya category discussed in \cref{sec:donaldson-fukaya} we obtain a Yoneda embedding
\[ \sY_F \colon \sW(T^*Q;R) \longrightarrow \homod{(\varSigma^\infty_+ \varOmega Q \wedge R)},\]
defined on objects as $L \mapsto HW(F,L)$. Via the triangle product $\mu^2$ (see \cref{defn:triangle_product}), this is a homotopy module over $HW(F,F)$, and hence a homotopy module over $\varSigma^\infty_+ \varOmega Q \wedge R$ by \cref{cor:hw_loops_alg}.
The aim of this section is to construct of lift of this functor to a functor
\[ \sY_{\varOmega Q} \colon \sW(T^*Q;R) \longrightarrow \Ho \mod{(\varSigma^\infty_+ \varOmega Q \wedge R)},\]
where the codomain is the homotopy category of (highly structured) modules over $\varSigma^\infty_+ \varOmega Q \wedge R$. 

We start by constructing a functor $\sY_{\varOmega Q}$ more generally for any closed Lagrangian $R$-brane $Q$ inside a stably polarized Liouville sector. It is constructed by equipping the relevant flow categories with local systems as defined in \cref{sec:flow_cat_local_system}. We then specialize to the case of the zero section in a cotangent bundle and verify that it indeed gives a lift of the functor $\sY_F$ above.

\subsection{Yoneda functor valued in local systems} 
\begin{notn}
    \begin{enumerate}
        \item Let $R$ be a commutative ring spectrum.
        \item Let $X$ be a stably polarized Liouville sector.
        \item Let $Q \subset X$ be a closed Lagrangian $R$-brane.
        \item Throughout this section we fix a basepoint $q_0 \in Q$ and let $F = T^\ast_{q_0}Q$.
        \item For any $q_1,q_2 \in Q$, we denote by $\sP_{q_1q_2}Q$ the space of Moore paths in $Q$ from $q_1$ to $q_2$.
    \end{enumerate} 
\end{notn}
\begin{thm}\label{thm:loc_sys_yoneda}
	There exists a functor
	\[ \sY_{\varOmega Q} \colon \sW(X;R) \longrightarrow \Ho \mod{(\varSigma^\infty_+ \varOmega Q \wedge R)}\]
	enriched over the homotopy category of $R$-modules.
\end{thm}
Construction of this functor occupies the rest of this subsection.
\subsubsection{Yoneda functor on objects}\label{subsec:obj_lift_gen}
Let $L$ be a Lagrangian $R$-brane inside $X$. We now construct a $(\varSigma^\infty_+ \varOmega Q \wedge R)$-module $\sY_{\varOmega Q}(L)$.

\begin{notn}
    For a Hamiltonian chord $x \in \Ob(\sC\sW(Q,L))$ we shall abuse the notation and sometimes write $x$ for the starting point $x(0) \in Q$.
\end{notn}

Similar to \eqref{eq:ev_loop} we have an evaluation map for any $x,y \in \Ob(\sC\sW(Q,L))$
\begin{align*}
    \ev \colon \sC\sW(Q,L)(x,y) &\longrightarrow \sP_{xy}Q \\
    u &\longmapsto u|_{\IR \times \{0\}}.
\end{align*}

\begin{defn}
    Let $L \subset X$ be any Lagrangian $R$-brane. Define the local system $\sE_{\varOmega Q} \colon \sC\sW(Q, L) \to \mod{(\varSigma^\infty_+ \varOmega Q \wedge R)}$ on objects by $\sE_{\varOmega Q}(x) \coloneqq \varSigma^\infty_+\sP_{q_0x}Q \wedge R$ and on morphisms by the composition
    \begin{align*}
        (\varSigma^\infty_+\sP_{q_0x}Q \wedge R) \wedge \varSigma^\infty_+\sC\sW(Q,L)(x,y) &\xrightarrow{\id \wedge \ev} (\varSigma^\infty_+\sP_{q_0x}Q \wedge R) \wedge \varSigma^\infty_+\sP_{xy}Q \\
        &\overset{P}{\longrightarrow} \varSigma^\infty_+\sP_{q_0y}Q\wedge R,
    \end{align*}
    where the second map is the concatenation of Moore paths. Let $\sC\sW(\varOmega Q,L) \coloneqq (\sC\sW(Q,L),\sE_{\varOmega Q})$ and define
    \[ HW(\varOmega Q,L) \coloneqq |\sC\sW(\varOmega Q,L),\fo_{Q,L}|.\]
\end{defn}

\subsubsection{Yoneda functor on morphisms}\label{subsec:mor_lift_gen}
    We now extend the assignment $L \mapsto HW(\varOmega Q,L)$ constructed in \cref{subsec:obj_lift_gen} to a functor
	\[ \sY_{\varOmega Q}\colon \sW(T^*Q;R) \longrightarrow \Ho\mod{(\varSigma^\infty_+ \varOmega Q \wedge R)}.\]
    Concretely, this amounts to constructing for any Lagrangian $R$-branes $L, K \subset T^*Q$, a map
    \[ \sY_{\varOmega Q} \colon HW(L,K) \longrightarrow \left(\Ho\mod{(\varSigma^\infty_+ \varOmega Q \wedge R)}\right) \left( HW(\varOmega Q,L) , HW(\varOmega Q, K) \right)\]
    Moreover, we need to verify that this assignment respects compositions up to homotopy.

    For any Lagrangian $R$-branes $L$ and $K$, define the $R$-oriented flow multimodule with local system 
    \[
    (\sR_{\varOmega Q},\fo_{\varOmega Q},\sF_{\varOmega Q}) \colon \sC\sW(\varOmega Q,L) , \sC\sW(L,K) \longrightarrow \sC\sW(\varOmega Q,K),
    \]
    completely analogously to the $R$-oriented flow multimodule with local system in \eqref{eq:twisted_multimod}. Analogously to \eqref{eq:action_filtered_twisted_mu2}, the flow multimodule $(\sR_{\varOmega Q},\fo_{\varOmega Q})$ restricts to an $R$-oriented flow multimodule with local system
    \[
    (\sR^{\ell}_{\varOmega Q},\fo^{\ell}_{\varOmega Q},\sF^{\ell}_{\varOmega Q}) \colon \sC\sW(\varOmega Q,L), \sC\sW^{\leq B_\ell}(L,K) \longrightarrow \sC\sW(\varOmega Q,K).
    \]
\begin{defn}\label{dfn:morphisms_twisted_yoneda}
    The CJS realizations of $(\sR^{\ell}_{\varOmega Q},\fo^{\ell}_{\varOmega Q},\sF^{\ell}_{\varOmega Q})$ induces a map of $(\varSigma^\infty_+\varOmega Q \wedge R)$-modules after passing to the homotopy colimit as $\ell \to \infty$ that we denote by
    \[ \mu^2_{\varOmega Q} \colon HW(\varOmega Q,L) \wedge_R HW(L,K) \longrightarrow HW(\varOmega Q, K).\]
\end{defn}
\begin{prop}\label{prop:mor_lift2}
    The following diagram commutes up to homotopy:
    \[
        \begin{tikzcd}
            HW(\varOmega Q,L) \wedge_R HW(L,K)  \wedge_R HW(K,M) \rar{\id \wedge_R \mu^2} \dar{\mu^2_{\varOmega Q} \wedge_R \id} & HW(\varOmega Q,K) \wedge_R HW(K,M) \dar{\mu^2_{\varOmega Q}}\\
            HW(\varOmega Q,K) \wedge_R HW(K,M) \rar{\mu^2_{\varOmega Q}} & HW(\varOmega Q,M).
        \end{tikzcd}
    \]
\end{prop}
\begin{proof}
    Similar to the proof of \cref{prop:mor_lift1} it suffices to pick strictly increasing sequences $\{B_\ell\}_{\ell=1}^\infty, \{C_m\}_{m=1}^\infty \subset \IR$ that diverge to $\infty$, and consider the following diagram on CJS realizations of the appropriately action filtered versions of the flow multimodules involved:
    \begin{equation}\label{eq:action_filtered_twisted_mu2_assoc}
        \begin{tikzcd}[column sep=large]
            \parbox{5cm}{$HW(\varOmega Q,L) \wedge_R HW^{\leq B_\ell}(L,K)\\
            \phantom{HW(\varOmega Q,L)}\wedge_R HW^{\leq C_m}(K,M)$} \rar{\id \wedge_R |\osr^{\ell,m}|} \dar{|\sR^{k,\ell}_{\varOmega Q}| \wedge_R \id} & HW(\varOmega Q,K) \wedge_R HW^{\leq B_\ell+C_m}(K,M) \dar{|\sR^{\ell+m}_{\varOmega Q}|}\\
            HW(\varOmega Q,K) \wedge_R HW^{\leq C_m}(K,M) \rar{|\sR^{m}_{\varOmega Q}|} & HW(\varOmega Q,M).
        \end{tikzcd}
    \end{equation}
    It suffices to prove that the above diagram is homotopy commutative, since passing to homotopy colimit as $\ell,m\to \infty$ finishes the proof.

    To that end, consider the $R$-oriented flow bordism $\sB$ defined by the assignment
    \[
    (p,q,r;s) \longmapsto \sR_3(p,q,r;s),
    \]
    where $\sR_3$ denotes the $R$-oriented flow bordism defined in \cref{dfn:moduli_assoc} and \cref{lem:moduli_associativity}(ii). It follows from \cref{lem:moduli_associativity} that $\sB$ defines an $R$-oriented flow bordism $\sR_{\varOmega Q} \circ_1 \sR_{\varOmega Q} \Rightarrow \sR_{\varOmega Q} \circ_2 \osr$. It therefore suffices to show that the local systems are respected. Equip $\sB$ with the local system $\sB \Rightarrow F_A(\sF_{\varOmega Q},\sE_M)$ specified by the composition of maps
    \[
    (\varSigma^\infty_+\sP_{q_0p}Q\wedge R) \wedge \varSigma^\infty_+\sB(p,q,r;s) \xrightarrow{\id \wedge \ev} (\varSigma^\infty_+\sP_{q_0p}Q \wedge R) \wedge \varSigma^\infty_+\sP_{ps}Q \overset{P}{\longrightarrow} \varSigma^\infty_+\sP_{q_0s}Q\wedge R
    \]
    where the first map is induced by the evaluation map $\sB(p,q,r;s) \to \sP_{ps}Q$, $u\mapsto u|_{(\partial S)_0}$ and the second map $P$ is concatenation of Moore paths. It is clear that this local system restricts to the corresponding local system of the two codimension $1$ boundary strata 
    \begin{align*}
        (\sR_{\varOmega Q} \circ_1 \sR_{\varOmega Q})(p,q,r;s) &\hooklongrightarrow \sB(p,q,r;s) \\
        (\sR_{\varOmega Q} \circ_2 \osr)(p,q,r;s) &\hooklongrightarrow \sB(p,q,r;s)
    \end{align*}
    The restriction of the $R$-oriented flow bordism with local system $\sB$ to the action filtered pieces yields an $R$-oriented flow bordism with local system
    \[
    \sR^{m}_{\varOmega Q} \circ_1 \sR^{k,\ell}_{\varOmega Q} \Longrightarrow \sR^{\ell+m}_{\varOmega Q} \circ_2 \osr^{\ell,m},
    \]
    and thus passing to CJS realizations \cref{lma:bordism_gives_homotopy_of_maps} shows that \eqref{eq:action_filtered_twisted_mu2_assoc} is homotopy commutative.
\end{proof}
This completes the proof of \cref{thm:loc_sys_yoneda}.

\subsection{Local systems corresponding to $R$-branes on the base Lagrangian}
In this subsection we analyze the restriction of the Yoneda functor $\sY_{\varOmega Q}$ from \cref{thm:loc_sys_yoneda} on the subcategory of $R$-branes supported on the Lagrangian $Q$. In the special case of the zero section in a cotangent bundle, this will be used in the proof of \cref{thm:intro_main} in \cref{sec:modules} below.

Let $Q$ be a closed Lagrangian $R$-brane, and let $Q'$ be any other $R$-brane structure on the Lagrangian submanifold underlying $Q$. The data of such an $R$-brane structure gives a map $\xi\colon Q_+ \to \BGL_1(R)$, which is equivalent to a ring map $\varOmega \xi \colon (\varOmega Q)_+ \to \GL_1(R)$. This in particular gives a $(\varSigma^\infty_+ \varOmega Q\wedge R)$-module structure on $R$ via the composition of ring maps
\begin{equation}\label{eq:loop_mod_struct_R}
    \varSigma^\infty_+ {\GL}_1(R) \longrightarrow \varSigma^\infty_+ \varOmega^\infty R \longrightarrow R,
\end{equation}
and we denote the resulting $(\varSigma^\infty_+ \varOmega Q\wedge R)$-module by $R_\xi$.

\begin{prop}\label{prop:brane_twisted_R_mod}
There exists a weak equivalence of $(\varSigma^\infty_+ \varOmega Q \wedge R)$-modules
\[ HW(\varOmega Q,Q') \simeq R_\xi.\]
\end{prop}
\begin{proof}
    Without loss of generality, we may pick an admissible Hamiltonian $H_f$ induced by a $C^2$-small Morse function $f \colon Q \to \IR$, cf.\@ \cref{notn:oc_morse,lem:PSS}. Using this Hamiltonian, we have an identification of $R$-oriented flow categories
    \[
        (\sC\sW(Q,Q),\fo_{Q,Q}) \cong (\sM^{-\bullet}(f,g),\fo_{(f,g)}^{-\bullet})
    \]
    by \cref{lem:PSS}(iii).

    Let $q \in \Ob(\sC\sW(Q,Q))$, and consider the spaces of abstract Floer strip caps $\sC^\#_{Q,Q}(q)$ and $\sC^\#_{Q,Q'}(q)$. There are compositions 
    \begin{align*}
        \sC^\#_{Q,Q}(q) &\longrightarrow \sD^\#(q^\ast \sE) \overset{\rho^\#}{\longrightarrow} \varOmega_{q}(\lag)^\# \\
        \sC^\#_{Q,Q'}(q) &\longrightarrow \sD^\#(q^\ast \sE) \overset{\rho^\#}{\longrightarrow} \sP_{q,q'}(\lag)^\#,
    \end{align*}
    where we abuse the notation and write $q,q' \in (\lag)^\#$ for the points $\sG^\#_Q(q)$ and $\sG^\#_{Q'}(q)$, respectively. By choosing a path $\gamma_q \in \sP_{q',q}(\lag)^\#$, concatenation with $\gamma_q$ yields a map $\sC^\#_{Q,Q'}(q) \to \sC^\#_{Q,Q}(q)$ by \cref{dfn:pullback_path}, and we obtain the following commutative diagram
    \begin{equation}\label{eq:index_bundles_path_compos}
        \begin{tikzcd}[sep=scriptsize]
            \sC^\#_{Q,Q'}(q) \rar{\rho^\#} \dar & \sP_{q,q'}(\lag)^\# \rar{\ind} \dar{P(-,\gamma_q)} & \BGL_1(R) \dar[equals] \\ 
            \sC^\#_{Q,Q}(q) \rar{\rho^\#} & \varOmega_q(\lag)^\# \rar{\ind} & \BGL_1(R)
        \end{tikzcd}.
    \end{equation}

    The map $\xi \colon Q_+ \to \BGL_1(R)$ is the induced map from the two choices of $R$-brane structure $\sG^\#_Q$ and $\sG^\#_{Q'}$, and $\xi_q$ denotes the pullback of $\xi$ under $\left\{q\right\} \hookrightarrow Q$. Recall by \cref{dfn:flow_cat_lag} that $-\fo_{Q,Q'}(q)$ is the rank one free $R$-module that is given by the pullback of the index bundle $\sV_{Q,Q'}(q) = \ind \circ \rho^\#$ under the inclusion of a point $\left\{q\right\} \hookrightarrow \sC^\#_{Q,Q'}(q)$. 

    We abuse the notation and denote the map $\left\{q\right\} \to \sP_{q',q}(\lag)^\#$, $q \mapsto \gamma_q$ by $\gamma_q$, and let $\iota_q$ denote the pullback of $\rho^\#$ under the inclusion $\left\{q\right\} \hookrightarrow \sC^\#_{Q,Q'}(q)$ above. Now we claim that we have a commutative diagram
    \[
        \begin{tikzcd}[sep=scriptsize]
            \left\{q\right\} \times \left\{q\right\}\dar{\iota_q \times \gamma_q} \ar[rr,"-\fo_{Q,Q'}(q) \times \xi_q"] && \BGL_1(R) \times \BGL_1(R) \dar \\
            \sP_{q,q'}(\lag)^\# \times \sP_{q',q}(\lag)^\# \rar{P} & \varOmega_q(\lag)^\# \rar{\ind} & \BGL_1(R)
        \end{tikzcd}.
    \]
    Combining this with \eqref{eq:index_bundles_path_compos} yields a commutative diagram
    \[
        \begin{tikzcd}[sep=scriptsize]
            \left\{q\right\} \times \left\{q\right\} \dar \ar[rrr,"-\fo_{Q,Q'}(q) \times \xi_q"] &&& \BGL_1(R) \dar[equals] \\
            \left\{q\right\} \rar & \sC^\#_{Q,Q}(q) \rar{\rho^\#} & \varOmega_q(\lag)^\# \rar{\ind} & \BGL_1(R)
        \end{tikzcd},
    \]
    where the lower horizontal composition is the rank one free $R$-module $\fo_{Q,Q}(q)$. This shows that we have an isomorphism of rank one free $R$-modules $\fo_{Q,Q'}(q) \cong \fo_{Q,Q}(q) \otimes_R \xi_{q}$ for any $q\in \Ob(\sC\sW(Q,Q))$. We conclude that we have an identification of $R$-oriented flow categories
    \[
        (\sC\sW(Q,Q'),\fo_{Q,Q'}) \cong (\sM^{-\bullet}(f,g),\fo_{\xi}^{-\bullet}).
    \]
    Finally, the local system $\sE_{\varOmega Q}$ on $\sC\sW(Q,Q)$ gets identified with the local system $\sP^{-\bullet}Q$ on $\sM^{-\bullet}(f,g)$ (see \eqref{eq:local_syst_cohomological_path}). Therefore we have an identification of $R$-oriented flow categories with local system
    \[
        (\sC\sW(\varOmega Q,Q), \fo_{Q,Q'}) \cong (\sM^{-\bullet}(f,g),\fo_{\xi}^{-\bullet},\sP^{-\bullet}Q).
    \]
    The result then follows from \cref{prop:morse_flow_path_coh}.
\end{proof}

\subsection{The case of cotangent bundles}
We now specialize to the case the zero section in a cotangent bundle. Our aim is to prove the following:
\begin{thm}\label{thm:yoneda_cot}
	Let $X=T^*Q$ and $L=Q$. Then the local system valued Yoneda functor
	\[ \sY_{\varOmega Q} \colon \sW(T^*Q;R) \longrightarrow \Ho \mod{(\varSigma^\infty_+ \varOmega Q \wedge R)}\]
	constructed in \cref{thm:loc_sys_yoneda} is a lift of the Yoneda functor 
	\[ \sY_F \colon \sW(T^*Q;R) \longrightarrow \homod{(\varSigma^\infty_+ \varOmega Q \wedge R)}.\]
	More explicitly, 
	\begin{enumerate}
		\item For any $L \in \Ob (\sW(T^*Q;R))$, there is an equivalence of homotopy $(\varSigma^\infty_+ \varOmega Q \wedge R)$-modules
			\[ HW(F,L) \overset{\simeq}{\longrightarrow} HW(\varOmega Q,L).\]
		\item For every pair of Lagrangian $R$-branes $L,K \subset T^*Q$, the following diagram commutes up to homotopy:
	    		\[
	    		\begin{tikzcd}
                    HW(F,L) \wedge_R HW(L,K) \rar{\mu^2} \dar{|\sN| \wedge_R \id} & HW(F,K) \dar{|\sN|}\\
                    HW(\varOmega Q,L) \wedge_R HW(L,K) \rar{\mu^2_{\varOmega Q}} & HW(\varOmega Q,K).
			    \end{tikzcd}
	    		\]
	\end{enumerate}
\end{thm}
We verify the conditions (i) and (ii) in \cref{thm:yoneda_cot} and \cref{prop:compos_bimod_struct} respectively.
\subsubsection{Lift at the level of objects}\label{subsec:obj_lift}
\begin{lem}\label{lem:loop_sp_local_sys}
    There is an equivalence of $R$-modules $HW(\varOmega Q,F) \simeq \varSigma^\infty_+ \varOmega Q \wedge R$.
\end{lem}
\begin{proof}
    Since $\sC\sW(Q,F)$ consists of a single object $q_0 \in Q \cap F$, we combine the observations in \cref{rem:trivial_local_sys} and \cref{exmp:cjspt} to get
    \begin{align*}
        HW(\varOmega Q,F) &= |\sC\sW(Q,F),\fo_{Q,F},\sE_{\varOmega Q}| \simeq |\sC\sW(Q,F),\fo_{Q,F}| \wedge_R \sE_{\varOmega Q}(q_0) \\
        &\simeq \sE_{\varOmega Q}(q_0) = \varSigma^\infty_+\varOmega Q \wedge R.
    \end{align*}
\end{proof}

Let $Q^\#$ be an $R$-brane structure supported on the zero section $Q \subset T^\ast Q$. The data of the $R$-brane structure gives a map $Q_+ \to \BGL_1(R)$, which is equivalent to a ring map $(\varOmega Q)_+ \to \GL_1(R)$. This gives a $(\varSigma^\infty_+ \varOmega Q\wedge R)$-module structure on $R$ via the composition of ring maps
 \[
 \varSigma^\infty_+ {\GL}_1(R) \longrightarrow \varSigma^\infty_+ \varOmega^\infty R \longrightarrow R.
 \]
\begin{prop}\label{prp:equiv_fiber_loop_space_local}
    There is an equivalence of $R$-modules 
    \[ HW(F,L) \simeq HW(\varOmega Q,L).\]
\end{prop}
\begin{proof}
    Consider the $R$-oriented flow category $\sC\sW(F,L)$ as being equipped with the trivial local system $\underline R$ (see \cref{rem:trivial_local_sys}). Consider the $R$-oriented flow bimodule 
    \begin{equation}\label{eq:bimod_fiber_loc}
        (\sN,\fm) \colon \sC\sW(F,L) \longrightarrow \sC\sW(\varOmega Q, L)
    \end{equation}
    defined by $\sN(p,q) \coloneqq \osr(q_0,p;q)$ (see \cref{dfn:moduli_product}), where $Q \cap F = \{q_0\}$. The $R$-orientation $\fm$ is obtained from \cref{lem:moduli_product}(ii). Equip $(\sN,\fm)$ with the local system
    \[
        \sE \colon \sN \Longrightarrow F_{\varSigma^\infty_+\varOmega Q \wedge R}(\underline R, \sE_{\varOmega Q}),
    \]
    induced by the map
    \begin{align*}
        \sE(p,q) \colon \sN(p,q) &\longrightarrow \sP_{q_0q} Q\\
        u &\longmapsto u|_{(\partial S)_0}.
    \end{align*}
    Pick a strictly increasing sequence $\{A_k\}_{k=1}^\infty \subset \IR$ diverging to $\infty$. The restriction of the local system $\underline R$ to the action filtered flow category $\sC\sW^{\leq A_k}(F,L)$ is denoted by $\underline R^{\leq A_k}$. Then $(\sN,\fm,\sE)$ restricts to an $R$-oriented flow bimodule with local system
    \[
    (\sN^{k},\fm^{k},\sE^{k}) \colon (\sC\sW^{\leq A_k}(F,L),\fo^{\leq A_k}_{(F,L)},\underline R^{\leq A_k}) \longrightarrow (\sC\sW(Q,L),\fo_{Q,L},\sE_{\varOmega Q}).
    \]
    Passing to the CJS realization yields a map
    \begin{equation}\label{eq:flow_bimod_loc}
    |\sN^k,\fm^k,\sE^k| \colon HW^{\leq A_k}(F,L) \longrightarrow HW(\varOmega Q,L).
    \end{equation}
    From \cref{cor:jmod_fun} the map of Morse chain complexes induced by $(\sN^k,\fm^k,\sE^k)$ coincides with the restriction of the map $\sF^1$ as defined in \cite[Lemma 4.5]{abouzaid2012wrapped} to the action filtered piece, and is thus a quasi-isomorphism. Using Whitehead's theorem \cref{thm:whitehead} it follows that \eqref{eq:flow_bimod_loc} is a weak equivalence. Passing to the homotopy colimit as $k\to \infty$ finishes the proof.
\end{proof}

Consider the $R$-oriented flow multimodule with local system
\begin{equation}\label{eq:twisted_multimod}
    \sR_{\varOmega Q} \colon \sC\sW(\varOmega Q,F),\sC\sW(F,L) \longrightarrow \sC\sW(\varOmega Q,L)
\end{equation}
whose underlying flow multimodule is $\osr$ that is defined in \cref{dfn:moduli_product}, with $R$-orientation $\fo_{\varOmega Q}$ obtained from \cref{lem:moduli_product}(ii). The local system is defined for every $p \in \Ob(\sC\sW(F,L))$ and $q\in \Ob(\sC\sW(Q,L))$ by the map
\[
\sF_{\varOmega Q}(q_0,p;q) \colon (\varSigma^\infty_+\varOmega_{q_0}Q\wedge R) \wedge \varSigma^\infty_+\sR_{\varOmega Q}(q_0,p;q) \longrightarrow (\varSigma^\infty_+ \sP_{q_0q}Q \wedge R)
\]
induced by the evaluation map $\sR_{\varOmega Q}(q_0,p;q) \to \sP_{q_0q}Q$, $u \mapsto u|_{(\partial S)_0}$ and concatenation of Moore paths.

Moreover, let $(\sN,\fm,\sE)$ be the $R$-oriented flow bimodule with local system defined in the proof of \cref{prp:equiv_fiber_loop_space_local}. Let $\sN_F \colon \sC\sW(F,L) \to \sC\sW(F,L)$ be defined similarly to $\sN$, with the Lagrangian $R$-brane $F$ instead of $L$.
\begin{prop}\label{prop:compos_bimod_struct}
    The following diagram is homotopy commutative
    \[
    \begin{tikzcd}[row sep=scriptsize, column sep=1cm]
        HW(F,F) \wedge_R HW(F,L) \rar{\mu^2} \dar{|\sN_F| \wedge_R \id} & HW(F,L) \dar{|\sN|} \\
        HW(\varOmega Q,F) \wedge_R HW(F,L) \rar{|\sR_{\varOmega Q}|} \dar{\varphi \wedge_R |\sN| } & HW(\varOmega Q,L) \dar[equal] \\
        (\varSigma^\infty_+ \varOmega Q \wedge R) \wedge_R HW(\varOmega Q,L) \rar{\beta} & HW(\varOmega Q,L)
    \end{tikzcd},
    \]
    where $\beta$ is the $(\varSigma^\infty_+\varOmega Q\wedge R)$-module structure on $HW(\varOmega Q,L)$, and $\varphi \colon HW(\varOmega Q,F) \overset{\simeq}{\to} \varSigma^\infty_+\varOmega Q \wedge R$ is the equivalence of $R$-modules from \cref{lem:loop_sp_local_sys}.
\end{prop}
\begin{proof}
    We first prove that the top square is homotopy commutative. We have two $R$-oriented flow multimodules with local system
    \[
        \sN \circ \osr,\; \sR_{\varOmega Q} \circ_1 \sN_F \colon\sC\sW(F,F),\sC\sW(F,L) \longrightarrow \sC\sW(\varOmega Q,L).
    \]
    The local system on $\sR_{\varOmega Q} \circ_1 \sN_F$ is given by the composition local system $\sE_1$ (see \cref{dfn:composition_local_system}), and the local system $\sE_2$ on $\sN \circ \osr$ is determined by the one on $\sN$. We now consider an $R$-oriented flow bordism $\sB$ defined by the assignment
    \begin{equation}\label{eq:flow_bordism_twisted_mu2}
        (a,p;q) \longmapsto \sR_3(q_0,a,p;q),
    \end{equation}
    where $\sR_3$ denotes the $R$-oriented flow bordism defined in \cref{dfn:moduli_assoc} and \cref{lem:moduli_associativity}(ii). It follows from \cref{lem:moduli_associativity} that $\sB$ defines an $R$-oriented flow bordism $\sB \colon \sN \circ \osr \Rightarrow \sR_{\varOmega Q} \circ_1 \sN_F$. It therefore suffices to show that local systems are respected. Equip $\sB$ with the local system $\sB \Rightarrow F_{\varSigma^\infty_+\varOmega Q \wedge R}(\sE_1,\sE_2)$ induced by the evaluation map
    \begin{align*}
        \sB(a,p;q) &\longrightarrow \sP_{q_0q}Q \\
        u &\longmapsto u|_{(\partial S)_0}.
    \end{align*}
    It is clear that this local system restricts to the corresponding local system of the two codimension $1$ boundary strata $(\sN\circ \osr)(a,p;q) \hookrightarrow \sB(a,p;q)$ and $(\sR_{\varOmega Q} \circ_1 \sN_F)(a,p;q) \hookrightarrow \sB(a,p;q)$. Pick a strictly increasing sequence $\{B_\ell\}_{k=1}^\infty \subset \IR$ that diverge to $\infty$. The flow multimodule $\sR_{\varOmega Q}$ restricts to the action filtered flow categories to an $R$-oriented flow multimodule with local system
    \begin{equation}\label{eq:action_filtered_twisted_mu2}
        (\sR_{\varOmega Q}^{\ell},\fo^{\ell}_{\varOmega Q},\sF^{\ell}_{\varOmega Q}) \colon \sC\sW(\varOmega Q,F),\sC\sW^{\leq B_\ell}(F,L) \longrightarrow \sC\sW(\varOmega Q,L),
    \end{equation}
    similar to the action filtered flow multimodule $\osr^{k,\ell}$ defined in \eqref{eq:multimodule_mu2}. Similarly, the flow bimodule $\sN_F$ restricts to the action filtered pieces
    \[
    \sN_F^k \colon \sC\sW^{\leq A_k}(F,F) \longrightarrow \sC\sW(\varOmega Q,F).
    \]
    Restricting the $R$-oriented flow bordism $\sB$ to appropriate action filtrations and passing to CJS realizations yields a homotopy commutative diagram
    \[
    \begin{tikzcd}[row sep=scriptsize]
        HW^{\leq A_k}(F,F) \wedge_R HW^{\leq B_\ell}(F,L) \rar{|\osr^{k,\ell}|} \dar{|\sN_F^k| \wedge_R \id} & HW^{\leq A_k+B_\ell}(F,L) \dar{|\sN^{k+\ell}|}\\
        HW(\varOmega Q,F) \wedge_R HW^{\leq B_\ell}(F,L) \rar{|\sR_{\varOmega Q}^{\ell}|} & HW(\varOmega Q,L)
    \end{tikzcd},
    \]
    which implies that the top square is homotopy commutative, after passing to the homotopy colimit as $k,\ell \to \infty$.

    Next we show that the bottom square is homotopy commutative. The map $\beta$ is induced from the natural transformation $(\varSigma^\infty_+\varOmega Q \wedge R) \wedge_R Z_{\varOmega Q,L} \Rightarrow Z_{\varOmega Q,L}$ that is defined on objects by the composition of Moore paths:
    \begin{align*}
        (\varSigma^\infty_+\varOmega_{q_0}Q\wedge R) \wedge_R Z_{\varOmega Q,L}(m) &= (\varSigma^\infty_+\varOmega_{q_0}Q\wedge R) \wedge_R \bigvee_{p \in \mu^{-1}(m)} ((\varSigma^\infty_+\sP_{q_0p}Q\wedge R) \wedge_R \fo_{Q,L}(p)) \\
        &= \bigvee_{p \in \mu^{-1}(m)} ((\varSigma^\infty_+\varOmega_{q_0} Q\wedge R) \wedge_R (\varSigma^\infty_+\sP_{q_0p}Q \wedge R) \wedge_R \fo_{Q,L}(p)) \\
        &\longrightarrow \bigvee_{p \in \mu^{-1}(m)} ((\varSigma^\infty_+\sP_{q_0p}Q \wedge R) \wedge_R \fo_{Q,L}(p)) = Z_{\varOmega Q,L}(m).
    \end{align*}
    The $R$-oriented flow multimodule with local system $(\sR_{\varOmega Q},\fo_{\varOmega Q},\sF_{\varOmega Q})$ defined in \eqref{eq:twisted_multimod} induces a natural transformation $\sR_{\varOmega Q}^{-I} \Rightarrow F_{\varSigma^\infty_+ \varOmega Q\wedge R}(Z_{\varOmega Q,F} \wedge_R Z_{F,L},Z_{\varOmega Q,L})$, that induces maps $\rho_{0,p;q}$ for any $p \in \Ob(\sC\sW(F,L))$ and $q\in \Ob(\sC\sW(Q,L))$ defined as the composition
    \[
        Z_{\varOmega Q,F}(0) \wedge_R Z_{F,L}(\mu(p)) \longrightarrow Z_{\varOmega Q,F}(0) \wedge_R Z_{F,L}(\mu(p)) \wedge \sR_{\varOmega Q}^{-I}(q_0,p;q) \longrightarrow Z_{\varOmega Q,L}(\mu(q)).
    \]
   Similarly, the $R$-oriented flow bimodule with local system $(\sN,\fo,\sE)$ induces a natural transformation $\sN^{-I} \Rightarrow F_{\varSigma^\infty_+ \varOmega Q\wedge R}(Z_{F,L}, Z_{\varOmega Q,L})$ that induces maps $\nu_{p,q}$ for any $p \in \Ob(\sC\sW(F,L))$ and $q\in \Ob(\sC\sW(Q,L))$ defined as the composition
    \[
    Z_{F,L}(\mu(p)) \longrightarrow Z_{F,L}(\mu(p))\wedge \sN^{-I}(p,q) \longrightarrow Z_{\varOmega Q,L}(\mu(q)).
    \]
    Now, for every $p \in \Ob(\sC\sW(F,L))$ and $q\in \Ob(\sC\sW(Q,L))$ the following diagram is commutative
    \begin{equation}\label{eq:natural_transf}
        \begin{tikzcd}[row sep=scriptsize,column sep=1cm]
            Z_{\varOmega Q,F}(0) \wedge_R Z_{F,L}(\mu(p)) \rar{\rho_{0,p;q}} \dar{\id \wedge \nu_{p,q}} & Z_{\varOmega Q,L}(\mu(q)) \dar[equal] \\
            (\varSigma^\infty_+\varOmega_{q_0}Q \wedge R) \wedge_R Z_{\varOmega Q,L}(\mu(q)) \rar{\beta} & Z_{\varOmega Q,L}(\mu(q))
        \end{tikzcd},
    \end{equation}
    since for every $p\in \Ob(\sC\sW(F,L))$ and $q\in \Ob(\sC\sW(\varOmega Q,L))$, we have $\sR_{\varOmega Q}(q_0,p;q) = \sN(p,q)$. The flow bimodule $\sN$ restricts to action filtrations to the flow bimodule
    \[
    \sN^\ell \colon \sC\sW^{\leq B_\ell}(F,L) \longrightarrow \sC\sW(\varOmega Q,L).
    \]
    Passing to the CJS realizations of $\sR_{\varOmega Q}^{\ell}$ and $\sN^\ell$ yields a diagram
    \[
    \begin{tikzcd}[row sep=scriptsize, column sep=1cm]
        HW(\varOmega Q,F) \wedge_R HW^{\leq B_\ell}(F,L) \rar{|\sR^{\ell}_{\varOmega Q}|} \dar{\varphi \wedge_R |\sN^\ell| } & HW(\varOmega Q,L) \dar[equal] \\
        (\varSigma^\infty_+\varOmega_{q_0}Q \wedge R) \wedge_R HW(\varOmega Q,L) \rar{\beta} & HW(\varOmega Q,L)
    \end{tikzcd},
    \]
    which is commutative up to homotopy, since the diagram \eqref{eq:natural_transf} is commutative after passing to action filtrations. Passing to the homotopy colimit as $k,\ell \to \infty$ finishes the proof.
\end{proof}
\begin{cor}\label{prop:equiv_of_loop-homod}
    The equivalence of $R$-modules $|\sN| \colon HW(F,L) \overset{\simeq}{\to} HW(\varOmega Q,L)$ is an equivalence of homotopy $(\varSigma^\infty_+ \varOmega Q \wedge R)$-modules.
    \qed
\end{cor}
\subsubsection{Lift of morphisms}\label{subsec:mor_lift}
We now turn to morphisms:
\begin{prop}\label{prop:mor_lift1}
    The following diagram commutes up to homotopy:
    \[
        \begin{tikzcd}
            HW(F,L) \wedge_R HW(L,K) \rar{\mu^2} \dar{|\sN| \wedge_R \id} & HW(F,K) \dar{|\sN|}\\
            HW(\varOmega Q,L) \wedge_R HW(L,K) \rar{\mu^2_{\varOmega Q}} & HW(\varOmega Q,K).
        \end{tikzcd}
    \]    
\end{prop}
\begin{proof}
    Pick strictly increasing sequences $\{A_k\}_{k=1}^\infty, \{B_\ell\}_{\ell=1}^\infty \subset \IR$ that diverge to $\infty$. Let $\osr^{k,\ell}$ be the action filtered flow multimodule defined in \eqref{eq:multimodule_mu2}. Passing to CJS realizations of the appropriately action filtered flow multimodules involved, we obtain a diagram
    \[
        \begin{tikzcd}
            HW^{\leq A_k}(F,L) \wedge_R HW^{\leq B_\ell}(L,K) \rar{|\osr^{k,\ell}|} \dar{|\sN^k| \wedge_R \id} & HW^{\leq A_k+B_\ell}(F,K) \dar{|\sN^{k+\ell}|}\\
            HW(\varOmega Q,L) \wedge_R HW^{\leq B_\ell}(L,K) \rar{|\sR^{\ell}_{\varOmega Q}|} & HW(\varOmega Q,K),
        \end{tikzcd}
    \]
    and to conclude the result it suffices to prove that this diagram is homotopy commutative. Similarly to the $R$-oriented flow bordism with local system $\sB$ defined in \eqref{eq:flow_bordism_twisted_mu2}, we may define an $R$-oriented flow bordism with local system which restricts on action filtrations to an $R$-oriented flow bordism with local system $\sN^{k+\ell} \circ \osr^{k,\ell} \Rightarrow \sR^{\ell}_{\varOmega Q} \circ_1 \sN^k$. Therefore \cref{lma:bordism_gives_homotopy_of_maps} yields that the above action filtered diagram is homotopy commutative, which finishes the proof.
\end{proof}

\section{Closed exact Lagrangians are isomorphic to the zero section}\label{sec:modules}
\begin{notn}
	\begin{enumerate}
	    \item Let $R$ be a commutative ring spectrum.
        \item Let $Q$ be a closed smooth manifold.
        \item Let $F \coloneqq T^\ast_\xi Q$ denote a fixed cotangent fiber.
	\end{enumerate}
\end{notn}

Recall from \cref{exmp:cotpol} that we equip $T^*Q$ with the standard polarization given by the tangent spaces of the cotangent fibers. It follows that $\mathcal W(T^*Q;R)$ is well-defined. Recall also that the stable Lagrangian Gauss map of the zero section is null-homotopic and hence it supports an $R$-brane structure. The cotangent fiber $F$ admits a canonical $R$-brane structure since it is contractible (see \cref{rem:brane}(iii)). More generally, any nearby Lagrangian admits an $R$-brane structure:
\begin{thm}[{\cite[Corollary 1.3]{jin2020microlocal}, \cite[Theorem B]{abouzaid2020twisted}}]
Let $L \subset T^\ast Q$ be a nearby Lagrangian. The composition
\[
L \longrightarrow \lag \longrightarrow {\B}^2{\GL_1}(\IS)
\]
is null-homotopic. In particular, any nearby Lagrangian supports an $R$-brane structure for any commutative ring spectrum $R$.
\qed
\end{thm}

\begin{notn}
    \begin{enumerate}
        \item Let $Q^\#$ denote a choice of $R$-brane structure on the zero section $Q$.
        \item Let $Q$ denote the $R$-brane structure on the zero section $Q$ determined by the canonical null-homotopy of the stable Lagrangian Gauss map.
    \end{enumerate}
\end{notn}

\begin{lem}\label{lma:rank1}
    Let $R$ be a connective ring spectrum. For any nearby Lagrangian $R$-brane $L \subset T^*Q$ we have an equivalence of $R$-modules $HW(F,L;R) \simeq \varSigma^\ell R$, for some $\ell \in \IZ$.
\end{lem}
\begin{proof}
    Let $k \coloneqq \pi_0R$.  Notice that an $R$-brane structure induces a natural $Hk$-brane structure. There is an associated $Hk$-orientation on the flow category $\mathcal{CW}(F,L)$, so that 
    \[ HW(F,L;Hk) \simeq HW(F,L;R) \wedge_{R} Hk.\]
    Following \cref{lem:wfuk_discrete_coeff} (and \cref{rem:wfuk_discrete_coeff_k}), we have that the ordinary wrapped Floer cohomology of $F$ and $L$ with coefficients in $k$ is given by 
    \[ HW^{-\bullet}(F,L;k) \cong \pi_\bullet(HW(F,L;R) \wedge_R Hk).\] 
    By \cite[Lemma C.1]{abouzaid2012nearby} and \cite[Corollary 1.3]{fukaya2008symplectic}, $HW^{-\bullet}(F,L;k) \cong k[\ell]$ for some $\ell \in \IZ$. Since $R$ is connective and $L$ is compact, $HW(F,L;R)$ is bounded below. The conclusion now follows from \cref{cor:whitehead_cor}.
\end{proof}
\begin{lem}\label{lem:iso_loop_yoneda}
    There exists an $R$-brane structure $Q^\#$ on $Q$, such that there is a weak equivalence of $(\varSigma^\infty_+ \varOmega Q \wedge R)$-modules
    \[
    HW(\varOmega Q,L;R) \simeq HW(\varOmega Q,Q^\#;R).
    \]
\end{lem}
\begin{proof}
    By combining \cref{prp:equiv_fiber_loop_space_local,lma:rank1}, we have an isomorphism of $R$-modules
    \[
    HW(\varOmega Q,L;R) \simeq \varSigma^\ell R,
    \]
    for some $\ell \in \IZ$. The $(\varSigma^\infty_+ \varOmega Q \wedge R)$-module structure on $HW(\varOmega Q,L;R)$ is therefore encoded by a map $\zeta \colon (\varOmega Q)_+ \to \GL_1(R)$, which precisely corresponds to a choice of $R$-brane structure $Q^\#$ on $Q$. Equipping $Q$ with this $R$-brane structure and appealing to \cref{prop:brane_twisted_R_mod} therefore finishes the proof.
\end{proof}

The following is our main result of this paper.
\begin{thm}\label{thm:modules_main}
	Let $L$ be a nearby Lagrangian $R$-brane in $T^*Q$. Then, there exists an $R$-brane structure $Q^\#$ on the zero section $Q$ such that $L \cong Q^{\#}$ in $\sW(T^\ast Q;R)$.
\end{thm}
\begin{rem}\label{rem:reduce_to_connective}
    Recall, from the discussion in \cref{subsec:DF_conn_cover}, that there is a functor
    \[
        \sW(T^*Q;R_{\geq 0}) \longrightarrow \sW(T^*Q;R)
    \]
    that is the identity on objects and reflects isomorphisms between compact $R$-branes (see \cref{cor:lifting_equivalences_to_conn_covers}).  Consequently, to prove \cref{thm:modules_main}, it suffices to prove the result under the assumption that $R$ is connective.  With this in mind, we now fix $R$ to be a connective commutative ring spectrum for the remainder of this section.
\end{rem}

\begin{proof}[Proof of \cref{thm:modules_main}]
    Let $k\coloneqq \pi_0R$. Recall the definition of the Hurewicz map in \eqref{eq:hurewicz_map}. Note that an $R$-brane structure determines a $k$-brane structure.  In particular, for any Lagrangian $R$-branes $L$ and $K$, the wrapped Floer cohomology with coefficients in the discrete ring $k$, $HW^{-\bullet}(L,K;k)$, is well-defined. Let $\mathcal{F}(T^*Q;R) \subset \mathcal W(T^\ast Q;R)$ denote the full subcategory of closed Lagrangian $R$-branes. Using \cref{prop:mor_lift1,prop:mor_lift2}, we have a commutative diagram as in \eqref{eq:diagram_yoneda}
    \[
        \begin{tikzcd}[row sep=scriptsize, column sep=1.5cm]
            \mathcal{F}(T^\ast Q;R) \dar[swap]{\Hw_{\sF(T^\ast Q;R)}} \rar{\sY_{\varOmega Q}} & \Ho\mod{(\varSigma^\infty_+ \varOmega Q \wedge R)} \dar{\Hw_{(\varSigma^\infty_+ \varOmega Q \wedge R)}} \\
            \mathcal{F}(T^\ast Q;Hk) \rar{\sY_{\varOmega Q}^{Hk}} & \Ho\mod{(\varSigma^\infty_+ \varOmega Q \wedge Hk)}
        \end{tikzcd},
    \]
    where $\sY_{\varOmega Q}^{Hk}$ is the local system valued Yoneda functor, as in \cref{sec:str_lift}, with $Hk$-coefficients. By construction, $\sY_{\varOmega Q}^{Hk}$ coincides with (the restriction to $\sF(T^*Q;Hk)$ of) the cohomological functor associated to the one from \cite[Theorem 1.1]{abouzaid2012wrapped}. The latter is fully-faithful as a consequence of the well-known fact that the wrapped Fukaya category of $T^\ast Q$ with discrete coefficients is generated by a cotangent fiber (see \cite{abouzaid2011a,ganatra2022sectorial}). It follows that $\sY_{\varOmega Q}^{Hk}$ is fully faithful. Moreover, we note that the restriction of $\sY_{\varOmega Q}$ on the full subcategory $\sW(T^*Q;R) \wedge Hk \subset \sW(T^*Q;Hk)$ coincides with $\sY_{\varOmega Q} \wedge Hk$. From \cref{lem:iso_loop_yoneda}, there exists an $R$-brane structure $Q^\#$ on the zero section so that $\sY_{\varOmega Q}(L) \simeq \sY_{\varOmega Q}(Q^\#)$ in $\mod{(\varSigma^\infty_+ \varOmega Q \wedge R)}$. Hence \cref{lma:whitehead_cats} implies that $L \cong Q^\#$ in $\sW(T^*Q;R)$.
\end{proof}

\appendix
    \section{Geometric background}\label{sec:geom_background}
	In this section, we give some background on the various geometric notions appearing in this paper. A \emph{Liouville domain} is a pair $(X^{2n},\lambda)$ of a smooth $2n$-dimensional manifold $X$ with boundary $\partial_\infty X$ and a one-form $\lambda$ such that $\omega \coloneqq d \lambda$ is symplectic and the $\omega$-dual of $\lambda$ (called the \emph{Liouville vector field}) is outwards pointing along $\partial_\infty X$. The completion of a Liouville domain $(X^{\mathrm{in}},\lambda^{\mathrm{in}})$ is the non-compact exact symplectic manifold $X \coloneqq X^{\mathrm{in}} \cup_{\{0\} \times \partial X^{\mathrm{in}}} ([0,\infty) \times \partial_\infty X^{\mathrm{in}})$ equipped with the one-form
    \[
    \lambda \coloneqq \begin{cases}
        \lambda^{\mathrm{in}}, & \text{in } X^{\mathrm{in}} \\
        e^r\lambda^{\mathrm{in}}|_{\partial X^{\mathrm{in}}}, & \text{in } [0,\infty)_r \times \partial_\infty X^{\mathrm{in}}
    \end{cases}.
    \]
    This is a smooth one-form because the Liouville vector field specifies a collar neighborhood $(-\varepsilon,0]_r \times \partial_\infty X^{\mathrm{in}} \subset X^{\mathrm{in}}$ on which the Liouville one-form is equal to $e^r\lambda^{\mathrm{in}}|_{\partial_\infty X^{\mathrm{in}}}$. A (finite type) \emph{Liouville manifold} is the completion of a Liouville domain $(X, \lambda)$. The \emph{boundary} of a Liouville manifold is defined to be the boundary of the Liouville domain it is the completion of. The \emph{core} of a Liouville manifold, denoted by $\Core X$, is the attractor of the backwards flow of the Liouville vector field. A \emph{Liouville embedding} $f \colon (X, \lambda) \to (Y, \mu)$ is a proper embedding  of underlying smooth manifolds  such that $f^\ast \mu - \lambda$ is exact and compactly supported. Note that the Liouville flow determines an identification $(X\smallsetminus \Core X,\lambda) \cong (\IR \times \partial_\infty X,e^r \lambda|_{\partial_\infty X})$.
	\begin{defn}[Liouville pair]
		A Liouville pair is a tuple $(X,F)$ consisting of a Liouville $2n$-manifold $X^{2n}$ and a Liouville $(2n-2)$-manifold $F$ together with a choice of Liouville embedding $(F, \lambda_F) \hookrightarrow (X \smallsetminus \Core X, \lambda_X)$ such that the induced map $F \to \partial_\infty X$ is an embedding. We call $F$ together with a choice of such a Liouville embedding a \emph{Liouville hypersurface}.
	\end{defn}
	\begin{rem}
		A Liouville hypersurface is sometimes called a \emph{stop}, cf.\@ Sylvan's original definition \cite{sylvan2019on} and \cite[Definition 2.33]{ganatra2020covariantly}. A Liouville pair is sometimes called a \emph{stopped Liouville manifold}.
	\end{rem}
	\begin{defn}[Liouville sector {\cite[Definition 1.1]{ganatra2020covariantly}}]\label{defn:liouville_sector}
		A \emph{Liouville sector} is a Liouville manifold-with-boundary $(X, \lambda, Z)$ for which there is a function $I \colon \partial X \to \IR$ such that the following holds.
		\begin{itemize}
			\item $I$ is \emph{linear at infinity}, meaning $ZI = I$ outside a compact set, where $Z$ denotes the Liouville vector field.
			\item The Hamiltonian vector field $X_I$ of $I$ is outward pointing along $\partial X$.
		\end{itemize}
	\end{defn}
    \begin{rem}
        Note that a Liouville sector $(X,\lambda,Z)$ has two kinds of boundaries: The boundary $\partial_\infty X$ of a defining Liouville domain for $X$, and the \emph{horizontal boundary} $\partial X$. E.g.\@ if $M$ is a smooth manifold with boundary $\partial M$, $T^\ast M$ is a Liouville sector with boundary $\partial_\infty T^\ast M = ST^\ast M$ and horizontal boundary $\partial T^\ast M = T^\ast M|_{\partial M}$.
    \end{rem}
    \begin{defn}[Inclusion of Liouville sectors]\label{dfn:inclusion_sectors}
        \begin{enumerate}
            \item ({\cite[Definition 2.4]{ganatra2020covariantly}}) An \emph{inclusion of Liouville sectors} is a proper map $i \colon (X,\lambda) \hookrightarrow (X',\lambda')$ which is a diffeomorphism onto its image, such that $i^\ast \lambda' - \lambda$ is exact and compactly supported.
            \item By an \emph{inclusion of stably polarized Liouville sectors}, we mean an inclusion of Liouville sectors $X \hookrightarrow X'$ such that the pullback of the stable polarization of $X'$ along the inclusion coincides (up to homotopy) with the stable polarization of $X$.
    \end{enumerate}
	\end{defn}
	By \cite[Proposition 2.25]{ganatra2020covariantly}, there is a conical neighborhood of $\partial X$ that is Liouville isomorphic to $(F^{2n-2} \times T^\ast [0,\varepsilon), \lambda_F+pdq)$ for some $\varepsilon > 0$, where $(q,p)$ are local coordinates in $T^\ast [0,\varepsilon)$ and $F^{2n-2}$ is a Liouville $(2n-2)$-manifold that is called the \emph{symplectic boundary} of $X$.

	For every Liouville sector $X$, one can modify the Liouville form to obtain a Liouville pair $(\overline X, F)$ called the \emph{convexification} of $X$, see \cite[Section 2.7]{ganatra2020covariantly}. Moreover, up to a contractible choice, there is a one-to-one correspondence between Liouville sectors and Liouville pairs \cite[Lemma 2.32]{ganatra2020covariantly}. The convexification comes with a smooth proper embedding
    \begin{equation}\label{eq:stop_proper_emb}
        \sigma \colon F \times \IC_{\mathrm{Re} < \varepsilon} \longrightarrow \overline X
    \end{equation}
    such that $\sigma^\ast \lambda_X = \lambda_F + \frac 12(xdy-ydx) + df$ for some compactly supported real-valued function $f \colon \overline X \to \IR$, see \cite[Lemma 2.31]{ganatra2020covariantly}, which coincides with the definition of a stop \cite[Definition 2.3]{sylvan2019on}.

    A \emph{Weinstein manifold} is a Liouville manifold $X$ that admits a Morse function $X \to \IR$ for which the Liouville vector field is a pseudogradient. A Weinstein manifold is of \emph{finite type} if there exists a Weinstein Morse function with finitely many critical points.
    \begin{defn}[Weinstein pair]
    	A \emph{Weinstein pair} is a Liouville pair $(X,F)$ such that both $X$ and $F$ are Weinstein manifolds, and the associated Liouville embedding $(F,\lambda_F) \hookrightarrow (X \smallsetminus \Core X, \lambda)$ preserves some choices of Weinstein Morse functions.
    \end{defn}
    \begin{defn}[Weinstein sector]
    	A \emph{Weinstein sector} is a Liouville sector such that its convexification admits the structure of a Weinstein pair.
    \end{defn}
    \begin{rem}
    	Any finite type Weinstein manifold can equivalently be defined as the result of a finite number of successive \emph{Weinstein handle attachments} \cite{weinstein1991contact}. A Weinstein handle of index $k$ of a $2n$-dimensional Weinstein manifold is topologically equivalent to $\ID^k \times \ID^{2n-k}$ with $k \in \{0,\ldots,n\}$. Simiarly, any finite type Weinstein sector admit a description in terms of successive attachments of Weinstein handles and Weinstein half-handles (see \cite[Definition 2.7]{chantraine2017geometric}).
    \end{rem}
    \begin{defn}[Subcritical Weinstein sector]\label{dfn:subcrit_weinstein_sector}
    	We say that a $2n$-dimensional Weinstein sector is \emph{subcritical} if it is the result of successive attachments of Weinstein handles and Weinstein half-handles of index $< n$.
    \end{defn}
    \section{Spectra and orientations}\label{sec:background_spectra}
Here we will discuss some basic notions related to spectra and orientations. The main references for this section are \cite{may2006ring}, \cite{elmendorf1997rings} and \cite[Section 1.2]{ando2008units}.

\begin{notn}
    Throughout this paper, whenever we refer to the category of spectra we mean the category of $\IS$-modules, as defined in \cite[Section II.1]{elmendorf1997rings}. This is a complete, cocomplete category with a (point-set level) symmetric monoidal structure given by the $\IL$-smash product $\wedge_{\IS}$ (see \cite[Proposition 1.4 and Theorem 1.6]{elmendorf1997rings}). We will often denote $\wedge_{\IS}$ by $\wedge$. 
\end{notn}

\begin{defn}[Ring spectra and modules over them]\label{defn:hgh_str_ring_spec}
    By a \emph{(commutative) ring spectrum}  $R$ we mean a (commutative) $\wedge_\IS$-monoid object in $\IS$-modules. A left (right) \emph{$R$-module} is a left (right) module object in $\IS$-modules over the $\wedge$-monoid $R$. We denote the category of $R$-modules, enriched over spectra, by $\mod{R}$.
\end{defn}
\begin{rem}\label{rem:hgh_str_ring_spec}
        Ring spectra, commutative ring spectra, and modules as defined above correspond to $A_\infty$ ring spectra, $E_\infty$ ring spectra, and $A_\infty$-modules in the classical sense. 
\end{rem}
The category of $R$-modules is closed under limits and colimits and, similarly as in usual algebra, admits a symmetric monoidal structure given by
    \[
        \begin{tikzcd}[row sep=scriptsize, column sep=scriptsize]
            M \wedge_R N \coloneqq \colim \big(M \wedge R \wedge N \rar[shift left]{} \rar[shift right,swap]{} &  M \wedge N\big)
        \end{tikzcd}.
    \]
There is a free $R$-module functor $X \mapsto R \wedge X$ which is right adjoint to the forgetful functor $\mod{R} \to \mod{\IS}$. For any $R$-module $Y$, free $R$-module $X$, we have an isomorphism
    \[
        Y \wedge_R (R \wedge X) \simeq Y \wedge X.
    \]
\begin{defn}[$R$-algebra]\label{defn:r_alg}
    A \emph{(commutative) $R$-algebra} is a (commutative) monoid object in the category of $R$-modules. 
\end{defn}
\begin{defn}[Space of units of a ring spectrum]\label{defn:space_of_units} The \emph{space of units} of $R$ is the topological space defined by the following homotopy limit
    \[
        \GL_1(R) \coloneqq \holim \left(\begin{tikzcd}[row sep=scriptsize, column sep=scriptsize]
            {} & \varOmega^\infty R \dar \\
            (\pi_0 R)^\times \rar[hook] & \pi_0 R
        \end{tikzcd}\right),
    \]
    where the vertical map is the composition $\varOmega^\infty R \to \pi_0 \varOmega^\infty R \xrightarrow{\cong} \pi_0 R$.
\end{defn}
\begin{rem} 
    \begin{enumerate} 
        \item Since $\varOmega^\infty R$ is an $E_\infty$-space it follows that $\GL_1(R)$ admits the structure of a group-like $E_\infty$-space, in other words it is an infinite loop space.
        \item Note that for a space $X$ we have $[X,\GL_1(R)] = H^0(X_+;R)^\times$.
   \end{enumerate}
\end{rem}
The space of units of $R$ admits a delooping $\BGL_1(R)$. In fact, there is a fibration sequence, see \cite[Section 3]{ando2008units} or \cite[Corollary 1.14]{ando2014an}, 
\[ \GL_1(R) \longrightarrow \EGL_1(R) \longrightarrow \BGL_1(R).\]
Moreover, the construction of $\BGL_1(R)$ is functorial and thus the unit map $\IS \to R$ induces a map $\BGL_1(\IS) \longrightarrow \BGL_1(R)$.

\begin{defn}[$R$-line bundle]
    \begin{enumerate}
        \item An \emph{$R$-line bundle} on a space $X$ is a map $X_+ \to \BGL_1(R)$.
        \item Given two $R$-line bundles $f, g \colon X_+ \to \BGL_1(R)$, an \emph{isomorphism} $\varphi \colon f \to g$ is a homotopy from $f$ to $g$.
        \item \emph{The trivial $R$-line bundle} on $X$ is the constant map $X_+ \to \BGL_1(R)$ mapping $X$ to the base point of $\BGL_1(R)$. We often denote this line bundle by $\underline R$.
        \item A \emph{trivialization} of an $R$-line bundle $f \colon X_+ \to \BGL_1(R)$ is an isomorphism between $f$ and the trivial line bundle on $X$. An $R$-line bundle is \emph{trivializable} if it admits a trivialization.
        \item Given two $R$-line bundles $f \colon X_+ \to \BGL_1(R)$ and $g \colon Y_+ \to \BGL_1(R)$, their \emph{tensor product} is denoted by $f \otimes_R g \colon (X \times Y)_+ \to \BGL_1(R)$ and is defined as the following composition
            \[
                \begin{tikzcd}[row sep=scriptsize, column sep=scriptsize]
                    (X \times Y)_+ \cong X_+ \wedge Y_+ \rar{f \times g} &\BGL_1(R) \wedge \BGL_1(R) \rar & \BGL_1(R)
                \end{tikzcd}
            \]
            where the second map is give by the ring structure on $\BGL_1(R)$.
    \end{enumerate}
\end{defn}

\begin{defn}[$R$-Thom spectrum]\label{dfn:R-Thom_spec}
    Let $f \colon X \to \BGL_1(R)$ be an $R$-line bundle and 
    \[
        P \coloneqq \holim \left(\begin{tikzcd}[row sep=scriptsize, column sep=scriptsize]
            {}& \EGL_1(R) \dar \\
            X \rar{f} & \BGL_1(R)
        \end{tikzcd}\right)
    \]
    be its homotopy fiber. We define the \emph{$R$-Thom spectrum} of $f$ to be the derived smash product
    \[ X^f \coloneqq P \wedge_{\GL_1(R)}^{\IL} R.\]
\end{defn}

\begin{rem}\label{rem:thommul}
    \begin{enumerate}
        \item For the trivial $R$-line bundle $\underline R$ on $X$, 
        \[X^{\underline R} \simeq X \wedge R.\]
        \item $R$-Thom spectra are compatible with tensor products in the sense that
        \[ (X \times Y)^{f \otimes_R g} \simeq X^f \wedge_R Y^g\]
    \end{enumerate}
\end{rem}

\begin{defn}[Rank one free $R$-module]\label{defn:free_rmod}
    By a \emph{rank one free $R$-module} we we mean an $R$-line bundle on the point space $\text{pt}$, in other words a map $\text{pt}_+ \to \BGL_1(R)$.
\end{defn}
\begin{rem}
    Associated to any rank one free $R$-module $f$, there is an $R$-module viz.\@ its $R$-Thom spectrum $\text{pt}_+^f$. Since $f$ is isomorphic to the trivial $R$-line bundle on $\text{pt}_+$, there exists an $R$-module isomorphism 
    \[\text{pt}_+^f \simeq \text{pt}_+^{\underline R} \simeq R.\]
\end{rem}
\begin{notn}\label{notn:thom_spectra}
    Let $f\colon \text{pt}_+ \to \BGL_1(R)$ be a rank one free $R$-module.
    \begin{enumerate}
        \item We often abuse notation and also denote the $R$-module $\text{pt}_+^f$ by $f$. 
        \item Given a space $X$, we will denote the pullback of $f$ to $X$ by $\underline f_X$. Note that $X^{\underline f_X} \simeq X \wedge f$.
    \end{enumerate}
\end{notn}
\begin{rem}
    An $\infty$-categorical approach is taken in \cite{ando2014an} to define $R$-line bundles, their trivializations, and Thom spectra. Let $\mod{R}$ be the $\infty$-category of $R$-modules and $\line{R}$ be subcategory of rank one free $R$-modules and isomorphisms between them. It is shown in \cite[Proposition 2.9]{ando2014an} that the topological space associated to the $\infty$-groupoid $\line{R}$ is equivalent to $\BGL_1(R)$. A line bundle as defined above can thus be seen to be equivalent to a functor $X \to \line{R}$, where $X$ is thought of as the $\infty$-groupoid associated to the space $X$. The latter is the definition of $R$-line bundle adopted there. Moreover, the Thom spectrum associated with the functor $f \colon X \to \line{R}$ is defined to be the $\infty$-categorical colimit
    \[ X^f \coloneqq \colim (X \overset{f}{\longrightarrow} \line{R} \longrightarrow  \mod{R}).\]
    It is shown in \cite[Proposition 3.20]{ando2014an} that this definition agrees with the one in \cref{dfn:R-Thom_spec}.
\end{rem}

\begin{defn}[The $R$-line bundle associated to a vector bundle]\label{dfn:assoc_rline}
    Given any stable vector bundle $f \colon X \to \BO$ we define its \emph{associated $R$-line bundle} $f_R \colon X \to \BGL_1(R)$ by the following composition
    \[
        f_R \colon X \overset{f}{\longrightarrow} \BO \longrightarrow \BGL_1(\IS) \longrightarrow \BGL_1(R).
    \]
\end{defn}
\begin{rem}\label{rem:assoc_vector_bundles}
\begin{enumerate}
    \item Let $f$ be a vector bundle and let $f_R$ be its associated $R$-line bundle. Then we have
        \[ X^{f_R} \simeq R \wedge X^f,\]
        where $X^f$ here denotes the usual Thom spectrum associated to the stable vector bundle $f$.
    \item Given two vector bundles $f$ and $g$, we have
        \[ (f \oplus g)_R \simeq f_R \otimes_R g_R.\]
    \item Given a non-stable rank $n$ vector bundle, $E \colon X \to \BO(n)$, one can consider its stabilization, denoted $f \colon X \to \BO$.  The classical Thom spectrum of $E$, denoted by $X^E$, is related to $X^f$ by $\varSigma^{-n}X^E \simeq X^f$.
\end{enumerate}
\end{rem}
\begin{rem}\label{rem:constant_R-line_isos}
    Let $f,g$ be rank one free $R$-modules and $X$ be a space such that there is an isomorphism of $R$-line bundles $\underline f_X \simeq \underline g_X$. The map on Thom spaces gives a $R$-module map
    \[ f \colon X \wedge f \longrightarrow g.\]
    This in turn induces a map $\varSigma_+^\infty X \to F_R(f,g)$.
\end{rem}

\begin{defn}[$R$-orientation on vector bundles]\label{dfn:ori_vb}
    A vector bundle $f \colon X \to \BO$ is \emph{$R$-orientable} if its associated line bundle $f_R$ is trivializable. An \emph{$R$-orientation} on $f$ is choice of trivialization of $f_R$.
\end{defn}
\begin{rem}\label{rem:ori_fb_torsor}
    If a vector bundle $f$ is $R$-orientable, the set of $R$-orientations up to homotopy is a torsor over $[X, \varOmega \BGL_1(R)] = [X, \GL_1(R)]$.
\end{rem}

This notion of orientations is related to the usual notion of orientations defined through Thom classes as follows: Let $f \colon X \to \BGL_1(R)$ be an $R$-line bundle and $X^f$ be the associated Thom spectrum. For any $x \in X$, let $(X^f)_x$ denote the $R$-Thom spectrum associated to the rank one free $R$-module
\[ \{x\} \longrightarrow X \overset{f}{\longrightarrow} \BGL_1(R).\]
\begin{defn}[$R$-Thom class]
    A \emph{Thom class} of an $R$-line bundle $f \colon X \to \BGL_1(R)$ is a map of $R$-modules
    \[ \tau \colon X^f \longrightarrow R \]
    such that the composition
    \[ (X^f)_x \longrightarrow X^f \overset{\tau}{\longrightarrow} R \]
    is a weak equivalence for all $x \in X$. An \emph{$R$-Thom class of a vector bundle} is a Thom class of the associated $R$-line bundle.
\end{defn}
\begin{thm}[{\cite[Theorem 2.24 and Corollary 2.26]{ando2014an}}]\label{thm:ori_thom}
    \begin{enumerate}
        \item The space of trivializations of an $R$-line bundle $f \colon X \to \BGL_1(R)$ is homotopy equivalent to the space of Thom classes of $f$.
        \item If $f$ is $R$-oriented and $\tau \colon X^f \to R$ is the associated Thom class, the following composition
        \[
            \begin{tikzcd}[row sep=scriptsize, column sep=scriptsize]
                X^f \rar{\mathrm{pr} \wedge \id} & X \wedge X^f \rar{\id \wedge \tau} & X \wedge R
            \end{tikzcd}
        \]
            is a weak equivalence.
    \end{enumerate}
    \qed
\end{thm}

    \section{Hurewicz maps and Whitehead theorems}\label{sec:Whitehead}
\begin{notn}
    \begin{enumerate}
        \item Let $R$ be a connective and commutative ring spectrum.
        \item Let $k \coloneqq \pi_0(R)$.
    \end{enumerate}
\end{notn}
By \cite[Proposition IV.3.1]{elmendorf1997rings}, there exists a map of $\IS$-algebras $\Hw \colon R \to Hk$ such that the induced map $\Hw_\bullet \colon \pi_0(R) \to \pi_0(Hk)$ is the identity isomorphism. As discussed in \cite[p.\@ 78]{elmendorf1997rings}, since $R$ is a commutative $\IS$-algebra, so is $Hk$, and moreover $\Hw$ endows $Hk$ with the structure of a commutative $R$-algebra.  If $M$ is an $R$-module, then we define the \emph{Hurewicz map} to be the map of $R$-modules
\begin{equation}\label{eq:hurewicz_map}
\Hw_M \colon M \overset{\simeq}{\longrightarrow} M \wedge_R R \xrightarrow{\id \wedge_R \Hw} M \wedge_R Hk.
\end{equation}
\begin{thm}[Hurewicz {\cite[Theorem IV.3.6]{elmendorf1997rings}}]\label{thm:Hurewicz}
    If $M$ is an $(n-1)$-connective $R$-module, then $\pi_i(M \wedge_R Hk) = 0$ for $i < n$ and
    \[ (\Hw_M)_n \colon \pi_n(M) \longrightarrow \pi_n(M \wedge_R Hk) \]
    is an isomorphism.
    \qed
\end{thm}
A corollary of the Hurewicz theorem for connective $R$-modules is the following Whitehead theorem for connective $R$-modules.
\begin{thm}[Whitehead]\label{thm:whitehead}
    Let $M$ and $M'$ be $n$-connective and $n'$-connective $R$-modules for some $n, n' \in \IZ$, respectively. If an $R$-module map $\varphi \colon M \to M'$ induces an equivalence
    \[ \varphi_\bullet \colon \pi_i(M \wedge_R Hk) \longrightarrow \pi_i(M' \wedge_R Hk)\]
    for all $i \in \IZ$, then $\varphi$ is an equivalence of $R$-modules.
\end{thm}
\begin{proof}
By considering the homotopy cofiber of the map $\varphi$ and using \cref{thm:Hurewicz}, the proof of this theorem is the same as the proof of the usual Whitehead theorem.
\end{proof}
\begin{cor}\label{cor:whitehead_cor}
If $M$ is an $(n-1)$-connective $R$-module and
\[
	\pi_\bullet(M \wedge_R Hk) = \begin{cases}
		k, & \bullet = n \\
		0, & \text{else}
	\end{cases},
\]
then there is an equivalence of $R$-modules $\varSigma^{n} R \overset{\simeq}{\to} M$.
\end{cor}
\begin{proof}
	Without loss of generality we may assume that $n = 0$.  We can also assume that $M$ is a CW $R$-module (see \cite[Definition III.2.5 and Theorem III.2.10]{elmendorf1997rings}). The $0$-skeleton, $M^0$, is a wedge of sphere $R$-modules of dimension $0$, $R \simeq S^0_R$.  Note that $M^0 \wedge_R Hk$ and $M \wedge_R Hk$ are CW $Hk$-modules and we have an exact sequence induced from the long exact sequence of the pair $(M \wedge_R Hk, M^0 \wedge_R Hk)$:
    \[
        \begin{tikzcd}[row sep=scriptsize, column sep=scriptsize]
            0 \rar & \pi_1(M \wedge_R Hk,M^0 \wedge_R Hk) \rar & \pi_0(M^0 \wedge_R Hk) \rar & \pi_0(M\wedge_R Hk) \rar & 0
        \end{tikzcd}.
    \]
	Using the CW $Hk$-module structure, it follows that there is a map $R \simeq S_R \to M^0$ of some dimension $0$ $R$-sphere, so that the composition $R \to M^0 \to M$ induces an isomorphism on $\pi_\bullet(- \wedge_R Hk)$.  The desired result now follows from \cref{thm:whitehead}.
\end{proof}

We also record other consequences of the Whitehead theorem that are used in the proofs of our main results. Let $\mathcal C$ denote a category enriched over the homotopy category of $R$-modules, $A$ be an $R$-algebra, and 
\[
\sY \colon \mathcal C \longrightarrow \mod{A} 
\]
be a functor of $\Ho \mod{R}$-enriched categories. Let $\mathcal C \wedge_R Hk$ denote the category with the same objects as $\mathcal C$ and $(\mathcal C \wedge_R Hk)(L,K) \coloneqq \mathcal C(L,K) \wedge_R Hk$. The Hurewicz map $\Hw$ in \eqref{eq:hurewicz_map} induces a functor
\begin{align*}
\Hw_{\mathcal C} \colon \mathcal C &\longrightarrow \mathcal C \wedge_R Hk \\
L &\longmapsto L \\
(L \to K) &\longmapsto \Hw(L \to K)
\end{align*}
The Hurewicz map also induces a functor 
\[
\Hw_{A} \colon \mod{A} \longrightarrow \mod{(A \wedge_R Hk)},
\]
such that the following diagram commutes
\begin{equation}\label{eq:diagram_yoneda}
\begin{tikzcd}[row sep=scriptsize]
	 \mathcal C \dar{\Hw_{\mathcal C}} \rar{\sY} & \mod{A} \dar[swap]{\Hw_{A}} \\
	 \mathcal C \wedge_R Hk \rar{\sY \wedge_R Hk} & \mod{(A \wedge_R Hk)}.
\end{tikzcd}
\end{equation}

\begin{lem}\label{lma:whitehead_cats}
Assume the following:
\begin{enumerate}
    \item $\sY \wedge_R Hk$ in \eqref{eq:diagram_yoneda} is fully faithful.
    \item $A \simeq \varSigma^\infty_+ \varOmega Q \wedge R$ for a path-connected finite CW-complex $Q$.
    \item $\sY(L)$ and $\sY(K)$ are both equivalent to $\varSigma^n R$ as $R$-modules, for some $n \in \IZ$.
    \item $\sC(L,K)$ and $\sC(K,L)$ are connective.
\end{enumerate}
Then, if $\sY(L)$ is isomorphic to $\sY(K)$ in $\mod{A}$, then $L$ is isomorphic to $K$ in $\sC$.
\end{lem}

\begin{proof}
Consider the following commutative diagram:
\[
    \begin{tikzcd}[row sep=1cm, column sep=scriptsize]
        \mathcal C(L,K) \rar \dar & F_{\varSigma^\infty_+ \varOmega Q \wedge R}(\sY(L),\sY(K)) \dar \\
        \mathcal C(L,K) \wedge_R Hk \rar \drar & F_{\varSigma^\infty_+ \varOmega Q \wedge R}(\sY(L),\sY(K)) \wedge_R Hk \dar \\
        {} & F_{\varSigma^\infty_+ \varOmega Q \wedge Hk}(\sY(L) \wedge_R Hk, \sY(K) \wedge_R Hk)
    \end{tikzcd}
\]
By assumption, the bottom arrow is an equivalence.  We claim that the lower right vertical arrow is also an equivalence.  Indeed, recall that via the path space fibration $\varOmega Q \to \sP Q \simeq \text{pt} \to Q$, $\IS = \varSigma^\infty_+ \text{pt}$ equipped with the trivial $\varSigma^\infty_+ \varOmega Q$-module structure is a finite CW $\varSigma^\infty_+ \varOmega Q$-module \cite[Section III.2]{elmendorf1997rings}.  From this, it follows that $R$ equipped with the trivial $(\varSigma^\infty_+ \varOmega Q \wedge R)$-module structure is a finite CW $(\varSigma^\infty_+ \varOmega Q \wedge R)$-module.  We denote this module by $\un{R}$.  Notice that since $\sY(L)$ and $\sY(K)$ are $\varSigma^\infty_+ \varOmega Q$-modules and since $\varSigma^\infty_+ \varOmega Q$ is group-like, the space $F_R(\sY(L), \sY(K))$ is an $\varSigma^\infty_+ \varOmega Q$-module via the diagonal $\varSigma^\infty_+ \varOmega Q$-action.  Using duality theory of modules over $\varSigma^\infty_+ \varOmega Q \wedge R$ and $R$ (see \cite[Section III.7]{elmendorf1997rings} and \cite[Section 16.4]{may2006parametrized}), we use the notation $\sC_{L,K} \coloneqq F_R(\sY(L),\sY(K))$ and derive the equivalence
\begin{align*}
    F_{(\varSigma^\infty_+ \varOmega Q \wedge R)}(\sY(L),\sY(K)) \wedge_R Hk 
    & \simeq F_{(\varSigma^\infty_+ \varOmega Q \wedge R)}(\un{R}, \sC_{L,K}) \wedge_{R} Hk \\
    & \simeq F_{(\varSigma^\infty_+ \varOmega Q \wedge R)}\left(\un{R}, \sC_{L,K}\right) \wedge_{(\varSigma^\infty_+ \varOmega Q \wedge R)} ((\varSigma^\infty_+ \varOmega Q \wedge R) \wedge_R Hk) \\
    & \simeq F_{(\varSigma^\infty_+ \varOmega Q \wedge R)}\left(\un{R}, \sC_{L,K} \wedge_{(\varSigma^\infty_+ \varOmega Q \wedge R)} ( (\varSigma^\infty_+ \varOmega Q \wedge R) \wedge_R Hk)\right) \\
    &\simeq F_{(\varSigma^\infty_+ \varOmega Q \wedge R)}(\un{R}, \sC_{L,K} \wedge_R Hk) \\
    & \simeq F_{(\varSigma^\infty_+ \varOmega Q \wedge R)}(\un{R}, F_R(\sY(L),\sY(K)\wedge_R Hk ))\\
    & \simeq  F_{(\varSigma^\infty_+ \varOmega Q \wedge R)}(\sY(L),\sY(K) \wedge_R Hk) \\
    & \simeq F_{(\varSigma^\infty_+ \varOmega Q \wedge Hk)}(\sY(L) \wedge_R Hk,\sY(K) \wedge_R Hk).
\end{align*}
The third equivalences uses the fact that $\un{R}$ is a finite CW $(\varSigma^\infty_+ \varOmega Q \wedge R)$-module and the fifth equivalence uses that $\sY(L)$ is a finite CW $R$-module. It follows from \cref{thm:Hurewicz} that $\mathcal{F}$ induces an equivalence of $R$-modules along the top horizontal arrow.  By using this and running the same argument above with the roles of $L$ and $K$ reversed, we find that the isomorphism of $\sY(L)$ and $\sY(K)$ in $\mod{A}$ implies an isomorphism of $L$ and $K$ in $\sC$.
\end{proof}

    \section{Connective covers}\label{sec:background_connective}

In this section we record a few basic facts about connective covers of spectra.

\begin{notn}
\begin{enumerate}
    \item Let $R$ be a commutative ring spectrum.
    \item Let $N$ be an $R$-module.
\end{enumerate}
\end{notn}

\begin{defn}
    The \emph{connective cover} of $N$ is a connective spectrum $N_{\geq 0}$ equipped with a map $N_{\geq 0} \to N$ which induces an isomorphism $\pi_i(N_{\geq 0}) \xrightarrow{\cong} \pi_i(N)$ for all $i \geq 0$.
\end{defn}
\begin{lem}[{\cite[Lemma II.2.11]{may2006ring}}]
    A connective cover $N_{\geq 0} \to N$ of $N$ exists and is unique up to equivalence. Moreover, if $N$ admits the structure of a ring spectrum, then $N_{\geq 0}$ carries a unique,  up to equivalence, ring spectrum structure so that $N_{\geq 0} \to N$ is a map of ring spectra.
    \qed
\end{lem}
\begin{rem}\label{rem:conn_cov_cnstr}
    Following the construction described in \cite[Proposition VII.3.2]{may2006ring}, we note that the map $\varOmega^\infty N \to N$  defines a connective cover of $N$.
\end{rem}
\begin{lem}\label{lem:conn_smash_conn_is_id}
Let $M$ be a $q$-connective $R_{\geq 0}$-module, for some $q \in \mathbb Z$.  The natural map $M_{\geq 0} \to (M \wedge_{R_{\geq 0}} R)_{\geq 0}$ is an equivalence of $R_{\geq 0}$-modules.
\end{lem}

\begin{proof}
First, it suffices to prove that the induced map $\pi_n(M) \to \pi_n((M \wedge_{R_{\geq 0}} R)_{\geq 0})$ is an isomorphism for $n \geq 0$.  Second, since $M$ is a $q$-connective $R_{\geq 0}$-module, it is weakly equivalent to a $R_{\geq 0}$-module of the form $\varSigma^q M'$ where $M'$ is a CW $R_{\geq 0}$-module (see \cite[Section III.2]{elmendorf1997rings}).  Using an inductive argument with the long exact sequence on (stable) homotopy groups and the five lemma, it suffices to prove the result for the case of $M = R_{\geq 0}$; however, this case is immediate.
\end{proof}

\begin{lem}[{\cite[Lemma II.2.10]{may2006ring}}]\label{lem:obstruction_lifting_conn_covers}
Let $M$ be a connective $R_{\geq 0}$-module.  The map $N_{\geq 0}\to N$ induces an equivalence $[M,N_{\geq 0}]_{R_{\geq 0}} \xrightarrow{\simeq} [M,N]_{R_{\geq 0}}$.
\qed
\end{lem}

\bibliographystyle{alpha}
\bibliography{nearlag}

\end{document}